%% file: Third_big.tex
\documentclass[a4paper,10pt]{report}
\usepackage{mathtext}
\usepackage[T1,T2A]{fontenc} 
\usepackage[cp1251]{inputenc}
\usepackage[english]{babel}
\usepackage{amsmath} 
\usepackage{amsfonts}
\usepackage{amssymb}
\usepackage{mathrsfs}
\usepackage{amsthm}
\usepackage{titling}
\usepackage{arcs}
\usepackage{graphicx}
\usepackage{wrapfig}
\usepackage{multicol}
\usepackage[nottoc]{tocbibind}

\usepackage{makeidx}
\usepackage{xcolor}
\usepackage{euscript}
\usepackage{imakeidx}

\DeclareMathOperator{\sign}{sign}
\DeclareMathOperator{\supp}{supp}
\DeclareMathOperator{\interior}{int}

\DeclareMathOperator{\indconv}{ind\,conv}
\DeclareMathOperator{\cl}{clos}

\DeclareMathOperator{\dist}{dist}

\newcommand{\B}{\boldsymbol{B}}

\newcommand{\df}{\buildrel\mathrm{def}\over=}
\newcommand{\bell}{\boldsymbol{b}}
\newcommand{\Bell}{\boldsymbol{B}_{\eps}}

\newcommand{\BMO}{\mathrm{BMO}}
\newcommand{\eps}{\varepsilon}
\newcommand{\epsmax}{\eps_\mathrm{max}}

\newcommand{\av}[2]{\langle {#1}\rangle_{{}_{#2}}}
\newcommand{\Ch}{\Omega_{\mathrm{ch}}}

\newcommand{\Rt}{\Omega_{\scriptscriptstyle\mathrm{R}}}
\newcommand{\Lt}{\Omega_{\scriptscriptstyle\mathrm{L}}}
\newcommand{\Ang}{\Omega_{\mathrm{ang}}}

\newcommand{\Troll}{\Omega_{\mathrm{tr}}}
\newcommand{\RTroll}{\Omega_{\mathrm{tr},{\scriptscriptstyle\mathrm{R}}}}
\newcommand{\LTroll}{\Omega_{\mathrm{tr},{\scriptscriptstyle\mathrm{L}}}}

\newcommand{\DR}{D_{\scriptscriptstyle\mathrm{R}}}
\newcommand{\DL}{D_{\scriptscriptstyle\mathrm{L}}}

\newcommand{\CR}{C_{\scriptscriptstyle\mathrm{R}}}
\newcommand{\CL}{C_{\scriptscriptstyle\mathrm{L}}}

\newcommand{\hreff}[1]{\hyperref[#1]{\ref{#1}}}
\newcommand{\Bird}{\Omega_{\mathrm{bird}}}

\newcommand{\FF}{\mathcal F} %Force function
\newcommand{\Fr}{F_{\scriptscriptstyle\mathrm{R}}}
\newcommand{\Fl}{F_{\scriptscriptstyle\mathrm{L}}}
\newcommand{\FFr}{\mathfrak{F}_{\scriptscriptstyle\mathrm{R}}}
\newcommand{\FFl}{\mathfrak{F}_{\scriptscriptstyle\mathrm{L}}}

\renewcommand{\le}{\leqslant}

\renewcommand{\leq}{\leqslant}
\renewcommand{\geq}{\geqslant}
\newcommand{\BG}{\mathfrak{B}}

\newcommand{\MTTR}{\Omega_{\mathrm{Mtr,{\scriptscriptstyle R}}}}
\newcommand{\MTTL}{\Omega_{\mathrm{Mtr,{\scriptscriptstyle L}}}}
\newcommand{\MTC}{\Omega_{\mathrm{Mcup}}}
\newcommand{\MTB}{\Omega_{\mathrm{Mbird}}}
\newcommand{\ClMTC}{\Omega_{\mathrm{ClMcup}}}
\newcommand{\ato}{a^{\mathrm{top}}}
\newcommand{\bto}{b^{\mathrm{top}}}
\newcommand{\abot}{a^{\mathrm{bot}}}
\newcommand{\bbot}{b^{\mathrm{bot}}}
\newcommand{\tr}{t^{\scriptscriptstyle\mathrm{R}}}
\newcommand{\tl}{t^{\scriptscriptstyle\mathrm{L}}}
\newcommand{\FixedBoundary}{\partial_{\mathrm{fixed}}}
\newcommand{\FreeBoundary}{\partial_{\mathrm{free}}}
\newcommand{\GammaFixed}{\Gamma^{\mathrm{fixed}}}
\newcommand{\GammaFree}{\Gamma^{\mathrm{free}}}
\newcommand{\uur}{u_{\mathrm{r}}}
\newcommand{\uul}{u_{\mathrm{l}}}
\newcommand{\E}{\mathbb{E}}

\newcommand{\kappachord}{\kappa_{\mathrm{chord}}}

\newcommand{\vf}{\varphi}
\newcommand{\Class}{\boldsymbol{A}}
\newcommand{\dfi}{\partial_{\mathrm{fixed}}}
\newcommand{\dfree}{\partial_{\mathrm{free}}}
\newcommand{\R}{\mathbb{R}}
\newcommand{\ff}{\mathrm{f}} % ýòî îáîçíà÷åíèå äëÿ ãðàíè÷íûõ äàííûõ íà æåñòêîé ãðàíèöå
\newcommand{\scalp}[2]{\left \langle #1, #2 \right\rangle}
\newcommand{\const}{\mathrm{Const}}

\newcommand{\vt}{\vartheta}
\newcommand{\vtl}{\vt^{^\mathrm{L}}}

\newcommand{\eq}[2]{\begin{equation}\label{#1}#2\end{equation}}
\newcommand{\eqstar}[1]{\begin{equation*}{#1}\end{equation*}}

\newcommand{\tors}{\mathfrak{K}} %ýòî îáîçíà÷åíèå äëÿ îòíîøåíèÿ \kappa_3'/\kappa_2'
\newcommand{\torsion}{\mathbf{T}} % ýòî îáîçíà÷åíèå äëÿ îïðåäåëèòåëÿ \gamma',\gamma'',\gamma'''
\newcommand{\geps}{\mathbf{g}} % ýòî îáîçíà÷åíèå äëÿ ñåìåéñòâà êðèâûõ

\newcommand{\wl}{w_{_{\mathrm{L}}}} %ýòî îáîçíà÷åíèå äëÿ òî÷êè êàñàíèÿ ëåâîé êàñàòåëüíîé
\renewcommand{\wr}{w_{_\mathrm{R}}} %ýòî îáîçíà÷åíèå äëÿ òî÷êè êàñàíèÿ ïðàâîé êàñàòåëüíîé

\newcommand{\tts}{s}
\newcommand{\ttl}{s_{_\mathrm{L}}} %ýòî îáîçíà÷åíèå äëÿ ïàðàìåòðà òî÷êè êàñàíèÿ ëåâîé êàñàòåëüíîé
\newcommand{\ttr}{s_{_\mathrm{R}}} %ýòî îáîçíà÷åíèå äëÿ ïàðàìåòðà òî÷êè êàñàíèÿ ïðàâîé êàñàòåëüíîé

\newcommand{\Sl}{S_{_\mathrm{L}}} %ýòî îáîçíà÷åíèå äëÿ ëåâîé êàñàòåëüíîé
\newcommand{\Sr}{S_{_\mathrm{R}}} %ýòî îáîçíà÷åíèå äëÿ ïðàâîé êàñàòåëüíîé
\newcommand{\laml}{\lambda_{_\mathrm{L}}} %ýòî îáîçíà÷åíèå äëÿ ëåâîé ëÿìáäû
\newcommand{\lamr}{\lambda_{_\mathrm{R}}} %ýòî îáîçíà÷åíèå äëÿ ïðàâîé ëÿìáäû

\newcommand{\psir}{\psi_{_\mathrm{R}}}
\newcommand{\psil}{\psi_{_\mathrm{L}}}

\newcommand{\kappal}{\kappa_{_\mathrm{L}}}
\newcommand{\kappar}{\kappa_{_\mathrm{R}}}

\newcommand{\set}[2]{\{{#1}\mid{#2}\}}
\newcommand{\Set}[2]{\Big\{{#1}\,\Big|\;{#2}\Big\}}

\theoremstyle{plain}
\theoremstyle{plain}\newtheorem{Le}{Lemma}[section]
\theoremstyle{definition}\newtheorem{Def}[Le]{Definition}
\theoremstyle{plain}\newtheorem{St}[Le]{Proposition}
\theoremstyle{plain}
\theoremstyle{plain}\newtheorem{Th}[Le]{Theorem}
\theoremstyle{plain}\newtheorem{Cor}[Le]{Corollary}
\theoremstyle{remark}\newtheorem{Rem}[Le]{Remark}
\theoremstyle{plain}
\theoremstyle{plain}\newtheorem{Cond}[Le]{Condition}
\numberwithin{section}{chapter}
\numberwithin{equation}{section}
\numberwithin{figure}{chapter}

\renewcommand{\chaptermark}[1]{}

\newcommand{\dett}[2]{\det\begin{pmatrix} #1\\ #2 \end{pmatrix}}
\newcommand{\dettt}[3]{\det\begin{pmatrix} #1\\ #2\\ #3 \end{pmatrix}}

\makeindex

\makeindex[name=symbol,title={Index of symbols}]
\begin{document}
\title{Bellman functions on simple non-convex domains in the plane}
\author{Paata~Ivanisvili \and Dmitriy~Stolyarov \and Vasily~Vasyunin \and Pavel~Zatitskii}
\date{}
\maketitle

\begin{abstract}
%In \cite{ISVZ2018} the authors built the Bellman function for integral functionals on the $\BMO$ space. The present paper 
%provides a development of the subject. To find the Bellman function for the  functionals on the $\BMO$ space the 
%corresponding foliation of the parabolic strip was constructed. Here we construct the required foliation for much 
%larger set of domains being deference of two convex domains. Thus, we provide the construction for integral functionals 
%on various functional spaces. The elaborated \emph{evolutional} approach allows us to treat the problem in its natural 
%setting. What is more, it lighten dynamical aspects of the Bellman function.

The present paper provides a generalization of the previous authors' work on Bellman functions for integral functionals on~$\BMO$. Those Bellman functions are the minimal locally concave functions on parabolic strips in the plane. Now we describe the algorithm for constructing minimal locally concave functions
on a planar domain that is a difference of two unbounded convex domains. This leads to many sharp estimates for functions in the classes like~$\BMO$,~$A_p$, or the Gehring classes. 
\let\thefootnote\relax\footnote{\bf The research is supported by RSF grant 19-71-10023}
\end{abstract}

%\bigskip
%
%%We are grateful to Dmitriy~Chelkak, Pavel~Galashin, Alexander~Logunov, Nikolay~Osipov, Fedor~Petrov, Leonid~Slavin, Alexander~Volberg, and Vladimir~Zolotov for sharing their ideas with us.
%
%\bigskip
%
%\bigskip
%
%\bigskip
%
%\bigskip
%
%\bigskip
%
%\bigskip
%
%\bigskip
%
%\bigskip
%
%\bigskip
%
%\bigskip
%
%\bigskip
%
%\bigskip
%
%\bigskip
%
%\bigskip
%
%\bigskip
%
%\bigskip
%
%\bigskip
%
%\bigskip
%
%\bigskip
%
%\bigskip
%
%\bigskip
%
%\bigskip
%
%\bigskip
%
%\bigskip
%
%\bigskip
%
%\bigskip
%
%\bigskip
%
%\bigskip
%
%
%
%Paata Ivanisvili
%
%
%University of California, Irvine, ???, USA
%
%
%ivanishvili.paata@gmail.com
%
%\bigskip
%
%Dmitriy M. Stolyarov
%
%
%Chebyshev Laboratory, St. Petersburg State University, 14th Line, 29b, St. Petersburg, 199178, Russia
%
%St. Petersburg Department of Steklov Mathematical Institute, Russian Academy of Sciences, 27 Fontanka, St. Petersburg, 191023, Russia 
%
%
%dms@pdmi.ras.ru
%
%http://www.chebyshev.spb.ru/DmitriyStolyarov
%
%\bigskip
%
%Vasily I. Vasyunin
%
%St. Petersburg Department of Steklov Mathematical Institute, Russian Academy of Sciences, 27 Fontanka, St. Petersburg, 191023, Russia 
%
%vasyunin@pdmi.ras.ru
%
%\bigskip
%
%Pavel B. Zatitskiy
%
%Chebyshev Laboratory, St. Petersburg State University, 14th Line, 29b, St. Petersburg, 199178, Russia
%
%St. Petersburg Department of Steklov Mathematical Institute, Russian Academy of Sciences, 27 Fontanka, St. Petersburg, 191023, Russia 
%
%pavelz@pdmi.ras.ru
%
%%http://www.chebyshev.spb.ru/pavelzatitskiy
%
%\newpage
\tableofcontents

\input{1Chapter_text.tex}

\input{2Chapter_text.tex}
\input{3Chapter_text.tex} 
\input{4Chapter_text.tex}
\input{5Chapter_text.tex}

\printindex[symbol]
\printindex
%\index{z@\addcontentsline{toc}{chapter}{Index}\quad\quad\quad\quad\quad\quad\quad\quad\quad\quad\quad\quad\quad\quad\quad\quad\quad\quad\quad\quad\quad\quad\quad\quad\quad\quad\quad\quad\quad\quad\quad\quad\quad\quad\quad\quad\quad\quad\quad\quad\quad\quad\quad\quad\quad\quad\quad\quad\quad\quad\quad\quad\quad\quad\quad\quad\quad\quad\quad\quad\quad\quad\quad\quad|phantom}
\bibliography{mybib}{}
\bibliographystyle{amsplain}

\medskip

\begin{multicols}{2}

Paata Ivanisvili

University of California, Irvine,

pivanisv@uci.edu;

\medskip

Dmitriy Stolyarov

St. Petersburg State University,

d.m.stolyarov@spbu.ru;

\medskip

Vasily Vasyunin

St. Petersburg State University,

vasyunin@pdmi.ras.ru;

\medskip

Pavel Zatitskii

University Cincinnati,

St. Petersburg State University,

pavelz@pdmi.ras.ru.

\end{multicols}
\end{document}

%% file: 1Chapter_text.tex
\chapter{Introduction}
The Bellman functions and their probabilistic counterparts, the Burkholder functions, play an important role in modern harmonic analysis and probability theory. In fact, these are the two names for one object. The foundations were laid in the groundbreaking papers~\cite{Burkholder1984} and~\cite{NT1997} (see~\cite{NTV2001} as well); we refer the reader to the monographs~\cite{Osekowski2012} and~\cite{VV2021} for a general description of the field; see shorter introductions in~\cite{Osekowski2013} and~\cite{SV2015}.  

The field may be, roughly speaking, divided into two parts. The first part uses Bellman functions as tools for proving inequalities without concern on their sharpness. Here, one usually has to come up with a sufficiently good \emph{Bellman supersolution}. The second part aims to provide sharp solutions to Bellman optimization problems. This significantly narrows the area of application since the exact Bellman function is often extremely difficult to describe. However, the result usually comes not only with the sharp constant for the inequality in question, but also with good understanding of the optimal functions. 

The present volume provides the proofs of the results announced in~\cite{IOSVZ2015_bis}. We will concentrate on a special class of optimization problems that are designed to study summability properties of functions in classes like~$\BMO$ or~$A_p$. We refer the reader to~\cite{IOSVZ2015_bis} and Subsections~\ref{s111} and~\ref{s112} below for a long list of specific already solved problems falling under the scope of our considerations. In a sense, our main target is to provide a theory for solving optimization problems of a very special kind; this encircles only a tiny part of inequalities amenable to the Bellman function method. However, the constructions and ideas arising in this small subfield are universal and are used in other more involved Bellman problems; see Subsections~\ref{s113},~\ref{s114}, and~\ref{S12} below. In the forthcoming section, we briefly motivate our study, describe the field, and outline the main examples. We also show how many parts of the present work serve as useful tools for related problems and explain the relationship with the classical moment method and themes in differential geometry. Section~\ref{S13} contains the plan of paper. In Section~\ref{S14}, we discuss some related questions that are not covered by the present volume.

\section{Motivation, examples, and links to classical topics}\label{S11}
Let~$\Xi_0$ be a strictly convex unbounded open proper subset of~$\R^2$. Let~$\Xi_1$ be another strictly convex unbounded open set whose closure lies in~$\Xi_0$ entirely; we set~$\Omega = \cl\Xi_0 \setminus \Xi_1$. Let~$I$ be an interval in~$\R$. Consider the class~$\Class_\Omega(I)$ of~$\R^2$-valued summable functions~$\varphi\colon I \to \partial \Xi_0$ such that the point~$\av{\varphi}{J}$ does not belong to~$\Xi_1$ for any subinterval~$J \subset I$; here and in what follows, the symbol~$\av{\psi}{E}$ denotes the average of a locally summable function~$\psi$ over a set~$E$ of non-zero measure:
\eqstar{
\av{\psi}{E} = \frac{1}{|E|} \int_E \psi(t)\,dt,\quad |E|\ \text{stands for the Lebesgue measure of}\ E\subset \R.
} 
We assume that both boundaries of~$\Xi_0$ and~$\Xi_1$ may be parametrized by the first coordinate. In other words, they are graphs of convex functions in the standard coordinate system.

Let~$\ff\colon \partial \Xi_0\to \R$ be a function. Assume~$\ff$ is measurable and does not grow too fast at infinity. We address the question of finding sharp estimates for the quantity~$\av{\ff(\varphi)}{I}$ under the conditions~$\varphi \in \Class_\Omega(I)$ and~$\av{\varphi}{I} = x$. This setting naturally leads to the introduction of the Bellman function
\eq{BellmanFunction}{
\B_\Omega(x;\ff) = \sup\Set{\av{\ff(\varphi)}{I}}{\av{\varphi}{I} = x, \quad \varphi \in \Class_{\Omega}(I)},\qquad x\in \Omega.
}
Using dilation invariance of the average, one may prove that the definition of~$\B$ does not depend on the particular choice of~$I$. There are two main classes of examples: those related to~$\BMO$ spaces and to the Muckenhoupt classes~$A_p$.

\subsection{$\BMO$ case}\label{s111}
The case where
\eqstar{
\Omega = \Set{x\in \R^2}{x_1^2 \leq x_2 \leq x_1^2 + 1},
}
corresponds to optimizing an integral functional~$\av{f(\varphi)}{I}$ on the ball of unit radius of the~$\BMO$ space. Here we have chosen
\eqstar{
\Xi_0 = \Set{x\in \R^2}{x_1^2 < x_2}\quad \text{and}\quad \Xi_1 = \Set{x\in \R^2}{x_1^2 + 1 < x_2}.
}
Note that if~$\varphi \colon I \to \partial \Xi_0$ belongs to~$\Class_{\Omega}(I)$, then its first coordinate,~$\varphi_1\colon I \to \R$, satisfies the inequality
\eqstar{
\av{\varphi_1^2}{J} - (\av{\varphi_1}{J})^2 \leq 1 
}
for any subinterval~$J\subset I$. This may be rewritten as
\eqstar{
\av{\big|\varphi_1 - \av{\varphi_1}{J}\big|^2}{J} \leq 1,\qquad \forall \ J \ \text{subinterval}\ I.
}
This is equivalent to the inequality~$\|\varphi_1\|_{\BMO} \leq 1$, provided we equip the space~$\BMO(I)$ with the quadratic seminorm. The general~$\BMO_p$ seminorm is 
\eqstar{
\|\psi\|_{\BMO_p(I)} = \sup\Set{\Big(\av{|\psi - \av{\psi}{J}|^p}{J}\Big)^\frac{1}{p}}{J\ \text{is a subinterval of}\ I}.
}
For our considerations, the main norm is~$\BMO_2$, so we omit the subscript~$2$ in this case. By the John--Nirenberg inequality, these norms define one and the same Banach space of functions. We refer the reader to~\cite{Stein1993} for the general~$\BMO$ theory.

The~$\BMO$ setting suggests it might be useful to consider intermediate domains
\eq{ParabolicStrip}{
\Omega_\eps = \Set{x\in \R^2}{x_1^2 \leq x_2 \leq x_1^2 + \eps^2},
}
which is the set theoretical difference of~$\Xi_0$ and
\eqstar{
\Xi_\eps = \Set{x\in \R^2}{x_1^2 + \eps^2 < x_2}.
}
The class~$\Class_{\Omega_\eps}(I)$ generated by~$\Omega_\eps$ corresponds to the~$\eps$-ball of~$\BMO$. From now on, we fix some~$\eps > 0$.

The choice~$\ff(t,t^2) = e^t$ and the analytic expression
\eq{ExponentialFormula}{
\B_{\Omega_\eps}(x; e^t) = \frac{1-\sqrt{x_1^2 - x_2 + \eps^2}}{1-\eps}e^{x_1 - \eps + \sqrt{x_1^2 - x_2 + \eps^2}},\qquad x\in \eqref{ParabolicStrip},\ \eps < 1,
}
for the Bellman function~\eqref{BellmanFunction} found in~\cite{SV2011}, lead to sharp integral forms of the John--Nirenberg inequality. This formula explains what we mean by the words `compute the Bellman function'. While~$\B$ is completely defined by~\eqref{BellmanFunction}, one can extract almost no information from this definition. By computation of a Bellman function we mean finding a good analytical expression for it. Usually our expressions are not as brief as~\eqref{ExponentialFormula}. They are composed of integrations, differentiations, and also include specific functions. What is more, usually we define functions by cases, and the splitting into cases is defined in terms of inequalities for solutions to certain analytic equations. Though sometimes such a description is implicit, we always prove that the solutions exist and are unique. 

The choice~$\ff(t,t^2)= |t|^p$ leads to sharp inequalities between~$\BMO_p$ norms; this case was considered in~\cite{SV2012}. The sharp constants in the classical forms of the John--Nirenberg inequalities were found in~\cite{VV2014}. The paper~\cite{Osekowski2015} derives sharp embeddings of~$\BMO$ into Lorentz spaces from the Bellman functions in~\cite{VV2014} (see~\cite{KO2021} for similar estimates for~$\BMO$ martingales) and~\cite{Osekowski2015_bis} uses the case where~$f(t) = \ff(t,t^2)$ is piecewise affine to derive the sharp embedding of~$\BMO$ into the weak~$L_\infty$-space. The paper~\cite{Osekowski2021} employs Bellman functions of the described type to find sharp constants in the embeddings of~$\BMO$ into weighted~$\BMO$. The papers~\cite{SVZ2020} and~\cite{VZZ2022} find the sharp constants in the multiplicative inequalities~$\|\varphi\|_{L_q} \leq c_{p,q}\|\varphi\|_{L_p}^{p/q}\|\varphi\|_{\BMO}^{1-p/q}$. The reasoning relies on a three-dimensional Bellman function. In the first paper, it is almost immediately constructed from two-dimensional functions from~\cite{SV2012}. The function of the second paper is quite complicated. For the sharp symmetric form of the John--Nirenberg inequality, see the recent preprint~\cite{Dobronravov2023}. The theory for the case of a generic~$\ff$ was created in~\cite{ISVZ2018} (see the short report~\cite{IOSVZ2012}, a simpler case study in~\cite{IOSVZ2015}, and an example of application of that theory in~\cite{Vasyunin2014}).

We note that the Bellman function technique described above is perfectly suited for working with~$\BMO_2$, however, it meets significant difficulties when applied to~$\BMO_p$ norms with general~$p$. All the papers cited above work with~$\BMO_2$ (the only application of the Bellman function technique to~$\BMO_p$ is given in~\cite{Slavin2015} and~\cite{SV2016}). For results about sharp constants in John--Nirenberg inequalities on~$\BMO_1$, see~\cite{Korenovski1990},~\cite{Lerner2013}, and~\cite{Slavin2015} (the first two papers do not use Bellman function techniques). 

The function~$\B_{\Omega_\eps}(\,\cdot\,;f)$ for different values of~$\eps$ are related to each other. The paper~\cite{ISVZ2018} suggested to compute them all simultaneously. To be more precise, first, the function~$\B$ is computed for sufficiently small~$\eps$. Then, one continuously increases~$\eps$ and tracks the evolution of~$\B$. This provides an algorithm for construction of~$\B$ for any~$\eps$. In this paper, we will follow a similar route. This justifies the need for a continuous family of convex domains~$\Xi_\eps$ that connect~$\Xi_0$ with~$\Xi_1$.

\subsection{The~$A_{p_1,p_2}$ classes}\label{s112}
Let~$p_1$ and~$p_2$,~$p_1 > p_2$, be real numbers and let~$Q > 1$. Consider the domain
\eq{Ap1p2}{
\Omega_{p_1,p_2,Q} = \Set{x\in \R^2}{x_1,x_2 > 0\ \text{and}\ x_2^{1/{p_2}} \leq x_1^{1/{p_1}} \leq Q x_2^{1/{p_2}}}. 
}
This domain corresponds to the so-called~$A_{p_1,p_2}$ classes introduced in~\cite{Vasyunin2008}. A weight~$w\colon I \to \R_+$ belongs to~$A_{p_1,p_2}$ if the quantity
\eqstar{
[w]_{p_1,p_2} = \sup\Set{\av{w^{p_1}}{J}^{\frac{1}{p_1}}\av{w^{p_2}}{J}^{-\frac{1}{p_2}}}{J\ \text{is a subinterval of}\ I}
}
is finite. One may see that~$[w]_{p_1,p_2}\leq Q$ if and only if~$\varphi \in \Class_{\Omega_{p_1,p_2,Q}}$, where~$\varphi = (w^{p_1}, w^{p_2})$.

The paper~\cite{Vasyunin2008} generalizes~\cite{Vasyunin2003}, where the first sharp Bellman functions for Muckenhoupt weights were found; those Bellman functions lead to sharp constants in the Reverse H\"older inequality. The scale~$A_{p_1,p_2}$ includes the classical Muckenhoupt classes~$A_p$. More precisely,~$A_{p,-\frac{1}{p-1}}$ coincides with~$A_p$ when~$p\in (1,\infty)$. The case~$p_1=1$ corresponds to the so-called Gehring classes. The paper~\cite{DW2009} is devoted to this particular case. The sharp constants for the embedding of a Gehring class into~$A_p$ are found. That paper uses the Bellman function generated by~$\ff$ of the form~$\ff(x) = x_1^q$ for some~$q$. The corresponding sharp weak-type Reverse H\"older inequalities for~$A_{p_1,p_2}$-weights were established in~\cite{Reznikov2013}; the latter paper uses the Bellman function for~$\ff(x) = \chi_{[1,\infty)}(x_1)$. See~\cite{Rey2016} for the limiting case of the Muckenhoupt class~$A_1$.

We note that the definition~\eqref{Ap1p2} extends naturally to the limiting cases where some of the parameters~$p_1$ and~$p_2$ is equal to~$1$ and~$\infty$. The case of~$A_\infty$ equipped with the so-called Hruschev's `norm' 
\eqstar{
[w]_{\infty} =  \sup\Set{\av{w}{J}\exp(-\av{\log w}{J})}{J\ \text{is a subinterval of}\ I}
}introduced in~\cite{Hruschev1984}, is of particular importance. This corresponds to the domain
\eq{AinftyDomain}{
\Omega_{1,\infty,Q} = \Set{x\in \R^2}{e^{x_1} \leq x_2 \leq Qe^{x_1}}, 
}
where~$Q > 1$, meaning~$[w]_{\infty} \leq Q$ whenever~$\varphi = (\log w, w)$ belongs to~$\Class_{\Omega_{1,\infty,Q}}$. This domain and the study of the corresponding Bellman functions go back to~\cite{Vasyunin2003}. See~\cite{BR2014} for the study of this particular case and applications of the obtained inequalities. The paper~\cite{Osekowski2019} delivers sharp forms of the principle that the logarithm of an~$A_\infty$ Muckenhoupt weight belongs to~$\BMO$; the reasoning relies on a Bellman function on the domain~\eqref{AinftyDomain} with~$\ff(x) = x_1^2$. In~\cite{Slavin2015} and~\cite{SV2016}, the Bellman functions for~$\ff(x) = |x_1|^p$ are computed and used to establish sharp forms of the John--Nirenberg inequalities for~$\BMO_p$ with~$p\ne 2$.

\subsection{Monotonic rearrangements and geometry of the Bellman function}\label{s113}
One may wonder why does a complicated optimization problem~\eqref{BellmanFunction} have a transparent solution. The heuristic reason might be described in several ways. First, the notion of a monotonic rearrangement appears useful in this context. For a function~$\xi\colon I \to \R$ define its non-increasing rearrangement~$\xi^*\colon [0,|I|] \to \R$ by the formula
\eqstar{
\xi^*(t) = \inf\Set{\lambda}{|\set{s\in I}{\xi(s) > \lambda}| \leq t}.
}
Note that~$\av{f(\xi)}{I} = \av{f(\xi^*)}{[0,|I|]}$ for any reasonable function~$f$. It was proved in~\cite{Klemes1985} that the monotonic rearrangement operator does not increase the~$\BMO_1$ norm\footnote{The said paper also contains hints to the proof that the same is true for arbitrary~$p\in [1,\infty)$.}. A similar principle also holds for the~$A_p$ constant:~$[w^*]_p \leq [w]_p$ (see~\cite{BSW1992} for the case~$p=1$ and~\cite{Korenovski1992} for  arbitrary~$p$). Recall we have assumed~$\partial\Xi_0$ may be parametrized with the first coordinate. A function~$\varphi\colon I \to \partial \Xi_0$ is called non-increasing, provided its first coordinate~$\varphi_1$ does not increase. The notion of monotonic rearrangement of a function~$\varphi \colon I \to \partial \Xi_0$ is defined accordingly: this is a non-increasing function~$\varphi^*\colon [0,|I|]\to \partial \Xi_0$ that has the same distribution. The paper~\cite{SZ2016} suggests a proof of a similar principle in the generality of~$\Class_\Omega$:~$\varphi^* \in\Class_\Omega([0,|I|])$, provided~$\varphi \in \Class_\Omega(I)$. The reasoning is based on Bellman ideas. See~\cite{SVZ2015} for dyadic classes and~\cite{SZ2023} for generalization to Campanato-type norms on~$\mathrm{VMO}$. Thus, since the monotonic rearrangement of a function in~$\Class_\Omega$ also belongs to~$\Class_\Omega$, we may restrict the class of functions~$\varphi$ in~\eqref{BellmanFunction} to non-increasing ones. This often simplifies the investigation of the function~$\B$ and gives hope for a closed formula for~$\B$.

Another reason for a good formula for~$\B$ is its geometric description proved in~\cite{SZ2016} (see~\cite{SZ2022} for a generalization): this function is the pointwise minimal among all locally concave functions on~$\Omega$ that satisfy the boundary condition~$\B(x) = \ff(x)$,~$x\in \partial \Xi_0$. The main idea of~\cite{SZ2016} is that the Bellman function~\eqref{BellmanFunction} has a probabilistic representation in terms of certain discrete martingales; it resembles the representation formula for harmonic functions in terms of the Brownian motion. See, e.\,g.,~\cite{Krylov1989}, for much more general representation formulas for solutions of a Hamilton--Jacobi--Bellman equation; the difference between~\cite{SZ2016} and~\cite{Krylov1989} is that the latter paper uses continuous time martingales. The development of these ideas allowed to transfer all the results about Bellman problems for functions on the interval to the line and the circle, see~\cite{SZ2021}. In particular, the sharp constants in various forms of the John--Nirenberg inequalities and Reverse H\"older inequalities for Muckenhoupt weights are the same for functions on the interval, the circle, and the line. We note that the question of finding any reasonable sharp forms of the said inequalities for functions of several variables is widely open (see~\cite{CSS2011} for questions about dimensional dependence of constants). Some dimensional estimates are obtained with the help of monotone rearrangement estimates (see~\cite{BDG2022}) and with the semigroup approach (see~\cite{SZ2016}).

\subsection{Miscellaneous direct applications}\label{s114}
We will describe two recent applications. The first one deals with sharp estimates of distribution of martingales whose square function is bounded. Namely, it appears that the Bellman function~\eqref{BellmanFunction} for the~$\BMO$ case (i.\,e., defined on~\eqref{ParabolicStrip} with arbitrary~$\ff$) delivers good bounds for the following optimization problem: maximize~$\E f(\psi)$ provided~$\psi$ is the limit value of a martingale with the square function uniformly bounded by~$1$ and prescribed mean. It appears that one may treat~$\psi$ as a function whose~$\BMO([0,1])$ norm is bounded by one. The paper~\cite{SVZZ2022} (see also the short report~\cite{SVZZ2019}) describes the class of~$f$ for which this principle leads to the exact solution of the initial problem; for other~$f$, the said principle provides a fine supersolution.

The second application in~\cite{SlZ2021} says that the Bellman function~\eqref{BellmanFunction} delivers dimension-free bounds for the classes of functions like~$\BMO(\R^d)$ or~$A_p(\R^d)$ provided one replaces averages over balls in the definitions of these spaces and Bellman functions with averages over specific semigroup kernels and uses the associated Garsia-type norms. The main feature of these semigroup averages is that they are of martingale nature, which allows to apply martingale representations from~\cite{SZ2016}. It is unclear whether the resulting estimates are sharp.

\subsection{Relationship with the moment method and differential geometry}\label{S12}
Assume for a while that~$\Xi_1$ is an empty set. Let~$(t,g(t))$ be a parametrization of~$\partial\Xi_0$. Then, we arrive at the following optimization problem:
\eqstar{
\int_0^1\varphi_1 = x_1;\quad \int_0^1 g(\varphi_1) = x_2;\qquad \int_0^1f(\varphi_1) \longrightarrow \max.
}
Here we have set~$I = [0,1]$. This is a simple example of the moment problem. We refer the reader to~\cite{KS1966},~\cite{Kemperman1973}, and~\cite{KN1977} for the description of the moment method. The solution to the problem is described geometrically. Let~$H$ be the convex hull of the three-dimensional curve~$\gamma(t) = (t,g(t),f(t))$. Then, the desired maximum equals
\eqstar{
\sup\Set{s\in \R}{(x_1,x_2,s) \in H}.
}
Similar to~\eqref{BellmanFunction}, we will denote this quantity by~$\bell(x)$. Assume for a while that~$\Xi_0$ is bounded. Using the Carath\'eodory theorem on convex hulls, we see that there exists an optimal~$\varphi_1$ that attains at most~$3$ values. One may establish a similar principle for unbounded domains. As a result of these considerations, we see that the graph of~$\bell$ is somehow flat: each point on this graph is a convex combination of at most~$3$ points of the boundary curve~$\gamma$. The most common case is that~$(x,\bell(x))$ is a convex combination of two points of the boundary curve,~$A$ and~$B$. Then, we call the segment~$[A,B]$ (as well as its projection onto the~$x_1x_2$-plane) that connects them a \emph{chord}. In differential geometry, the tangent plane to the graph of~$\bell$ at~$(x,\bell(x))$ is called a \emph{bitangent plane}, because it is tangent to the boundary curve at both points~$A$ and~$B$. For us, the condition on the points~$A$ and~$B$ for existence of the bitangent plane is expressed in the \emph{cup equation} (see equation~\eqref{urlun1} below). The cup equation described in the present paper has been already used by the authors in several more complicated Bellman problems, see~\cite{HSVZ2018},~\cite{Ivanisvili2015},~\cite{SVZ2020},~\cite{VZZ2022}, and~\cite{ISZ2015}.

A solution to the  classical moment method as described in the books above works  with the case where the boundary curve is somehow regular (say, it has positive torsion), this is related to the notion of a Chebyshev system. According to Chapters III and IV in \cite{KN1977} the boundary of the convex hull of a space curve $(\gamma_{1}(t), \ldots, \gamma_{n+1}(t))$ given on an interval $[a,b]$, where $\gamma_{j}(t)$ are continuous functions,  can be described in terms of {\em upper and lower principal representations} provided that the systems $(1, \gamma_{1}(t), \ldots, \gamma_{n}(t))$ and $(1, \gamma_{1}(t), \ldots, \gamma_{n+1}(t))$ are $T_{+}$-system on $[a,b]$. Recall that the system of continuous functions $(\gamma_{0}(t), \ldots, \gamma_{m}(t))$ is called $T_{+}$-system if 
\begin{align}\label{chebplus}
\mathrm{det}(\{ \gamma_{i}(t_{j})\}_{i,j=0}^{m}) >0 \quad \text{on} \quad \Sigma = \{ a\leq t_{0}<\ldots<t_{m}\leq b\}.
\end{align}
Verification of the condition (\ref{chebplus}) on the set $\Sigma$ may be cumbersome. If the map $t \mapsto (\gamma_{0}(t), \ldots, \gamma_{m}(t))$ is of class $C^{m}((a,b))$ then (Chapter VIII, \cite{KS1966}) the following easier condition 
\begin{align}\label{markovplus}
\mathrm{det}(\{\gamma_{i}^{(j)}(t)\}_{i,j=0}^{k})>0 \quad \text{on} \quad (a,b)
\end{align}
 for all $k=1,\ldots, m$, implies (\ref{chebplus}). The system $(\gamma_{0}(t), \ldots, \gamma_{m}(t))$  satisfying (\ref{markovplus}) is called $M_{+}$-system. Clearly if the system $(\gamma_{0}(t), \ldots, \gamma_{m}(t))$ is $M_{+}$-system then any system $(\gamma_{0}(t), \ldots, \gamma_{k}(t))$ is also $M_{+}$-system for any $k$, $0\leq k \leq m$. Notice that the system $(1, t, g(t), f(t))$ is $M_{+}$-system if and only if $g''(t)>0$  and $g''(t)f'''(t)-f''(t)g'''(t)>0$ on $(a,b)$.

 If~$\gamma$ does not satisfy any regularity assumptions of type (\ref{chebplus}), the combinatorial structure of the boundary of the convex hull can be extremely complicated. We refer the reader to~\cite{Sedykh1977} and~\cite{Zakalyukin1977} for descriptions of singularities of convex hulls of generic curves. See also~\cite{Ghomi2017} and~\cite{Sedykh1994} for relationship with the four vertex theorems. 

The minimal concave and locally concave functions may be treated as concave (and locally concave) solutions of the degenerate homogeneous Monge--Amp\`ere equation (see Subsection~\ref{s223} below). See~\cite{CNS1986} and~\cite{Guan1998} for regularity results. Note that the latter paper studies minimal locally concave functions on non-convex domains from a different point of view. In that paper,  the boundary conditions are imposed on the whole boundary. We impose the boundary conditions on the convex (or \emph{fixed}, see~\cite{SZ2016}) part of the boundary. This comes naturally from the Bellman function setting. The difference between the two settings is huge.

\section{Description of exposition}\label{S13}
The main body of the paper is divided into four chapters. The general line of reasoning follows~\cite{ISVZ2018}. 

Chapter~\ref{C2} contains a more formal introduction. In Section~\ref{s21}, we state the problem, list the requirements on the domains and the function~$f$, define the Bellman function, and describe its simple properties. By simple properties we mean those that follow from the definition directly. We also provide some examples that explain the meaning of various conditions imposed on~$\Omega$ and~$f$. Section~\ref{s22} surveys the duality theory from~\cite{SZ2016}, introduces the strategy of proof, in particular, the notion of optimizers, and ends with formal description of our results. We suggest consulting Chapter~\ref{C2} for definitions while reading the further plan of the paper or return to this section after reading Chapter~\ref{C2}. 

We describe the local possible foliations in Chapter~\ref{C3}. In Section~\ref{s31}, we present a general description of what a foliation may look like. Section~\ref{s31bis} introduces a \emph{fence} --- a foliation that consists of line segments attached to the fixed boundary.  A fence may consist of chords, or tangents, or both of them. We study the conditions for the function constructed from this foliation to be locally concave. Section~\ref{s32} contains the description of tangent domains and the corresponding standard candidates. In Section~\ref{s33}, we characterize chordal domains. Though the reasoning here follows the lines of~\cite{ISVZ2018}, some non-trivial modifications are needed to make calculations simpler. Section~\ref{s33old} is devoted to the construction of a cup. Here the construction significantly differs from the one in~\cite{ISVZ2018}, it is quite lengthy and involved; Section~\ref{s33old} is split into several subsections. We introduce~\emph{forces} in Section~\ref{s34bis}. The forces are certain quantities that express the concavity of the constructed candidate. The main feature of these quantities is that they are negative and decrease as~$\eps$ increases. The construction of the Bellman function relies heavily on these monotonicity formulas for the forces. We state and prove them in Section~\ref{s34bis}. We also introduce the notion of the~\emph{tail} of a force. We warn the reader that we have slightly changed the point of view concerning forces: now the domain of a force coincides with its tails by definition. In~\cite{ISVZ2018}, the forces were defined on rays or the entire line and could attain both positive and negative values. In the present paper, they are always negative. Section~\ref{s34} contains the classification of linearity domains and Section~\ref{s35} describes combinatorial properties of foliations. The material of the latter two sections is similar to the material of the corresponding sections in~\cite{ISVZ2018}.

Chapter~\ref{C4} describes the evolution of Bellman candidates and, in particular, contains the proof that a Bellman candidate exists for any admissible~$f$ and~$\eps$. In Section~\ref{s41}, we describe the case of simple picture, when the foliation consists of tangents, angles, cups, and, possibly, simple multicups. We also prove that, given an admissible~$f$, there exists~$\eps_1$ such that for any~$\eps < \eps_1$, there exists a simple Bellman candidate for~$f$ and~$\eps$. Section~\ref{s42} contains the local monotonicity formulas for forces, the proofs of the existence of roots of balance equations, and other lemmas that will be used in further proofs. In Section~\ref{s43}, we describe the local evolutional scenarios. Roughly speaking, the cups and full multicups grow, the trolleybuses shrink, the multibirdies and multitrolleybuses desintegrate. We state and prove a rigorous proposition for each of these principles. Section~\ref{s44} contains the combinatorial reasoning that glues together local evolutional scenarios. Since this combinatorial reasoning is identical to the one in~\cite{ISVZ2018}, we omit it, providing a detailed citation. 

Chapter~\ref{C5} provides the theory of optimizers. In particular, it is proved that all the Bellman candidates constructed in Chapter~\ref{C4} coincide with the corresponding Bellman functions. The opening Section~\ref{s51} contains geometric description of optimizers: it is useful to represent them as special curves in~$\Omega$ called delivery curves (they deliver an optimizer to a given point in~$\Omega$). Section~\ref{s52} considers all local foliations and constructs optimizers for each of them. Then, these ``local'' optimizers are glued together in Section~\ref{s53}. The concluding Section~\ref{s54} considers various situations when the summability conditions on~$f$ are violated. We study the question of finiteness of the Bellman function in this setting and provide simple if and only if conditions on~$f$ assuming this function satisfies the regularity condition that the torsion of~$\gamma$ changes sign a finite number of times. Sometimes, the Bellman function is finite when the summability assumptions are violated. In this case, the optimizers do not exist. However, we provide an optimizing sequence for each point.

\begin{figure}[h!]
\includegraphics[width=1\textwidth]{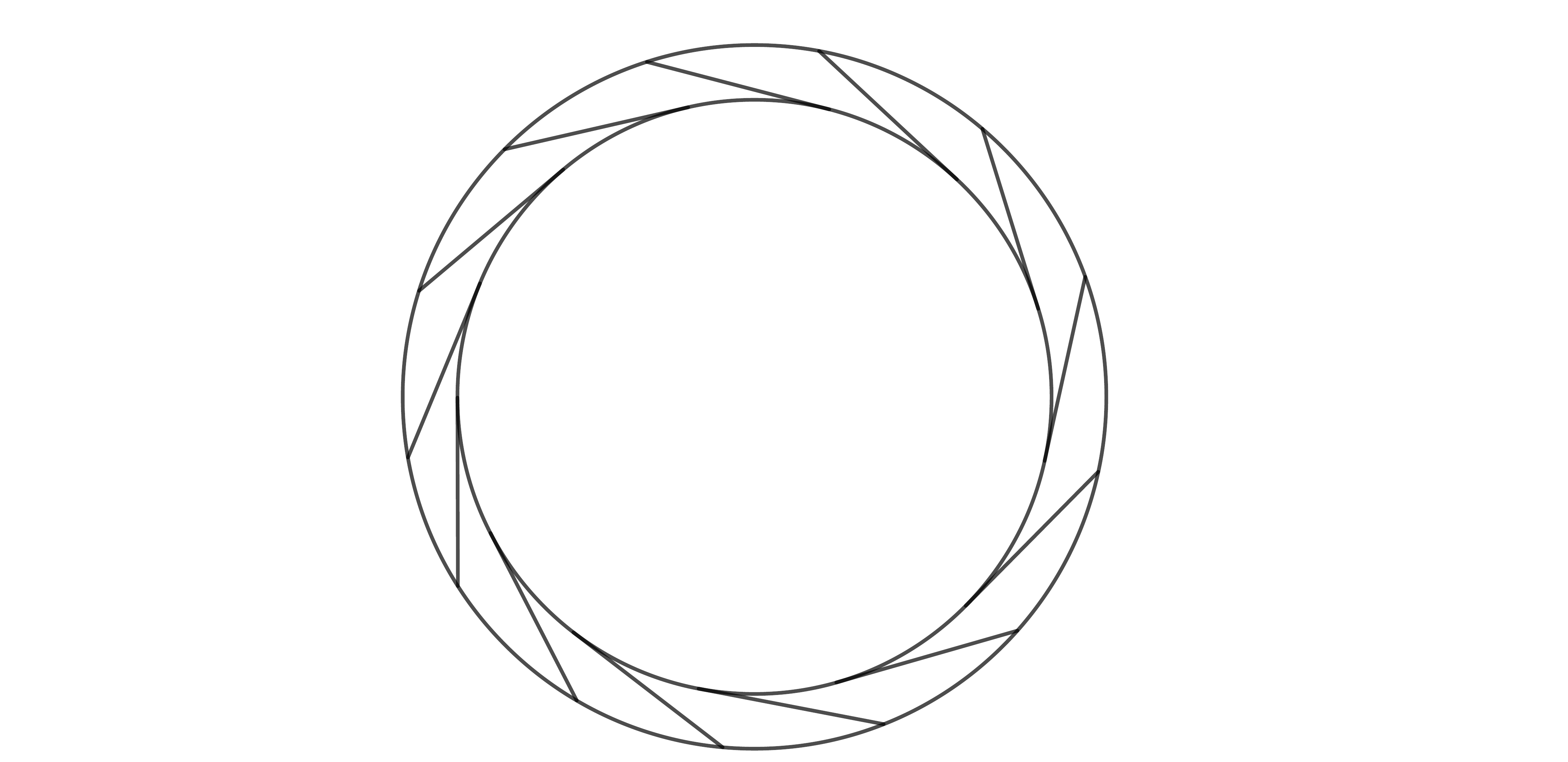}
\caption{Impossible foliation.}
\label{fig:imp}
\end{figure}

\section{Related questions}\label{S14}
\subsection{Case of bounded~$\Omega$}\label{s141}
We may consider the case where~$\Omega_0$ and~$\Omega_1$ are both bounded domains. Say, the case where
\eqstar{
\Omega_0 = \Set{x\in \R^2}{x_1^2 + x_2^2 < 1},\qquad \Omega_1 = \Set{x\in \R^2}{x_1^2 + x_2^2 < 1-\delta}
}
naturally corresponds to the~$\delta$-`ball' of the class of~$\BMO$ mappings of~$[0,1]$ into the unit circle~$S^1$. See the third example in Section~$2$ of~\cite{SZ2022}. While it seems that all the constructions of the present paper are applicable in this situation, there are two subtleties that distinguish the case where~$\Omega_0$ is bounded.

First, the boundaries of~$\Omega_0$ and~$\Omega_1$ are not graphs of functions anymore. Therefore, to adjust the methods of the current paper to this situation, we need to find coordinate-free geometric versions of our reasonings. For some of them, this is not difficult. Say, a torsion is a geometric notion that does not depend on the choice of the coordinate system. Similar, forces can be expressed geometrically as exterior products of certain vectors. However, we are not able to find a geometric interpretation for most of the material of Section~\ref{s33old}.

Second, some topological effects related to the four vertex theorem~(see~\cite{Ghomi2017}) come into play. That theorem says that a regular curve that bounds a convex surface changes the sign of its torsion at least four times. Based on the experience of the present paper, it is natural to expect that the points where the torsion changes sign from~$+$ to~$-$ are origins of cups (at least for the case where~$\delta$ is sufficiently small) and there are vertices of angles in neighborhoods of the points where the torsion changes sign from~$-$ to~$+$. In particular, the foliation of the type presented on Figure~\ref{fig:imp} is impossible.

\subsection{Higher dimensional cases}\label{s142}
One may consider the case where~$\Omega_0$ and~$\Omega_1$ are subsets of~$\R^d$ and state a similar Bellman function problem for them. For example, such a setting appears if one wishes to work with vector-valued~$\BMO$ functions (see the second example in~\cite{SZ2022}). As it was proved in~\cite{SZ2022}, the assertion that the Bellman function is the minimal locally concave function holds true in this generality. Its geometry for arbitrary boundary conditions might be involved and difficult to describe. However, if one assumes that the domains~$\Omega_0$ and~$\Omega_1$ posses rotational symmetry, and~$f$ is also rotationally symmetric, then, maybe, the computation of the Bellman function may be reduced to the computation of certain functions of the types described in the present paper. In particular, the problems for vector-valued~$\BMO$ functions fit into this framework. This has an analogy in probabilistic setting: many Burkholder problems in~\cite{Osekowski2012} allow generalizations to martingales attaining values in a Hilbert space.

%% file: 2Chapter_text.tex
\chapter{Setting and sketch of proof}\label{C2}

\section{Setting}\label{s21}

{%\color{blue}
\subsection{Domains}

\label{cond}

Suppose that we have a family of $C^3$-smooth curves $\geps(\,\cdot\,; \eps)\colon \mathbb{R} \to \mathbb{R}^2$, 
\hbox{$\eps\in[0,\epsmax)$,}\index[symbol]{$\epsmax$} let $\geps(t;\eps) = (\geps_1(t;\eps),\geps_2(t;\eps))$, $t \in \mathbb{R}$. In particular, we assume $\frac{\partial}{\partial t}\geps(t;\eps) \ne 0$.
We assume that $\geps(t,\eps)$ is $C^{3,1}$-smooth with respect to $(t,\eps)$ meaning { the mixed derivative $\frac{\partial^4}{\partial t^3 \partial \eps} \geps$ is continuous}. For any $\eps$ fixed we will write 
$\geps(t) = (\geps_1(t), \geps_2(t))$ omitting $\eps$ and considering it as a function of one variable. 
By $g(t)$ we will denote $\geps(t;0)$. Suppose that:

\eq{eq250702}{
\geps_1'>0;
}
\eq{eq230402}{
\geps_2''\geps_1'-\geps_2'\geps_1''>0,
}
in particular, $\geps(\mathbb{R})$ is a graph of a strictly convex function.

Let $\Xi_\eps$\index[symbol]{$\Xi_\eps$} be the strict epigraph of $\geps(\,\cdot\,;\eps)$:
$$
\Xi_\eps = \big\{(\geps_1(t;\eps),x_2) \in \mathbb{R}^2 \colon t \in \mathbb{R}, \ x_2>\geps_2(t;\eps)\big\}.
$$

We assume that $\Xi_{\eps_1} \supset \Xi_{\eps_2}$ if $\eps_1<\eps_2$. Moreover, we assume that 
$\geps(t;\eps_1) \notin \cl \Xi_{\eps_2}$ for any $t \in \mathbb{R}$ and $\eps_1<\eps_2$.

For any $\eps \in(0,\epsmax)$ we assume that for any $t \in \mathbb{R}$ there are two tangent lines from 
the point $g(t)$ to $\Xi_\eps$, and that each of these lines intersects with the curve $g(\,\cdot\,)$ 
twice. We note that this assumption slightly differs from the cone condition used in~\cite{SZ2016} and~\cite{SZ2022}: the cone condition 
imposed on the domain by those two papers is weaker than the present one. 
Let $\wl(t;\eps) = \geps(\ttl;\eps)$ and $\wr(t;\eps) = \geps(\ttr;\eps)$ be the tangency points with the curve 
$\geps(\,\cdot\,;\eps)$, here $\ttl$ and $\ttr$ are some functions of $t$ and $\eps$; we choose $\ttr<\ttl$. 
Let $\Sl(t;\eps)= [g(t), \wl(t;\eps)]$ and $\Sr(t;\eps)= [g(t), \wr(t;\eps)]$ \index[symbol]{$\wl$}\index[symbol]{$\wr$}\index[symbol]{$\Sl$}\index[symbol]{$\Sr$} be the tangent segments. 
Here $\mathrm{L}$ and $\mathrm{R}$ mean left and right with respect to tangency point correspondingly. 
We will often omit indices $\mathrm{L}$ and $\mathrm{R}$ and write simply $S(t;\eps)$ and $w(t;\eps)$ or 
even $S(t)$ and $w(t)$. 

\begin{Rem}\label{rem250701}
For $\eps$ fixed the functions $t\mapsto \ttl$ and $t\mapsto \ttr$ are $C^2-$smooth and have positive derivatives: $\ttl'>0, \ttr'>0$.
\end{Rem}
\begin{proof}
It follows from~\eqref{eq250702} that the functions $\ttl$ and $\ttr$  are increasing.

Let either $\tts = \ttl$ or $\tts = \ttr$. This number is defined by the tangency equation
\eq{eq250701}{
\big(\geps_2(t; 0) - \geps_2(\tts; \eps)\big)\geps_1'(\tts;\eps) - \big(\geps_1(t; 0) - \geps_1(\tts; \eps)\big)\geps_2'(\tts;\eps) =0.
}
Partial derivative of the left hand side with respect to $t$ is
$$
\geps_2'(t; 0)\geps_1'(\tts;\eps) - \geps_1'(t; 0)\geps_2'(\tts;\eps),
$$ 
which is nonzero, because the segment connecting $\geps(t;0)$ and $\geps(\tts;\eps)$ is transversal to the curve $\geps(\,\cdot\,;0)$. 

Partial derivative of the left hand side of~\eqref{eq250701} with respect to $\tts$ is equal to
$$
\big(\geps_2(t; 0) - \geps_2(\tts; \eps)\big)\geps_1''(\tts;\eps) - \big(\geps_1(t; 0) - \geps_1(\tts; \eps)\big)\geps_2''(\tts;\eps).
$$
This is nonzero because $\geps_1(t; 0) - \geps_1(\tts; \eps) \ne 0$ and due to~\eqref{eq250701} and \eqref{eq230402}.

The implicit function theorem guarantees that $\tts$ and $t$ are $C^2$-smooth functions of each other and that $\tts'(t)$ is finite and nonzero. Therefore, $\tts'(t)>0$.
\end{proof}

\begin{figure}[h!]
\includegraphics[width=\textwidth]{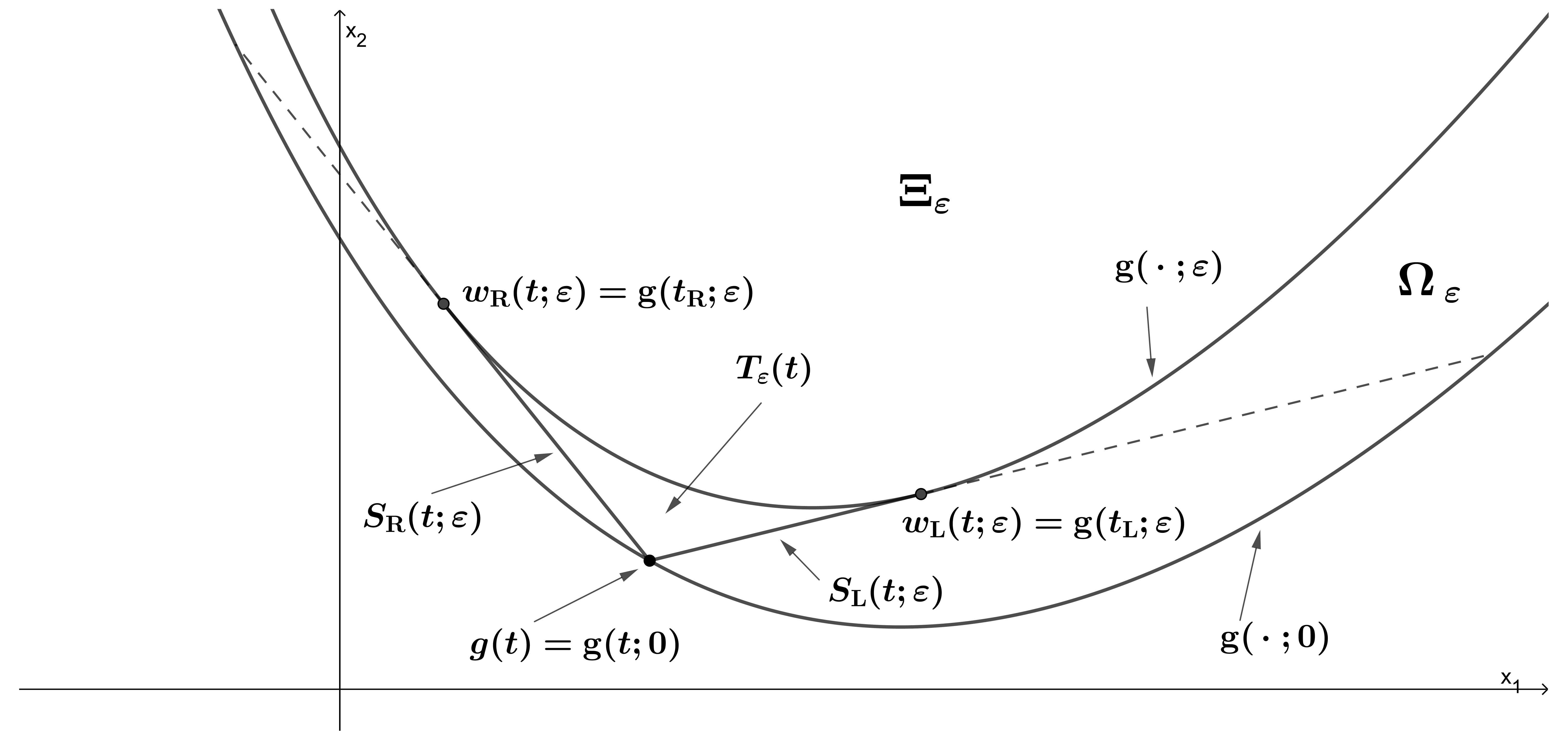}
\caption{Triangle $T_\eps(t)$.}
\label{fig:triangle}
\end{figure}

Let $T_\eps(t)$\index[symbol]{$T_\eps$} be the closed curvilinear triangle with the vertex $g(t)$ whose sides are $\Sl(t;\eps)$, 
$\Sr(t;\eps)$, and the part of the curve $\geps(\,\cdot\,;\eps)$ between the two tangency points, 
see Figure~\ref{fig:triangle}. We define a domain\index[symbol]{$\Omega_\eps$} 
\eq{eq230401}{
\Omega_\eps = \cup_{t \in \mathbb{R}} T_\eps(t).
}
One may see that~$\Omega_\eps = \cl\Xi_0 \setminus \Xi_\eps$.
We give several examples of admissible domains $\Omega_\eps$ on Figure~\ref{fig:domain_examples}, see Subsection~\ref{sec280301} for more information concerning these examples.
\begin{figure}[h!]
\includegraphics[width=0.45\textwidth]{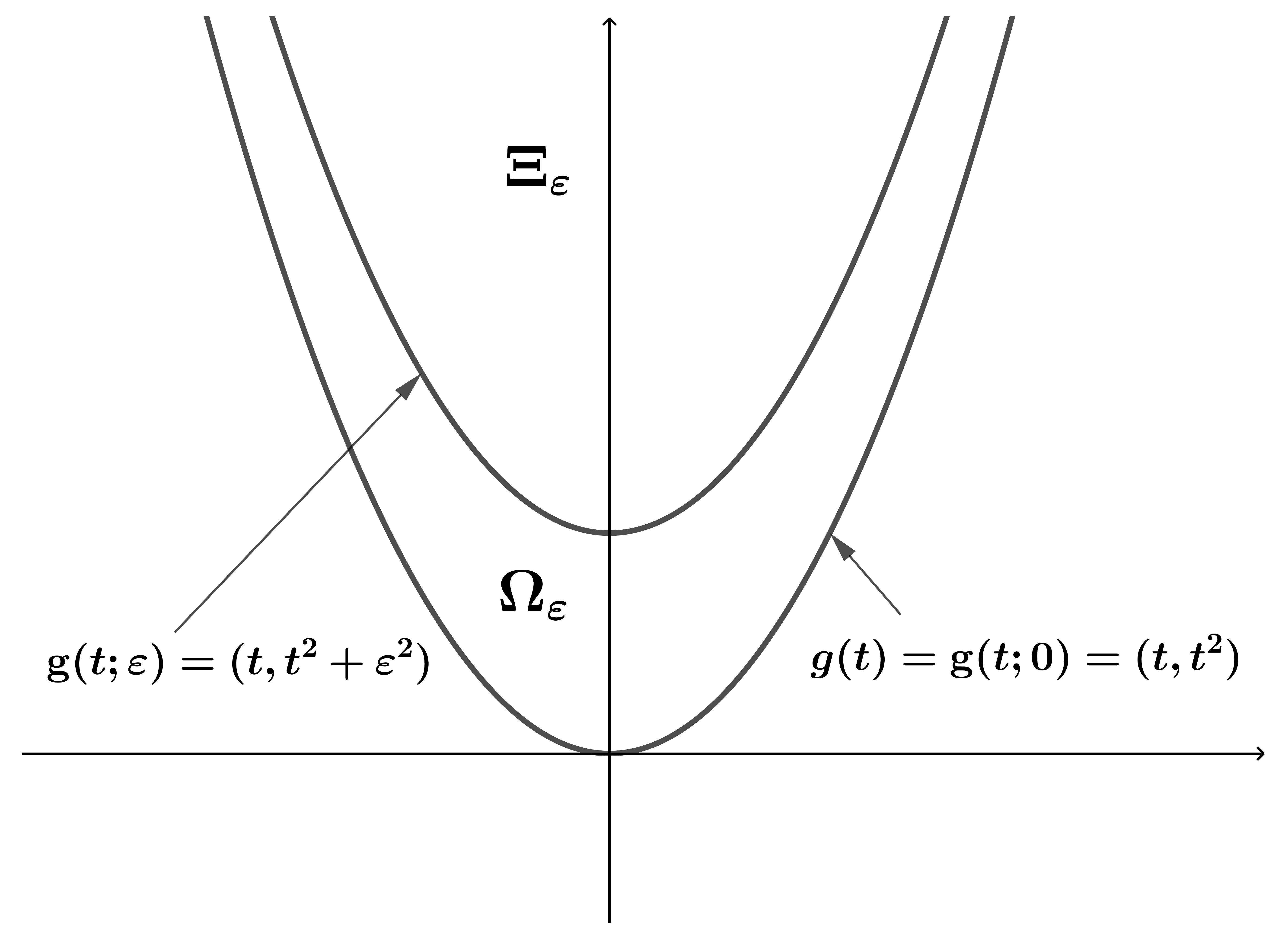}
\includegraphics[width=0.45\textwidth]{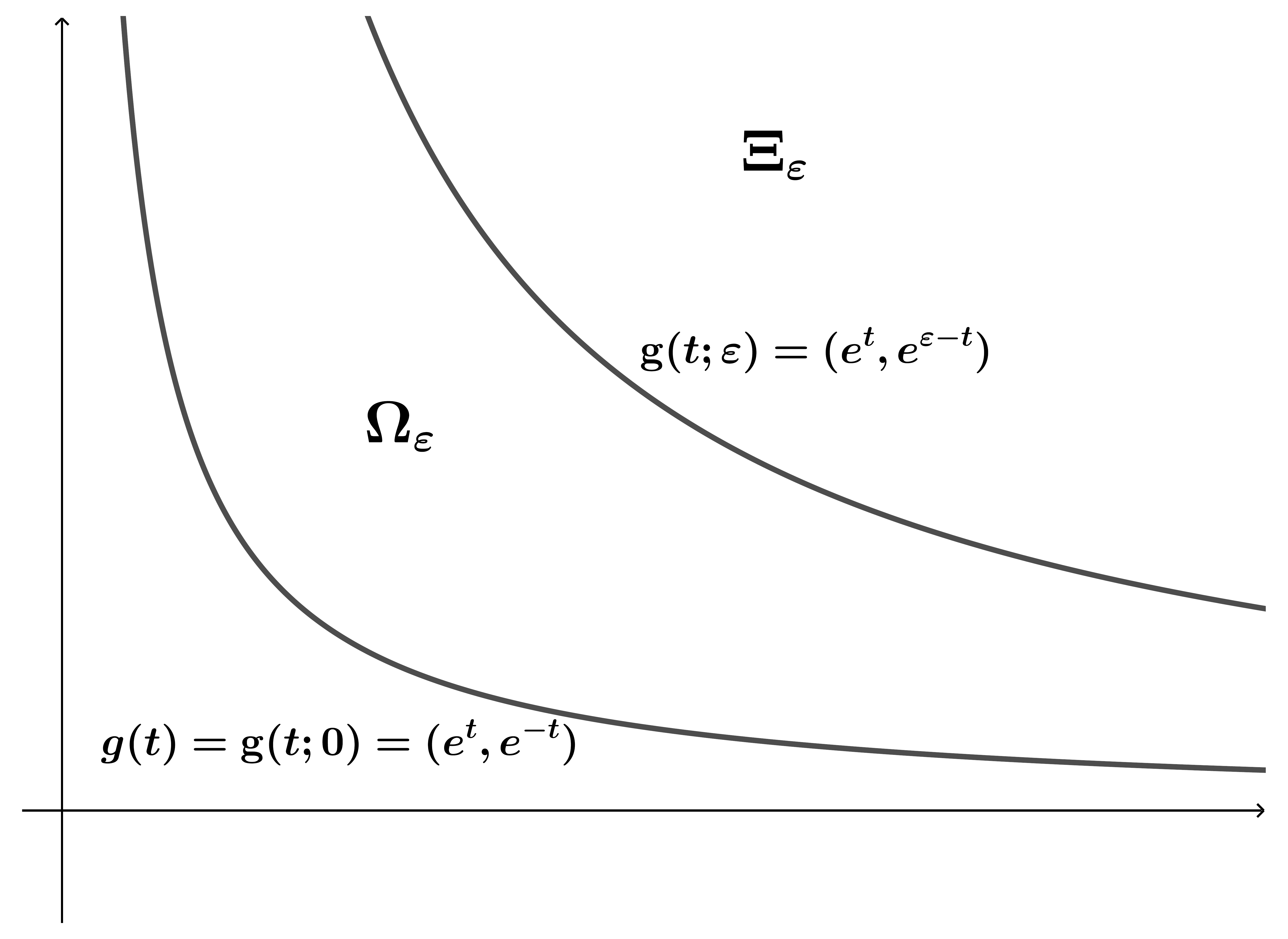}
\includegraphics[width=0.45\textwidth]{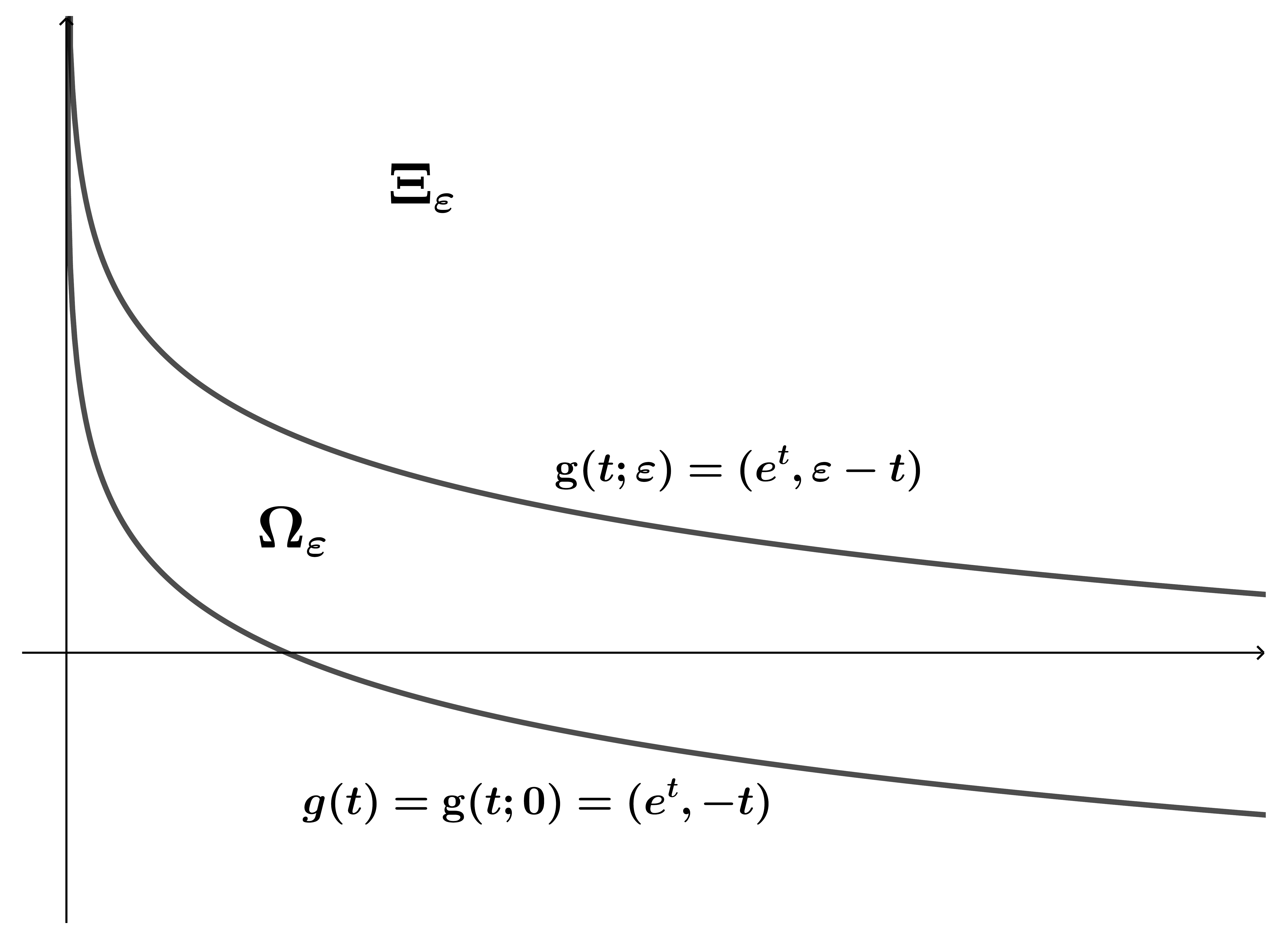}
\includegraphics[width=0.45\textwidth]{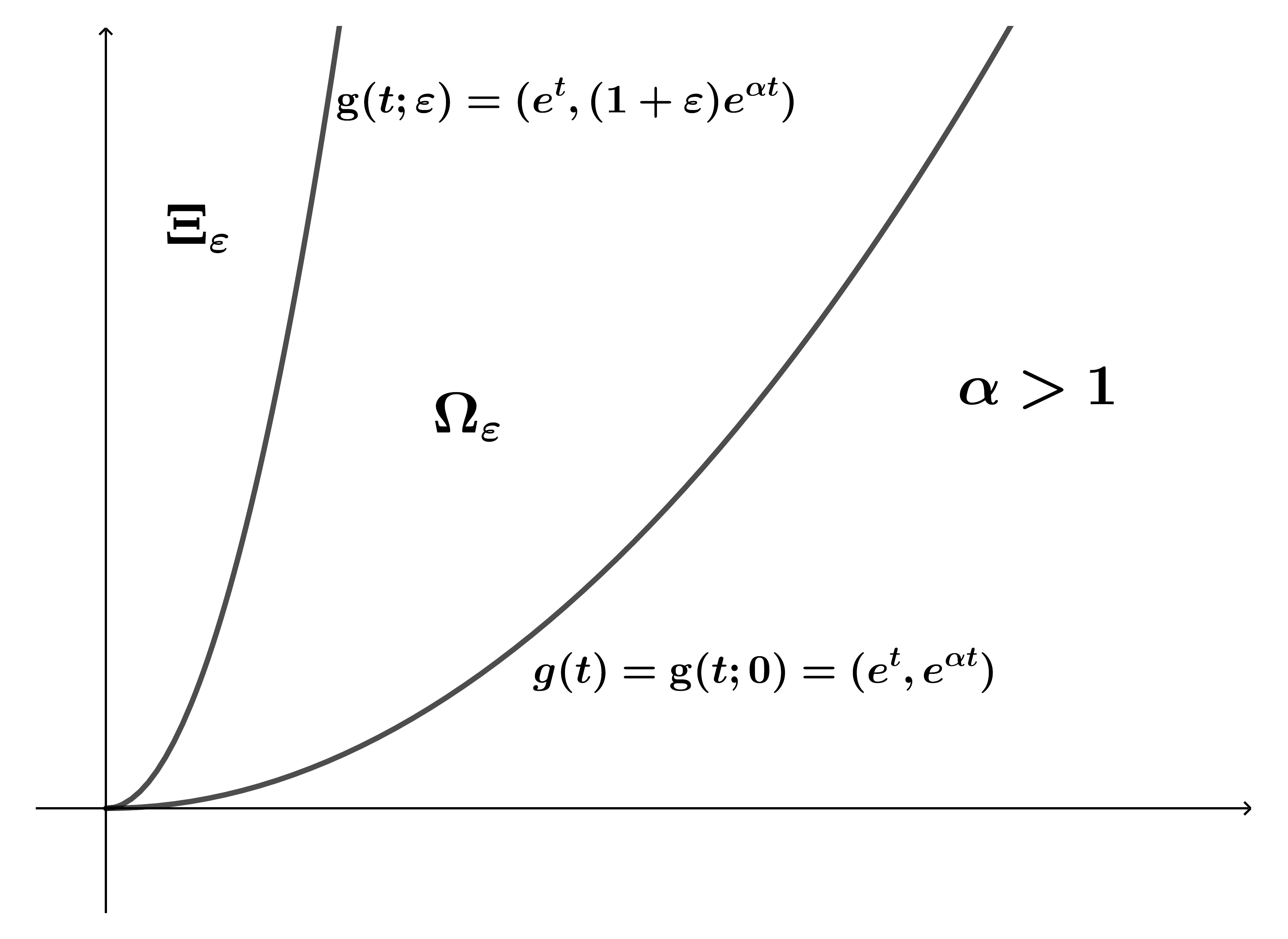}
\includegraphics[width=0.45\textwidth]{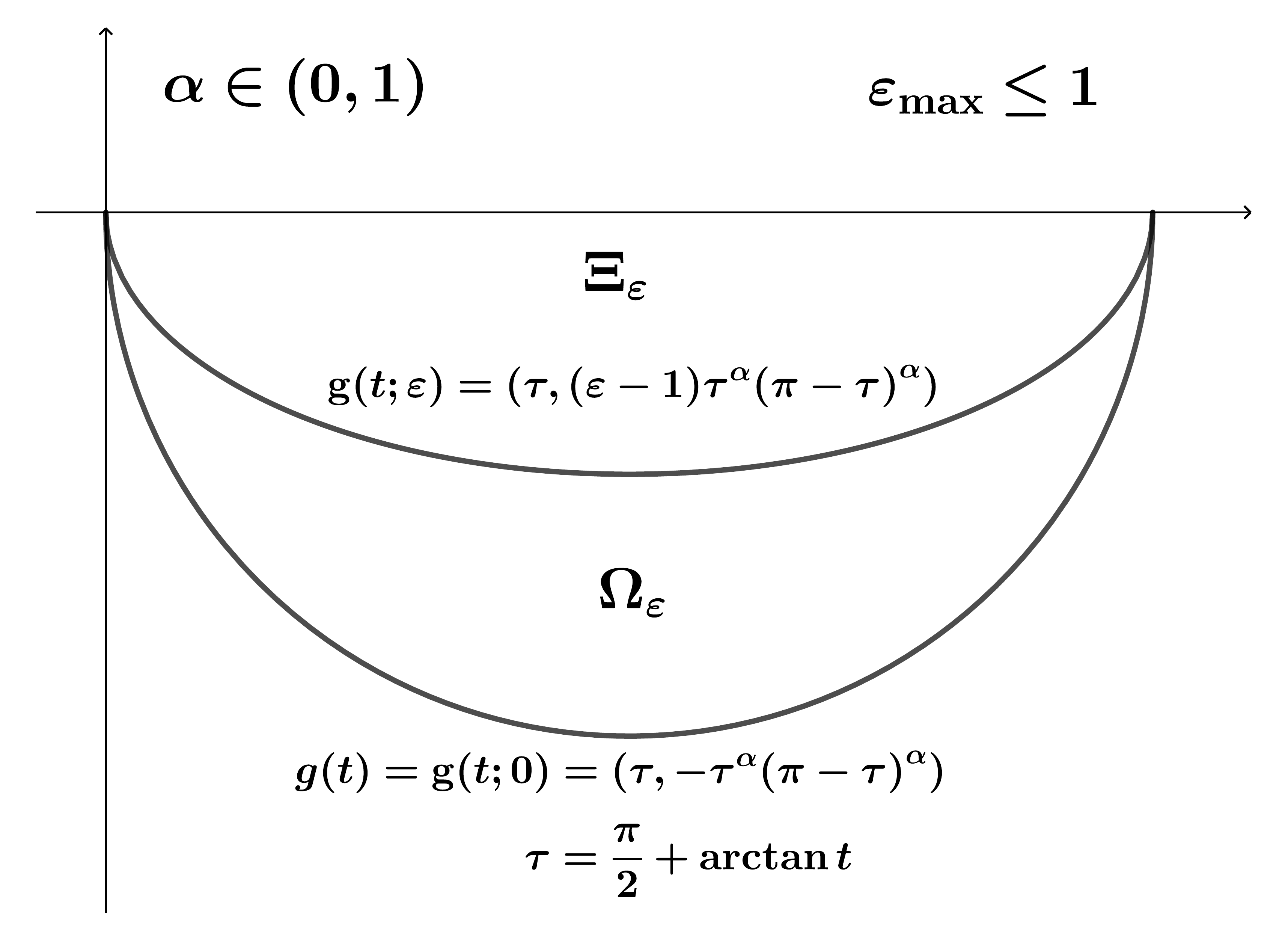}
\includegraphics[width=0.45\textwidth]{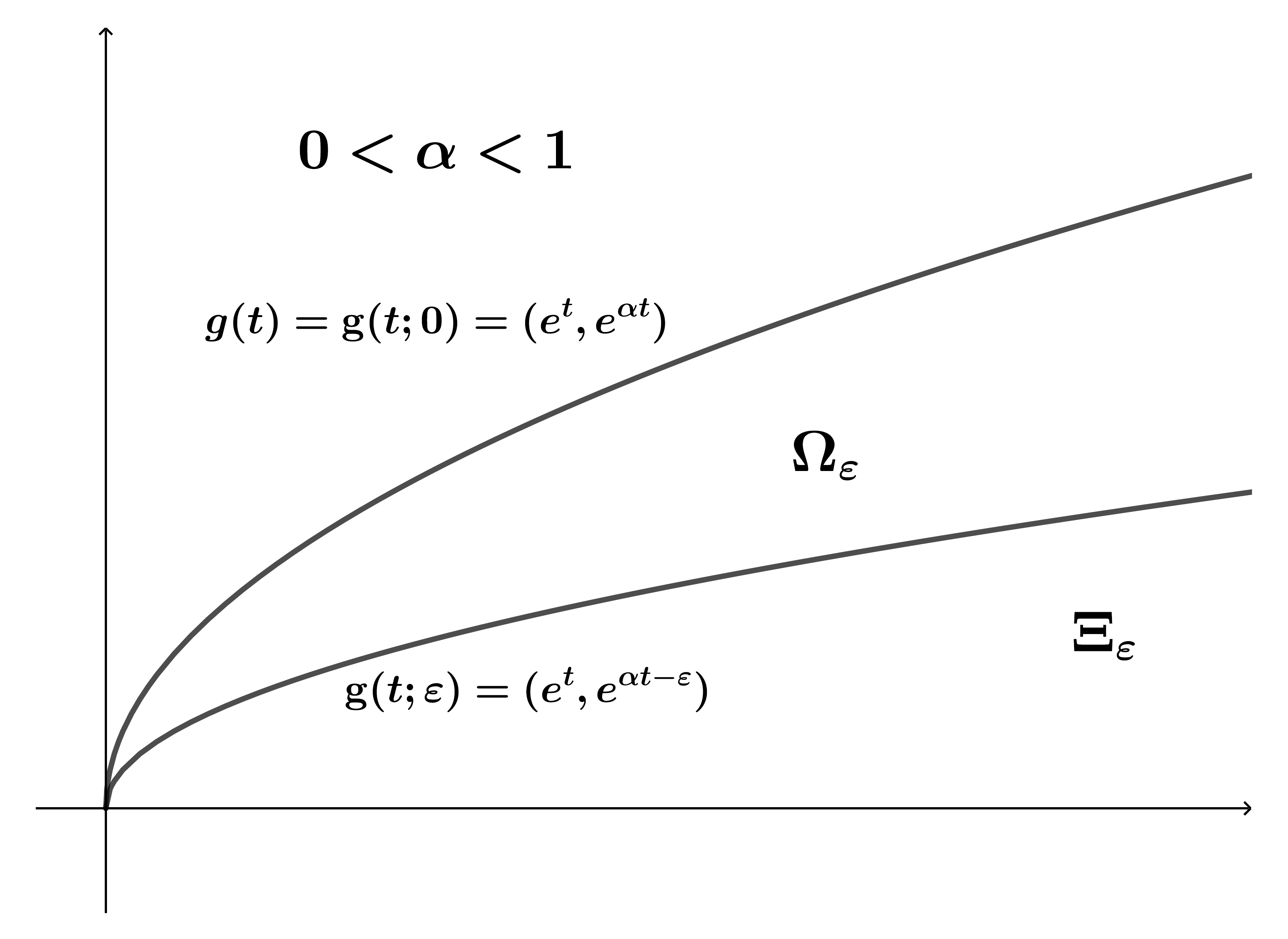}
\caption{Examples of the domains $\Omega_\eps$.}
\label{fig:domain_examples}
\end{figure}
The domain on the latter picture 
on Figure~\ref{fig:domain_examples} formally does not satisfy~\eqref{eq230402}, but after 
application of a suitable isometry of the plane, it becomes admissible for the theory. The parametrization 
by $t$ running from $-\infty$ to $+\infty$ is convenient for theoretical construction, but sometimes there 
are more natural parametrizations, as we can see on Figure~\ref{fig:domain_examples}.

We also make the following technical assumption. Let $\kappal(t;\eps)$ and $\kappar(t;\eps)$ be the slopes 
of the sides $\Sl$ and $\Sr$ of the triangle $T_\eps(t)$, i.\,e.:\index[symbol]{$\kappa$}
\eq{eq211002}{
\kappa = \frac{w_2 - g_2}{w_1 - g_1}.
}
It follows from the implicit function theorem that the functions $\ttl,\ttr,\wl,\wr,$ and $\kappa$ are 
$C^{2,1}$-smooth with respect to $(t,\eps)$ in the same sense as before. The domains $\Xi_\eps$ decrease by inclusion, therefore,
the function $\kappa$ is monotone with respect to $\eps$: $\kappal$ increases while $\kappar$ decreases. 
We assume a bit more: $\frac{\partial}{\partial \eps}\kappa(t,\eps) \ne 0$ for any $t$ and~$\eps$.

The curve $g(\mathbb{R})$ is called the {\it fixed boundary} \index{fixed boundary} of the domain 
$\Omega_\eps$ and is denoted by $\dfi\Omega_\eps$. The curve $\geps(\mathbb{R};\eps)$ is called 
the {\it free boundary} \index{free boundary}  of $\Omega_\eps$ and is denoted by $\dfree\Omega_\eps$. 
We will use the sign $\preceq$~\label{ordering} for two points on $\dfi\Omega$ to indicate their disposition: 
$g(t)\preceq g(s)$ if $t \leq s$.

Since the segment $S(t)$ is tangent to the curve $\tau \mapsto \geps(\tau;\eps)$, there exist a scalar function 
$\lambda$\index[symbol]{$\lambda$} satisfying the following vector-valued identity:
\eq{eq150701}{
\lambda(t) (g(t) - w(t)) = w'(t).
}
Here are two cases: $w = \wr$ defines $\lambda = \lamr(t;\eps)$ and $w = \wl$ defines $\lambda=\laml(t;\eps)$, 
correspondingly. Let us note that  $\lamr>0$ and $\laml<0$ pointwise, the strict inequalities follow from $w'(t) = \geps'(\tts;\eps)\tts'(t) \ne 0$ that is guarantied by~Remark~\ref{rem250701}. 
We formulate important assumptions on the domain $\Omega$ in terms of $\lambda$ and $w$: for all $\eps \in (0,\epsmax)$,
\begin{align}
\big(1+|\wr(t;\eps) - g(t)|\big)\cdot \exp\Big(-\int_t^{t_0} \lamr(\tau;\eps)\,  d\tau\Big) \to 0, \qquad 
t \to -\infty, \quad t_0 \in \mathbb{R};\label{eq111001}
\\
\big(1+|\wl(t;\eps)- g(t)|\big)\cdot \exp\Big(\int_{t_0}^t \laml(\tau;\eps)\,  d\tau\Big) \to 0, 
\qquad t \to +\infty, \quad t_0 \in \mathbb{R}. \label{eq111002}
\end{align}
These conditions are technical. If some of them fail, the general theory works, but some explanations 
need minor changes. The details are given in~Subsection~\ref{s54}.

Recall a definition from~\cite{IOSVZ2015_bis}. %(see also~\cite{SZ2016}).

\begin{Def}
Let $\Omega=\Omega_\eps$, let $I\subset\R$ be an interval and $\vf \colon I\to\dfi\Omega$ 
a summable function. We say that the function $\vf$ belongs to the class $\Class_\Omega(I)$\index[symbol]{$\Class_\Omega$} if
$\av\vf{J}\in\Omega$ for every subinterval~$J\subset I$.
\end{Def}

\subsection{Extremal problem and the Bellman function}\label{s211}

\begin{Def}
Let $\ff \colon \dfi\Omega_\eps \to \mathbb{R}$ be some measurable function. We define the Bellman function
\begin{equation}\label{Bf}
	\Bell(x;\,\ff) \;\df 
	\sup\big\{\av{\ff(\vf)}{I}  \colon \av{\vf}{I} = x,\;\vf \in \Class_{\Omega_{\eps}}(I)\big\}.
\end{equation}\index[symbol]{$\Bell$}

\end{Def}

In this subsection, we summarize some easy properties of the Bellman function defined by formula~\eqref{Bf} 
and try to explain the choice of the Bellman function in view of the extremal problem we study. We will often 
use the function $f \colon \mathbb{R}\to \mathbb{R}$ associated with $\ff$ given above:
$$
f(t) = \ff\big(g(t)\big), \quad t \in \mathbb{R}.
$$

We will also use symbols $\Bell(\,\cdot\,)$ and $\Bell(\,\cdot\, ;\,f)$ for the Bellman function~\eqref{Bf} 
depending on what we would like to emphasize. When $\eps$ and $f$ are fixed we will omit them and write simply 
$\boldsymbol{B}(x)$. Now, we formulate the easiest properties of the Bellman function  that do not need 
any conditions on $f$.

\begin{Rem}\label{Rescaling}
The Bellman function $\Bell$ does not depend on the interval~$I$ where the class $\Class_{\Omega_{\eps}}$ is defined.
\end{Rem}

Indeed, using a linear change of variables one can transform a function $\vf \in \Class_{\Omega_{\eps}}(I)$ 
into another function $\tilde{\vf} \in \Class_{\Omega_{\eps}}(\tilde{I})$ so that all the averages in 
formula~\eqref{Bf} do not change. Thus, the supremum defined by formula~\eqref{Bf} is taken over the same 
subset of the real numbers.

The next remark allows us to estimate integral functionals from below. One can consider another Bellman function,
$$
\Bell^{\min}(x;\,\ff) \;\df 
	\inf\big\{\av{\ff(\vf)}{I}  \colon \av{\vf}{I} = x,\;\vf \in \Class_{\Omega_{\eps}}(I)\big\}.
$$
Of course, the minimal Bellman function can be easily expressed in terms of the maximal one.

\begin{Rem}\label{minBell}
$\Bell^{\min}(x;\,f) = -\Bell(x;\,-f)$.
\end{Rem}

Now we begin to study the domain of $\Bell$. The following definition seems to be useful for all further reasoning.

%\begin{Def}\label{BellmanPoint}\index{Bellman! point}
%Consider a function $\vf\colon I \to \dfi \Omega_\eps$. We call the point $\av{\vf}{J}$
%the \emph{Bellman point} of the function $\vf$ on the interval $J$, $J\subset I$.
%\end{Def} 
%{\vv A nado li vvodit' termin, kotoryi my ispol'zuem vseto odin raz neskol'kimi strokami nizhe?}

The Bellman function~\eqref{Bf} is defined everywhere on $\mathbb{R}^2$. However, this function is equal 
to $-\infty$ if the supremum is taken over the empty set. 
%This happens for the points $x$ such that there are no functions $\vf \in \Class_{\Omega_{\eps}}$ satisfying $\av{\vf}{I} = x$. 
We drop such points
from the domain of the Bellman function. 
\begin{Def}\index{Bellman! domain}
By the \emph{Bellman domain} of the function $\Bell$ we mean the set of points $x \in \mathbb{R}^2$  for which there are no functions $\vf \in \Class_{\Omega_{\eps}}$ satisfying $\av{\vf}{I} = x$.
\end{Def}

\begin{St}\label{parstr}
The domain $\Omega_\eps$ defined in~\eqref{eq230401} coincides with the Bellman domain of~$\Bell$.
\end{St} 

\begin{proof}
First, by the definition of $\Class_{\Omega_{\eps}}$, for any $\vf \in \Class_{\Omega_{\eps}}$ we have 
$\av{\vf}{} \in \Omega_{\eps}$. Second, for any given point $x \in \Omega_\eps$ there exists 
some $t \in \mathbb{R}$ such that $x \in T_\eps(t)$. We draw the tangent to $\dfree\Omega_\eps$ passing through $x$. 
It follows from conditions on the domain that this line intersects $\dfi\Omega_\eps$ twice, say, at $g(a)$ and $g(b)$. 
It is important that $x \in [g(a),g(b)] \subset \Omega_\eps$. Let $\alpha \in [0,1]$ be such that 
$x = \alpha g(a) + (1-\alpha)g(b)$. Define the function $\vf$ as follows:
\begin{equation*}			
\vf = g(a) \chi_{_{[0,\alpha)}}+ g(b) \chi_{_{[\alpha,1]}}.
\end{equation*}
It is clear that $\av{\vf}{[0,1]} = x$ and $\av{\vf}{J}\in [g(a),g(b)] \subset \Omega_\eps$ for any 
$J \subset [0,1]$. Therefore $\vf \in \Class_{\Omega_\eps}$.
\end{proof}

This proposition says nothing whether the Bellman function is finite on~$\Omega_\eps$, it only cuts off those 
points of the plane where it is \emph{a priori} infinite.

\begin{St}\label{Pr040601}
$\Bell(g(t); f) = f(t)$ for all $t \in \mathbb{R}$.
\end{St}

\begin{proof}
The function $\vf$ whose average $\av{\vf}{}$ lies on the fixed boundary is constant almost everywhere, 
because the curve $\dfi\Omega_\eps$ is a graph of a strictly convex function. This constant coincides 
with the average, thus $\vf = g(t)$. Consequently, the set we are taking the supremum over consists of 
a single number $f(t)$, therefore $\Bell(g(t); f) = f(t)$.
\end{proof}

\subsection{Conditions on $f$}\label{s212}
The function~$f$ will be under some restrictions. These restrictions are of two types. The first type comes 
from quantitative aspects. We are interested in finite Bellman functions, we want the function~$\ff(\vf)$ 
to be integrable for any function $\vf \in \Class_{\Omega_\eps}$, $\eps<\eps_{\max}$, in order to have well 
defined averages in~\eqref{Bf}. These conditions are expressed in terms of summability properties of $f$. 
The second type of conditions corresponds to regularity properties of $f$. These conditions make 
the structure of the Bellman function less complicated, thus able to be described. 

We begin with conditions of the second type. We require $f \in C^2(\mathbb{R})$. Consider the curve 
\eq{defgamma}{
\gamma(t) = \big(g(t),f(t)\big), \qquad t \in \mathbb{R},
}\index[symbol]{$\gamma$}
which is the graph of the boundary condition, see Proposition~\ref{Pr040601}. We introduce the following object: 
\eq{eq28042001}{
\torsion (t) = \dettt{\gamma'(t)}{\gamma''(t)}{\gamma'''(t)} = 
\det
\begin{pmatrix}
g_1'(t)& g'_2(t) & f'(t)\\
g_1''(t)& g''_2(t) & f''(t)\\
g_1'''(t)& g_2'''(t) & f'''(t)\\
\end{pmatrix},
\qquad t \in \mathbb{R}.}
We will impose such conditions on $f$ that $\torsion$\index[symbol]{$\torsion$} will be a signed measure, see Condition~\ref{reg}. The sign of $\torsion$ 
coincides with the sign of the torsion of the curve $\gamma$. Since we assume $\gamma$ to be $C^2$, 
the measure $\torsion$ does not have atoms.

\begin{Def}\label{solidroot}
Let $\mu$ be a signed measure on the line. The complement of its support is an open subset of the line, 
so it is a union of several intervals (finite or countable number of them). We call the closure of each 
such interval \emph{a solid root} of $\mu$. If $\mu$ is neither positive nor negative in every neighborhood 
of its solid root, then such a solid root is called \emph{essential}. 
\end{Def}\index{solid root}\index{essential root}

Let $\torsion = \torsion _+ + \torsion_-$ be the Hahn decomposition of the measure $\torsion$. The set 
$\supp \torsion_+ \cap \supp \torsion_-$ is the set of points where $\torsion$ ``changes its sign''. 
This set is closed. The points of~$\supp \torsion_+ \cap \supp \torsion_-$ are also called essential roots. 
Therefore, an essential root is a maximal by inclusion connected subset of the line such that $\torsion$ 
vanishes on it but is neither negative nor positive in every its neighborhood. The regularity condition 
we will impose on $\torsion$ is that it has only finite number of essential roots. If $\gamma$ were $C^3$-smooth, 
this condition would be the same as if the function $\torsion$ had only finite number of changes of sign. 

Let us introduce the following functions:\index[symbol]{$\kappa_2$}\index[symbol]{$\kappa_3$}\index[symbol]{$\tors$}
\eq{eqKappas}{
\kappa_2 = \frac{g_2'}{g_1'}, \qquad \kappa_3 = \frac{f'}{g_1'}, \qquad \tors = \frac{\kappa_3'}{\kappa_2'}.
}
Note that $\kappa_2$ and $\kappa_3$ are the slopes of the projections of $\gamma'$  onto $(x_1x_2)$- and 
$(x_1x_3)$-coordinate planes.

Since $g_1'>0$, $\kappa_2'\ne 0$ (see~\eqref{eq250702} and~\eqref{eq230402}), and
\eq{eq211003}{
\torsion = \tors' \cdot (g_1')^3 \cdot (\kappa_2')^2,
}
the signs of $\torsion$ and $\tors'$ coincide. We formulate the condition discussed above in terms of $\tors$:

\begin{Cond}\label{reg}
\index{condition! conditions on~$f$}
The function~$f$ is two times continuously differentiable\textup,~$\tors$ is piecewise monotone
and has only finite number of monotonicity intervals.
\end{Cond}

\begin{Rem}\label{Rem100301}
Condition~\ref{reg} implies that the functions $f$ and $\kappa_3$ are piecewise monotone and change their 
sign only finite number of times.  
\end{Rem}
\begin{proof}
The function $\tors$ changes its sign only finite number of times, $\kappa_2'>0$, therefore the same is true for 
$\kappa_3' = \kappa_2' \tors$, and hence for $\kappa_3$. But $g_1'>0$ and $f' = \kappa_3 g_1'$, thus $f'$ has 
only finite number of changes of sign, therefore the same is true for $f$.
\end{proof}

The essential roots of $\torsion$ will play a significant role in what follows, therefore we fix the notation for them.

\begin{Def}
\label{roots}
The essential roots of~$\torsion$ (or $\tors'$) are closed intervals (which can be single points or rays) 
$c_0, c_1, \ldots, c_n$ and $v_1, v_2, \ldots, v_n$ such that $c_0 < v_1 < c_1 < v_2 < \cdots < v_n < c_n$. 
The measure $\torsion$ ``changes sign'' from '$-$' to '+' at $v_i$, from '+' to '$-$' at~$c_i$. 
\end{Def}

We make an agreement that if in a neighborhood of $-\infty$ we have $\torsion < 0$, then $c_0 = -\infty$. 
Similarly, if in a neighborhood of $+\infty$ we have $\torsion> 0$, then $c_n = +\infty$. What is more, $v_i$ or $c_i$ 
is an interval (not a point) if and only if it is an essential solid root in the sense of Definition~\ref{solidroot}. 

In the light of our definition, sometimes we will have to treat the intervals as if they were points. We write 
$\dist(x,y)$ for the usual distance between subsets $x$ and $y$ of the real line. We will need it only to denote 
the distance between either two intervals or an interval and a point. Moreover, sometimes we will write, for 
example, $a_n \rightarrow w$ where $w$ is a root, e.\,g., it can be an interval. In such situations we mean that 
for every neighborhood of $w$ all but finite number of members of $\{a_n\}_n$ lie in it. What is more, the set 
of intervals has an essential ordering: $[a,b]$ is less than $[c,d]$ if and only if $b < c$. We have already 
used this ordering in Definition~\ref{roots}. We will also often use the notation~$\mathfrak{a}^{\mathrm{r}}$ 
and~$\mathfrak{a}^{\mathrm{l}}$ for the right and the left endpoints of the interval~$\mathfrak{a}$.

Let us turn to the conditions of the first type (summability conditions at infinities).

\begin{Cond}
\label{sum}
\index{condition! conditions on~$f$}

For any $\eps< \eps_{\max}$ we have
\eq{eq141001a}{
\int_{-\infty}^v f'(\tau) \exp\Big(\!-\!\int_{\tau}^v \lamr(s;\eps)\,  ds \Big)\ d\tau > - \infty.
}
\eq{eq211001a}{
\int_v^{+\infty} f'(\tau) \exp\Big(\int_v^{\tau} \laml(s;\eps)\,  ds \Big)\ d\tau < + \infty. 
}
for any $v\in \mathbb{R}$. Recall that $\laml$ and $\lamr$ are defined in~\eqref{eq150701}.
\end{Cond}

These conditions look cumbersome, but it will be clear that they appear quite naturally, see Chapter~\ref{C5}, 
Propositions~\ref{st140701} and~\ref{st140702}. If~\eqref{eq141001a} fails, but~\eqref{eq111001} holds, 
then the function $B$ is infinite because one can construct a special test function $\vf$ (that is an optimizer 
in Proposition~\ref{OptimizerRightTangentsInfty}, see also Definition~\ref{Opt})  such that $\av{\ff(\vf)}{} = +\infty$. 
Similarly, if~\eqref{eq211001a} fails, but~\eqref{eq111002} holds, then the function $B$ is also infinite. See Subsection~\ref{sec100101} for the details.

\subsection{Smoothness conditions}
\label{Sec091201}

Here we collect all the smoothness conditions we impose on the functions we discussed in this section. 

We suppose that $\geps$ is $C^{3,1}$-smooth and then the functions $\ttl, \ttr, \wl,\wr, \kappal,\kappar$ 
are $C^{2,1}$-smooth with respect to $(t,\eps)$ and $\kappa_2$ is $C^2$-smooth. We also have 
$\frac{\partial}{\partial t}\geps(t,\eps)>0$, $\frac{\partial}{\partial \eps}\kappa(t,\eps) \ne 0$. 
The functions $\lamr, \laml$ are $C^{1,1}$. The function $f$ is $C^2$ with $f''$ being piecewise monotone 
with finite number of monotonicity intervals, therefore $\kappa_3 \in C^1$ and $\tors$ is a piecewise 
monotone continuous function with finite number of monotonicity intervals, and $\tors'$ (as well as $\torsion$) 
is a signed continuous measure. 

\section{Locally concave functions}\label{s22}

We will use abstract results from~\cite{SZ2016} concerning relation between the Bellman function and locally 
concave functions. We start with the definition of local concavity.

\begin{Def}
\label{LocalConcavity}\index{locally concave function}
The function $G\colon \Omega \mapsto \mathbb{R}$ is said to be \emph{locally concave} on $\Omega$ if 
it is concave on every convex subdomain of $\Omega$. We introduce the set of locally concave functions 
on $\Omega$ majorising $\ff$ on $\dfi\Omega$:
\eq{Lambda}{
\Lambda_{\Omega,f} \df \big\{G\colon G \text{ is locally concave on } \Omega,\; G|_{\dfi\Omega}\geq\ff\big\}.
}\index[symbol]{$\Lambda_{\Omega,f}$}
Define the function~$\BG_{\Omega,f}$\index[symbol]{$\BG_{\Omega,f}$} to be the pointwise minimum of the functions from $\Lambda_{\Omega,f}$:
\eq{MinimalLocallyConcave}{
\BG_{\Omega,f}(x) = \inf_{G \in \Lambda_{\Omega,f}} G(x), \quad x \in \Omega.
}
\end{Def}

It is not difficult to see that a function is locally concave if and only if it is concave on every segment 
that belongs to $\Omega$ entirely.

The proposition below resembles Proposition~6.1 in~\cite{SZ2016}. It follows from Proposition~6.1 in~\cite{SZ2016} 
for the case of infinite limits $\lim_{t\to \pm\infty}g(t)$. If some of the limits is finite, one should use 
projective transform trick to reduce it to the case of infinite limits, see Appendix~A in~\cite{SZ2016}.

\begin{St}
\label{Summability}
Let the function~$f$ be locally bounded from below\textup, and let the function $\BG_{\Omega_\eps,\max(f,0)}$ 
be finite. Then the integral~$\av{\ff(\vf)}{I}$ is well defined for all~$\vf \in \Class_{\Omega_\eps},$ 
and~$\Bell(\, \cdot\, ;f) = \BG_{\Omega_\eps,f}$.
\end{St}

Proposition~\ref{Summability} will serve as one of the main ingredients in the proof of Theorem~\ref{MT}.

\subsection{Monge--Amp\`ere equation and ruled surfaces}\label{s223}

Using simple convex geometry arguments one can see that~$\BG_{\Omega_\eps,f}$ is affine in some directions. 
The precise statement is as follows.

\newcommand{\itememph}[1]{\item \emph{#1}}
\begin{Th}
\label{ConvexGeometry}
Let~$f$ satisfy Conditions~\textup{\ref{reg}, \ref{sum}}. Then the function $\BG = \BG_{\Omega_\eps,f}$  
satisfies the following conditions.
	\textup{\begin{enumerate}
	\itememph{For every point $x \in \interior\Omega_{\eps}$ there is a nonzero vector $\Theta(x)$ such that $\BG$ 
	is affine along the line $\ell(x)= x + \mathbb{R}\Theta(x)$ in a neighborhood of $x$.  If there are at least 
	two noncollinear vectors $\Theta_1$ and $\Theta_2$ with this property\textup, then $\BG$ is affine in a neighborhood 
	of~$x$. If this vector is unique, by the \index{extremal segment}{\bf extremal segment} we call the maximal \textup{(}by inclusion\textup{)} segment of $\ell(x)$ containing $x$ such that $\BG$ is affine on this segment.}
	\itememph{The function~$\BG$ is differentiable and its gradient is constant along each extremal segment.}
	\itememph{The extremals cannot intersect the free boundary transversally\textup, only tangentially.}
	\end{enumerate}}
\end{Th}

This theorem provides a partition of $\Omega_{\eps}$ into sets of two types. The sets of the first type are 
extremal segments, along which $\BG$ is affine. We note (and this can be easily proved) that the extremal segments
cannot ``stop'' inside $\Omega_{\eps}$. One of the endpoints of an extremal segment belongs to the fixed boundary, the second one could lie either on the fixed boundary as well, or is the point of tangency with the free boundary. The sets of the second type are the two-dimensional domains where $\BG$ is affine, 
we call them \emph{domains of linearity}.

We should make a remark on the term ``domain''. We call a domain every open connected set united with some 
(or none) part of its boundary.

We will not prove Theorem~\ref{ConvexGeometry}, because, from a formal point of view, we do not need it 
(however, it will follow from our general considerations, e.\,g., Theorem~\ref{Final} far below). It only 
helps us to guess the Bellman function. It leads us to the notion of a \index{Bellman! candidate} \emph{Bellman candidate}.

\begin{Def}
\label{candidate}\index{foliation}
Let $B$ be a continuous locally concave function on a subdomain $\omega$ of $\Omega_{\eps}$ satisfying the boundary 
condition on $\omega\cap \dfi\Omega_\eps$. The function $B$ is called a \emph{Bellman candidate} on $\omega$ 
if there exists a \emph{foliation} on $\omega$. This means that $\omega$ can be represented as a union of several 
domains,~$\omega = \cup_{i}\omega^i$,  such that $B \in C^1({\omega^i})$ and $B$ is either affine in $\omega^i$ 
(thus $\omega^i$ is a domain of linearty) or $\omega^i$ is foliated by straight line segments along which the 
differential of $B$ is constant (they are the extremal segments).
\end{Def}
 
If $B$ is twice differentiable at some inner point $x \in \Omega_{\eps}$, then $\frac{d^2\!B}{dx^2} \le 0$, i.\,e., 
the second differential of $B$ is negative-definite, because $B$ is locally concave. This matrix 
has the vector $\Theta(x)$ in the kernel, thus its determinant is zero. This remark clarifies the name of 
the subsection: the achieved equation is called the \index{Monge--Amp\'ere equation} \emph{homogeneous Monge--Amp\`ere equation}:

\eq{MA}{
B_{x_1x_1}B_{x_2x_2} - B_{x_1x_2}^2 = 0.
}
The homogeneous Monge-Amp\`ere equation must hold almost everywhere for the function $\BG$, because a locally 
concave function is almost everywhere twice differentiable. However, it does not have to hold everywhere, 
because the Bellman function does not have to be $C^2$-smooth even for very smooth boundary values~$f$. For example, see the function $N_{\eps,p}$ from Lemma 6.1 in \cite{SV2012} with $p>2$: it is not $C^2$-smooth, however its boundary data is $C^\infty$ for $p$ being even integer greater than $2$. 
 
\subsection{Optimizers}\label{s224}

Let $B$ be a Bellman candidate in the whole domain $\Omega_{\eps}$, i.\,e., let it satisfy Definition~\ref{candidate} 
with $\omega = \Omega_{\eps}$. This subsection provides a method of verification that the candidate $B$ coincides 
with the Bellman function. By Proposition~\ref{Summability}, there is the inequality $\Bell \leq B$. To prove the 
reverse inequality, $B(x) \leq \Bell(x)$, for a point $x$, $x \in \Omega_{\eps}$, it is sufficient to find a function 
$\vf \in \Class_{\Omega_{\eps}}$ with $\av{\vf}{} = x$ such that $B(x) \leq \av{\ff(\vf)}{}$. Indeed, by the 
definition of the Bellman function, formula \eqref{Bf}, $\av{\ff(\vf)}{} \leq \Bell(x)$, consequently, 
$B(x) \leq \Bell(x)$. We introduce some notions.

\begin{Def}
\label{TestFunction}\index{test function}
Let $x\in\Omega_\eps$. We call a function $\vf\in\Class_{\Omega_\eps}$ a \emph{test function} for $x$ if $\av{\vf}{}=x$. 
\end{Def}

\begin{Def}
\label{Opt}\index{optimizer}
Let $x\in\Omega_\eps$ and let $B$ be a Bellman candidate. We call a measurable function $\vf$ an \emph{optimizer} 
for $B$ at $x$ if it satisfies two conditions:
  \begin{itemize}
  \item $\vf$ is a test function for $x$;  
  \item $B(x) = \av{\ff(\vf)}{}$.
 \end{itemize} 
\end{Def}

So, in order to prove that a candidate $B$ coincides with the Bellman function, it suffices to provide at 
least one optimizer for each point $x$ in $\Omega_{\eps}$ and the candidate $B$. What is the way to do this? 
We may consider only monotone optimizers (with respect to the ordering $\preceq$ defined in Subsection~\ref{s21}). 
This follows from the fact that the class $\Class_{\Omega_{\eps}}$ is invariant with respect to monotonic 
rearrangement\index{monotonic rearrangement}. This fact is proved in~\cite{SZ2016}, Corollary~3.12.

The following property of monotonic rearrangement is useful for us:
$\av{\ff(\vf)}{} = \av{\ff(\vf^*)}{},$
where $\vf^*$ is a monotone rearrangement  of $\vf$. Therefore, $\vf^*$ is an optimizer provided $\vf$ is. 
We do not need this consideration formally, but it helps us to guess the optimizers. In the light of it, 
we consider only monotone optimizers. A more detailed discussion concerning optimizers is postponed until 
Chapter~\ref{C5}.

\subsection{General principles and description of results}

One of our main aims is to prove the following theorem. 

\begin{Th}
\label{MT}
Let $f$ satisfy Conditions~\textup{\ref{reg},~\ref{sum}}. Then 
\begin{equation*}
\Bell(x;\,f) = \BG_{\Omega_\eps,f}.
\end{equation*}
\end{Th}

In fact, in the whole paper we study how to construct the minimal locally concave function~$\BG_{\Omega_\eps, f}$. 
In Chapters~\ref{C3} and~\ref{C4}, for every $f$ satisfying Conditions~\ref{reg} and~\ref{sum}, we construct specific 
locally concave functions~$G$ on $\Omega_\eps$ with the boundary conditions $G|_{\dfi\Omega_\eps}=\ff$. This functions are called 
Bellman candidates, see Definition~\ref{candidate}.  Due to Remark~\ref{Rem100301}, the function $f$ does not 
change its sign out of some compact set, therefore it is easy to find a function $\tilde f$ satisfying 
Conditions~\ref{reg} and~\ref{sum}, which coincides with $\max(f,0)$ outside a compact set. This would imply 
that the function $\BG_{\Omega_\eps,\tilde f}$ is finite, whence $\BG_{\Omega_\eps, \max(f,0)}$ is finite as well. 
Proposition~\ref{Summability} now gives the statement of Theorem~\ref{MT}. In particular, the function $G$ 
constructed in our considerations satisfies the inequality $\Bell(\,\cdot\,;\,f) \leq G$ on~$\Omega_{\eps}$. 
We use optimizers to prove that $G$ coincides with $\Bell(\,\cdot\,;\,f)$, this will be done in Chapter~\ref{C5}, see
Theorem~\ref{GlobalOptimizer}. 
\bigskip

\bigskip

\bigskip

\centerline{\bf Main results}

\begin{itemize}
\item[1.] We provide an algorithm to calculate the Bellman function $\Bell$ for $f$ satisfying Conditions~\ref{reg}, \ref{sum} and describe the evolution of this function with respect to~$\eps$; as a corollary, we obtain Theorem~\ref{MT}.
\item[2.] We show that for any $C^1$ function $f$ not satisfying either of Conditions~\ref{sum} 
$\Bell(\,\cdot\,;\,f)=+\infty$ everywhere except the fixed boundary.
\item[3.] We investigate the case when the technical conditions~\eqref{eq111001} and~\eqref{eq111002} for the domain do not
hold. The theory still works with minor modifications: it could happen that there exists no optimizer for some points $x \in \Omega_\eps$;
in such a case, instead of one optimizer we construct an optimizing sequence of functions $\vf_n \in \Class_{\Omega_\eps}$ such 
that $\av{\vf_n}{I} = x$ and $\av{\ff(\vf_n)}{I}\to\Bell(x;\, f)$. We also give criteria for the function $\Bell$ to be finite when conditions~\eqref{eq111001} and~\eqref{eq111002} could fail.
\end{itemize} 

\bigskip

\bigskip

Some explanation is needed. By building the Bellman function we mean mainly the construction of the 
corresponding foliation. This foliation evolves continuously and obeys certain monotonicity rules that are also described. In the first 
point we intend to provide some expression for $\Bell$ that contains integrals, differentiation, and solution 
of some implicit equations. We always prove that those equations are well solvable, i.\,e.,  do not have 
infinite number of solutions.
}

%% file: 3Chapter_text.tex
\newcommand{\kappaw}{\kappa}

\chapter{Patterns for Bellman candidates}\label{C3}

\section{Preliminaries}
\label{s31}
{%\color{blue}
The purpose of this section is to construct Bellman candidates (see Definition~\ref{candidate}) on various domains. 
The global foliation for the Bellman function may occur to be rather complicated, but its local structure is easy 
to describe. We give some heuristics to classify local Bellman candidates.

Consider a minimal locally concave function and its foliation provided by Theorem~\ref{ConvexGeometry}. 
We recall that this foliation consists of segments, which are called extremal segments, and linearity 
domains\index{domain! linearity domain}. 

The extremal segments are of two types: those that connect two points on the fixed boundary and those that connect 
a point on the fixed boundary with a point on the free one. First type extremal segments are called 
\emph{chords}\index{chords}, second type extremal segments are called \emph{tangents}\index{tangents}. 
We note that a chord can be tangent to the free boundary. Such a chord is called a \emph{long} one.

It is convenient to classify linearity domains by the number of their points on the fixed boundary.
We distinguish the linearity domains 
that have one point on the fixed boundary, the ones that have two points on the fixed boundary, and all 
the others. A more detailed classification will be provided later in Subsection~\ref{s34}.

A global foliation is glued from local ones. We explain the informal meaning of the word ``glue'' we use. 
Consider two subdomains~$\Omega^1$ and~$\Omega^2$ of~$\Omega_{\eps}$. Let~$B_1$ be a Bellman candidate 
on~$\Omega^1$, let~$B_2$ be a Bellman candidate on~$\Omega^2$. Suppose that~$B_1 = B_2$ on~$\Omega^1\cap\Omega^2$. 
Consider the function~$B$ defined on the union domain~$\Omega = \Omega^1 \cup \Omega^2$ as a concatenation 
of~$B_1$ and~$B_2$ (i.\,e.,~$B = B_1$ on~$\Omega^1$ and~$B = B_2$ on~$\Omega^2$). Suppose that this 
function~$B$ is~$C^1$-smooth. In such a case, it is locally concave, provided the functions~$B_1$ and~$B_2$ are, see Proposition~\ref{ConcatenationOfConcaveFunctions} below. 
Thus it is a Bellman candidate on~$\Omega$. Its foliation coincides with the foliation for~$B_1$ on~$\Omega^1$ 
and with the foliation for~$B_2$ on~$\Omega^2$. We say that the foliation for~$B$ is glued from the foliations 
for~$B_1$ and~$B_2$. 

We have used the following fact in the explanation: a~$C^1$-concatenation of two locally concave functions 
is locally concave. To formulate this claim rigorously, we need several new notions. 

\begin{Def}
\label{InducedConvexSet}\index{induced convex set}\index{induced convex set}
Suppose that~$\Omega$ is a subdomain of~$\Omega_{\eps}$. We call~$\Omega$ an \emph{induced convex} 
set if for every segment~$l \subset \Omega_{\eps}$ the set~$\Omega \cap l$ is convex. As usual, for 
any set $\omega \subset \Omega_\eps$ we define its \emph{induced convex hull} $\indconv(\omega)$ as 
the minimal induced convex set which contains $\omega$. 
\end{Def}

All the domains we use for building Bellman candidates are induced convex.

\begin{St}
\label{ConcatenationOfConcaveFunctions}
Suppose that the domains~$\Omega^1$ and~$\Omega^2$ are induced convex in~$\Omega_{\eps}$. 
Suppose that a~$C^1$-smooth function~$B$ is locally concave on each of the domains~$\Omega^i,$ $i=1,2$. 
Then it is locally concave on~$\Omega^1 \cup\Omega^2$.
\end{St}

\begin{proof}
To prove the claim we establish that the restriction of~$B$ to every segment~$l\subset\Omega^1\cup\Omega^2$ 
is concave. We have~$l = \big(l\cap\Omega^1\big) \cup \big(l\cap\Omega^2\big)$. Each of the sets~$l\cap\Omega^1$ 
and~$l\cap\Omega^2$ is convex, i.\,e., they are either segments or empty sets (the latter case is trivial). 
By the hypothesis,~$B$ is concave on each of these segments. Using~$C^1$-smoothness of~$B$ in a common point 
of these segments, we get that~$B|_{l}$ is concave.  
\end{proof}

Now we can state 
that a~$C^1$-smooth concatenation of two Bellman candidates is a Bellman candidate provided their domains 
are induced convex. 

We turn to building Bellman candidates. Usually, we will give only sufficient conditions for a foliation 
and a function~$f$ that generate a Bellman candidate. However, to be ready to construct the Bellman function, 
we have to examine all possible local Bellman candidates. So, the conditions we provide are usually also necessary. 
To make the story shorter, sometimes we will not prove this necessity, because we do not need it.

To describe combinatorial properties of foliations, we associate a special oriented graph with each foliation. 
Generally, its vertices correspond to the linearity domains, whereas its edges correspond to the domains 
of extremal segments. A vertex is incident to an edge if the corresponding two domains are adjacent. We postpone 
a more detailed description of the graph to Subsection~\ref{s345}. 

\section{Fence}\label{s31bis}
We start with investigation of the properties of a Bellman candidate defined on a family of extremal segments with an endpoint on the fixed boundary. 

Recall that $g=(g_1,g_2)\colon\R \to \dfi\Omega$  is the parametrization of the fixed boundary of $\Omega$ 
such that $g_1'>0$ and $\kappa_2' = \big(g_2'/g_1'\big)'>0$. Let $I$ be some interval in $\mathbb{R}$. 
Suppose that there exists a family of segments $S(t)=[g(t),w(t)]$, $t \in I$, which foliate a subdomain 
$\Omega(I)\subset \Omega$. We assume that the foliation is sufficiently smooth: $w \in C^1(I)$. We can consider 
$t$ as a function from $\Omega(I)$ to $I$ such that $x \in S(t(x))$ for $x \in \Omega(I)$.  We assume 
the function $t$ to be~$C^1$-smooth as well, and also $\frac{\partial}{\partial x_2}t  \ne 0$ (it implies that 
the segments $S(t)$ are not vertical, i.\,e., $w_1(t)\ne g_1(t)$). We denote the slope of the segment $S(t)$ 
by $\kappaw(t)$: 
\eq{kappa}{
\kappaw(t) \df \frac{w_2(t)-g_2(t)}{w_1(t)-g_1(t)}.
}\index[symbol]{$\kappa$}

\begin{St}
\label{St280701}
Suppose $B$ is a $C^1$-smooth function on $\Omega(I)$ that is affine on each segment $S(t),$ and 
its gradient $\beta = (\beta_1, \beta_2) = \nabla B$\index[symbol]{$\beta$} is constant on $S(t),$ $t \in I$. Then $B$ 
is $C^2$-smooth. Moreover\textup, there is a representation
\eq{fenceB2}{
B(x) = f(t) + \big(x_1-g_1(t)\big)\Big[\kappa_3(t)+\big(\kappa(t)-\kappa_2(t)\big)\beta_2(t)\Big], \quad t = t(x),
}
where $\beta_2$ is given by 
\eq{beta2new}{
\beta_2(t) = \exp\Big(\!-\!\int_{t_0}^t\frac{\kappa_2'}{\kappa_2-\kappa}\Big)
\left(\int_{t_0}^t \exp\Big(\int_{t_0}^\tau\frac{\kappa_2'}{\kappa_2-\kappa}\Big)
\frac{\kappa_3'(\tau)}{\kappa_2(\tau)-\kappa(\tau)}\,d\tau+\beta_2(t_0) \right), \quad t, t_0 \in I,
}
and $\kappa_2,$ $\kappa_3$ are defined in~\eqref{eqKappas}\textup, here $\beta_2(t_0)$ is an arbitrary number.
\end{St}

\begin{proof} 
We have 
\eq{fenceB}{
B(x) = f(t) + \beta_1(t) (x_1-g_1(t))+ \beta_2(t) (x_2-g_2(t)),
}
since the function $B$ is affine on $S(t)$. 
Differentiating the boundary equality $B(g(t))=f(t)$, we obtain
\eq{nabla1}{
\scalp{\beta(t)}{g'(t)} = f'(t),
}
which is equivalent to (divide by $g_1'$) 
\eq{nabla11}{
\beta_1+\kappa_2\beta_2=\kappa_3.
}

Plugging this into~\eqref{fenceB}, we get
\eq{eq200702}{
B(x) = f(t) + (\kappa_3(t) - \kappa_2(t) \beta_2(t))(x_1-g_1(t))+ \beta_2(t) (x_2-g_2(t)).
}
We use the formula
\eq{tofxy}{
\kappaw(t) =\frac{x_2-g_2(t)}{x_1-g_1(t)}, \quad x \in S(t),
}
to rewrite~\eqref{eq200702} in the form~\eqref{fenceB2}.

It remains to prove~\eqref{beta2new}. Let us note that~\eqref{fenceB2} at the point $x = w(t)$ implies 
$C^1$-smoothness of $\beta_2$, which in its turn implies $C^2$-smoothness of $B$ because  $\beta_1$ is 
$C^1$-smooth as well due to~\eqref{nabla11}. Differentiating~\eqref{eq200702} with respect to $x_2$ and 
using the relations $\kappa_2 g_1'=g_2'$, $\kappa_3 g_1'=f'$, we get
\eq{eq200703}{
\Big((\kappa_3'-\kappa_2'\beta_2-\kappa_2\beta_2')(x_1-g_1(t))+\beta_2'(x_2-g_2(t))\Big)\cdot
\frac{\partial}{\partial x_2}t = 0.
}
We use that $\frac{\partial}{\partial x_2}t\ne 0$ and~\eqref{tofxy} to obtain
\eq{diffbeta2}{
(\kappa_2-\kappa)\beta_2'+\kappa_2'\beta_2=\kappa_3'.
}

Solution $y$ of the equation $y'(t)+K_1(t)y(t) =K_2(t)$ is given by the formula
\eq{firstordersolution}{
y(t) = \exp\Big(-\int_{t_0}^t K_1(\tau)\, d\tau \Big)
\left(\int_{t_0}^t \exp\Big(\int_{t_0}^\tau K_1(s)\,d s\Big)K_2(\tau)\,d\tau+y(t_0)\right),
}
where $y(t_0)$ is an arbitrary parameter. 
Applying this to~\eqref{diffbeta2} with $y = \beta_2$, we obtain~\eqref{beta2new}. 
\end{proof}

In what follows we will need a slightly different representation for $\beta_2$ that can be obtained using integration by parts.

\begin{Cor}
Under conditions of Proposition~\textup{\ref{St280701}} for $t, t_0 \in I$ we have\textup:
\eq{beta2}{
\beta_2(t) = \tors(t)-\exp\Big(-\int_{t_0}^t\frac{\kappa_2'}{\kappa_2-\kappa}\Big)
\left(\int_{t_0}^t \exp\Big(\int_{t_0}^\tau\frac{\kappa_2'}{\kappa_2-\kappa}\Big)\tors'(\tau)\,d\tau+\const \right),
}
\eq{betaprime}{
\beta_2'(t) = \frac{\kappa_2'(t)}{\kappa_2(t)-\kappa(t)}\exp\Big(\!-\!\int_{t_0}^t\frac{\kappa_2'}{\kappa_2-\kappa}\Big)
\left(\int_{t_0}^t \exp\Big(\int_{t_0}^\tau\frac{\kappa_2'}{\kappa_2-\kappa}\Big)\tors'(\tau)\,d\tau+\const \right),
}
where $\const = \beta_2'(t_0)\frac{\kappa_2(t_0)-\kappa(t_0)}{\kappa_2'(t_0)} =\tors(t_0)-\beta_2(t_0)$.
\end{Cor}

\begin{proof}
Let us integrate by parts in the right hand side of~\eqref{beta2new} using the fact that 
$$
\exp\Big(\int_{t_0}^\tau\frac{\kappa_2'}{\kappa_2-\kappa}\Big)\frac{\kappa_3'(\tau)}{\kappa_2(\tau)-\kappa(\tau)} = 
\bigg(\exp\Big(\int_{t_0}^\tau\frac{\kappa_2'}{\kappa_2-\kappa}\Big)\bigg)'\! \cdot \tors(\tau).
$$
We note that $\tors'$ is a signed measure due to Condition~\ref{reg}, and the integrals in~\eqref{beta2} 
and~\eqref{betaprime} are considered as the integrals with respect to this measure. Formula~\eqref{betaprime} 
immediately follows from~\eqref{beta2} and~\eqref{diffbeta2}, this also proves that the constants in~\eqref{betaprime} 
and~\eqref{beta2} are the same.
\end{proof}

We formulate a proposition converse to~\ref{St280701}.

\begin{St}
\label{St280702}
Suppose that the function $\beta_2$ is given by~\eqref{beta2new} and the function $B$ is defined 
by~\eqref{fenceB2} on the domain $\Omega(I)$. Then the function $B$ is $C^2$-smooth on $\Omega(I)$ 
and affine on each segment~$S(t)$. Its gradient coincides with $(\beta_1(t),\beta_2(t))$ on $S(t),$  
here $\beta_1$ is given by~\eqref{nabla11}.
\end{St}

\begin{proof}
If we prove $\nabla B = \beta$, then we automatically get that $B$ is $C^2$-smooth, since both $\beta$ 
and $t$ are $C^1$-smooth. 

Using~\eqref{nabla11} and~\eqref{tofxy} we rewrite~\eqref{fenceB2} in the form~\eqref{fenceB}. 
Differentiating \eqref{fenceB} with respect to $x$ we obtain
\eq{gradB}{
\nabla B = f'(t) \nabla t +\beta(t) + \big(\scalp{\beta'(t)}{x-g(t)}-\scalp{\beta(t)}{g'(t)}\big)\nabla t, 
\qquad t \in I,\  x \in S(t).
}
We recall that~\eqref{nabla11} is equivalent to~\eqref{nabla1}, and therefore $\scalp{\beta(t)}{g'(t)}=f'(t)$. 
Thus, to prove that $\nabla B = \beta$  we only need to show that $\scalp{\beta'(t)}{x-g(t)} = 0$. 
The function $\beta_2$ given by~\eqref{beta2new} satisfies~\eqref{diffbeta2}. Again, using~\eqref{nabla11} 
and~\eqref{tofxy} we obtain
$$
(\kappa_3' - \kappa_2' \beta_2 - \kappa_2 \beta_2')(x_1-g_1(t))+ \beta_2' (x_2-g_2(t)) = 0,
$$ 
which is equivalent to 
\begin{equation}
\label{nabla2}
\scalp{\beta'(t)}{x-g(t)} = 0, \quad x \in S(t).%\qedhere
\end{equation}
\end{proof}

\begin{Rem}
Dividing~\eqref{nabla2} by~$x_1-g_1(t),$ we obtain the following relation:
\eq{nabla21}{
\beta_1'+\kappa\beta_2'=0.
}
\end{Rem}

Concluding this subsection, we name the domains $\Omega(I)$ as defined above by the term \emph{fences}\index{fence}. 
We will need two types of fences: the points $w(t)$ are either on $\dfree\Omega$ or on $\dfi \Omega$. In the 
first case we have a family of tangents to $\dfree\Omega$, and $\Omega(I)$ is called a tangent domain, 
and in the second case we have a family of chords, and $\Omega(I)$ is called a chordal domain. We will 
use all the relations from this subsection for both types of fences in what follows. Now we consider these 
two cases separately and show how to find the vector valued function $\beta$ and determine the candidate~$B$.

\section{Tangent domains}
\label{s32}\index{domain! tangent domain}

As it was mentioned in Section~\ref{s31}, the extremal segments are of two types: the chords and the tangents. 
This section provides a study of \emph{tangent domains}, i.\,e., fences that consist of the segments $S(t)$ 
tangent to $\dfree \Omega$, see Figure~\ref{fig:tang}.

\begin{figure}[h!]
\includegraphics[width=0.8\textwidth]{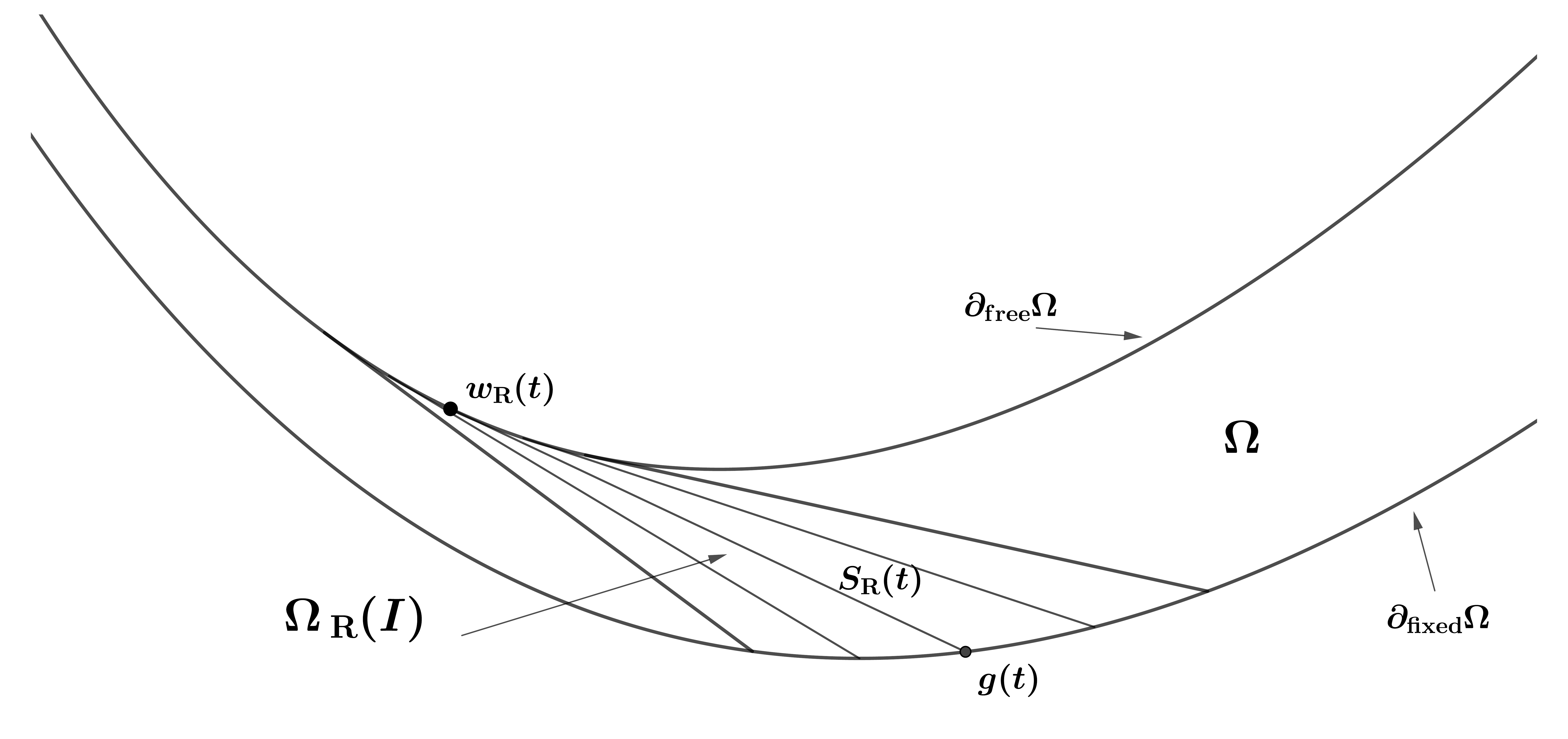}
\includegraphics[width=0.8\textwidth]{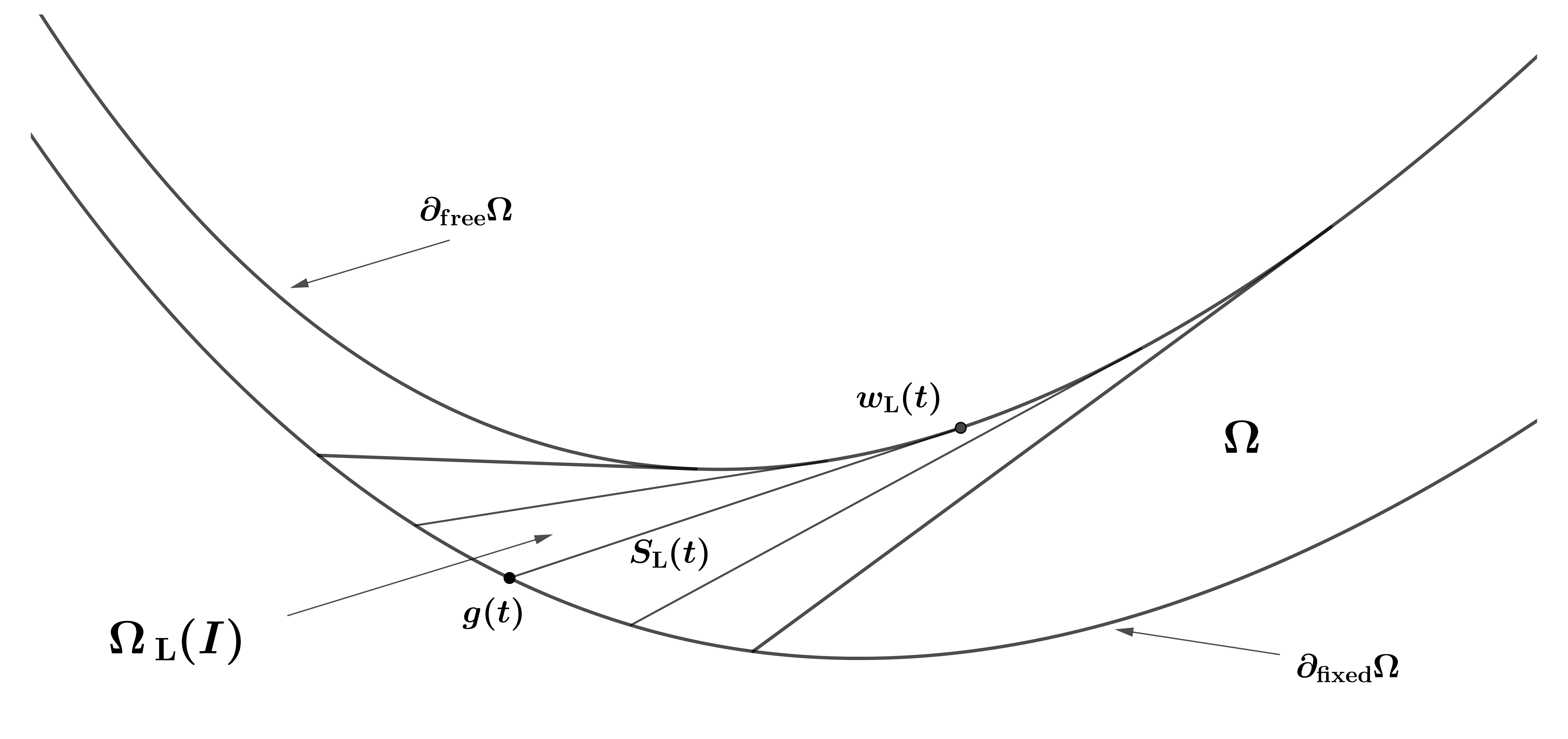}
\caption{Domains~$\Rt$ and~$\Lt$ with right and left tangents.}
\label{fig:tang}
\end{figure}

\begin{Def}
A fence $\Omega(I)$ with the foliation $S(t)=[g(t),w(t)]$, $t \in I$, is called a tangent domain if 
$w(t)\in\dfree\Omega$ and the segment $S(t)$ is tangent to $\dfree\Omega$ for $t \in I$.
\end{Def}

We note that in general there are two possibilities: either $S(t)=\Sr(t)$ (right tangents) or $S(t)=\Sl(t)$ 
(left ones), see Subsection~\ref{cond}. If we consider one of these two cases, we use notation $\Rt(I)$ 
and $\Lt(I)$, respectively. The segment $S(t)$ is tangent to $\dfree\Omega$ if and only if the function 
$\kappa$ defined in~\eqref{kappa} satisfies the following equation:
\eq{tangent}{
\kappaw(t) =\frac{w_2'(t)}{w_1'(t)}.
}

Differentiating \eqref{kappa} and using \eqref{tangent}, one can easily check that 
\eq{kappa'}{
\kappaw'=\frac{g_1'\kappaw-g_2'}{w_1-g_1}.
}

Differentiating \eqref{tofxy} with respect to $x_2$, using \eqref{kappa'} and \eqref{kappa} once more,
we find the partial derivative of the function $t=t(x)$:
\eq{t_y}{
t_{x_2}\df\frac{\partial t}{\partial x_2}=\frac1{g_2'+(x_1-g_1)\kappaw'-g_1'\kappaw}=
\frac{w_1-g_1}{(\kappaw g_1'-g_2')(x_1-w_1)}=\frac{w_1-g_1}{g_1'(\kappaw -\kappa_2)(x_1-w_1)}.
}

\begin{Rem}
\label{Rem201001}
Note that $x_1$ lies between $g_1$ and $w_1$, $\kappaw > \kappa_2$ for $\Lt(I)$ and $\kappaw < \kappa_2$ 
for $\Rt(I)$. Therefore $t_{x_2}>0$ for $\Rt(I)$ and $t_{x_2}<0$ for $\Lt(I)$. Recall that $B$ is defined by~\eqref{fenceB2} and its Hessian is degenerate, 
therefore concavity of $B$ follows from $B_{x_2x_2} < 0$. The sign of $B_{x_2x_2} = \beta_2'(t) t_{x_2}$, 
is determined by the sign of $\beta_2'$, and we obtain the following propositions.
\end{Rem}

\begin{St}
\label{NewRightTangentsCandidate}
Consider a tangent domain~$\Rt(I)$\index[symbol]{$\Rt(I)$} foliated by the family of segments $\Sr(t)=[g(t),w(t)]$ such that 
$w(t)\in \dfree\Omega$ and $\Sr(t)$ is a right tangent to $\dfree\Omega$ for $t \in I$. Suppose that 
the function~$\beta_2$ given by~\eqref{beta2new} satisfies the inequality~$\beta_2'(t)\leq 0$ for~$t\in I$. 
Then the function~$B$ given by~\eqref{fenceB2} is a Bellman candidate on~$\Rt(I)$.
\end{St}

\begin{proof}
The function $\beta_2$ given by~\eqref{beta2new} and $\beta_1$ defined by~\eqref{nabla11} 
are the components of the gradient $\nabla B$ by Proposition~\ref{St280702}. 
Thus, we only need to verify the concavity of~$B$. If $\beta_2'(t) =0$, then 
$\beta_1'(t) =0$ by~\eqref{nabla21} and the Hessian of $B$ is the zero matrix. If $\beta_2'(t) < 0$, then 
$B_{x_2x_2} = \beta_2'(t) t_{x_2}<0$ and the Hessian is non-positive definite by Silvester's criterion. 
Thus, in both cases the proposition is proved.
\end{proof}

\begin{Rem}
If $I=[t_-,t_+]$, $\beta_2'(t_-)\leq 0,$ and the function $\torsion$ corresponding to $\gamma$ 
{\bf(}defined in~\eqref{eq28042001}{\bf)} is nonpositive on $I,$ then $\beta_2'\leq 0$ on the 
entire interval~$I$ because of~\eqref{betaprime} and~\eqref{eq211003}.
\end{Rem}

We state the following symmetrical proposition:

\begin{St}
\label{NewLeftTangentsCandidate}
Consider a tangent domain~$\Lt(I)$\index[symbol]{$\Lt(I)$} foliated by the family of segments $\Sl(t)=[g(t),w(t)]$ such that 
$w(t)\in \dfree\Omega$ and $\Sl(t)$ is a left tangent to $\dfree\Omega$ for $t \in I$. Suppose that 
the function~$\beta_2$ given by~\eqref{beta2new} satisfies the inequality~$\beta_2'(t) \geq 0$ 
for~$t \in I$. Then the function~$B$ given by~\eqref{fenceB2} is a Bellman candidate on~$\Lt(I)$.
\end{St}

In the final part of this subsection we consider the case of infinite tangent domains.

\begin{St}
\label{NewRightTangentsCandidateInfty}
Let $I=(-\infty,t_2),$ $t_2 \in \mathbb{R}\cup\{+\infty\}$. Consider a tangent domain~$\Rt(I)$ foliated 
by the family of segments $\Sr(t)=[g(t),w(t)]$ such that $w(t)\in \dfree\Omega$ and $\Sr(t)$ is a right 
tangent to $\dfree\Omega$ for $t \in I$. Suppose that the function~$\beta_2$ given by~\eqref{beta2} 
with $\const=0$ and $t_0=-\infty,$ i.\,e.\textup,
\eq{eq120701}{
\beta_2(t) = 
\tors(t)-\int_{-\infty}^t \exp\Big(\!-\!\int_{\tau}^t\frac{\kappa_2'}{\kappa_2-\kappa}\Big)\tors'(\tau)\,d\tau, \qquad t \in I,
}
is finite and satisfies the inequality 
\eq{eq120702}{
\beta_2'(t) = \frac{\kappa_2'(t)}{\kappa_2(t)-\kappa(t)} \int_{-\infty}^t 
\exp\Big(\!-\!\int_{\tau}^t\frac{\kappa_2'}{\kappa_2-\kappa}\Big)\tors'(\tau)\,d\tau \leq 0,\qquad t \in I.
} 
Then the function~$B$ given by~\eqref{fenceB2} is a Bellman candidate on~$\Rt(I)$.
\end{St}

\begin{St}
\label{NewLeftTangentsCandidateInfty}
Let $I=(t_1,+\infty),$ $t_1 \in \mathbb{R}\cup\{-\infty\}$. Consider a tangent domain~$\Lt(I)$ foliated 
by a family of segments $\Sl(t)=[g(t),w(t)]$ such that $w(t)\in \dfree\Omega$ and $\Sl(t)$ is a left 
tangent to $\dfree\Omega$ for $t \in I$. Suppose that the function~$\beta_2$ given by formula \eqref{beta2} 
with $\const=0$ and $t_0=+\infty,$ i.\,e.\textup,
\eq{eq190701}{
\beta_2(t) = 
\tors(t)- \int_t^{+\infty}\exp\Big(\int_t^{\tau}\frac{\kappa_2'}{\kappa_2-\kappa}\Big)\tors'(\tau)\,d\tau, \qquad t \in I,
}
is finite and satisfies the inequality 
\eq{eq190702}{
\beta_2'(t) = \frac{\kappa_2'(t)}{\kappa_2(t)-\kappa(t)}\int_t^{+\infty}
\exp\Big(\int_t^{\tau}\frac{\kappa_2'}{\kappa_2-\kappa}\Big)\tors'(\tau)\,d\tau \geq 0,\qquad t \in I.
}
Then the function~$B$ given by formula~\eqref{fenceB2} is a Bellman candidate on~$\Lt(I)$.
\end{St}

\begin{Def}
\label{211101}
The functions $B$ constructed in Propositions~\ref{NewRightTangentsCandidate}, \ref{NewLeftTangentsCandidate}, 
\ref{NewRightTangentsCandidateInfty}, and \ref{NewLeftTangentsCandidateInfty} are called \emph{standard candidates} 
for the corresponding tangent domains $\Rt(I)$ and $\Lt(I)$ if the inequalities for $\beta_2'$ are strict in 
the interior of $I$, i.\,e., $\beta_2'<0$ for~$\Rt(I)$ and $\beta_2'> 0$ for~$\Lt(I)$.
\end{Def}

We would like to comment on the convergence of the integrals in~\eqref{eq120701} and~\eqref{eq190701}.

\begin{Rem}
Convergence of the integral in~\eqref{eq120701} is equivalent to convergence of the integral in 
Condition~\eqref{eq141001a}, this will be proved in Lemma~\ref{Lem190701} below. Moreover, if the integral 
in~\eqref{eq141001a} is equal to $+\infty$, then the inequality in~\eqref{eq120702} fails, i.\,e., there is no standard 
candidate on~$\Rt(-\infty,t_2)$ for any $t_2\in\mathbb{R}$.

Symmetrically, convergence of the integral in~\eqref{eq190701} is equivalent to convergence of the 
integral in Condition~\eqref{eq211001a}. Moreover, if the integral in~\eqref{eq211001a} is equal to $-\infty$, 
then the inequality in~\eqref{eq190702} fails, i.\,e., there is no standard candidate on~$\Lt(t_1,+\infty)$ for any $t_1\in\mathbb{R}$.
\end{Rem}

\section{Chordal domains}
\label{s33}
As it was mentioned in Section~\ref{s31}, the extremal segments are of two types: the chords and the tangents. 
In this section we study \emph{chordal domains}, i.\,e., the domains that consist of chords 
(see Figure~\ref{fig:chd}).\index{domain! chordal domain}
\begin{figure}[h!]
\includegraphics[width=0.8\textwidth]{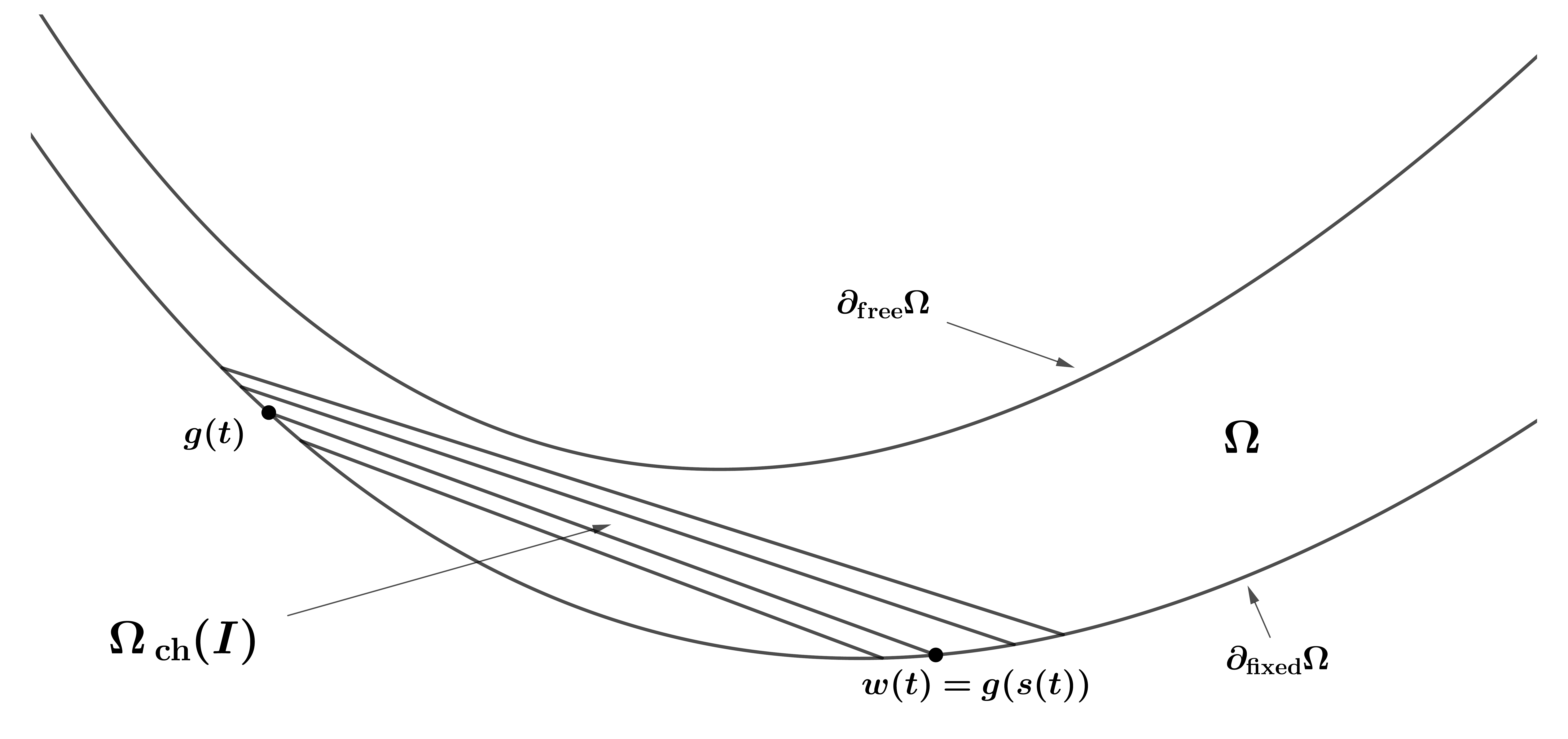}
\caption{Chordal domain~$\Ch(I)$.}
\label{fig:chd}
\end{figure}

\begin{Def}
A fence $\Omega(I)$ with the foliation $S(t)=[g(t),w(t)]$, $t \in I$, is called a \emph{chordal domain} if 
$w(t)\in \dfi \Omega$ for $t \in I$. A chordal domain is denoted by $\Ch(I)$\index[symbol]{$\Ch(I)$}.
\end{Def}

The function $B$ is uniquely determined in a chordal domain by its boundary values and linearity on segments:
\eq{eq250901}{
B\big(\alpha g(t)+(1-\alpha)g(s)\big)=\alpha f(t)+(1-\alpha)f(s), \qquad \alpha \in[0,1],
}
where the function $s \colon I \to \mathbb{R}$ is defined by the relation $g(s(t))=w(t)$ for $t \in I$. 
Let us note that the function $s$ is decreasing. Our aim in this subsection is to provide conditions when 
the function $B$ given by~\eqref{eq250901} supplies us with a Bellman candidate.

The vectors $\gamma'(t),\gamma'(s),$ and $\gamma(t)-\gamma(s)$ belong to the same tangent plane to the graph of $B$, therefore,
\eq{urlun1}{
\det
\begin{pmatrix}
\gamma'(t)\\
\gamma'(s)\\
\gamma(s)-\gamma(t)
\end{pmatrix}
=
\det
\begin{pmatrix}
g_1'(t)& g'_2(t) & f'(t)\\
g_1'(s)& g'_2(s) & f'(s)\\
g_1(s)-g_1(t)& g_2(s)-g_2(t) & f(s)-f(t)\\
\end{pmatrix}=0.
}
We will refer to this equation as the \emph{cup equation}.\index{equation!cup equation}\index{cup!cup equation}

Denote by $\theta = \theta(x_1,x_2)$ the vector orthogonal to the graph of $B$ at the point $(x_1,x_2,B(x_1,x_2))$: 
\eq{eq28042002}{
\theta=\nabla (B(x_1,x_2)-x_3) = (\beta(t(x_1,x_2)),-1).
} 
We can consider $\theta$ as a function of $t$ only: $\theta(t) = (\beta(t),-1)$. 
This vector is orthogonal to the tangent vector of the boundary curve $\gamma$, i.\,e., 
$\scalp{\theta(t)}{\gamma'(t)}=0$, which is equivalent to
\eq{nabla1a}{
\scalp{\beta(t)}{g'(t)}=f'(t).
}

Now we are going to find $\beta'$. We differentiate equation \eqref{nabla1a} and obtain 
\eq{eqbetaprime}{
\scalp{\beta'(t)}{g'(t)}=f''(t)-\scalp{\beta(t)}{g''(t)}= -\scalp{\theta(t)}{\gamma''(t)},
}
where $\theta$ is defined by~\eqref{eq28042002}.

We know that $\scalp{\theta(t)}{\gamma'(t)}=\scalp{\theta(t)}{\gamma'(s)}=0$ and $\scalp{\theta(t)}{(0,0,1)}=-1$. 
Therefore, for any $v \in \mathbb{R}^3$ we have $\scalp{\rule{0pt}{9pt}\theta(t)}{v+\scalp{\theta(t)}{v}(0,0,1)}=0$, 
whence 
$$
\det
\begin{pmatrix}
\gamma'(t)\\
\gamma'(s)\\
v+\scalp{\theta(t)}{v}(0,0,1)
\end{pmatrix}
=
0,
$$
or
$$
\scalp{\theta(t)}{v}
\det
\begin{pmatrix}
g'(t)\\
g'(s)\\
\end{pmatrix}
=
-\det
\begin{pmatrix}
\gamma'(t)\\
\gamma'(s)\\
v
\end{pmatrix}.
$$
We use it in \eqref{eqbetaprime} and obtain
\eq{eqbetaprime2}{
\scalp{\beta'(t)}{g'(t)} 
=-\scalp{\theta(t)}{\gamma''(t)}
=\frac{
\det
\begin{pmatrix}
\gamma'(t)\\
\gamma'(s)\\
\gamma''(t)
\end{pmatrix}
}
{
\det
\begin{pmatrix}
g'(t)\\
g'(s)\\
\end{pmatrix}
}=
\frac{
\det
\begin{pmatrix}
g_1'(t)& g'_2(t) & f'(t)\\
g_1'(s)& g'_2(s) & f'(s)\\
g_1''(t)& g_2''(t) & f''(t)\\
\end{pmatrix}}
{
\det
\begin{pmatrix}
g_1'(t)& g'_2(t)\\
g_1'(s)& g'_2(s)\\
\end{pmatrix}
}
.
}

We have a system of equations~\eqref{eqbetaprime2} and~\eqref{nabla21} (with $w(t) = g(s)$) for $\beta'(t)$. 
We can solve it and obtain
\eq{betaprime2}{
\beta'_2(t)= 
\frac{1}{g_1'(t)}
\cdot 
\frac{1}{\kappa_2(t)-\kappaw(t)}
\cdot
\frac{
\det
\begin{pmatrix}
g_1'(t)& g'_2(t) & f'(t)\\
g_1'(s)& g'_2(s) & f'(s)\\
g_1''(t)& g_2''(t) & f''(t)\\
\end{pmatrix}}
{
\det
\begin{pmatrix}
g_1'(t)& g'_2(t)\\
g_1'(s)& g'_2(s)\\
\end{pmatrix}
},
}
where $\kappaw(t)=\frac{g_2(s)-g_2(t)}{g_1(s)-g_1(t)}$ was defined by~\eqref{kappa} and $\kappa_2=\frac{g_2'}{g_1'}$. 
We recall that $g_1'>0$ {\bf(}see~\eqref{eq250702}{\bf)}. 
It follows from convexity of $\dfi \Omega$ that for $t<s$ the following inequality holds:
\eq{260701}{
\frac{g_2'(t)}{g_1'(t)} = \kappa_2(t) < \kappaw(t) < \kappa_2(s)= \frac{g_2'(s)}{g_1'(s)}
}
(and for $s<t$ both inequalities are opposite), therefore,
$\sign\beta_2'(t)=-\sign\det
\begin{pmatrix}
\gamma'(t)\\
\gamma'(s)\\
\gamma''(t)
\end{pmatrix}.$

Note that $\sign \frac{\partial t}{\partial x_2} = \sign (t-s)$, whence 
\eq{concB}{
\sign \frac{\partial^2 B}{\partial x_2^2} = \sign\big[\beta'_2(t) \frac{\partial t}{\partial x_2}\big]=
\sign\Bigg[(s-t)\det
\begin{pmatrix}
\gamma'(t)\\
\gamma'(s)\\
\gamma''(t)
\end{pmatrix}\Bigg].
}

We formulate a proposition that gives sufficient conditions for the function~$B$ defined by~\eqref{eq250901} to be a Bellman candidate on the chordal domain~$\Ch(I)$.

\begin{St}
\label{NewLightChordalDomainCandidate}
Consider a chordal domain~$\Ch(I)$ foliated by the family of segments $S(t)=[g(t),g(s(t))],$ where the function 
$s\colon I \to \mathbb{R}$ satisfies \eqref{urlun1}. Let $B$ be defined on $\Ch(I)$ by 
the boundary conditions $B\circ g= f$ and by linearity on each segment $S(t),$ $t \in I,$ i.\,e.\textup, by~\eqref{eq250901}. Suppose that 
\eq{eq260901}{
(t-s)\det
\begin{pmatrix}
\gamma'(t)\\
\gamma'(s)\\
\gamma''(t)
\end{pmatrix}\geq 0,\qquad t \in I.
}
Then the function $B$ is a Bellman candidate on $\Ch(I)$.
\end{St}

\begin{proof}
We need to check that $\nabla B$ is constant on each chord.
Using~\eqref{urlun1} we find a vector $\beta=\beta(t)$, such that $\theta=(\beta,-1)$ is orthogonal 
to $\gamma'(t), \gamma'(s)$ and $\gamma(s)-\gamma(t)$. The function $\beta$ chosen in such a way satisfies 
relations~\eqref{nabla1a}, \eqref{nabla2}. Identity~\eqref{fenceB} holds for $x=g(t)$ and $x=g(s)$, 
therefore, by linearity it holds for any $x \in S(t)$. Then, by~\eqref{gradB}, $\nabla B=\beta(t)$, therefore $\nabla B$ is 
constant on $S(t)$. Under the hypothesis of the proposition, concavity of $B$ follows from~\eqref{concB}.
\end{proof}

\begin{Def}\index{standard candidate! cup}\index{cup! standard candidate} 
The function $B$ constructed in Proposition~\ref{NewLightChordalDomainCandidate} is called 
a \emph{standard candidate} for the chordal domain $\Ch(I)$ if the inequality in~\eqref{eq260901} is strict 
for $t$ in the interior of~$I$.
\end{Def}

\begin{Def}\index{cup}\index{cup! origin}
Let $I = [t_-,t_+]$ and either $s(t_-)=t_-$ or $s(t_+)=t_+$ for a function $s\colon I \to \mathbb{R}$ 
corresponding to the chordal domain~$\Ch(I)$. Then the chordal domain $\Ch(I)$ is called a \emph{cup} 
originated at $g(c)$, where $c=s(c)$ coincides with either $t_-$ or $t_+$, respectively. The point $g(c)$ 
as well as the value of the parameter $c$ is called the \emph{origin of the cup}. 
\end{Def}

The origin $g(c)$ of a cup $\Ch(I)$ is a very special point on the curve $g$. It should satisfy 
the following torsion equation:
\eq{eqtor}{
\det
\begin{pmatrix}
\gamma'(c)\\
\gamma''(c)\\
\gamma'''(c)
\end{pmatrix}=0,
}
which can be obtained by accurate passage to the limit $t\to c$ and $s \to c$ in~\eqref{urlun1}.

Geometrically the chordal domains $\Ch(I)$ and $\Ch(s(I))$ are the same. Usually we do not need to distinguish them, 
so we will use the following common notation. Let $I=[t_-,t_+]$. Denote the numbers $t_\pm$ and $s(t_\pm)$ by 
$a_0,a_1,b_1$ and $b_0$ in the increasing ordering, i.\,e., $[g(a_0),g(b_0)]$ is the top chord of the domain 
and $[g(a_1),g(b_1)]$ is the bottom one. Denote such a chordal domain by~$\Ch([a_0,b_0],[a_1,b_1])$. 
If one of these chords is unimportant, we will use $\ast$ instead: 
$$
\Ch([a_0,b_0],\ast) \qquad \text{or} \qquad \Ch(\ast,[a_1,b_1]).
$$

\section{Around the cup}
\label{s33old}

\subsection{Differentials}

\begin{Def}
\label{differentials}\index{differentials}\index[symbol]{$\DL$}\index[symbol]{$\DR$}
Suppose that a pair~$(s,t)$ satisfies the cup equation~\eqref{urlun1}. Let $s<t$. We call the 
following two expressions the differentials, the left and the right one correspondingly:
\eq{e334}{
\DL(s,t) \df \frac{
\det
\begin{pmatrix}
\gamma'(s)\\
\gamma'(t)\\
\gamma''(s)
\end{pmatrix}}
{
\det
\begin{pmatrix}
g'(s)\\
g'(t)
\end{pmatrix}
},
\qquad\qquad		
\DR(s,t) \df \frac{
\det
\begin{pmatrix}
\gamma'(t)\\
\gamma'(s)\\
\gamma''(t)
\end{pmatrix}}
{
\det
\begin{pmatrix}
g'(t)\\
g'(s)
\end{pmatrix}
}.
}
\end{Def}

It is clear that formally $\DL(s,t) = \DR(t,s)$, but it will be more convenient to keep two different symbols 
for these expressions. 

\begin{Le}
\label{Lem310803}
Suppose that the pair~$(s,t)$ satisfies the cup equation~\eqref{urlun1}. Then\index[symbol]{$\CR$}\index[symbol]{$\CL$}
\eq{lincoef}{
\gamma(t)-\gamma(s)=\CR\gamma'(t)-\CL\gamma'(s),
}
where
$$
\CL(s,t)=\frac{\det
\begin{pmatrix}
g(s)-g(t)\\
g'(t)
\end{pmatrix}
}{\det
\begin{pmatrix}
g'(s)\\
g'(t)
\end{pmatrix}
},
\qquad
\qquad
\CR(s,t)=\frac{\dett{g(t)-g(s)}{g'(s)}}{\dett{g'(t)}{g'(s)}}.
$$
\end{Le}

\begin{Rem}
If $u,v,w \in \mathbb{R}^2$, then
\eq{3videntity}{
u \dett{w}{v}+v \dett{u}{w}+w\dett{v}{u}=0.
}
\end{Rem}

\begin{proof}[Proof of Lemma~\ref{Lem310803}]
It follows from~\eqref{3videntity}, that
\eq{lincoef2}{
g(t)-g(s)=\CR g'(t)-\CL g'(s).
}
The vectors $\gamma(t)-\gamma(s)$, $\gamma'(s)$, $\gamma'(t)$ are linearly dependent and the coefficients of the 
linear dependence are the same as for their two-dimensional projections $g(t)-g(s)$, $g'(t)$ and $g'(s)$.
\end{proof}

\begin{Rem}
\label{Rem310801}
For $s \ne t$ we have 
$\sign(\CR)=-\sign(\CL)=\sign(t-s)$ {\bf(}see~\eqref{260701}{\bf)}.
\end{Rem}

\begin{Cor}
Let $\Ch(I)$ be a chordal domain with the corresponding function $s=s(t),$ $t\in I$. Then 
\eq{clcr1}{
\CL \dettt{\gamma'(s)}{\gamma'(t)}{\gamma''(s)}ds+\CR \dettt{\gamma'(t)}{\gamma'(s)}{\gamma''(t)}dt=0.
} 
\end{Cor}

\begin{proof}
We differentiate~\eqref{urlun1} and obtain
$$
0=\dettt{\gamma''(t)}{\gamma'(s)}{\gamma(t)-\gamma(s)}dt + 
\dettt{\gamma'(t)}{\gamma''(s)}{\gamma(t)-\gamma(s)}ds=
\CR \dettt{\gamma''(t)}{\gamma'(s)}{\gamma'(t)}dt-\CL \dettt{\gamma'(t)}{\gamma''(s)}{\gamma'(s)}ds
$$
according to~\eqref{lincoef}.
\end{proof}

By~\eqref{e334}, the equality~\eqref{clcr1} can be rewritten in the following form:
\begin{equation}
\CL\DL ds=\CR\DR dt.
\end{equation}

\begin{Cor}
\begin{equation}\label{DLDRform1}
\DR=\frac{\dettt{\gamma'(t)}{\gamma(t)-\gamma(s)}{\gamma''(t)}}{\dett{g'(t)}{g(t)-g(s)}},\qquad\qquad
\DL=\frac{\dettt{\gamma'(s)}{\gamma(s)-\gamma(t)}{\gamma''(s)}}{\dett{g'(s)}{g(s)-g(t)}}.
\end{equation}
\end{Cor}

\begin{proof}
This immediately follows from~\eqref{lincoef}, \eqref{lincoef2}, and~\eqref{e334}.
\end{proof}

\begin{Le}
Let $\Ch(I)$ be a chordal domain with the corresponding function $s=s(t),$ $t\in I$. Then\textup, 
\eq{eq121201}{
d\DR(s,t)=\frac{\dettt{\gamma'(t)}{\gamma''(t)}{\gamma'''(t)}}{\dett{g'(t)}{g''(t)}}dt+
\left(\frac{\dett{g'(t)}{g'''(t)}}{\dett{g'(t)}{g''(t)}}-
\frac{\dett{g''(t)}{g(t)-g(s)}}{\dett{g'(t)}{g(t)-g(s)}}\right)\DR(s,t) dt.
}
\end{Le}

\begin{proof}
Let us differentiate $\DR$ using the representation~\eqref{DLDRform1}. First, note that the partial derivative 
with respect to $s$ is equal to zero:
$$
\frac{\dettt{\gamma'(t)}{-\gamma'(s)}{\gamma''(t)}}{\dett{g'(t)}{g(t)-g(s)}}-
\frac{\dettt{\gamma'(t)}{\gamma(t)-\gamma(s)}{\gamma''(t)}}{\dett{g'(t)}{g(t)-g(s)}^2}\dett{g'(t)}{-g'(s)}
\stackrel{{\genfrac{}{}{0pt}{-2}{\scriptscriptstyle\eqref{DLDRform1}}{\scriptscriptstyle\eqref{e334}}}}{=}
-\DR \frac{\dett{g'(t)}{g'(s)}}{\dett{g'(t)}{g(t)-g(s)}}+\DR \frac{\dett{g'(t)}{g'(s)}}{\dett{g'(t)}{g(t)-g(s)}}.
$$

Second, the partial derivative with respect to $t$ is equal to:
\eq{dr3}{
\frac{\dettt{\gamma'(t)}{\gamma(t)-\gamma(s)}{\gamma'''(t)}}{\dett{g'(t)}{g(t)-g(s)}}-
\frac{\dett{g''(t)}{g(t)-g(s)}}{\dett{g'(t)}{g(t)-g(s)}}\DR.
}
We use formulas~\eqref{lincoef} 
and~\eqref{lincoef2} to rewrite the first summand in~\eqref{dr3}:
\eq{dr1}{
\frac{\dettt{\gamma'(t)}{\gamma(t)-\gamma(s)}{\gamma'''(t)}}{\dett{g'(t)}{g(t)-g(s)}}=
\frac{\dettt{\gamma'(t)}{\gamma'(s)}{\gamma'''(t)}}{\dett{g'(t)}{g'(s)}}.
}

It follows from~\eqref{3videntity} for the vectors $g'(t)$, $g''(t)$, and $g'''(t)$, that 
$$
\dett{g'(t)}{g''(t)}\gamma'''(t)+\dett{g'''(t)}{g'(t)}\gamma''(t)+\dett{g''(t)}{g'''(t)}\gamma'(t)=
\left(0,\ 0,\  
\dettt{\gamma'(t)}{\gamma''(t)}{\gamma'''(t)}
\right).
$$
We express $\gamma'''(t)$ from this identity and substitute it into~\eqref{dr1} to obtain
$$
\frac{\dettt{\gamma'(t)}{\gamma'(s)}{\gamma'''(t)}}{\dett{g'(t)}{g'(s)}}=
\frac{\dettt{\gamma'(t)}{\gamma''(t)}{\gamma'''(t)}}{\dett{g'(t)}{g''(t)}}-
\frac{\dettt{\gamma'(t)}{\gamma'(s)}{\gamma''(t)}}{\dett{g'(t)}{g'(s)}}
\frac{\dett{g'''(t)}{g'(t)}}{\dett{g'(t)}{g''(t)}}.
$$

This formula together with~\eqref{dr1} and~\eqref{dr3}, and definition~\eqref{e334} finishes the proof.
\end{proof}

\begin{Rem}
Formula~\eqref{eq121201} can be rewritten in terms of the functions $\kappa$, $\kappa_2$, $\kappa_3$, and $\tors$ {\bf(}see~\eqref{eq211002} and~\eqref{eqKappas}{\bf)}:
\eq{eq150101}
{
\frac{d\DR(s(t),t)}{dt} = \kappa_2'g_1'\tors'+ \Big[\frac{(g_1'\kappa_2')'}{g_1'\kappa_2'} - 
\frac{\kappa_2'}{\kappa_2-\kappa}\Big]\DR(s(t),t).
}
\end{Rem}

\begin{proof}
The coincidence of the first summands of~\eqref{eq150101} and~\eqref{eq121201} follows from~\eqref{eq211003} and the identity
$$
\dett{g'(t)}{g''(t)} = g_1'g_2''-g_1''g_2' = (g_1')^2\kappa_2'.
$$
As for the second summand, we write the following chain of elementary equalities:
$$
\dett{g'(t)}{g(t)-g(s)}=g_1'\big(g_2(t)-g_2(s)\big)-g_2'\big(g_1(t)-g_1(s)\big)=
g_1'\big[g_2(t)-g_2(s)-\kappa_2\big(g_1(t)-g_1(s)\big)\big].
$$
The coefficient in the big parentheses in~\eqref{eq121201} is equal to
\begin{align*}
\frac{\dett{g'(t)}{g'''(t)}}{\dett{g'(t)}{g''(t)}}\!-\!
\frac{\dett{g''(t)}{g(t)\!-\!g(s)}}{\dett{g'(t)}{g(t)\!-\!g(s)}}
&=\frac{\partial}{\partial t}\log \frac{\dett{g'(t)}{g''(t)}}{\dett{g'(t)}{g(t)\!-\!g(s)}}
=
\frac{\partial}{\partial t}\log \frac{\kappa_2'g_1'}{g_2(t) \!-\! g_2(s) \!-\! \kappa_2\big(g_1(t)\!-\!g_1(s)\big)} = \!\!\!\!\!\!
\\
&=\rule{0pt}{20pt}\frac{(\kappa_2'g_1')'}{\kappa_2'g_1'}- \frac{g_2' - 
\kappa_2' \big(g_1(t)-g_1(s)\big)-\kappa_2g_1'}{\big(g_1(t)-g_1(s)\big)(\kappa-\kappa_2)}=
\frac{(\kappa_2'g_1')'}{\kappa_2'g_1'}+\frac{\kappa_2'}{\kappa-\kappa_2}.\qedhere
\end{align*}
\end{proof}

\subsection{Cup construction}

\begin{St}
\label{St3108}
Suppose that the torsion of $\gamma$ changes its sign from $+$ to $-$ at the point $t_0,$ i.\,e.\textup, 
$\torsion$ is positive in a left neighborhood of $t_0$ and is negative in a right one. Then\textup, for 
$\delta>0$ sufficiently small there exists a chordal domain $\Ch(t_0,t_0+\delta)$ with the standard Bellman 
candidate on it.
\end{St}

Proposition~\ref{St3108} emphasizes the importance of Condition~\ref{reg} because cups can originate only at the points $t_0 = c_i$, see Definition~\ref{roots}.

First, we need to obtain several useful formulas.

\begin{Le}
\label{Le310801}
For any $s,t$ the following identity holds\textup:

\begin{align}
\label{eq280701}
\det
\begin{pmatrix}
\gamma'(t)\\
\gamma'(s)\\
\gamma(t)-\gamma(s)
\end{pmatrix}
&=g_1'(t)g_1'(s)\int_{s}^t\tors'(\sigma)
\bigg(\int_{s}^\sigma\int_\sigma^t\kappa_2'(u)\kappa_2'(v)\big(g_1(v)-g_1(u)\big)dv\,du\bigg)d\sigma=
\\
\label{eq150801}
&=-g_1'(t)g_1'(s) \int_{s}^t \tors(\sigma)
\bigg(\int_s^t \kappa_2'(\sigma) \kappa_2'(v)\big(g_1(v)-g_1(\sigma)\big)dv \bigg)d\sigma.
\end{align}
\end{Le}

\begin{proof}
\begin{multline*}
\det
\begin{pmatrix}
\gamma'(t)\\
\gamma'(s)\\
\gamma(t)-\gamma(s)
\end{pmatrix}
=
\int_s^t
\det
\begin{pmatrix}
\gamma'(t)\\
\gamma'(s)\\
\gamma'(\tau)
\end{pmatrix} d\tau
=
\int_s^t
g_1'(t)g_1'(s)g_1'(\tau)
\det
\begin{pmatrix}
1& \kappa_2(t)&\kappa_3(t) \\
1& \kappa_2(s)&\kappa_3(s) \\
1& \kappa_2(\tau)&\kappa_3(\tau) 
\end{pmatrix} d\tau=
\allowdisplaybreaks\\
=-g_1'(t)g_1'(s)
\int_s^t
g_1'(\tau)
\det
\begin{pmatrix}
\kappa_2(t)-\kappa_2(\tau)&\kappa_3(t)-\kappa_3(\tau) \\
\kappa_2(\tau)-\kappa_2(s) &\kappa_3(\tau)-\kappa_3(s)
\end{pmatrix} d\tau=
\allowdisplaybreaks\\
=-g_1'(t)g_1'(s)
\int_s^t
g_1'(\tau)
\int_\tau^t\int_s^\tau
\det
\begin{pmatrix}
\kappa_2'(v)&\kappa_3'(v) \\
\kappa_2'(u)&\kappa_3'(u)
\end{pmatrix}  du\,dv\,d\tau=
\allowdisplaybreaks
\\
=g_1'(t)g_1'(s)
\int_s^t
g_1'(\tau)
\int_\tau^t\int_s^\tau
\kappa_2'(v)\kappa_2'(u)
\int_{u}^v \tors'(\sigma)\,d\sigma\,du\,dv\,d\tau=
\allowdisplaybreaks
\\
=g_1'(t)g_1'(s)
\int_s^t
\tors'(\sigma)
\bigg(
\int_s^\sigma 
\int_\sigma^t
\kappa_2'(u)\kappa_2'(v)
\int_u^v
g_1'(\tau)\,d\tau\,dv\,du
\bigg) d\sigma
=\allowdisplaybreaks
\\
=g_1'(t)g_1'(s)\int_{s}^t \tors'(\sigma)
\bigg(\int_{s}^\sigma\int_\sigma^t\kappa_2'(u)\kappa_2'(v)\big(g_1(v)-g_1(u)\big)dv\,du\bigg)d\sigma 
\stackrel{\hbox{\tiny integration by parts}}{=}
\allowdisplaybreaks
\\
 =-g_1'(t)g_1'(s)\int_{s}^t\tors(\sigma)\bigg(\int_\sigma^t \kappa_2'(\sigma)\kappa_2'(v)\big(g_1(v)-g_1(\sigma)\big)dv-
\int_{s}^\sigma\kappa_2'(u)\kappa_2'(\sigma)\big(g_1(\sigma)-g_1(u)\big)du\bigg)d\sigma=
\allowdisplaybreaks
\\
=-g_1'(t)g_1'(s)\int_{s}^t\tors(\sigma)\bigg(\int_s^t\kappa_2'(\sigma)\kappa_2'(v)\big(g_1(v)-g_1(\sigma)\big)dv\bigg)d\sigma.
\end{multline*}
\end{proof}

\begin{Cor}
For any $s,t$ the following identity holds true\textup:
\eq{eq150802}{
\begin{aligned}
\det
\begin{pmatrix}
\gamma''(t)\\
\gamma'(s)\\
\gamma(t)-\gamma(s)
\end{pmatrix}
&=\frac{g_1''(t)}{g_1'(t)}
\det \begin{pmatrix}
\gamma'(t)\\
\gamma'(s)\\
\gamma(t)-\gamma(s)
\end{pmatrix}-\\
&\quad-g_1'(t)g_1'(s)\kappa_2'(t)\int_s^t \big(\tors(u) -\tors(t)\big)\kappa_2'(u)\big(g_1(t)-g_1(u)\big)du.
\end{aligned}
}
\end{Cor}

\begin{proof}
The result is obtained by the direct differentiation of~\eqref{eq150801}: the first summand is obtained 
from the differentiation of the factor $g_1'(t)$, the second on is the result of the differentiation of the integral.
\end{proof}

\begin{Le}
\label{eq310802}
Let $s,t,c \in \mathbb{R}$ and $s<c<t$. Suppose that a function $\Phi$ is continuous on $[s,t],$ nondecreasing 
on $(s,c),$ nonincreasing on $(c,t),$ and not constant on $[s,t]$. Let functions~$\phi$ and~$\psi$ on $[s,t]$ 
be such that $\phi>0$ and $\psi'>0$. Suppose that
\eq{eq150804}{
\int_{s}^t \Phi(\sigma) \bigg(\int_s^t \phi(\sigma) \phi(v)\big(\psi(v)-\psi(\sigma)\big)dv \bigg)d\sigma=0.
}
Then 
\eq{eq150803}{
\int_s^t \big(\Phi(u) -\Phi(t)\big)\phi(u)\big(\psi(t)-\psi(u)\big)du>0.
}
\end{Le}

\begin{proof}
The inequality~\eqref{eq150803} and the equality~\eqref{eq150804} do not change if we add a constant to the function~$\psi$. 
Therefore we may assume that $\psi(t)=0$. Under this assumption $\psi<0$ on $(s,t)$.

We rewrite~\eqref{eq150804} in the following way:
\eq{eq210801}{
\int_s^t \Phi(\sigma) \phi(\sigma)\,d\sigma \ \int_s^t\phi(v)\psi(v)\,dv 
= \int_s^t \Phi(\sigma)\phi(\sigma)\psi(\sigma)\,d\sigma \ \int_s^t \phi(v)\,dv.
}
We see that the problem is not sensitive to adding a constant to $\Phi$. The integral $\int \phi \psi$ is 
strictly negative, so we can assume $\int_s^t \Phi \phi \psi =0$ by adding a constant to $\Phi$, if necessary. It now follows from~\eqref{eq210801} that 
$\int_s^t \Phi \phi =0$. Under  our assumptions,~\eqref{eq150803} takes the following form 
$$
\Phi(t) \int \phi \psi>0,
$$
which is equivalent to $\Phi(t)<0$. If it is not true, then $\Phi(t)\geq 0$ and $\Phi \geq 0$ on $[c,t]$. 
We can find a point $\sigma_0\in(s,c)$, such that $\Phi\geq0$ on $(\sigma_0,t)$ and $\Phi\leq0$ on $(s,\sigma_0)$. However,
$$
\int_s^t \Phi(\sigma) \phi(\sigma) (\psi(\sigma)-\psi(\sigma_0))d\sigma = 0,
$$ 
while the integrand is non-negative. Thus the integrand is identically zero, therefore $\Phi$ vanishes as well. 
This contradicts our requirement that $\Phi$ is not constant. 
\end{proof}

\begin{Cor}
\label{Cor310801}
Let $s,t,t_0 \in \mathbb{R}$ and $s<t_0<t$. Suppose that $\tors$ is nondecreasing on $(s,t_0),$ nonincreasing 
on $(t_0,t),$ and is not constant on $[s,t]$. Suppose that the pair $(s,t)$ satisfies the cup equation~\eqref{urlun1}. 
Then $\DR(s,t)<0$ and $\DL(s,t)<0$.
\end{Cor}

\begin{proof}
We will prove only $\DR(s,t)<0$, the remaining inequality is symmetric. 
 
We want to use Lemma~\ref{eq310802} with $\Phi =\tors$, $\psi = g_1$, $\phi = \kappa_2'$, and $c=t_0$. 
Then,~\eqref{eq150804} holds due to~\eqref{eq150801}. Thus we obtain~\eqref{eq150803}, 
which is equivalent to 
$$
\det
\begin{pmatrix}
\gamma''(t)\\
\gamma'(s)\\
\gamma(t)-\gamma(s)
\end{pmatrix}<0
$$ 
by~\eqref{eq150802}. We use~\eqref{lincoef} to rewrite this in terms of $\DR$:
$$
0> \dettt{\gamma''(t)}{\gamma'(s)}{\gamma(t)-\gamma(s)}=\CR \dettt{\gamma''(t)}{\gamma'(s)}{\gamma'(t)}
=-\CR\DR\dett{g'(t)}{g'(s)}.
$$
We recall that $\CR>0$ and $\dett{g'(t)}{g'(s)}<0$ when $t>s$. Thus, we obtain the claimed inequality $\DR<0$. 
\end{proof}

\begin{proof}[Proof of Proposition~\ref{St3108}]
Take $\theta_0>0$ to be a small number such that the torsion $\torsion$ of $\gamma$ is positive
 on $(t_0-\theta_0,t_0)$ and is negative on $(t_0,t_0+\theta_0)$. Recall that the sign of $\torsion$ 
coincides with the sign of $\tors'$, see~\eqref{eq211003}. 

Consider the following function $h$: 
$$h(t,\theta)=\dettt{\gamma'(t)}{\gamma'(t-\theta)}{\gamma(t)-\gamma(t-\theta)},\quad
\theta\in(0,\theta_0],\ t\in [t_0,t_0+\theta].
$$
This function is continuous and differentiable on its domain. It follows from the conditions on the curve $g$ 
that $\kappa_2'(u)\kappa_2'(v)(g_1(v)-g_1(u))>0$ for any $u,v$ such that $u<v$. For any $\theta\in(0,\theta_0]$ 
the measure $\tors'$ is positive on $(t_0-\theta,t_0)$ and is negative on $(t_0,t_0+\theta)$, therefore 
formula~\eqref{eq280701} implies that $h(t_0,\theta)>0>h(t_0+\theta,\theta)$ for any $\theta\in(0,\theta_0]$. 
So, for any such $\theta$ there exists $t_\theta \in (t_0,t_0+\theta)$ such that $h(t_\theta,\theta)=0$. 
We will show that such $t_\theta$ is unique for any $\theta \in (0,\theta_0]$.

Corollary~\ref{Cor310801} implies that $\DR = \DR(t_\theta-\theta,t_\theta)$ and 
$\DL = \DL(t_\theta-\theta,t_\theta)$ are negative. We use formula~\eqref{lincoef} to obtain
$$
\frac{\partial h}{\partial t}(t_\theta,\theta)
=\dettt{\gamma''(t_\theta)}{\gamma'(t_\theta-\theta)}{\gamma(t_\theta)-\gamma(t_\theta-\theta)}
+\dettt{\gamma'(t_\theta)}{\gamma''(t_\theta-\theta)}{\gamma(t_\theta)-\gamma(t_\theta-\theta)}
=\dett{g'(t_\theta)}{g'(t_\theta-\theta)}(\CL\DL-\CR\DR)<0, 
$$
because $\dett{g'(t_\theta)}{g'(t_\theta-\theta)}<0$, $\CR>0>\CL$, $\DR<0$, and $\DL<0$. It follows that 
the function $h(\cdot,\theta)$ changes its sign from plus to minus at any root on $(t_0,t_0+\theta)$. 
Therefore, it has only one root $t_\theta$ there.

We now calculate 
$$
\frac{\partial h}{\partial \theta}(t_\theta,\theta)
=-\dettt{\gamma'(t_\theta)}{\gamma''(t_\theta-\theta)}{\gamma(t_\theta)-\gamma(t_\theta-\theta)}
=-\dett{g'(t_\theta)}{g'(t_\theta-\theta)}\CL\DL>0. 
$$
It now follows from the implicit function theorem that $t_\theta$ is a differentiable function of $\theta$ and 
$$
\frac{\partial t_\theta}{\partial \theta} = \frac{\CL \DL}{\CL\DL-\CR\DR} \in(0,1)
$$ 
for $\theta \in (0,\theta_0)$. Thus, $t_\theta -\theta$ decreases. 
Moreover, $t_\theta \in (t_0,t_0+\theta)$ for any $\theta$, therefore, $\lim_{\theta \to 0+} t_\theta = t_0$. 

Let $\delta = t_{\theta_0}-t_0$, $I = [t_0,t_{\theta_0}]$. We define the function 
$s\colon I \to [t_{\theta_0}-\theta_0,t_0]$ by the rule $s(t_\theta) = t_\theta-\theta$ for $\theta \in (0,\theta_0]$, 
and $s(t_0)=t_0$. This function decreases, and the corresponding family of segments $S(t) =[g(t),g(s(t))]$, $t \in I$, 
forms a chordal domain. By the definition of $t_\theta$ the pair $(s(t),t)$ satisfies the cup equation~\eqref{urlun1} 
for any $t \in I$, and moreover $\DR(s,t)<0$, $\DL(s,t)<0$. Therefore this chordal domain satisfies the conditions of 
Proposition~\ref{NewLightChordalDomainCandidate} and the proof is finished. 
\end{proof}

\subsection{How to grow a chordal domain from a single chord}

Now we turn to a more general situation. In this subsection we work under the following assumption.

\begin{Cond}
\label{StartChord}
There is a pair of numbers $a_0,b_0,$ $a_0<b_0,$ satisfying the cup equation~\eqref{urlun1}\textup, i.\,e.\textup, 
$$
\dettt{\gamma'(a_0)}{\gamma'(b_0)}{\gamma(b_0)-\gamma(a_0)} = 0,
$$
such that for some $\theta_0>0$ the determinant 
$$
L(t)=\dettt{\gamma'(t)}{\gamma'(a_0)}{\gamma'(b_0)} 
$$ 
is negative for $t \in (b_0,b_0+\theta_0)$ and is positive for $t\in (a_0-\theta_0,a_0)$.
\end{Cond}

\begin{Rem}\label{Rem230301}
If $(a_0,b_0)$ satisfies the cup equation~\eqref{urlun1} and the inequalities $\DR(a_0,b_0)<0$ and $\DL(a_0,b_0)<0$, 
then Condition~\ref{StartChord} holds.
\end{Rem}

\begin{proof}
It follows immediately from the fact that $L(a_0)=L(b_0)=0$ and 
\begin{equation*}
\frac{L'(a_0)}{\DL(a_0,b_0)}=\frac{L'(b_0)}{\DR(a_0,b_0)} = \dett{g'(a_0)}{g'(b_0)}>0.\qedhere
\end{equation*}
\end{proof}

We are ready to present the main result of this subsection.
\begin{St}
\label{St010901}
Under Condition~\textup{\ref{StartChord}} there exists $\delta>0$ and two differentiable functions 
$a,b \colon [0,\delta] \to \mathbb{R}$ such that
\begin{enumerate}
\item $a'<0<b'$\textup, $a(0)=a_0$, $b(0)=b_0$\textup;
\item $b(\theta)-a(\theta)=b_0-a_0+\theta$ and the pair $(a(\theta),b(\theta))$ satisfies the cup equation~\eqref{urlun1} 
for any $\theta \in [0,\delta]$\textup;
\item $\DL(a(\theta),b(\theta))<0$ and $\DR(a(\theta),b(\theta))<0$ for any $\theta \in (0,\delta]$\textup.
\end{enumerate}
\end{St}

\begin{proof}
We can always assume that $f'(a_0)=f'(b_0)=0$. If it is not the case, then we can replace the function $f$ by 
$f+c_1 g_1+c_2g_2$ with appropriate $c_1,c_2 \in \mathbb{R}$. 
In such a case, the determinants in the cup equation and in the differentials do not depend on the choice of $c_1$ and $c_2$. Since $\dett{g'(a_0)}{g'(b_0)}\ne 0$, 
we can choose the constants $c_1,c_2$ in such a way that $f'(a_0)=f'(b_0)=0$ for the modified function. 
In this situation the sign of $L(t)$ coincides with the sign of $f'(t)$, therefore Condition~\ref{StartChord} 
implies that $f'$ is negative on the right neighborhood of $b_0$. From the cup~\eqref{urlun1} we obtain $f(a_0)=f(b_0)$.

Let us prove that $\kappa_3'$ is non-positive in some right neighborhood of $b_0$. The function $\tors$ is 
piecewise monotone by Condition~\ref{reg}, therefore we can find a right neighborhood of $b_0$ where it is 
monotone and does not change its sign. Then, $\kappa_3' = \tors \kappa_2'$ does not change its sign 
in the same right neighborhood of $b_0$ as well. We know that $\kappa_3(b_0)=0$ and  
\eq{eq290703}{
\kappa_3 = \frac{f'}{g_1'}< 0 \quad \text{to the right of } b_0,
}
therefore $\kappa_3'\leq 0$ there. Without loss of generality, in what follows we assume that 
\eq{eq280702}{
\kappa_3'\leq 0 \quad \text{ on } (b_0,b_0+\theta_0).   
}

We write the following inequality for $t\in (b_0,b_0+\theta_0)$:
$$
f(t)-f(a_0)=f(t)-f(b_0)=\int_{b_0}^t f'(\tau)\,d\tau
=\int_{b_0}^t g_1'(\tau)\kappa_3(\tau)\,d\tau\geq\big(g_1(t)-g_1(b_0)\big)\kappa_3(t)
$$
because $g_1'>0$ and $\kappa_3'\leq 0$. We use this inequality to estimate the determinant:
$$
\begin{aligned}
\dettt{\gamma'(t)}{\gamma'(a_0)}{\gamma(t)-\gamma(a_0)} &\leq 
\det
\begin{pmatrix}
g_1'(t)& g_2'(t)&f'(t)\\
g_1'(a_0)& g_2'(a_0)& 0\\
g_1(t)-g_1(a_0)& g_2(t)-g_2(a_0)& (g_1(t)-g_1(b_0))\kappa_3(t)
\end{pmatrix}=
\\
&=f'(t)g_1'(a_0)
\det
\begin{pmatrix}
1& \kappa_2(t)& 1\\
1& \kappa_2(a_0)& 0\\
g_1(t)-g_1(a_0)& g_2(t)-g_2(a_0)& g_1(t)-g_1(b_0)
\end{pmatrix}.
\end{aligned}
$$
When $t$ approaches $b_0$ from the right, the latter determinant tends to 
$$
g_2(b_0)-g_2(a_0)-\kappa_2(a_0)\big(g_1(b_0)-g_1(a_0)\big).
$$
This expression is strictly positive due to convexity of the curve $g$. 
Since $f'(t)<0$, this leads to the inequality 
\eq{eq180901}{
\dettt{\gamma'(t)}{\gamma'(a_0)}{\gamma(t)-\gamma(a_0)}<0
}
for $t>b_0$ sufficiently close to $b_0$. Reducing $\theta_0$ if needed, 
we can assume that~\eqref{eq280702} and~\eqref{eq180901}  
hold for all $t \in (b_0,b_0+\theta_0)$. 

Similarly, we can prove that for all $t$ in a left neighborhood of $a_0$ we have $\kappa_3'(t)\geq 0$ and 
\eq{eq180902}{
\dettt{\gamma'(b_0)}{\gamma'(t)}{\gamma(b_0)-\gamma(t)}>0.
}
We can reduce $\theta_0$ once more and in what follows assume that~\eqref{eq180902} holds for all 
$t \in (a_0-\theta_0,a_0)$. Consider the function 
$$
h(t,\theta) = \dettt{\gamma'(t)}{\gamma'(t-(b_0-a_0)-\theta)}{\gamma(t)-\gamma(t-(b_0-a_0)-\theta)}, 
\qquad \theta \in (0,\theta_0], \quad t \in [b_0,b_0+\theta].
$$ 
This function is continuous and differentiable on its domain. It follows from~\eqref{eq180901} 
and~\eqref{eq180902} that 
$$
h(b_0,\theta)>0>h(b_0+\theta,\theta), \qquad \theta \in (0,\theta_0).
$$
Thus, for any such $\theta$ there exists $b_\theta \in (b_0,b_0+\theta)$ such that $h(b_\theta,\theta)=0$. 
We will show that such $b_\theta$ is unique. Let $a_\theta=b_\theta-(b_0-a_0)-\theta$. We would like
to verify that the functions $a(\theta)=a_\theta$ and $b(\theta)=b_\theta$ are the function whose existence
is stated in the proposition.

The remaining part of the proof is based on the following lemma.

\begin{Le}
\label{Lem180901}
If $\theta \in (0,\theta_0)$ is sufficiently small\textup, then $\DR(a_\theta,b_\theta)<0$ 
and $\DL(a_\theta,b_\theta)<0$.
\end{Le} 

The proof of this lemma is rather technical, we present it after completing the proof of the proposition. 

The remaining part of the proof %of Proposition~\ref{St010901} 
literally repeats the proof of Proposition~\ref{St3108}. We can assume that~$\theta_0$ is so small that the statement 
of Lemma~\ref{Lem180901} holds for all $\theta$ in the interval $(0,\theta_0)$. 
We calculate derivative of~$h$:
\eq{eq280703}{
\frac{\partial h}{\partial t}(b_\theta,\theta)
=\dett{g'(b_\theta)}{g'(a_\theta)}\cdot\big[\CL\DL-\CR\DR\big](a_\theta,b_\theta)<0.
}
Therefore the function $h(\,\cdot\,,\theta)$ has unique root on $(b_0,b_0+\theta)$. We save the notation $b_\theta$ 
for this unique root. One can calculate 
the partial derivative with respect to $\theta$:
\eq{eq280704}{
\frac{\partial h}{\partial \theta}(b_\theta,\theta)
=-\dett{g'(b_\theta)}{g'(a_\theta)}\cdot \big[\CL\DL\big](a_\theta,b_\theta)>0. 
}
The implicit function theorem implies that $b_\theta$ is a differentiable function of $\theta$ and 
\eq{eq280705}{
\frac{\partial b_\theta}{\partial \theta} = \frac{\CL \DL}{\CL\DL-\CR\DR} \in(0,1)
} 
for $\theta \in (0,\theta_0)$. It follows that $a_\theta=b_\theta-\theta-(b_0-a_0)$ decreases and 
$$
\lim_{\theta \to 0+} b_\theta = b_0, \qquad \lim_{\theta \to 0+} a_\theta = a_0.
$$
Proposition~\ref{St010901} is proved.
\end{proof}

\begin{proof}[Proof of Lemma~\ref{Lem180901}]
We will deal with $\DR$ only, the other inequality is symmetric. Since  $f'(b_0)=0$ and $f'(t)< 0$ for 
$t \in (b_0,b_0+\theta_0)$, we have $f''(b_0)\leq 0$. Since also $f'(a_0)=0$, this yields $\DR(a_0,b_0)\leq 0$. If $\DR(a_0,b_0)<0$, 
then $\DR(a_\theta,b_\theta)<0$ for small $\theta$ by continuity. The same reasoning shows that $\DL(a_0,b_0)\leq 0$, 
and if $\DL(a_0,b_0)<0$, then $\DL(a_\theta,b_\theta)<0$ for $\theta$ sufficiently small.

Let us consider the case when one of $\DL(a_0,b_0)$ and $\DR(a_0,b_0)$ is strictly negative while the other one 
is zero. The case where both differentials vanish will be considered later. Without loss of generality assume $\DL(a_0,b_0)< 0$ and $\DR(a_0,b_0) = 0$. Then $f''(b_0)=0$, therefore 
$$
\kappa_3'(b_0) = \frac{f''(b_0)g_1'(b_0) - f'(b_0)g_1''(b_0)}{(g_1'(b_0))^2} = 0; 
\qquad \tors(b_0) = \frac{\kappa_3'(b_0)}{\kappa_2'(b_0)}=0.
$$
Condition~\ref{reg} implies that $\tors'$ does not change its sign in a right neighborhood of $b_0$.  
From~\eqref{eq280702} we know that $\tors$ is non-positive on $(b_0,b_0+\theta_0)$, therefore $\tors'\leq 0$ in 
a right neighborhood of $b_0$ as well. Let us note that 
\eq{eq290704}{
\tors' \text{ is not identically zero in any right neighborhood of }  b_0.}
Assume the contrary, $\tors'$ is identically zero in a right neighborhood of $b_0$. Then the function $\tors$ is constant there, and coincides with $\tors(b_0)=0$. Therefore $\kappa_3'=0$ in the same neighborhood. Then $\kappa_3$ is constant there, and is equal to $\kappa_3(b_0)=0$. This contradicts~\eqref{eq290703}.

Inequalities~\eqref{eq280703} and~\eqref{eq280704} hold for $\theta =0$, therefore by continuity they 
also hold for $\theta$ sufficiently small. Thus, 
\eq{eq280706}{
\frac{\partial b_\theta}{\partial \theta} = \frac{\CL \DL}{\CL\DL-\CR\DR} > 0. 
}
Let $F(\theta) = \DR(a_\theta,b_\theta)$. From~\eqref{eq150101} we obtain
$$
F'(\theta) = \Big(\kappa_2'(b_\theta)g_1'(b_\theta)\tors'(b_\theta) + \Big[\frac{(g_1'\kappa_2')'}{g_1'\kappa_2'}(b_\theta) - 
\frac{\kappa_2'}{\kappa_2-\kappa}(b_\theta)\Big]F(\theta) \Big) \frac{\partial b_\theta}{\partial \theta}.
$$
We use the integrating factor
$$
M(\theta) = \exp\Big(-\int_{0}^\theta \Big[\frac{(g_1'\kappa_2')'}{g_1'\kappa_2'}(b_\eta) - 
\frac{\kappa_2'}{\kappa_2-\kappa}(b_\eta)\Big]\frac{\partial b_\eta}{\partial \eta} d \eta\Big), \qquad \theta>0,
$$
to conclude
\eq{eq290702}{
\frac{d}{d \theta} \big(F(\theta) M(\theta)\big) 
= M(\theta)\kappa_2'(b_\theta)g_1'(b_\theta)\tors'(b_\theta) \frac{\partial b_\theta}{\partial \theta} \leq 0.
}
We know that $F(0)=0$, therefore $F(\eta) \leq 0$ for $\eta>0$. If $F(\eta) = 0$ for some $\eta>0$, then 
the expression in~\eqref{eq290702} vanishes on $(0,\eta)$. Therefore $\tors'(b_\theta) = 0$ on $(0,\eta)$ as well, 
which contradicts to~\eqref{eq290704}. Hence, $\DR(a_\theta,b_\theta) = F(\theta)  < 0$ for sufficiently 
small $\theta>0$. The inequality $\DL(a_\theta,b_\theta) <0$ follows from continuity and our assumption $\DL(a_0,b_0)<0$.

It remains to prove the statement for the case $\DL(a_0,b_0)=\DR(a_0,b_0) =0$. In this case $f''(a_0)=f''(b_0)=0$, 
$\tors'\leq 0$ in a right neighborhood of $b_0$ and $\tors'\geq 0$ in a left neighborhood of $a_0$. 
Instead of $f$ consider the function 
$$
\tilde f = (1-\chi_{[a_0,b_0]})f + \chi_{[a_0,b_0]}f(a_0),
$$
and the corresponding function $\tilde\tors$. The function $\tilde f$ is $C^2$-smooth. For any $\theta$ 
the pair $(a_\theta,b_\theta)$ satisfies the cup equation~\eqref{urlun1} with $\tilde f$ instead of $f$. 
Also, outside $(a_0,b_0)$ we have $\tilde\tors = \tors$. The function $\tilde\tors = \tors$ is not constant 
on $(b_0,b_\theta)$ by~\eqref{eq290704}. For any $t_0 \in (a_0,b_0)$ we obtain that $\tilde\tors$ satisfies 
the assumptions of Corollary~\ref{Cor310801}, which yields $\DR(a_\theta,b_\theta)<0$ and $\DL(a_\theta,b_\theta)<0$.
\end{proof}

\section{Forces}\label{s34bis}

\begin{Def}
\label{Def031001}
Let $B$ be a Bellman candidate on a fence $\Omega(I)$. The function
\eq{eq121001}{
\FF_I=\frac{\kappa_2-\kappa}{\kappa_2'}\beta_2' 
}
on $I$ is called its \emph{force function}.\index[symbol]{$\FF_I$}
\end{Def}

\begin{Rem}
Suppose that $\Ch(I)$ is a chordal domain with the corresponding function $s\colon I \to \mathbb{R}$, and $B$ is 
the standard Bellman candidate on $\Ch(I)$. Then 
$$
\FF_I(t) = \frac{D}{g_1'(t)\kappa_2'(t)}, \qquad t \in I,
$$ 
where $D=\DR(s,t)$ if $t>s(t)$ and $D=\DL(t,s)$ if $t<s(t)$.
\end{Rem}
\begin{proof}
This is a direct consequence of formulas~\eqref{betaprime2},~\eqref{e334}, and~\eqref{eq121001}.
\end{proof}

\begin{Rem}
The force function $\FF$ is non-positive on the domain of its definition.
\end{Rem}
\begin{proof}
If $\Omega(I)$ is a right tangent domain, then $\beta_2'\leq 0$ (see Remark~\ref{Rem201001}) and 
$\kappa_2-\kappa>0$ (this follows from geometric properties of $\Omega$). For the case of left tangents 
both inequalities are opposite. Therefore, \hbox{$\FF \leq 0$}. 

For the case of a chordal domain both differentials $\DL$ and $\DR$ are negative, thus $\FF \leq 0$.
\end{proof}

\subsection{Gluing fences}\label{SecGluFen}
In this part we present necessary and sufficient conditions for gluing two Bellman candidates on fences with 
a common point. By this we mean the following situation. 

\begin{St}\label{St260901}
Let $t_-,t_0,t_+ \in \mathbb{R},$ $t_-\leq t_0\leq t_+,$ let $I_- = [t_-,t_0],$ $I_+=[t_0,t_+]$. Suppose that 
$\Omega(I_\pm)$ are two fences with Bellman candidates $B_\pm$ on them with the corresponding functions $w_\pm,$ 
$S_\pm, \beta_\pm$. Let $\interior\Omega(I_-)\cap \interior\Omega(I_+)=\varnothing$. Suppose that 
\eq{eq260902}{
\beta_+(t_0)=\beta_-(t_0)\df \beta_0.
}
Then the function $B$ defined by the formula
\eq{eq250902}{
B(x)=
\begin{cases}
B_-(x), & x \in \Omega(I_-);\\
f(t_0)+\scalp{\beta_0}{x-g(t_0)}, & x \in \indconv(S_-(t_0)\cup S_+(t_0));\\
B_+(x), & x \in \Omega(I_+)
\end{cases}
}
is a $C^1$-smooth Bellman candidate on its domain \textup(see Definition~\ref{InducedConvexSet} of induced convex hull\textup).
\end{St}
\begin{proof}
The $C^1$-smoothness of $B$ follows from \eqref{eq260902} and the $C^1$-smoothness of $B_\pm$. The concavity of~$B$ is then implied by Proposition~\ref{ConcatenationOfConcaveFunctions}. 
%follows from the fact that the graph of the function $B$ on the added domain 
%$\conv(S_-(t_0),S_+(t_0))$ lies in the common tangent plane to the graphs of $B_\pm$ at points of $S_\pm(t_0)$.
\end{proof}

\begin{Cor}\label{cor250901}
Conclusion of Proposition~\textup{\ref{St260901}} holds true if we replace~\eqref{eq260902} by the forces 
equality $\FF_{I_+}(t_0)=\FF_{I_-}(t_0)$.
\end{Cor}
\begin{proof}
Since $\kappa_2$ and $\kappa_3$ do not depend on the fences, it follows from~\eqref{nabla11} 
and~\eqref{diffbeta2} applied to~$t_0$ that $\beta_+(t_0)=\beta_-(t_0)$. 
\end{proof}

Now we consider specific cases of gluing fences we have investigated in Sections~\ref{s32},~\ref{s33}. 

\subsection{Forces and tails}
\label{Subsec362}

In Definition~\ref{Def031001} forces were defined on fences. For a given chordal domain we have 
a force function defined inside this domain. We wish to extend the forces outside for the purpose of continuation of a Bellman candidate from a chordal domain via tangents. 

Let $I=[t_-,t_+]$ be an interval, $\Ch(I)$ be a chordal domain with the corresponding function 
$s\colon I \to \mathbb{R}$, such that $s(t)\leq t$ for $t \in I$. We define a right force of the chordal 
domain $\Ch(I)$ in the following way. We define $\tr$ to be the supremum of numbers $\tau$, $\tau\geq t_+$, such that there exists 
a standard Bellman candidate on the right tangent domain $\Rt([t_+,\tau])$ satisfying  
$$
\FF_{[t_-,t_+]}(t_+)=\FF_{[t_+,\tau]}(t_+).
$$
Note that $\tr$ can be equal to $+\infty$. The concatenation of the functions $\FF_I$ and $\FF_{[t_+,\tr]}$ will be called the right force function of the chordal domain $\Ch(I)$.

\begin{Def}
\label{Def031002}\index[symbol]{$\Fr(t,\Ch(I))$}
The function 
$$
\Fr(t, \Ch(I)) = 
\begin{cases}
\FF_I(t),& t \in I;\\
\FF_{[t_+,\tr]}(t),& t \in [t_+,\tr],
\end{cases}
$$
is called the \emph{right force function} of the chordal domain $\Ch(I)$. The set $(t_-,\tr)$ is called 
the \emph{right tail} of the chordal domain $\Ch(I)$.\index[symbol]{$\tr$}
\end{Def}

\begin{Rem}\label{rem220202}
If $\tr<+\infty$, then $\FF_{[t_+,\tr]}(\tr)=0$.
\end{Rem}
\begin{proof}
This immediately follows from maximality of $\tr$, continuity and nonpositivity of the force.
\end{proof}

\begin{Rem}
\label{rem220203} 
If $t_+<\tr<+\infty$, then there exists $\bar t$, $\bar t<\tr$, such that the torsion of the curve $\gamma$ 
is strictly positive on $(\bar t,\tr)$. If $\tr>t_1\geq t_+$ and the torsion of the curve $\gamma$ is nonpositive 
on an interval $(t_1,t_2)$, then $\tr>t_2$.
\end{Rem}

\begin{proof}
Due to Condition~\ref{reg} we can find $\bar t$, $\bar t<\tr$, sufficiently close to $\tr$ such that either 
$\tors'>0$ on $(\bar t,\tr)$ or $\tors'\leq 0$ on $(\bar t,\tr)$. If  $\tors'\leq 0$, then we use~\eqref{betaprime} 
for $t_0=\bar t$, $t= \tr$ and obtain that $\FF_{[t_+,\tr]}(\tr)<0$, which contradicts to Remark~\ref{rem220202}. 
Thus $\tors'>0$ on $(\bar t,\tr)$. The second claim follows immediately from the first one.
\end{proof}

Similarly, we define the left force function and the left tail of the chordal domain. Let $I=[t_-,t_+]$ be 
an interval, $\Ch(I)$ be a chordal domain with the corresponding function $s\colon I \to \mathbb{R}$ such that 
$s(t)\geq t$ for $t \in I$. We define $\tl$ to be the infimum of the numbers $\tau$, $\tau\leq t_-$, such that there exists 
a standard Bellman candidate on the left tangent domain $\Lt([\tau,t_-])$  satisfying  
$$
\FF_{[\tau,t_-]}(t_-)=\FF_{[t_-,t_+]}(t_-).
$$
Note that $\tl$ can be equal to $-\infty$.

\begin{Def}
\label{Def031003}
\index[symbol]{$\Fl(t,\Ch(I))$}
The function 
$$
\Fl(t, \Ch(I)) = 
\begin{cases}
\FF_{[\tl,t_-]}(t),& t \in [\tl,t_-];\\
\FF_I(t),& t \in I
\end{cases}
$$
is called the \emph{left force function} of the chordal domain $\Ch(I)$. The set $(\tl,t_+)$ is called 
the \emph{left tail} of the chordal domain $\Ch(I)$.\index[symbol]{$\tl$}
\end{Def}

The following remark is the ``left'' analog of Remarks~\ref{rem220202} and~\ref{rem220203}.
\begin{Rem}\label{Rem200401}
If $\tl>-\infty$, then $\FF_{[\tl,t_-]}(\tl)=0$. If $-\infty<\tl<t_-$, then there exists $\bar t$, $\tl<\bar t$, 
such that the torsion of the curve $\gamma$ is strictly negative on $(\tl,\bar t)$. If $\tl<t_2\leq t_-$ and 
the torsion of the curve $\gamma$ is nonnegative on an interval $(t_1,t_2)$, then $\tl<t_1$.
\end{Rem}

\begin{Rem}
Suppose that $\Ch(I)$ is a chordal domain with the upper chord $[g(a),g(b)]$. Since the force function outside 
the chordal domain does not depend on the foliation inside it, we will use the notation $\Fr(t; a,b)$ instead of 
$\Fr(t,\Ch(I))$ and $\Fl(t; a,b)$ instead of $\Fl(t,\Ch(I))$ for $t$ outside the chordal domain. The forces 
satisfy the following identities:
\eq{eq101101}
{\Fr(t;a,b) =\exp\Big(-\int_{b}^t\frac{\kappa_2'}{\kappa_2-\kappa}\Big)
\left(\int_{b}^t \exp\Big(\int_{b}^\tau\frac{\kappa_2'}{\kappa_2-\kappa}\Big)\tors'(\tau)\,d\tau
+\frac{\DR(a,b)}{g_1'(b)\kappa_2'(b)} \right), \quad t \in [b,\tr]; 
}
\eq{eq101102}
{\Fl(t;a,b) =\exp\Big(\int_{t}^a\frac{\kappa_2'}{\kappa_2-\kappa}\Big)
\left(\!-\!\int_{t}^a \exp\Big(\!-\!\int_{\tau}^a\frac{\kappa_2'}{\kappa_2-\kappa}\Big)\tors'(\tau)\,d\tau
+\frac{\DL(a,b)}{g_1'(a)\kappa_2'(a)} \right), \quad t \in [\tl,a]. 
}
\end{Rem}

Analogously we define the forces and tails from the infinities. 
Suppose that  $t \in \mathbb{R}$ and there exists the standard Bellman candidate on the right tangent domain 
$\Rt(-\infty,t)$  (see Proposition~\ref{NewRightTangentsCandidateInfty}). Let $\tr \leq +\infty$ be the supremum 
of such $t$.

\begin{Def}
\label{Def031004}
The function 
$$
\Fr(t, -\infty) = \FF_{(-\infty,\tr)}(t),\quad t \in (-\infty,\tr)
$$
is called the \emph{right force function of $-\infty$}. The ray $(-\infty, \tr)$ is called the 
\emph{right tail of $-\infty$}.
\end{Def}

Similarly, we define the left force function and the left tail of $+\infty$. 
Let $\tl \geq -\infty$ be the infimum of such $t \in \mathbb{R}$ that there exists the standard Bellman candidate 
on the left tangent domain $\Lt(\tl,+\infty)$  (see Proposition~\ref{NewLeftTangentsCandidateInfty}).

\begin{Def}
\label{Def031005}
The function 
$$
\Fl(t, +\infty) = \FF_{(\tl,+\infty)}(t),\quad t \in (\tl,+\infty)
$$
is called the \emph{left force function of $+\infty$}. The ray $(\tl, +\infty)$ is called the 
\emph{left tail of $+\infty$}.
\end{Def}

\subsection{Properties of forces}\label{s332}

Though the expressions for the forces are well defined for arbitrary pairs of points~$(a_0, b_0)$, when we 
write a force concerning such a pair, we always assume that the pair~$(a_0,b_0)$ satisfies the cup equation~\eqref{urlun1}. 
We study differential properties of forces. 

\begin{St}
\label{difforcepou}
Let~$\Ch([a,b],*)$ be a chordal domain. Then its forces satisfy the following differential equation on 
the corresponding tails\textup:
\begin{equation}\label{leftder}
F' = \tors' - \beta_2' =\tors'- \frac{\kappa_2'}{\kappa_2-\kappa}F,
\end{equation}
where $F$ is $\Fr$ or $\Fl$. 
\end{St}

\begin{proof}
Taking into account the definition~\eqref{eq121001}, this statement may be obtained via division 
of~\eqref{diffbeta2} by $\kappa_2'$ and differentiation.
\end{proof}

\begin{Rem}\label{Rem150401}
The functions $\Fr-\tors$ and $\Fl-\tors$ are strictly increasing and  strictly decreasing on their domains, respectively.
\end{Rem}

\begin{proof}
According to~\eqref{leftder}, $F'-\tors' = -\frac{\kappa_2'}{\kappa_2-\kappa}F$. We always have $F<0$ and $\kappa_2'>0$. It remains to notice that $\kappa_2-\kappa>0$ in the case of the right force, 
and $\kappa_2-\kappa<0$ in the case of the left force.
\end{proof}

\begin{Le}
\label{DifForcePoL}
Suppose that $\Ch(I)$ is a chordal domain with the upper chord $[g(a),g(b)]$. For a fixed~$t,$ $t>b,$ consider 
the right force function $\Fr(t,a(b),b)$  as a function of~$b$. Then its  derivative with respect to $b$ satisfies 
the following equality\textup:
\eq{eq170101}{
\frac{\partial \Fr(t,a(b),b)}{\partial b} = 
\frac{\DR(a,b)}{g_1'(b)}\cdot\Big[\frac1{\kappa_2(b)\!-\!\kappar(b)}-\frac1{\kappa_2(b)\!-\!\kappachord(b)}\Big]
\cdot\exp\Big(\!-\!\!\int_{b}^t\!\frac{\kappa_2'(\tau)}{\kappa_2(\tau)\!-\!\kappar(\tau)}d\tau\Big),
}
where $\kappar$ is the slope of the right tangent {\bf(}see~\eqref{eq211002}{\bf)}\textup, and $\kappachord(b)$ 
is the slope of the upper chord of the chordal domain\textup, i.\,e.\textup,
\eq{eq160901}{
\kappachord(b) = \frac{g_2(b)-g_2(a)}{g_1(b)-g_1(a)}.
}
For a fixed $t,$ $t<a,$ we get a symmetric formula\textup: 
\eq{eq170102}{
\frac{\partial \Fl(t,a,b(a))}{\partial a} = 
\frac{\DL(a,b)}{g_1'(a)}\cdot\Big[\frac1{\kappa_2(a)\!-\!\kappal(a)}-\frac1{\kappa_2(a)\!-\!\kappachord(a)}\Big]
\cdot\exp\Big(\int_t^{a}\!\frac{\kappa_2'(\tau)}{\kappa_2(\tau)\!-\!\kappal(\tau)}d\tau\Big),
}
where $\kappal$ is the slope of the left tangent {\bf(}see~\eqref{eq211002}{\bf)}\textup, and $\kappachord(a) =\kappachord(b)$ is 
 given by~\eqref{eq160901}.
\end{Le}

\begin{proof}
We differentiate $\Fr(t,a,b)$ (see \eqref{eq101101}) with respect to $b$, regarding $a = a(b)$:
\begin{eqnarray}
\frac{\partial \Fr(t,a(b),b)}{\partial b}\!\! &=& \!\!
\displaystyle \frac{\kappa_2'(b)}{\kappa_2(b)-\kappa(b)} \Fr(t,a,b)+ 
\bigg\{\!-\tors'(b)-\frac{\kappa_2'(b)}{\kappa_2(b)-\kappa(b)} 
\int_{b}^t  \exp\Big(\int_{b}^\tau\frac{\kappa_2'}{\kappa_2-\kappa}\Big)\tors'(\tau)\,d\tau+ \nonumber
\\ 
&\ &\rule{0pt}{25pt} +\frac{1}{g_1'(b)\kappa_2'(b)}\frac{d\DR(a(b),b)}{d b}- 
\frac{(g_1'(b)\kappa_2'(b))'}{(g_1'(b)\kappa_2'(b))^2}\DR(a,b)\bigg\}
\exp\Big(\!-\!\!\int_{b}^t\frac{\kappa_2'}{\kappa_2-\kappa}\Big) 
\label{eq150103}
\\
\!\!&=&\!\!\displaystyle \rule{0pt}{25pt}\bigg\{\!-\tors'+ \frac{\kappa_2'}{\kappa_2-\kappa} \cdot 
\frac{\DR}{g_1'\kappa_2'}+\frac{1}{g_1'\kappa_2'}\cdot\frac{d\DR}{d b}- 
\frac{(g_1'\kappa_2')'}{(g_1'\kappa_2')^2}\DR\bigg\}
\exp\Big(\!-\!\!\int_{b}^t\frac{\kappa_2'}{\kappa_2-\kappa}\Big). \nonumber
\end{eqnarray}
We use formula~\eqref{eq150101} to simplify the expression in the braces. We note that $\kappa$ 
in~\eqref{eq150103} differs from the one in~\eqref{eq150101}. We write $\kappar$ for that which 
appears in~\eqref{eq150103} and $\kappachord$ for the one in~\eqref{eq150101}:

\begin{align*}
\bigg\{\!-\tors'&+ \frac{\kappa_2'}{\kappa_2-\kappar} \cdot \frac{\DR}{g_1'\kappa_2'}+\frac{1}{g_1'\kappa_2'}\cdot
\frac{d\DR}{d b}- \frac{(g_1'\kappa_2')'}{(g_1'\kappa_2')^2}\DR \bigg\}=
\\ 
&=-\tors'+ \frac{\kappa_2'}{\kappa_2-\kappar} \cdot \frac{\DR}{g_1'\kappa_2'}+\frac{1}{g_1'\kappa_2'}\cdot 
\Big[\kappa_2'g_1'\tors'+ \Big(\frac{(g_1'\kappa_2')'}{g_1'\kappa_2'} - 
\frac{\kappa_2'}{\kappa_2-\kappachord}\Big)\DR\Big]
- \frac{(g_1'\kappa_2')'}{(g_1'\kappa_2')^2}\DR=
\\
&=\frac{\DR}{g_1'}\Big(\frac{1}{\kappa_2-\kappar}-\frac{1}{\kappa_2-\kappachord}\Big). 
\end{align*}
\end{proof}

\section{Linearity domains}
\label{s34}\index{domain! linearity domain}

As it was stated in Section~\ref{s31}, we classify linearity domains by the number of points on the fixed 
boundary. For each linearity domain~$\mathfrak{L}$ we will define the corresponding force function on $\mathfrak{L}\cap \dfi\Omega$.

\subsection{Angle}
\label{s341}
 The first linearity domain we study, an \emph{angle}\index{angle}, has only one point~$g(u)$ on the 
fixed boundary. Recall that $\Sl(u)$ and $\Sr(u)$ are the left and the right tangent segments to the free 
boundary of $\Omega$ starting at the point $g(u)$, and $T(u)$ is the closed curvilinear triangle with the 
vertex $g(u)$ whose sides are $\Sl(u)$, $\Sr(u)$, and the part of $\dfree\Omega$ between the two tangency points. 
We will use the symbol  $\Ang(u)$ for this domain of linearity; note that geometrically $\Ang(u)=T(u)$.

\begin{Def}
An affine function $B$ on an angle $\Ang(u)$\index[symbol]{$\Ang$} is called a \emph{standard candidate} on $\Ang(u)$ if the vector 
$\gamma'(u)$ is parallel to the graph of $B$ on $\Ang(u)$.
\end{Def}

An angle is the linearity domain that appeared in Proposition~\ref{St260901} in the case where $\Omega(I_-)$ is 
a right tangent domain and $\Omega(I_+)$ is a left one, see Figure~\ref{fig:ang}. The function $B$ defined by~\eqref{eq250902} is a standard candidate on the angle. See the graphical representation of this 
situation in Subsection~\ref{ElementaryGraphs}, Figure~\ref{fig:anggraph}.
 
\begin{figure}[!h]
\begin{center}
\includegraphics[width = 0.8 \linewidth]{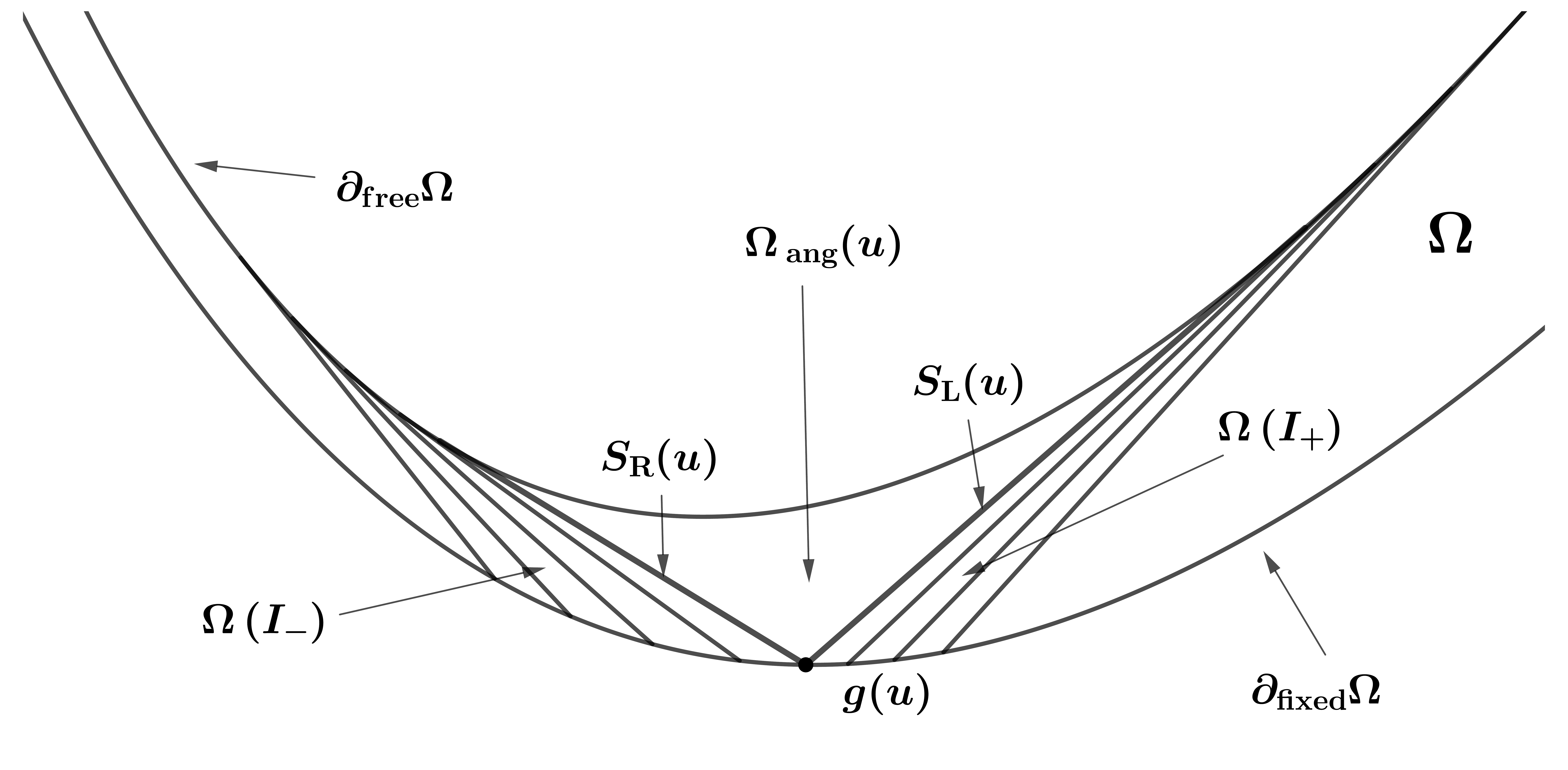}
\caption{An angle $\Ang(u)$ with adjacent domains.}
\label{fig:ang}
\end{center}
\end{figure}

\subsection{Linearity domains with two points on the fixed boundary}
\label{s342}

Consider a linearity domain~$\mathfrak{L}$ that has two points~$g(a)$ and~$g(b)$ on the fixed boundary 
$\dfi\Omega$, assuming~$a<b$ and $[g(a),g(b)] \subset \Omega$. Surely, the segment~$[g(a),g(b)]$ is 
a part of the boundary of the linearity domain. It is natural to assume that there are two extremal segments 
tangent to the free boundary, $S(a)$ and $S(b)$, and bounding our linearity domain from the left and right. 
If they have the same orientation, namely they are either $\Sl(a)$ and $\Sl(b)$ or $\Sr(a)$ and $\Sr(b)$, then 
the linearity domain is called a~\emph{trolleybus}\index{trolleybus}, the \emph{left} one or the 
\emph{right} one correspondingly. These trolleybuses will be denoted by~$\LTroll(a,b)$\index[symbol]{$\LTroll$} and~$\RTroll(a,b)$\index[symbol]{$\RTroll$}, 
see Figure~\ref{fig:ltr} and Figure~\ref{fig:ptr}.

\begin{figure}[!h]
\begin{center}
\includegraphics[width = 0.8 \linewidth]{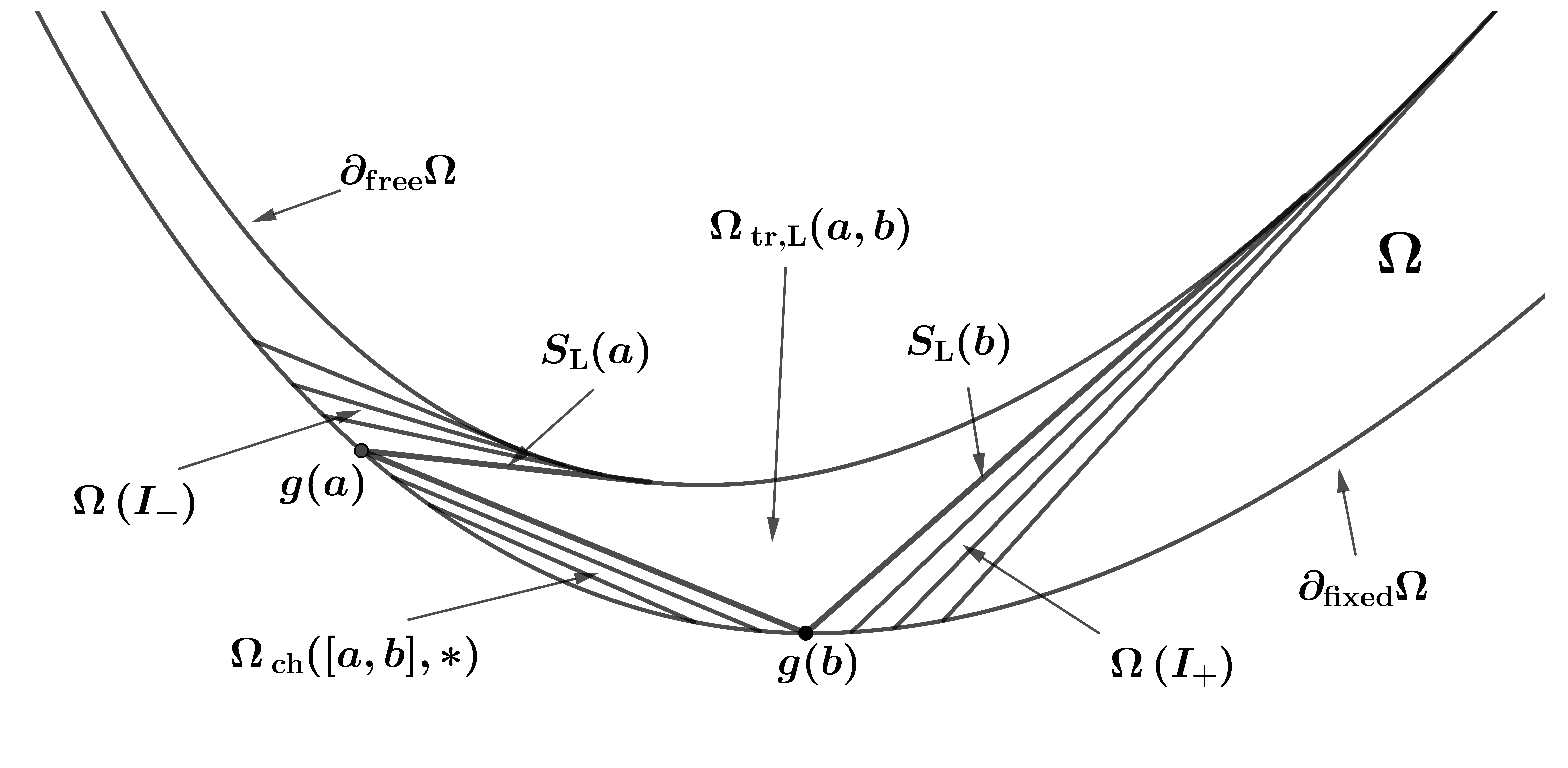}
\caption{Trolleybus~$\LTroll(a,b)$ with adjacent domains.}
\label{fig:ltr}
\end{center}
\end{figure}

\begin{figure}[!h]
\begin{center}
\includegraphics[width = 0.8 \linewidth]{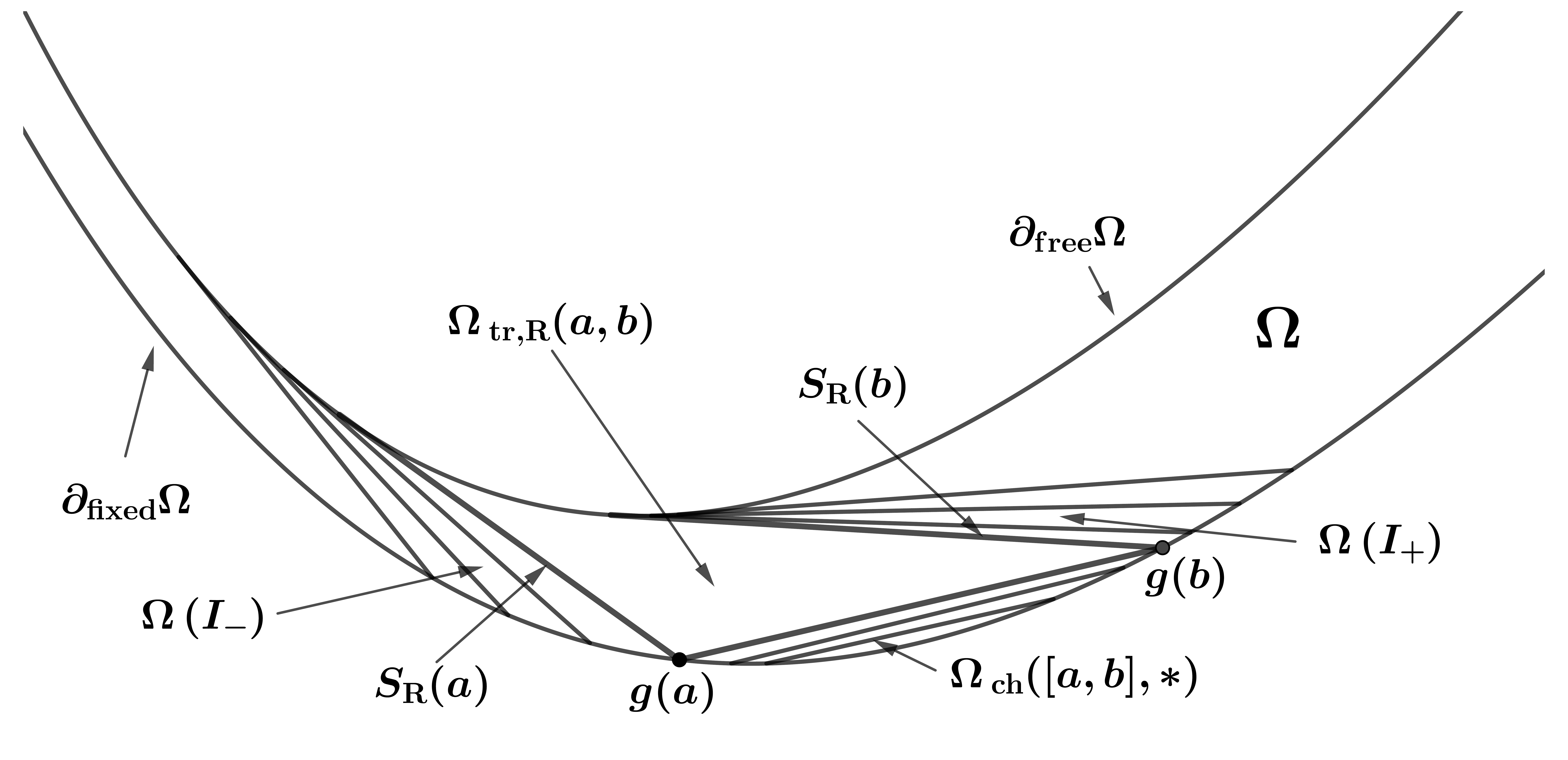}
\caption{Trolleybus~$\RTroll(a,b)$ with adjacent domains.}
\label{fig:ptr}
\end{center}
\end{figure}

\begin{figure}[!ht]
\begin{center}
\includegraphics[width = 0.8\linewidth]{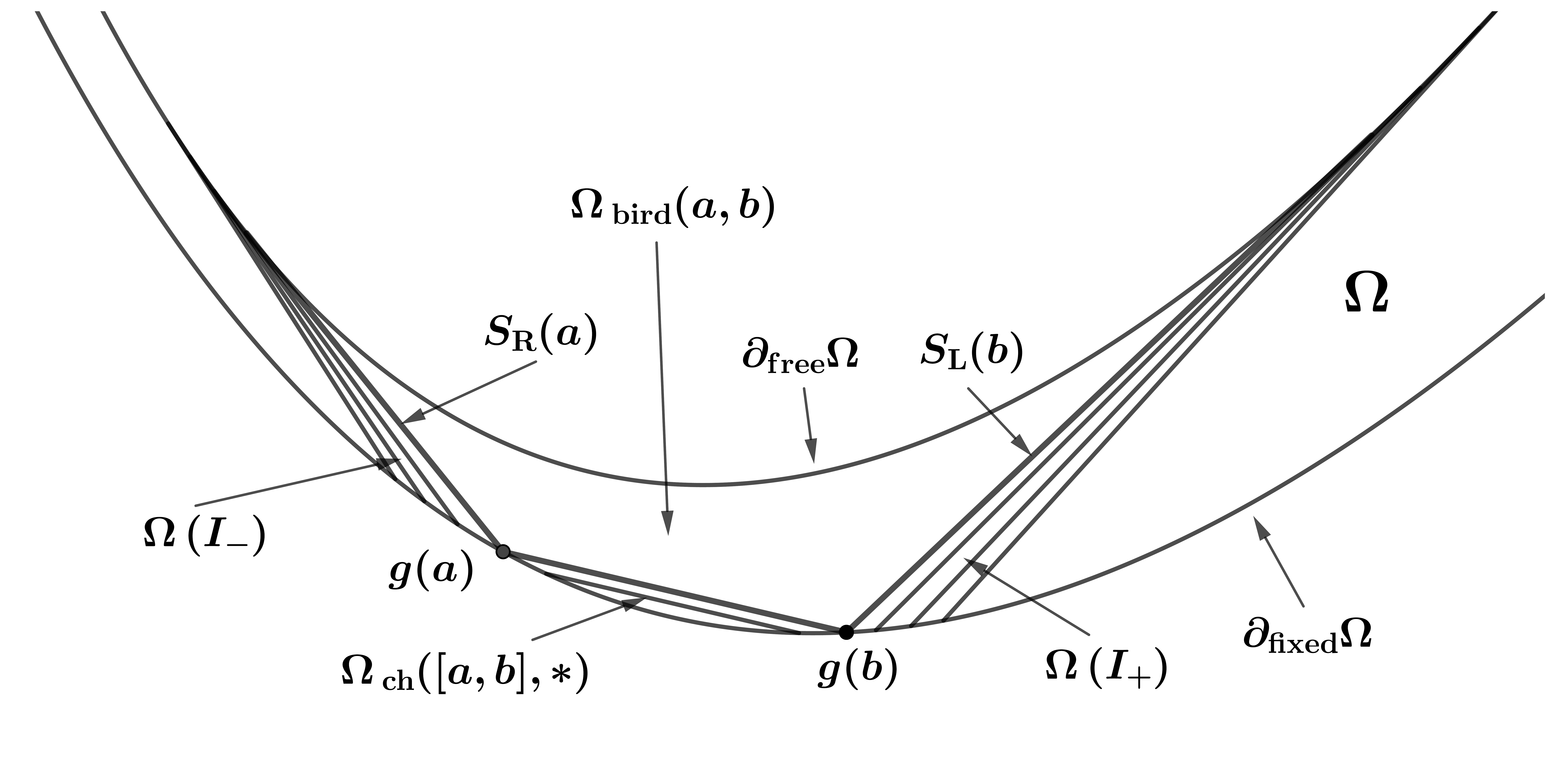}
\caption{A birdie $\Bird(a,b)$ with the adjacent domains.}
\label{fig:birdie}
\end{center}
\end{figure}

A linearity domain whose border tangents have different orientation, i.\,e., $\Sr(a)$ and $\Sl(b)$, 
is called a~\emph{birdie}\index{birdie}, see Figure~\ref{fig:birdie}, we denote it by~$\Bird(a,b)$.\index[symbol]{$\Bird$} 
It is clear that the opposite situation, $\Sl(a)$ and $\Sr(b)$, is impossible because these two segments 
intersect each other.

As in the case of an angle, we suppose that $g(a)$ and $g(b)$ are the points where we glue two different fences. 
So, there are three fences: a tangent domain $\Omega(I_-)$ adjacent to~$\mathfrak{L}$ along the tangent~$S(a)$, 
a tangent domain $\Omega(I_+)$ adjacent to~$\mathfrak{L}$ along~$S(b)$, and a chordal domain 
$\Ch([a,b],\ast)$. 

\begin{St}
\label{St131001}
Let $I_- = [t_-,a],$ $I_+=[b,t_+].$ Suppose that $B_\pm$ are Bellman candidates on $\Omega(I_\pm),$ 
and $B_{\Ch([a,b],\ast)}$ is the Bellman candidate on $\Ch([a,b],\ast)$. Suppose that the gluing conditions 
hold for each of two pairs of fences\textup, i.\,e.\textup, 
\eq{eq131001}{
\FF_{I_-}(a)=\Fl(a,\Ch([a,b],\ast)),\qquad
\FF_{I_+}(b)=\Fr(b,\Ch([a,b],\ast)).
}
Let $B_\mathfrak{L}(x)$ be the affine function defined by the equation
\eq{eq201001}{
\det
\begin{pmatrix}
g_1'(a)& g_2'(a)& f'(a)\\
g_1'(b)& g_2'(b)& f'(b)\\
x_1-g_1(a)& x_2-g_2(a)& B_\mathfrak{L}(x) - f(a)
\end{pmatrix}=0, \quad x \in \mathfrak{L}.
}
Then the function $B$ defined by the formula
\begin{equation*}
B(x)=
\begin{cases}
B_-(x), & x \in \Omega(I_-);\\
B_\mathfrak{L}(x), & x \in \mathfrak{L};\\
B_+(x), & x \in \Omega(I_+);\\
B_{\Ch([a,b],\ast)}, & x \in \Ch([a,b],\ast)
\end{cases}
\end{equation*}
is a $C^1$-smooth Bellman candidate on its domain.
\end{St}

\begin{proof}
Condition~\eqref{eq230402} implies 
$$\det\begin{pmatrix}
g_1'(a)& g_2'(a)\\
g_1'(b)& g_2'(b)
\end{pmatrix} \ne 0,$$
therefore~\eqref{eq201001} defines $B_\mathfrak{L}$ correctly. The cup equation~\eqref{urlun1} guaranties 
that the affine function $B_\mathfrak{L}$ coincides with the unique affine extension of~$B_{\Ch([a,b],\ast)}$ to $\mathfrak{L}$. Therefore, it also coincides with the affine parts of the functions from~\eqref{eq250902} 
for both pairs of glued fences: for $\Omega(I_-)$ and $\Ch$, and for $\Ch$ and $\Omega(I_+)$. Thus, 
Proposition~\ref{St260901} gives that the function $B$ is a $C^1$-smooth Bellman candidate.
\end{proof}

\begin{Def}\index{standard candidate! in trolleybus}\index{standard candidate! in birdie}
The function~$B_\mathfrak{L}$ defined by~\eqref{eq201001} in the linearity domain~$\mathfrak{L}$ with 
two points on the fixed boundary is called a \emph{standard candidate} there. 
\end{Def}

\begin{Rem}
By~\eqref{eq201001} we have the following formula for the derivative of the affine function~$B_\mathfrak{L}\!:$
\eq{eq201002}{
\beta_2 = \frac{\partial B_\mathfrak{L}}{\partial x_2} = \frac{
\det\begin{pmatrix}
g_1'(a) & f'(a)\\
g_1'(b) & f'(b)
\end{pmatrix}}
{\det\begin{pmatrix}
g_1'(a) & g_2'(a)\\
g_1'(b) & g_2'(b)
\end{pmatrix}
}
=
\frac{\kappa_3(b)-\kappa_3(a)}{\kappa_2(b)-\kappa_2(a)}.
}
\end{Rem}

\begin{Def}\label{Def190401}
Let $\mathfrak{L}$ be a domain of linearity with two points on the fixed boundary: 
$g(a), g(b) \in \mathfrak{L} \cap \FixedBoundary\Omega$, $g(a) \prec g(b)$. 
Define the force function $F$ on $\{a,b\}$ as follows: 
$$
F(a) = \frac{\DL(a,b)}{g_1'(a)\kappa_2'(a)}, \qquad F(b) = \frac{\DR(a,b)}{g_1'(b)\kappa_2'(b)}.
$$
\end{Def}

Let us note that a birdie can be regarded as a union of a trolleybus and an angle (there are two symmetric ways):
\begin{equation}\label{FirstFormula}
\begin{aligned}
\Bird(a,b) &= \RTroll(a,b) \biguplus \Rt(b,b) \biguplus \Ang(b);\\
\Bird(a,b) &= \Ang(a)\biguplus \Lt(a,a) \biguplus \LTroll(a,b). 
\end{aligned}
\end{equation}
Note that the right and the left sides of the equalities are equal as 
planar domains provided we substitute~$\bigcup$ for~$\biguplus$\index[symbol]{$\biguplus$}. 
The symbol~$\biguplus$ in~\eqref{FirstFormula} means 
the following: if a function~$B$ on this domain is continuous and its restriction to each 
single subdomain of one side of the formula is a standard candidate, then this function~$B$ is a standard 
candidate for each subdomain of the other side of the formula. This easily follows from Proposition~\ref{St260901} 
and Proposition~\ref{St131001}. Graphical \label{biguplusdef} representation of this equality is presented on Figure~\ref{fig:BirdieAsAnglePlusTrol} (see Section~\ref{ElementaryGraphs} below). 

\index{standard candidate! in a single chord}
It is convenient to introduce two more ``linearity domains'' for the purposes of formalization. First, 
sometimes we will treat a single chord~$[g(a),g(b)]$, where~$(a,b)$ satisfies the cup equation~\eqref{urlun1}, 
as a linearity domain~$\mathfrak{L}$. The standard candidate~$B$ inside~$[g(a),g(b)]$ is given by linearity. 
We note that Proposition~\ref{St131001} is valid for the following two ``hidden'' subcases. The first one is 
when the chord $[g(a),g(b)]$ is tangent to the free boundary of $\Omega$, $\Omega(I_-)=\Lt(I_-)$ is a left tangent 
domain, and $\Omega(I_+)=\Rt(I_+)$ is a right tangent domain. We obtain the chordal domain with the glued 
tangent domains (see Figure~\ref{fig:cupmeets}).

\begin{figure}[!h]
\begin{center}
\includegraphics[width = 0.8 \textwidth]{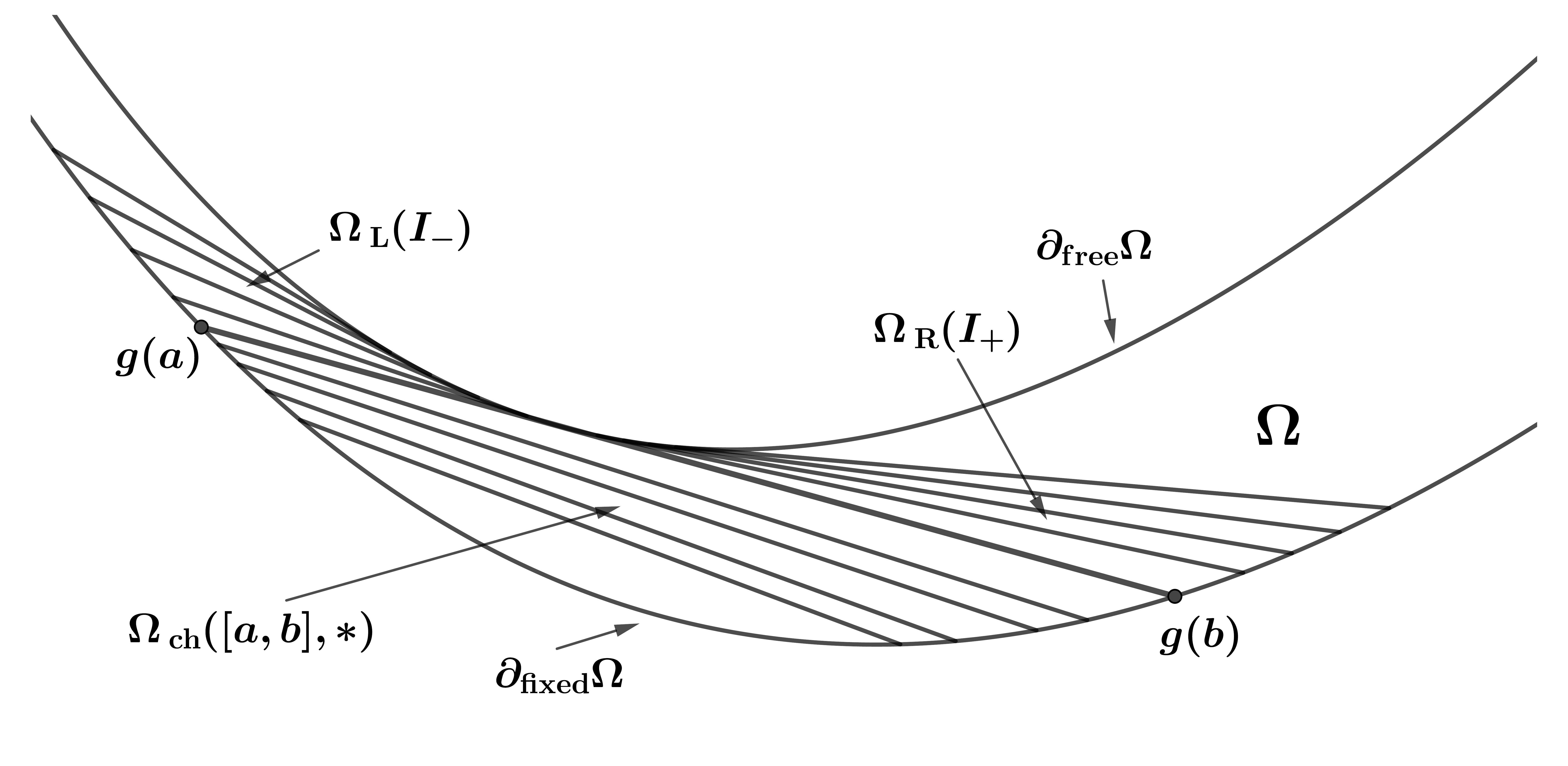}
\caption{A chordal domain $\Ch([a, b],*)$ with tangent domains attached to it.}
\label{fig:cupmeets}
\end{center}
\end{figure}

The second subcase is when the chord~$[g(a),g(b)]$ is not tangent to the free boundary and $\Omega(I_\pm)=\Ch(*,[a,b])$ 
are the right and the left parametrization of the same chordal domain, which lies above the chord $[g(a),g(b)]$. 
In this subcase we obtain two glued chordal domains $\Ch(*,[a,b])$ and $\Ch([a,b], *)$, which will appear when 
it is impossible to consider their union as one chordal domain (if either $\DR(a,b)=0$ or $\DL(a,b)=0$), 
see Figure~\ref{fig:twochdom}. 

\begin{figure}[!h]
\begin{center}
\includegraphics[width = 0.8 \linewidth]{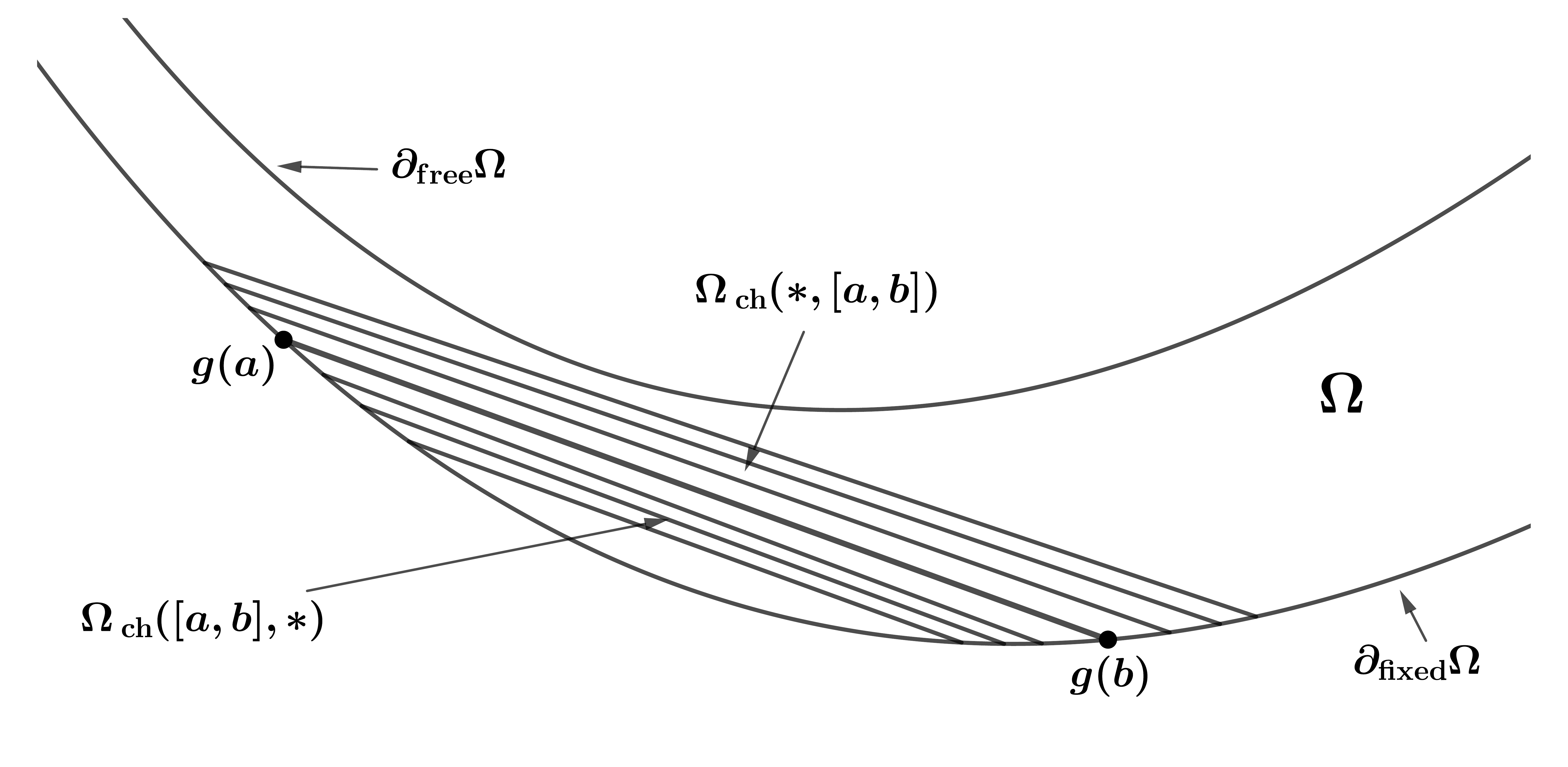}
\caption{Two chordal domains $\Ch([a, b],*)$ and $\Ch(*,[a,b])$ with the chord between them.}
\label{fig:twochdom}
\end{center}
\end{figure}

Second, sometimes it is useful to treat a single tangent~$\Rt(u,u)$ or~$\Lt(u,u)$ as a linearity 
domain~$\mathfrak{L}$. Moreover, no matter how strange it seems, it is natural to think of it as of a domain 
with two points on the fixed boundary. We treat this single tangent as a trolleybus of zero width, i.\,e., 
its base is the chord~$[g(u),g(u)]$. This can be considered as a limit of a sequence of trolleybuses on the 
chord shrinking to the point~$u$. It will appear only when~$u = c_i$ for some~$i$, where~$c_i$ is a single 
point root from Definition~\ref{roots}. The standard candidate~$B$ in this domain is obtained by passing to the limit in~\eqref{eq201001}:
\eq{Beta2InTheFifthType}{\index{standard candidate! in a single tangent}
\det
\begin{pmatrix}
g_1'(u)& g_2'(u)& f'(u)\\
g_1''(u)& g_2''(u)& f''(u)\\
x_1-g_1(u)& x_2-g_2(u)& B_\mathfrak{L}(x) - f(u)
\end{pmatrix}=0, \quad x \in \mathfrak{L}.
}
This construction will appear between two tangent domains of the same direction when the standard candidates 
on them cannot be concatenated  to be a standard candidate on the union of the domains. The obstacle of such 
a concatenation is the vanishing of the forces at the point $u$.

\subsection{Graphical representation of the elementary domains}\label{ElementaryGraphs}

\begin{wrapfigure}{r}{0.5\linewidth}
\begin{center}
\vspace{-10pt}
\includegraphics[width = \linewidth]{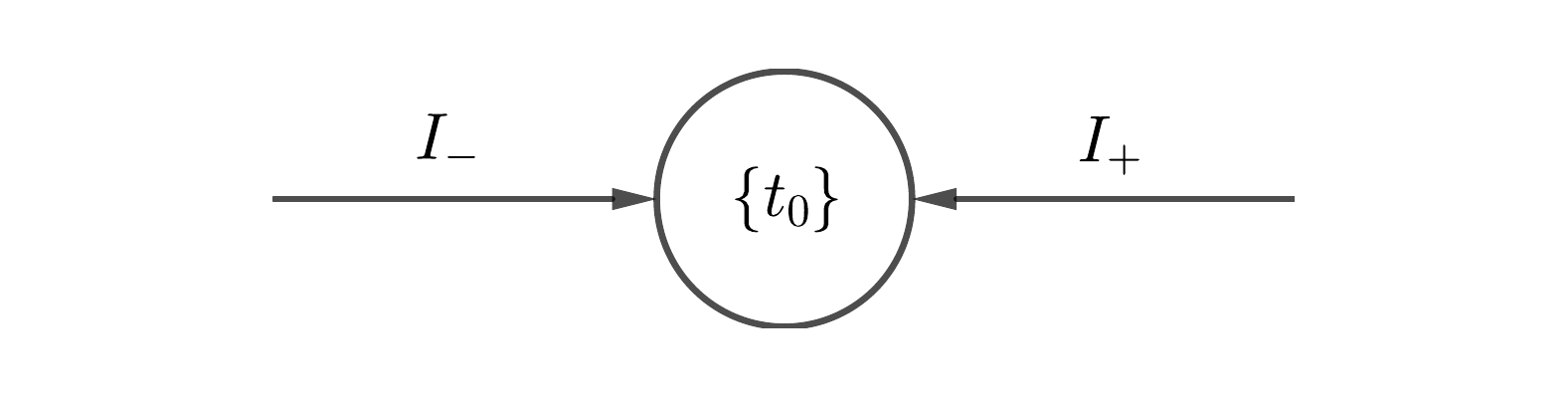}
\vspace{-10pt}
\caption{Graphical representation of an angle $\Ang(t_0)$ with adjacent tangent domains.}
\label{fig:anggraph}
\end{center}
\end{wrapfigure}

As it was said in Subsection~\ref{s31}, we give a graphical representation for combinatorial structure 
of foliations. The material of this subsection essentially repeats analogous constructions in~\cite{ISVZ2018}. 
We start with the representation of the simplest local foliations: fences (tangent domains and chordal domains) 
and some linearity domains. In Subsection~\ref{s345} we will give a graphical representation 
describing a foliation of the whole domain~$\Omega_\eps$.

\begin{figure}[h]
\begin{center}
\includegraphics[width = 0.45 \linewidth]{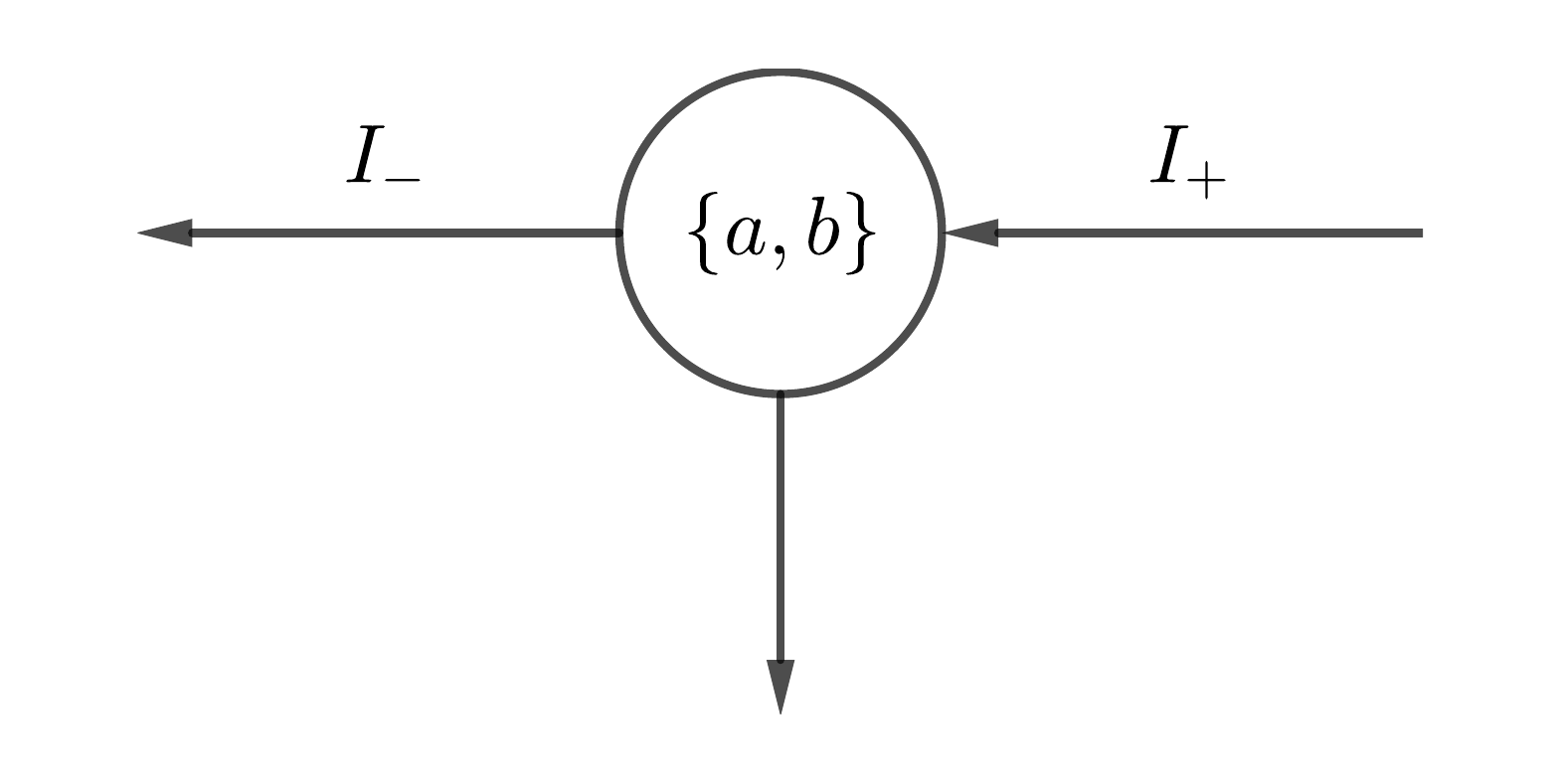}
\includegraphics[width = 0.45 \linewidth]{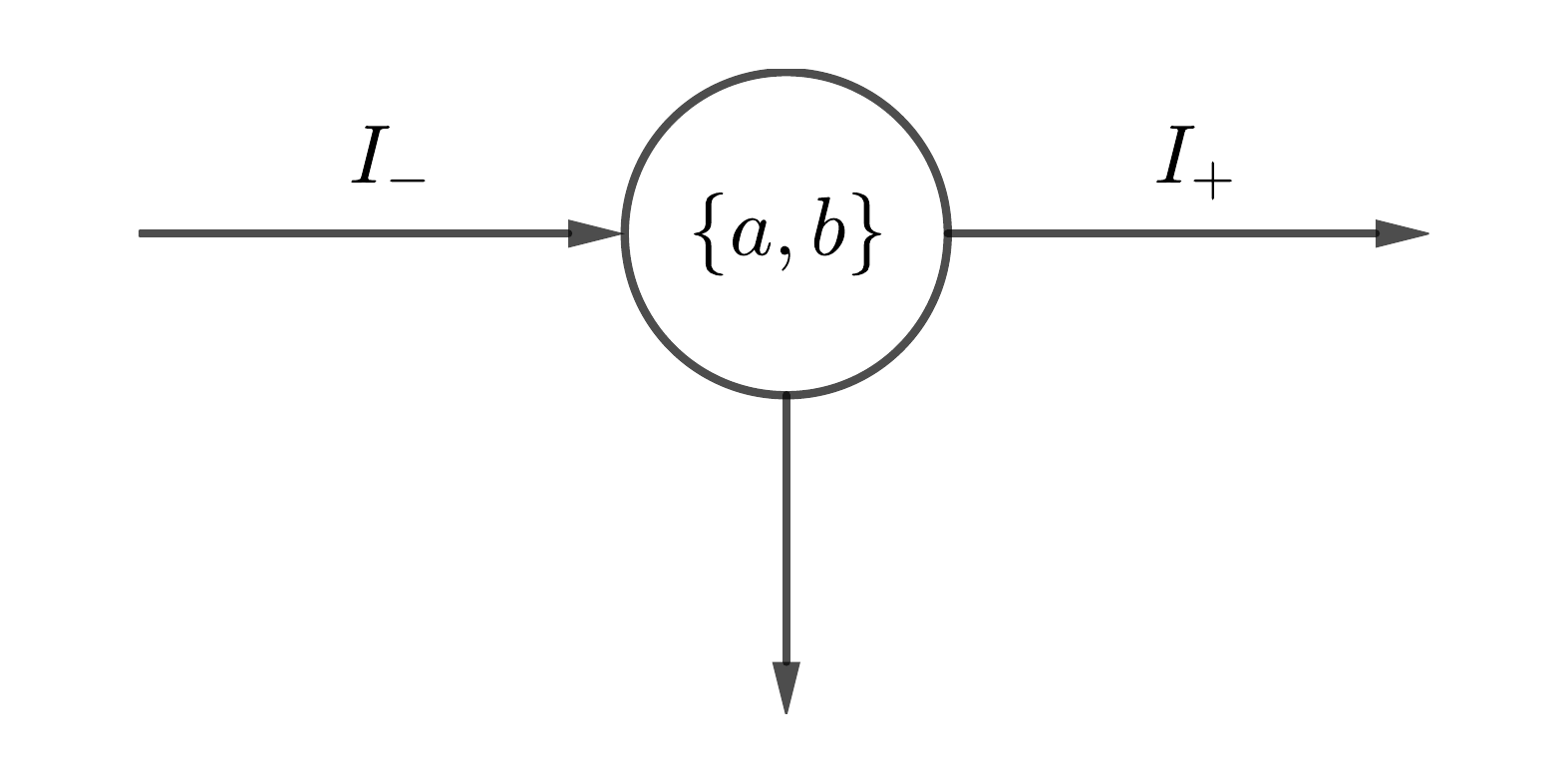}
\caption{Graphical representation of trolleybuses $\LTroll(a,b)$ and $\RTroll(a,b)$ with adjacent tangent domains and a chordal domain.}
\label{fig:gtr}
\end{center}
\end{figure}

\begin{wrapfigure}{r}{0.4\linewidth}
\begin{center}
\vspace{-50pt}
\includegraphics[width = 0.3\linewidth]{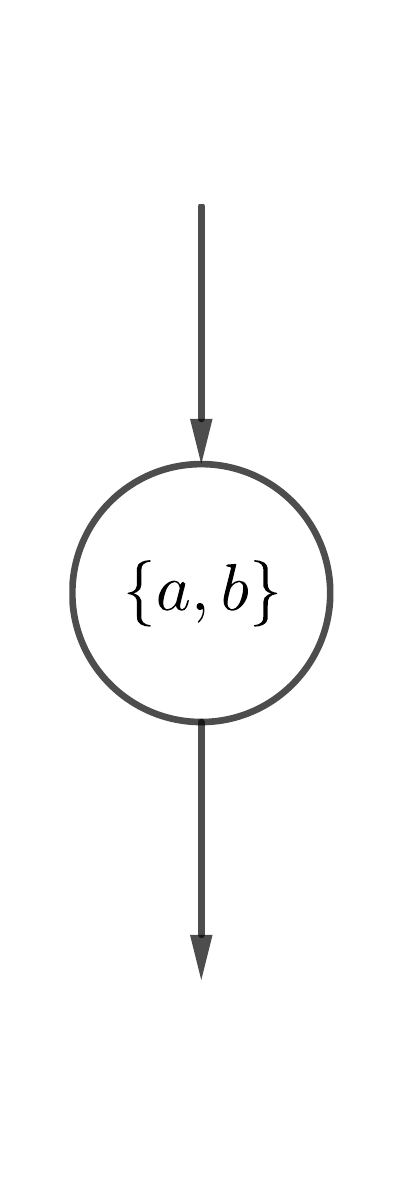}
\caption{Graphical representation of two chordal domains $\Ch([a, b],*)$ and $\Ch(*,[a, b])$ glued along the chord.}
\vspace{-30pt}
\label{fig:twochdomgraph}
\end{center}
\end{wrapfigure}
Usually vertices will correspond to linearity domains and oriented edges will always correspond to fences. 
We draw a vertex of the graph that corresponds to an angle and two incoming edges, which correspond to the 
neighbor tangent domains, see Figure~\ref{fig:anggraph}. We equip the elements of the graph with numerical 
parameters corresponding to the points on the fixed boundary.

Now we give a graphical representation of possible variants described in Proposition~\ref{St131001} (trolleybus 
and birdie). We draw a vertex for the domain of linearity and edges for the fences --- a chordal domain and two 
tangent domains. The edge corresponding to the chordal domain is outgoing. We will draw the edges corresponding to the tangent domains horizontally, and their directions agree 
with the directions of the tangents (either left or right). A trolleybus always has one incoming and two outgoing edges, 
and for a birdie we have two incoming and one outgoing edge, see Figure~\ref{fig:gtr} for the graphs corresponding to trolleybuses and the left drawing on Figure~\ref{fig:cup_graph} 
for the birdie.

We give a graphical representation of a full chordal domain with two neighbor tangent domains, see 
the right drawing on Figure~\ref{fig:cup_graph}. The vertex corresponds to the ``linearity domain'' being the chord $[g(a),g(b)]$, and the 
edges correspond to the chordal domain and tangent domains. Figure~\ref{fig:twochdomgraph} gives a graphical 
representation of two chordal domains $\Ch([a,b],*)$ and $\Ch(*,[a,b])$ glued along the chord 
(Figure~\ref{fig:twochdom}). The vertex here corresponds to the chord $[g(a),g(b)]$, and the incoming and outgoing 
edges correspond to $\Ch(*,[a,b])$ and $\Ch([a,b],*)$ respectively.

\begin{figure}[h!]
\includegraphics[width = 0.45\linewidth]{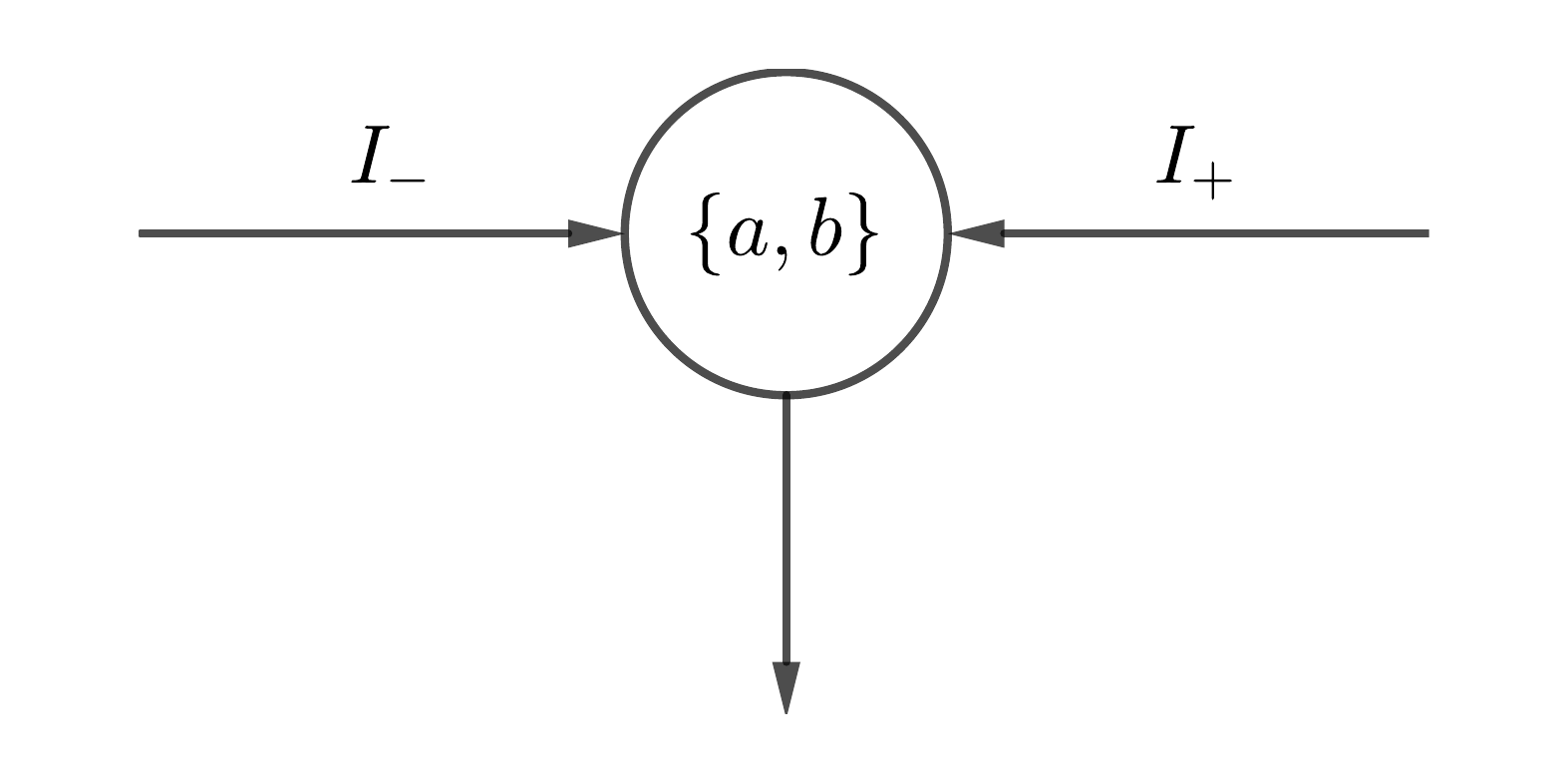}
\includegraphics[width = 0.45\linewidth]{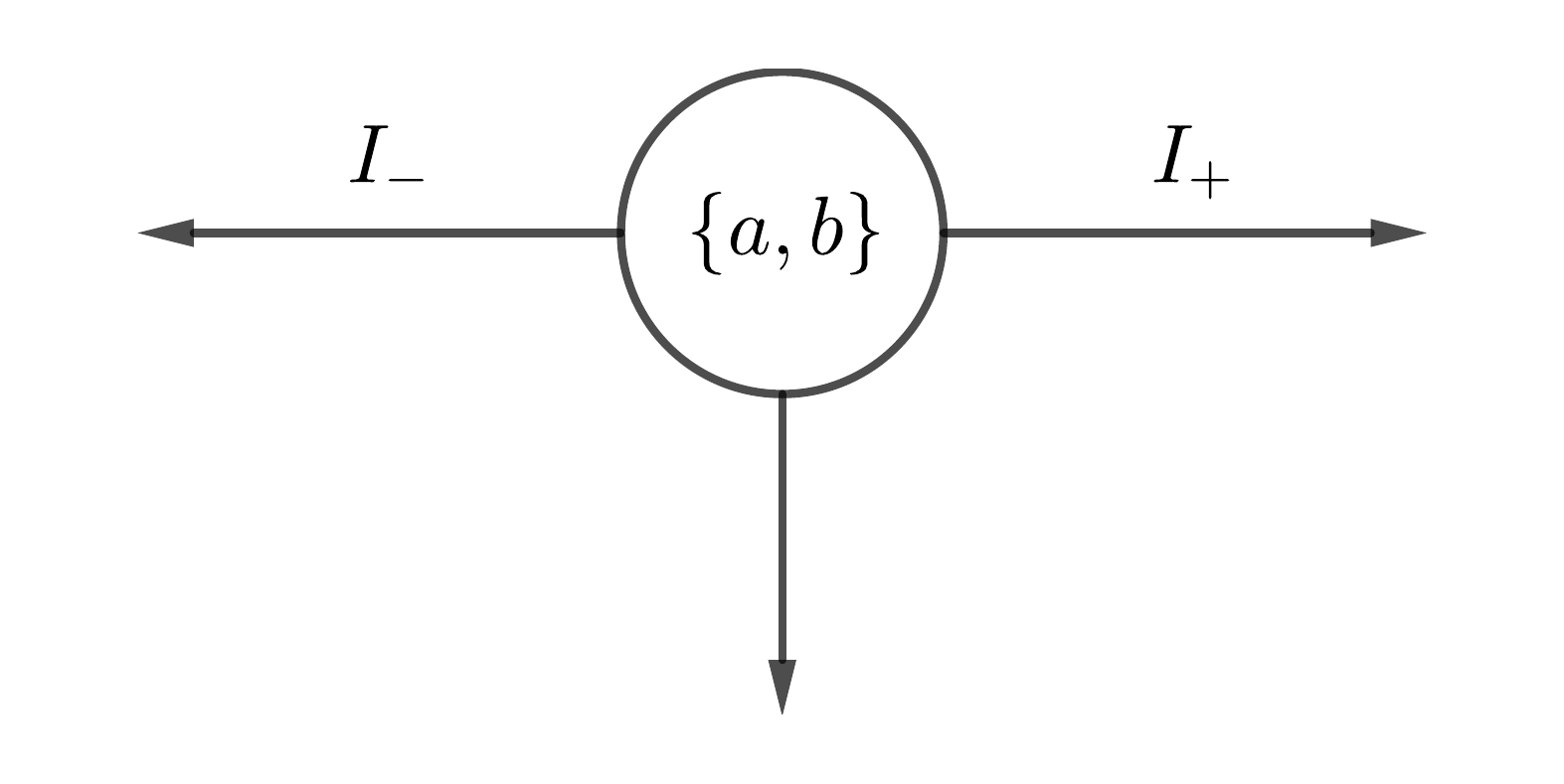}
\caption{Graphical representation of a birdie $\Bird(a,b)$ and of a full chordal domain $\Ch([a, b],*)$ with adjacent domains.
}
\label{fig:cup_graph}
\end{figure}

Figure~\ref{fig:twotandomgraph} gives a graphical representation of two tangent domains $\Lt(t_-,u)$ and 
$\Lt(u,t_+)$ glued along a ``trolleybus''~$\Lt(u,u)$ of the zero width, and for the symmetric case when 
$\Rt(t_-,u)$ and $\Rt(u,t_+)$ glued along a ``trolleybus''~$\Rt(u,u)$. The vertex here corresponds to the 
``trolleybus'' and the edges --- to the tangent domains.

\begin{figure}[!h]
\includegraphics[width = 0.45\linewidth]{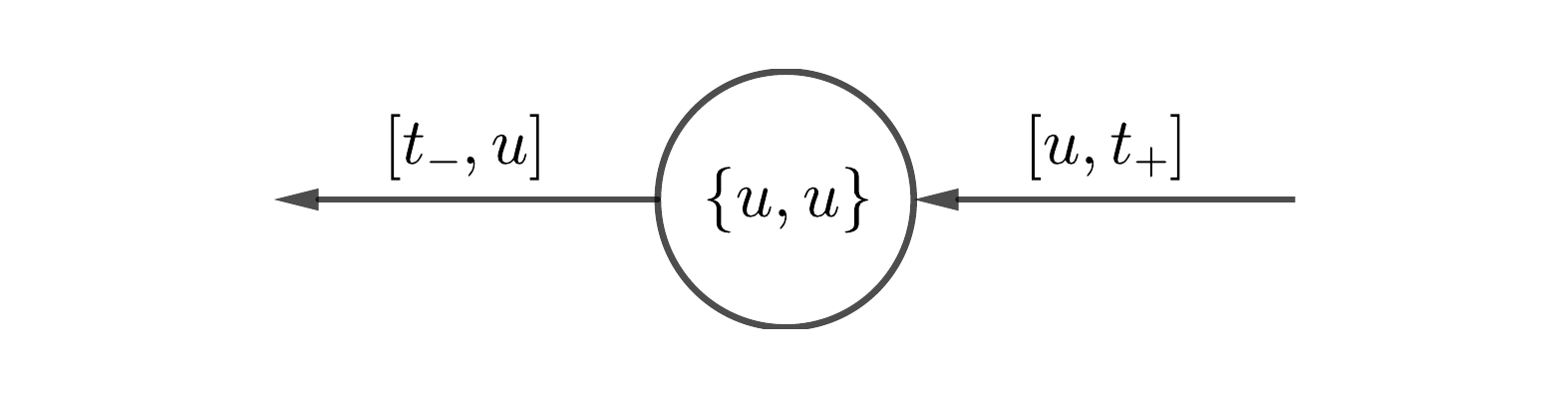}
\includegraphics[width = 0.45\linewidth]{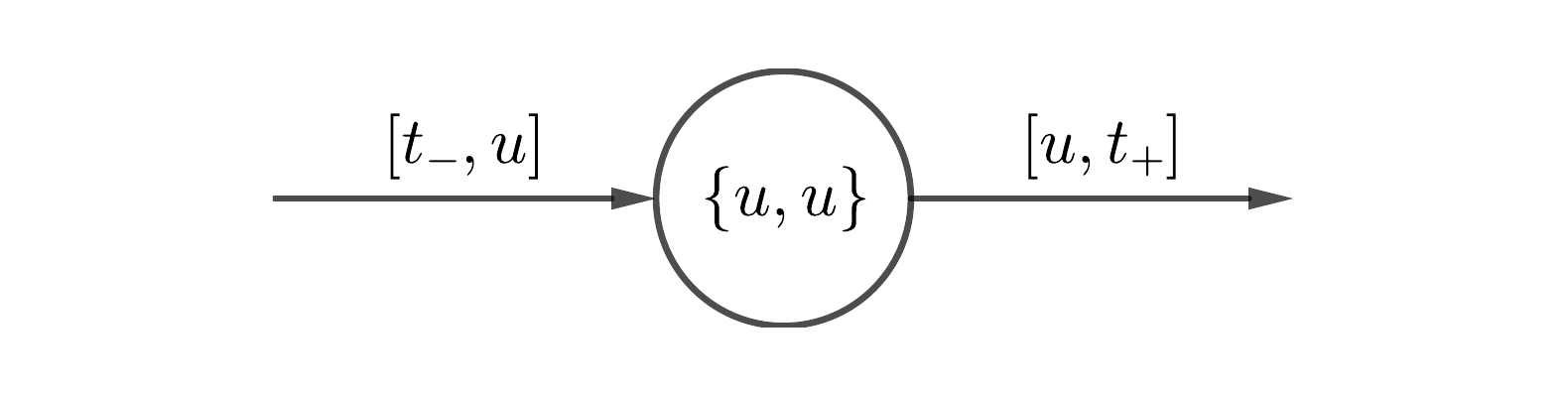}
\caption{Graphical representation of two tangent domains glued along a ``trolleybus'' of the zero width: the left and the right cases.}
\label{fig:twotandomgraph}
\end{figure}

Finally, Figure~\ref{fig:BirdieAsAnglePlusTrol} demonstrates the graphical representation of the equality~\eqref{FirstFormula}. 
\begin{figure}[!h]
\includegraphics[width = \linewidth]{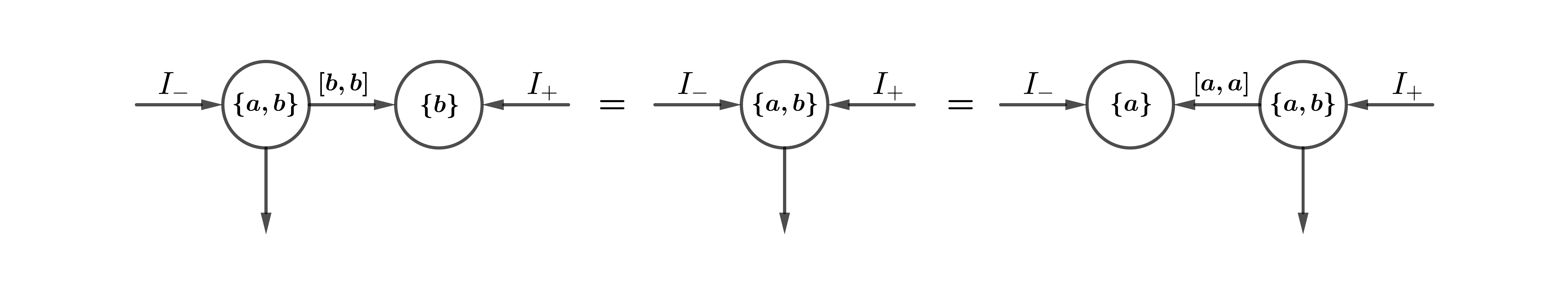}
\caption{Graphical representation of the equality ``birdie = angle + trolleybus''.}
\label{fig:BirdieAsAnglePlusTrol}
\end{figure}

\subsection{Multifigures}
\index{multifigures}\label{s343}

We begin with a structural agreement. For each linearity domain~$\mathfrak{L}$ we make the following finiteness assumption: the intersection 
$\mathfrak{L} \cap \FixedBoundary\Omega$ is assumed to be a union of finite number of arcs (one or two of these arcs 
may be unbounded, i.\,e., parametrized by a ray),
\begin{equation*}
\mathfrak{L} \cap \FixedBoundary\Omega = \cup_{i=1}^k \{g(t) \mid t \in \mathfrak{a}_i\},
\end{equation*} 
where~$\{\mathfrak{a}_i\}_{i=1}^k$ is a finite set of disjoint closed intervals, which can be single points. 
The curvilinear arc that corresponds to~$\mathfrak{a}_i$ is $g(\mathfrak{a}_i)$. We remind the reader the notation 
introduced in Subsection~\ref{s212}: the left endpoint of~$\mathfrak{a}_i$ is~$\mathfrak{a}_{i}^{\mathrm l}$ 
and the right endpoint is~$\mathfrak{a}_i^{\mathrm r}$. As we will see, due to Condition~\ref{reg} all the linearity domains we need satisfy this finiteness assumption. 

Consider a linearity domain~$\mathfrak{L}$. We know that the 
set~$\{\gamma(a) \colon g(a) \in \mathfrak{L} \cap \FixedBoundary\Omega\}$ lies in a two-dimensional plane 
in~$\mathbb{R}^3$. Therefore, there exists a function~$P_{\mathfrak{L}}$, which is a linear combination of 
$g_1$, $g_2,$ and a constant function, such that 
\begin{equation}
\label{PolynomialForLinearityDomain}
f(a) = P_{\mathfrak{L}}(a), \quad g(a) \in \mathfrak{L} \cap \FixedBoundary\Omega.
\end{equation}
Surely, the converse is also true: if there exists some linear combination~$P_{\mathfrak{L}}$ of $g_1,g_2$, 
and a constant function such that~\eqref{PolynomialForLinearityDomain} holds true, then there 
exists an affine function~$B_{\mathfrak{L}}$ such that~$B_{\mathfrak{L}}(g(a)) = f(a)$ for 
all~$g(a) \in \mathfrak{L} \cap \FixedBoundary\Omega$. 
Namely, if~$P_{\mathfrak{L}} = \beta_0+\beta_1 g_1+ \beta_2 g_2$, then
\begin{equation}\label{CandidateInMultifigure}
B_{\mathfrak{L}}\big(x_{1},x_{2}\big) = \beta_0+\beta_1 x_1+\beta_2 x_2.
\end{equation}
This function~$B_{\mathfrak{L}}$ is a Bellman candidate in~$\mathfrak{L}$. 

\begin{Rem}
Similar to the case where the linearity domain has only two points on the fixed boundary\textup, 
the identities~\eqref{eq201001} and~\eqref{eq201002} hold true for any~$g(a),g(b)\in \mathfrak{L}\cap\FixedBoundary$. 
All the vectors~$\gamma'(a)$ such that $g(a) \in\mathfrak{L} \cap \FixedBoundary\Omega$ lie in one plane orthogonal to the vector 
$(\beta_1,\beta_2,-1)$.
\end{Rem}

\begin{Def}\index{standard candidate! in multifigure}\index{standard candidate! in multifigure}
The function~$B$ defined by formula~\eqref{eq201001} in the linearity domain~$\mathfrak{L}$, where~$g(a)$ 
and~$g(b)$ are arbitrary points from $\mathfrak{L} \cap \FixedBoundary\Omega$\textup,
is called the \emph{standard candidate} in~$\mathfrak{L}$. 
\end{Def}

As we have verified, the standard candidate in~$\mathfrak{L}$ does not depend on the choice of~$a$ and~$b$ 
in the definition and coincides with the function given by~\eqref{CandidateInMultifigure}.
 
In the following lemma, we use  Definition~\ref{differentials}.
\begin{Le}
\label{ThreePointsOnOneLine}
Let~$a_1,$ $a_2,$ and~$a_3$ be such that~$\gamma'(a_i),$ $i = 1,2,3,$ lie in one plane. 
If~$a_1 \leq a_2 \leq a_3,$ then
\begin{align*}
\DL(a_1,a_2)&=\DL(a_1,a_3);
\\
\DR(a_1,a_3)&=\DR(a_2,a_3);
\\
\DR(a_1,a_2)&=\DL(a_2,a_3).
\end{align*}
\end{Le}

\begin{proof}
Let us prove the first identity, the others are similar. We find the coefficients $\alpha_1$ and $\alpha_2$ such that 
$\gamma'(a_3)=\alpha_1\gamma'(a_1)+\alpha_2\gamma'(a_2)$. Then $g'(a_3)=\alpha_1 g'(a_1)+\alpha_2 g'(a_2)$. 
We substitute this to $\DL(a_1,a_3)$ and obtain
$$
\DL(a_1,a_3) = \frac{
\det
\begin{pmatrix}
\gamma'(a_1)\\
\gamma'(a_3)\\
\gamma''(a_1)
\end{pmatrix}}
{
\det
\begin{pmatrix}
g'(a_1)\\
g'(a_3)
\end{pmatrix}
}=
\frac{\alpha_2
\det
\begin{pmatrix}
\gamma'(a_1)\\
\gamma'(a_2)\\
\gamma''(a_1)
\end{pmatrix}}
{\alpha_2
\det
\begin{pmatrix}
g'(a_1)\\
g'(a_2)
\end{pmatrix}
}=
\DL(a_1,a_2).
$$ 
\end{proof}

\begin{Def}
Let $\mathfrak{L}$ be a linearity domain such that the intersection $\mathfrak{L} \cap \FixedBoundary\Omega$ contains at least two points, 
let $g(a) \in \mathfrak{L} \cap \FixedBoundary\Omega$. By analogy with Definition~\ref{Def190401} we define the value of
the force function $F$ at $a$ as 
$$
F(a) = \frac{D}{g_1'(a)\kappa_2'(a)},
$$
where $D=\DL(a,b)$ if $a<b$ and $D=\DR(b,a)$ if $a>b$, and $b \ne a$ is any point, such that 
$g(b) \in \mathfrak{L} \cap \FixedBoundary\Omega$. 
\end{Def}

\begin{Rem}
The force function is well defined \textup(does not depend on the choice of the point $b$\textup) on the set $\{a \in \mathbb{R}\colon g(a) \in \mathfrak{L}\}$ due to 
Lemma~\ref{ThreePointsOnOneLine}.
\end{Rem}

\begin{Le}
\label{ThreePointsAndTheArea}
Let~$a_1$\textup,~$a_2$\textup, and~$a_3$ be such that~$\gamma'(a_i)$\textup, $i = 1,2,3$\textup, lie in one plane. 
Suppose that the pairs~$(a_1,a_2)$ and~$(a_2,a_3)$ satisfy the cup equation. Then~$(a_1,a_3)$ satisfies the cup 
equation as well.
\end{Le}

\begin{proof}
The cup equations for the pairs $(a_1,a_2)$ and $(a_2,a_3)$ mean 
that the vectors $\gamma(a_2)-\gamma(a_1)$ and $\gamma(a_3)-\gamma(a_2)$ lie in the same plane as all the $\gamma'(a_i)$, hence 
$\gamma(a_3)-\gamma(a_1)$ also lies there. Therefore the cup equation holds for the pair $(a_1,a_3)$.
\end{proof}

Now we are equipped to describe all the remaining linearity domains. We start with the domains 
that are not separated from the free boundary of~$\Omega$. 
\begin{figure}[h]
\begin{center}
\includegraphics[width = 0.59\linewidth]{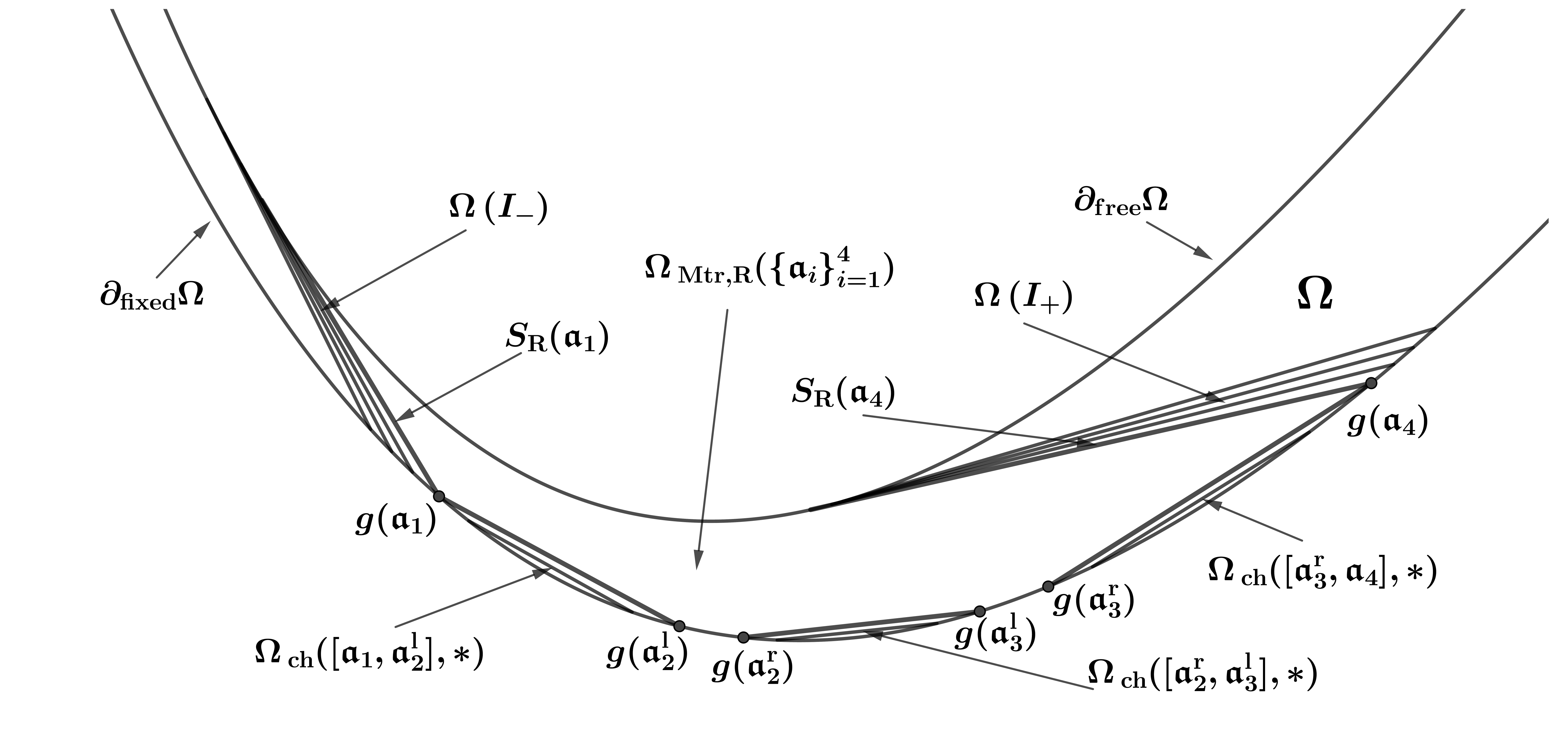}
\includegraphics[width = 0.39\linewidth]{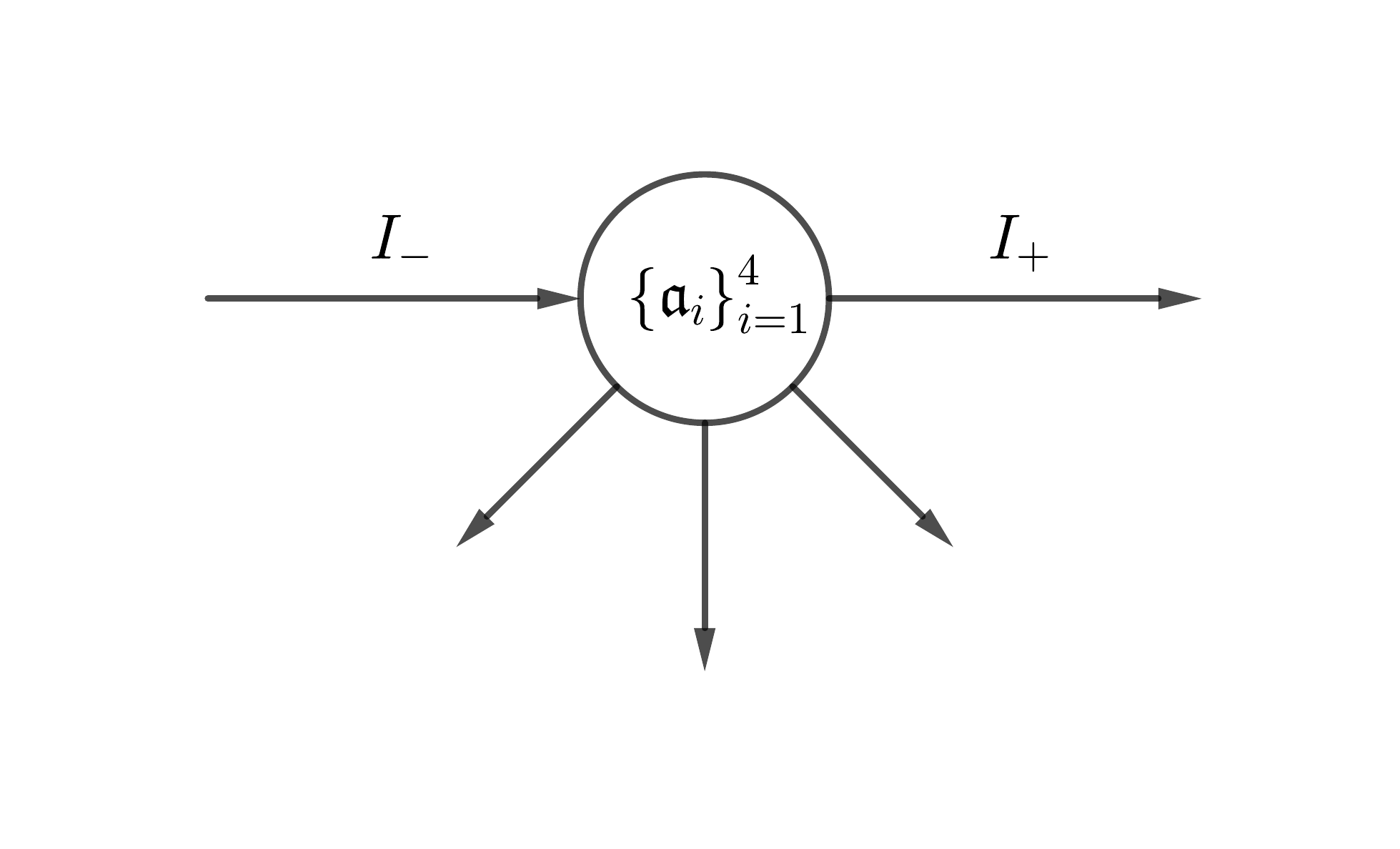}
\caption{A right multitrolleybus for $k = 4$ with adjacent domains and their graphical representation.}
\label{fig:multitroll}
\label{fig:MtrGr}
\end{center}
\end{figure}
The boundary of such a domain, provided it is compact, consists of the arcs~$g(\mathfrak{a}_i)$, $i = 1,2,\ldots, k$, 
the chords~$[g(\mathfrak{a}_i^{\mathrm r}),g(\mathfrak{a}_{i+1}^{\mathrm l})]$, $i = 1,2, \ldots, k-1$, 
two tangents $S(\mathfrak{a}_1^{\mathrm l})$ and $S(\mathfrak{a}_{k}^{\mathrm r})$,  
and the arc of the free boundary. 
We classify the multifigures with respect to the orientation of these tangents. Namely, if 
we have $\Sr(\mathfrak{a}_1^{\mathrm l})$ and $\Sr(\mathfrak{a}_{k}^{\mathrm r})$, then we get 
a \emph{right multitrolleybus}\index{multitrolleybus} 
denoted by~$\MTTR(\{\mathfrak{a}_i\}_{i=1}^k)$\index[symbol]{$\MTTR$}, see Figure~\ref{fig:multitroll}; if they are 
$\Sl(\mathfrak{a}_1^{\mathrm l})$ and $\Sl(\mathfrak{a}_{k}^{\mathrm r})$, then we have 
a \emph{left multitrolleybus} denoted by~$\MTTL(\{\mathfrak{a}_i\}_{i=1}^k)$\index[symbol]{$\MTTL$}.
If we have $\Sl(\mathfrak{a}_1^{\mathrm l})$ and $\Sr(\mathfrak{a}_{k}^{\mathrm r})$, then the linearity domain is called 
a~\emph{multicup}\index{multicup} and is denoted by~$\MTC(\{\mathfrak{a}_i\}_{i=1}^k)$\index[symbol]{$\MTC$}, see Figure~\ref{fig:multicup4}. 
\begin{figure}[h]
\begin{center}
\includegraphics[width = 0.59\linewidth]{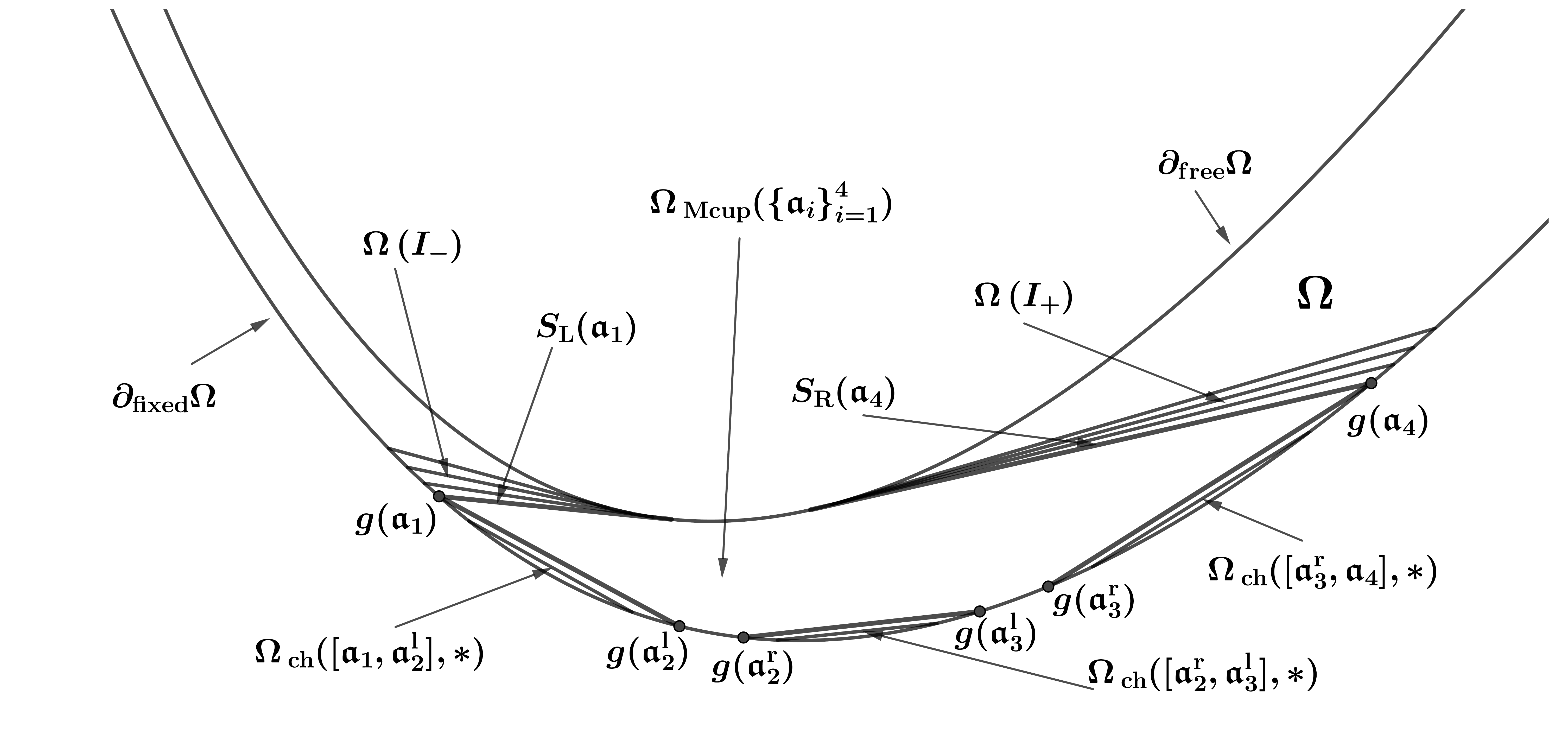}
\includegraphics[width = 0.4\linewidth]{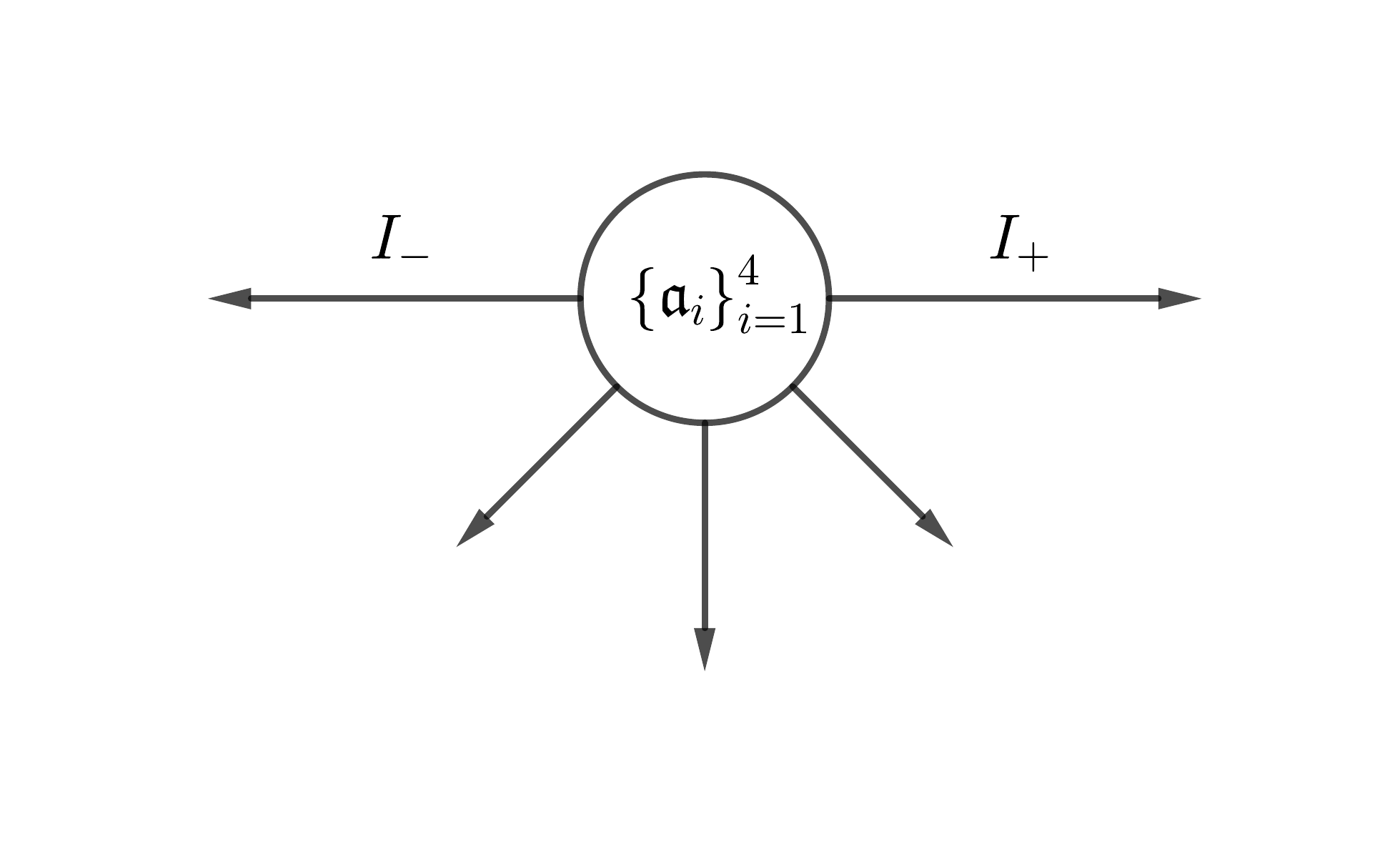}
\caption{A mutlicup for $k=4$ with adjacent domains and their graphical representation.}
\label{fig:multicup4}
\label{fig:McupGr}
\end{center}
\end{figure}

We distinguish the case where the two border tangents $\Sl(\mathfrak{a}_1^{\mathrm l})$ and 
$\Sr(\mathfrak{a}_{k}^{\mathrm r})$ lie on one line and say that in this case the multicup is \emph{full}. 
Finally, if we have  $\Sr(\mathfrak{a}_1^{\mathrm l})$ and $\Sl(\mathfrak{a}_{k}^{\mathrm r})$, then 
the domain of linearity $\mathfrak{L}$ is called a~\emph{multibirdie}\index{multibirdie} and is denoted 
by~$\MTB(\{\mathfrak{a}_i\}_{i=1}^k)$\index[symbol]{$\MTB$}, see Figure~\ref{fig:multibirdie}.
\begin{figure}[h]
\begin{center}
\includegraphics[width = 0.59\linewidth]{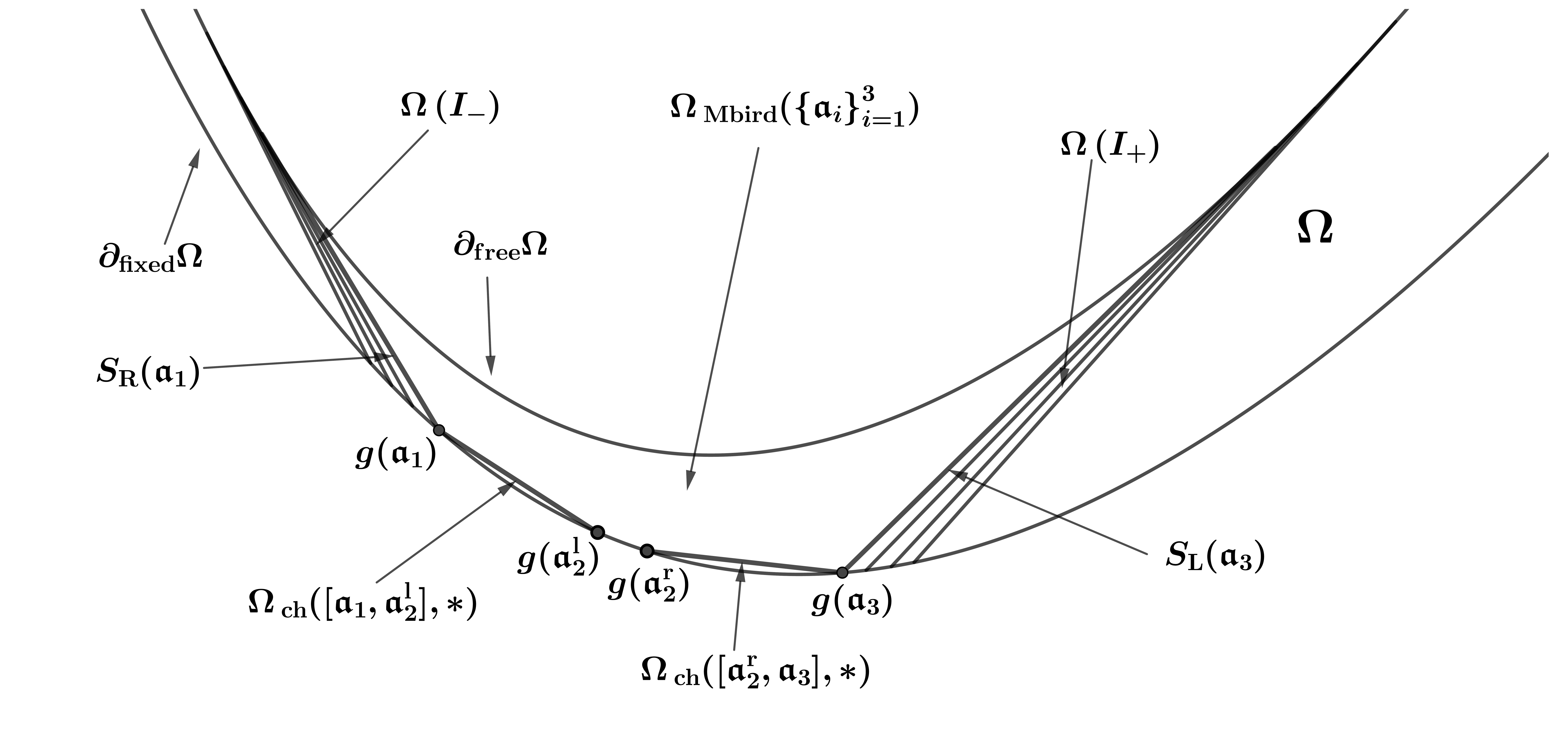}
\includegraphics[width = 0.4\linewidth]{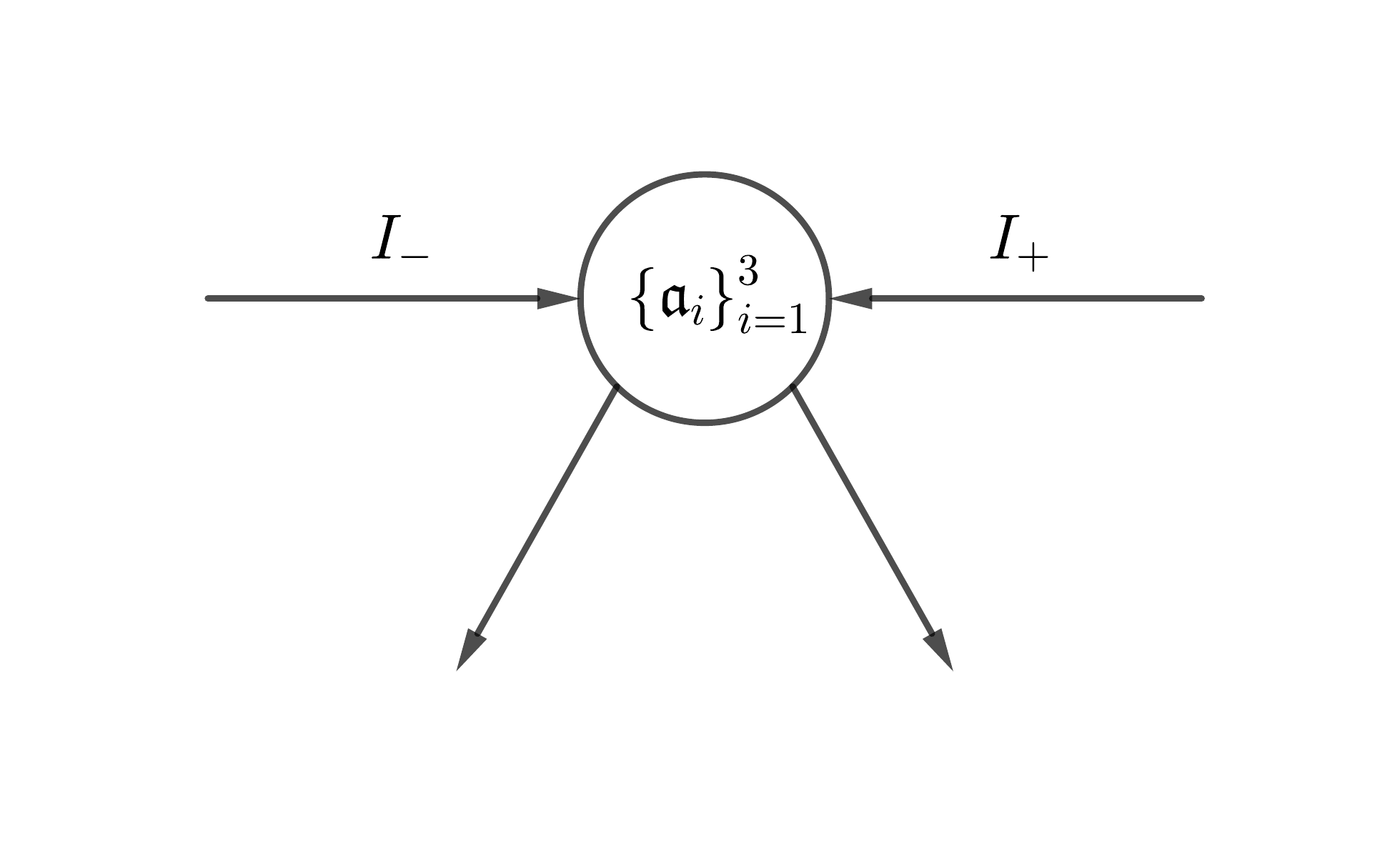}
\caption{A multibirdie for $k=3$ with adjacent domains and their graphical representation.}
\label{fig:multibirdie}
\label{fig:MtbGr}
\end{center}
\end{figure}
Graphical representation for a multifigure~$\mathfrak{L}$ built over~$\{\mathfrak{a}_i\}_{i=1}^k$ 
is drawn by the following rule. The domain~$\mathfrak{L}$ corresponds to a single vertex. 
It has~$k-1$ outgoing edges representing the chordal
domains~$\Ch([\mathfrak{a}_i^{\mathrm r},\mathfrak{a}_{i+1}^{\mathrm l}],*)$,~$i = 1,2,\ldots,k-1$. 
There are two more edges corresponding to two tangent domains surrounding~$\mathfrak{L}$. 
They are both outgoing if~$\mathfrak{L}$ is a multicup and both incoming in the case
where~$\mathfrak{L}$ is a multibirdie. If~$\mathfrak{L}$ is a multitrolleybus, then it 
has one incoming and one outgoing edge. We provide examples of graphs for the multifigures drawn on
Figures~\ref{fig:multitroll},~\ref{fig:multicup4}, and~\ref{fig:multibirdie}.

We will also consider unbounded domains of linearity: a multicup $\MTC(\{\mathfrak{a}_i\}_{i=1}^k)$ with at least one of the intervals $\mathfrak{a}_1$ or $\mathfrak{a}_k$ being a ray, a right multitrolleybus $\MTTR(\{\mathfrak{a}_i\}_{i=1}^k)$, where $\mathfrak{a}_k$ is a ray that lasts to $+\infty$, or a left multitrolleybus $\MTTL(\{\mathfrak{a}_i\}_{i=1}^k)$, where $\mathfrak{a}_1$ is a ray that lasts to $-\infty$. Such domains do not have one of border tangents. In such a case, the vertex representing this domain does not have the corresponding edge. 

Consider the case of a linearity domain $\mathfrak{L}$ that is separated from the free boundary. 
The boundary of $\mathfrak{L}$ consists of the arcs~$g(\mathfrak{a}_i)$,~$i = 1,2,\ldots,k$, 
the chords~$[g(\mathfrak{a}_{i}^{\textup r}),g(\mathfrak{a}_{i+1}^{\textup l})]$,~$i = 1,2,\ldots,k-1$, 
and the chord~$[g(\mathfrak{a}_{1}^l),g(\mathfrak{a}_{k}^{r})]$. Such a construction is called 
a \emph{closed multicup}\index{multicup! closed multicup} and is denoted by $\ClMTC(\{\mathfrak{a}_i\}_{i=1}^k)$\index[symbol]{$\ClMTC$}. 
It is represented graphically in the following way. It has one incoming edge 
representing~$\Ch(*,[\mathfrak{a}_1^{\mathrm l},\mathfrak{a}_k^{\mathrm r}])$ and several outgoing edges 
corresponding to the chordal 
domains~$\Ch([\mathfrak{a}_i^{\mathrm r},\mathfrak{a}_{i+1}^{\mathrm l}],*)$,~$i=1,2,\ldots,k-1$. 
For example, it may look like the one on Figure~\ref{fig:ClMcupGr}.
\begin{figure}[h]
\begin{center}
\includegraphics[width = 0.59\linewidth]{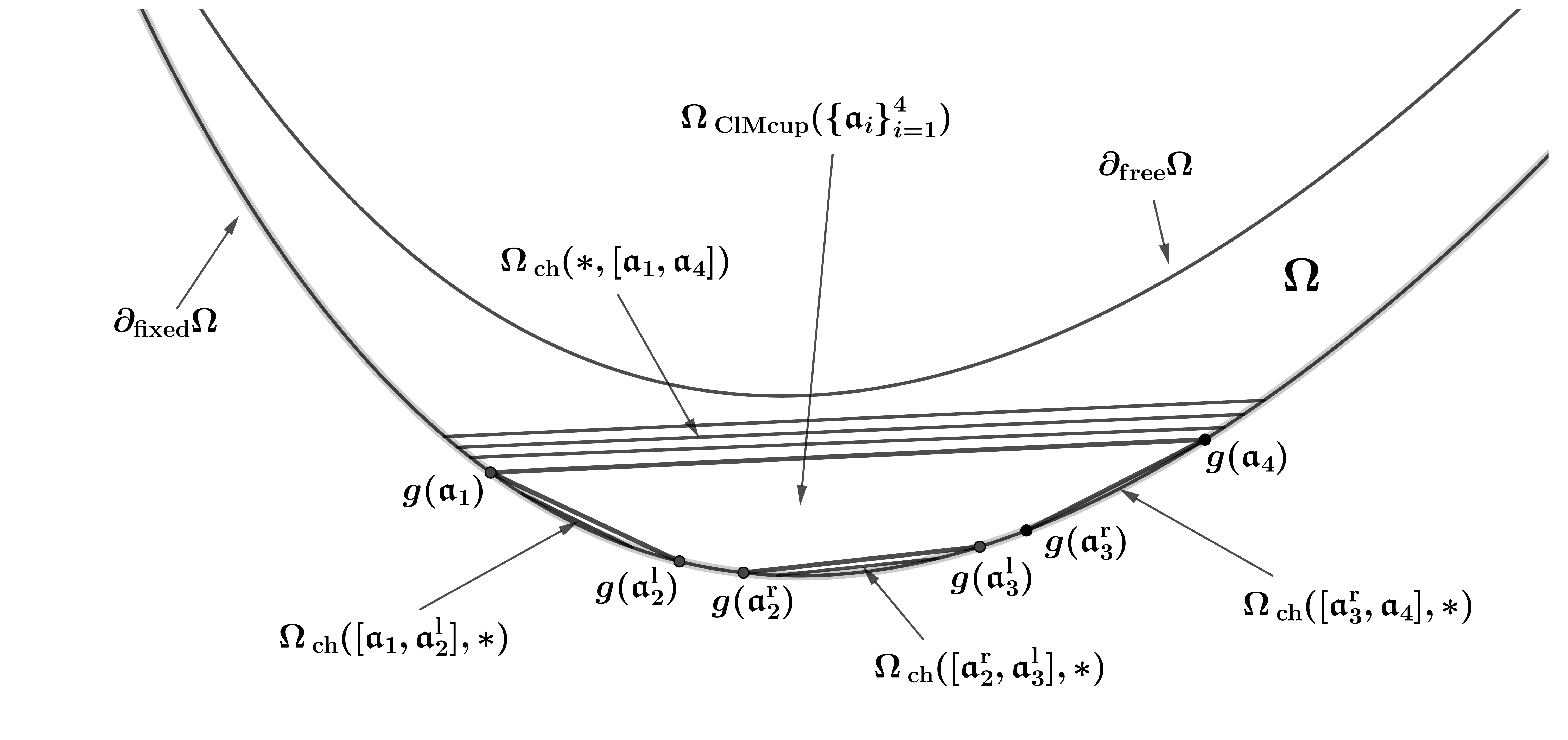}
\includegraphics[width = 0.4\linewidth]{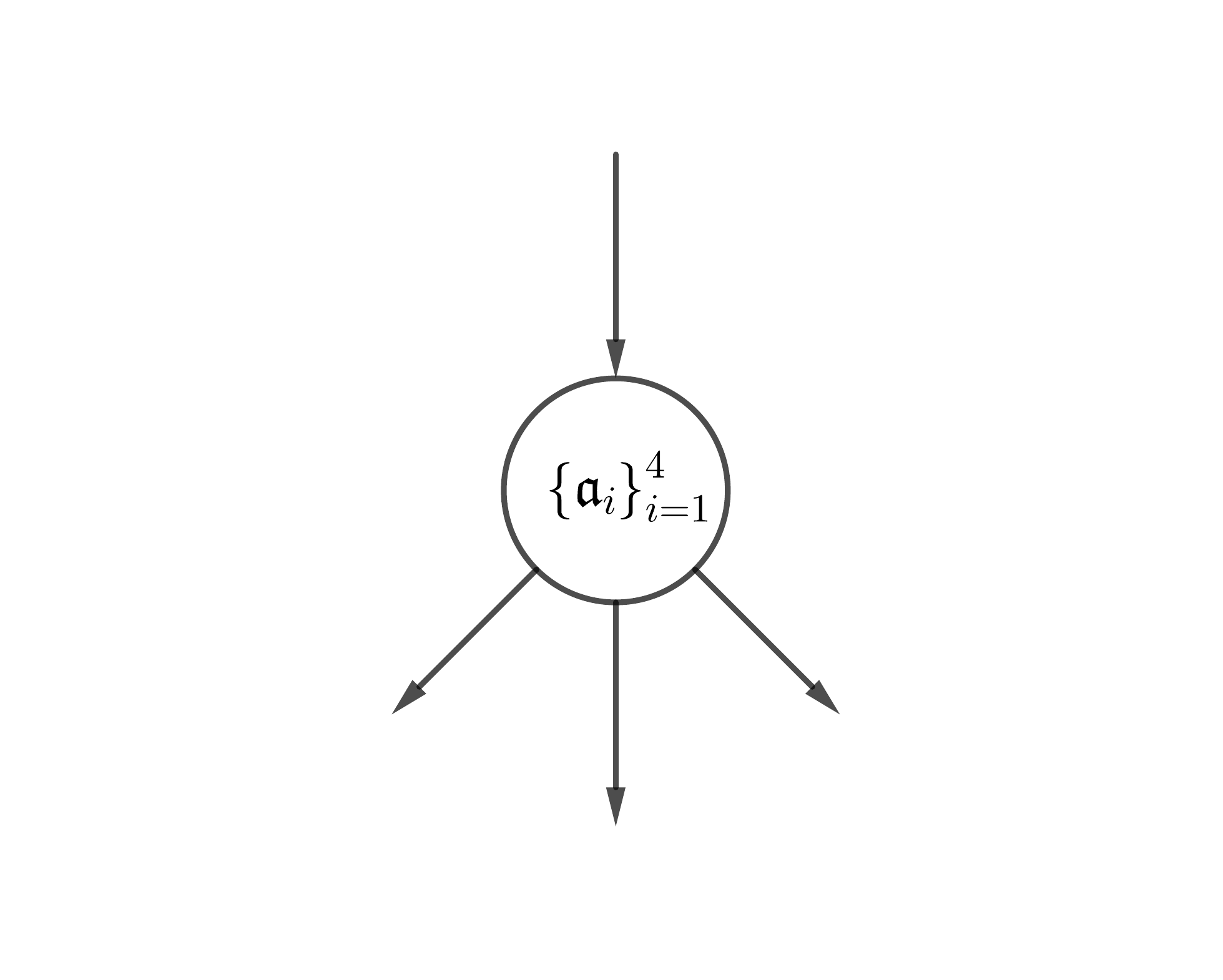}
\caption{
An example of the graph for a closed multicup with adjacent chordal domains.}
\label{fig:ClMcupGr}
\end{center}
\end{figure}

The following proposition gives sufficient conditions for concatenation of a linearity domain with 
the surrounding fences (tangent domains and chordal domains). Formally, it is more general 
than~Proposition~\ref{St131001}, but the proof is similar.

\begin{St}
\label{St211101}
Let $\mathfrak{L}$ be a domain of linearity surrounded by several fences. Suppose that the function $B$ is 
a standard candidate on $\mathfrak{L}$ and on each surrounding fence. 
Let $a =\inf\{t \in \mathbb{R} \colon g(t)\in \mathfrak{L}\}$ and 
$b=\sup\{t \in \mathbb{R} \colon g(t)\in \mathfrak{L}\}$. If the force function is continuous at $a$ and $b$, 
then $B$ is a $C^1$-smooth Bellman candidate on its domain.   
\end{St}

\begin{Rem}
In the previous proposition, if $\mathfrak{L}$ is a closed multicup, then it is surrounded by the chordal 
domains and the condition of continuity of the force function always takes place. 
\end{Rem}

Now is the time to define tails of a linearity domain $\mathfrak{L}$ that contains at least two points on 
the lower boundary. As in~Proposition~\ref{St211101} we define 
$a=\inf\{t\in\mathbb{R}\colon g(t)\in\mathfrak{L}\}$ and $b=\sup\{t\in\mathbb{R}\colon g(t)\in\mathfrak{L}\}$. 

\begin{Def}
\label{Rem211101}
We define the tails and the forces of a linearity domain~$\mathfrak{L}$ as the tails and the forces for 
the chord $[g(a),g(b)]$ \textup(which possibly does not lie in $\Omega$\textup).
Namely, see Definitions~\ref{Def031002},~\ref{Def031003} of the tails and 
formulas~\eqref{eq101101},~\eqref{eq101102} for the corresponding forces.
\end{Def}

\section{Combinatorial properties of foliations}
\label{s35} 

The material of this section essentially repeats Section~3.5 in~\cite{ISVZ2018}. The reason for this repetition 
is a slight change of the notation we are forced to do since some natural parameters used in~\cite{ISVZ2018} do not 
exist in our general setting. 

\subsection{Gluing composite figures}
\label{s344}
In this subsection we present several formulas that allow us to consider a part of the foliation as 
a union of elementary domains in different ways. An example has already been given in~\eqref{FirstFormula} 
(see also the description of the symbol~$\biguplus$ on page~\pageref{biguplusdef}).

We start with the formula which describes gluing of an angle $\Ang(a)$ with a long chord $[g(a),g(b)]$ 
(see Figure~\ref{CupPlusAngleRGraph}). Their union forms a trolleybus $\Troll(a,b)$:
\eq{CupPlusAngleR}{
\Ang(a) \biguplus \Lt(a,a) \biguplus \ [g(a),g(b)] = 
\RTroll(a,b).
}
\begin{figure}[h]
\includegraphics[width = 0.59\linewidth]{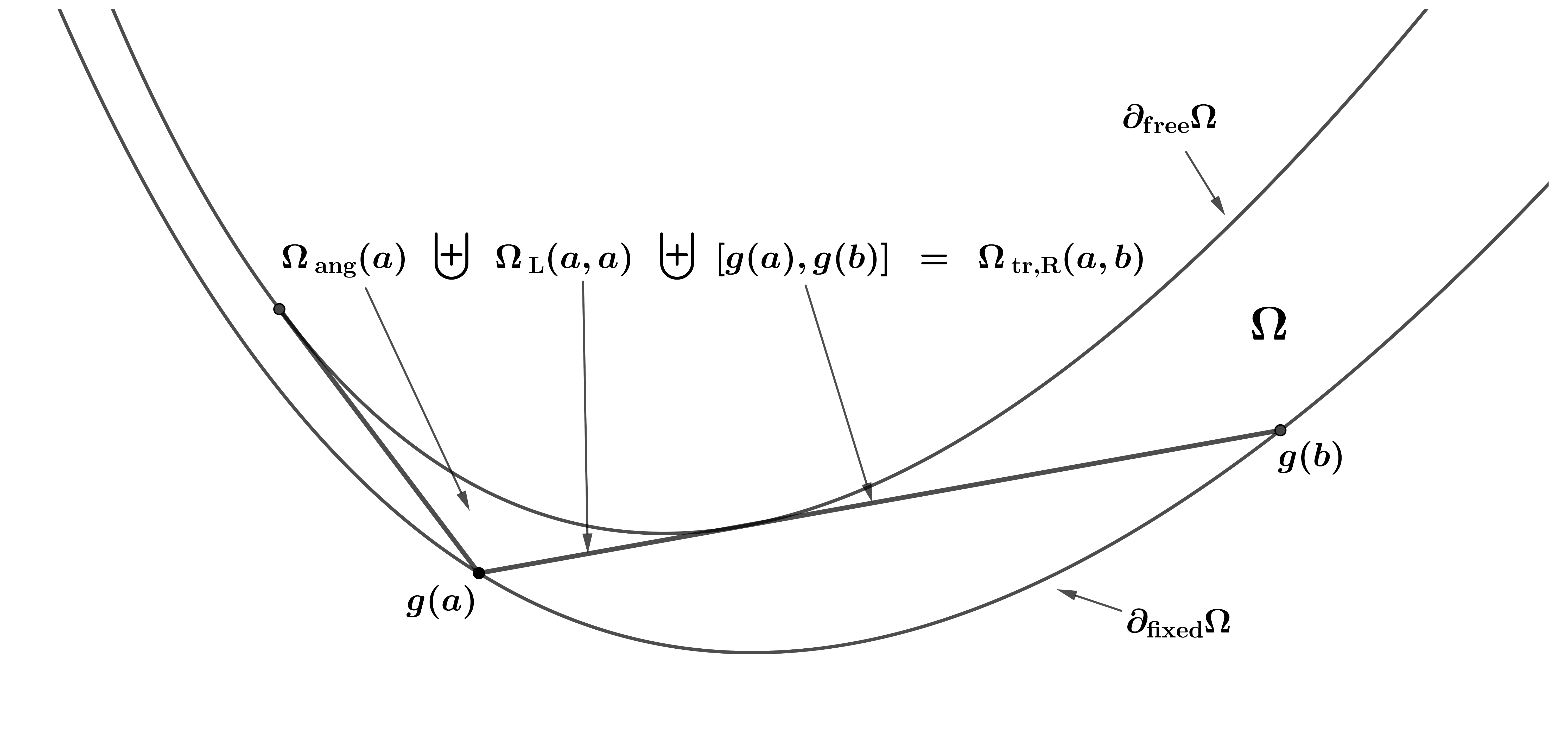}
\raisebox{35pt}{\includegraphics[width = 0.4\linewidth]{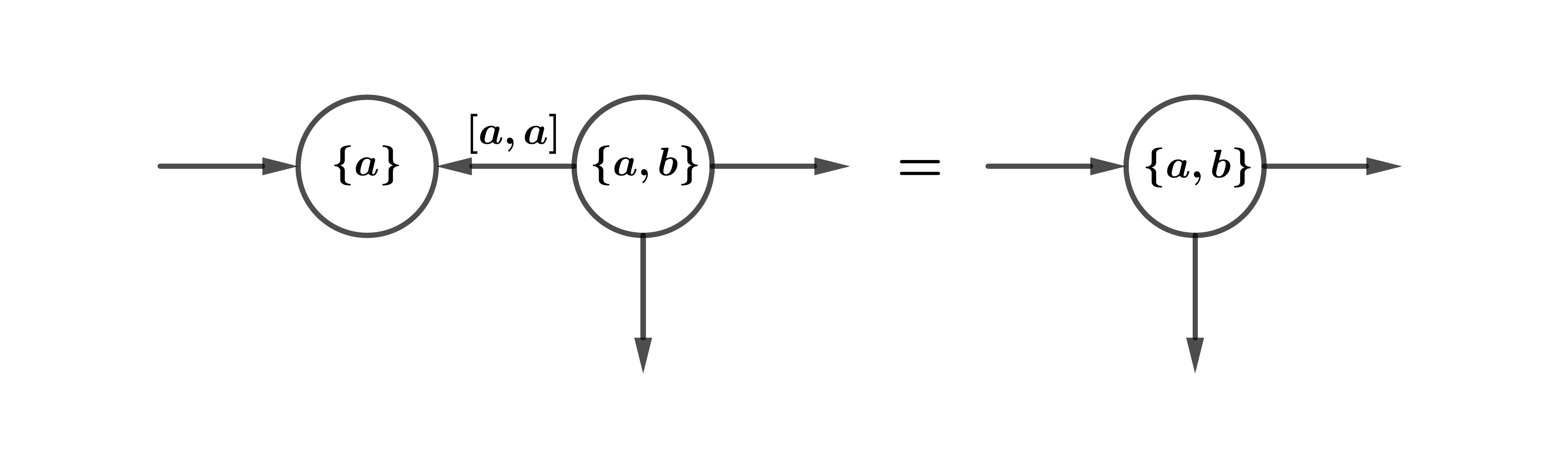}}
\caption{A graphical representation of formula~\eqref{CupPlusAngleR}.}
\label{CupPlusAngleRGraph}
\end{figure}

Similarly,
\eq{CupPlusAngleL}{
[g(a),g(b)]\ \biguplus \Rt(b,b) \biguplus \Ang(b)   = \LTroll(a,b).
}
Both these formulas can be informally named as~``{\bf angle + long chord = trolleybus}''.

We have already considered an example of a more complicated formula~\eqref{FirstFormula}:
\eq{RTrolleybusPlusAngle}{
\RTroll(a,b)  \biguplus \Rt(b,b) \biguplus \Ang(b) = \Bird(a,b);
}
\eq{LTrolleybusPlusAngle}{
\Ang(a) \biguplus \Lt(a,a)  \biguplus \LTroll(a,b)= \Bird(a,b),
}
which can be informally named as ``{\bf birdie = angle + trolleybus}''.

We provide the same-fashioned formulas for other domains. We leave their verification to the reader.

\medskip
\paragraph{{\bf Angle + multicup = multitrolleybus}}
\eq{AnglePlusMulticupLeft}{
\Ang(\mathfrak{a}_1^{\mathrm l}) \biguplus \Lt(\mathfrak{a}_1^{\mathrm l},\mathfrak{a}_1^{\mathrm l}) 
\biguplus \MTC(\{\mathfrak{a}_i\}_{i=1}^k) = \MTTR(\{\mathfrak{a}_i\}_{i=1}^k);
}

\begin{figure}[h]
\includegraphics[width=0.49\linewidth]{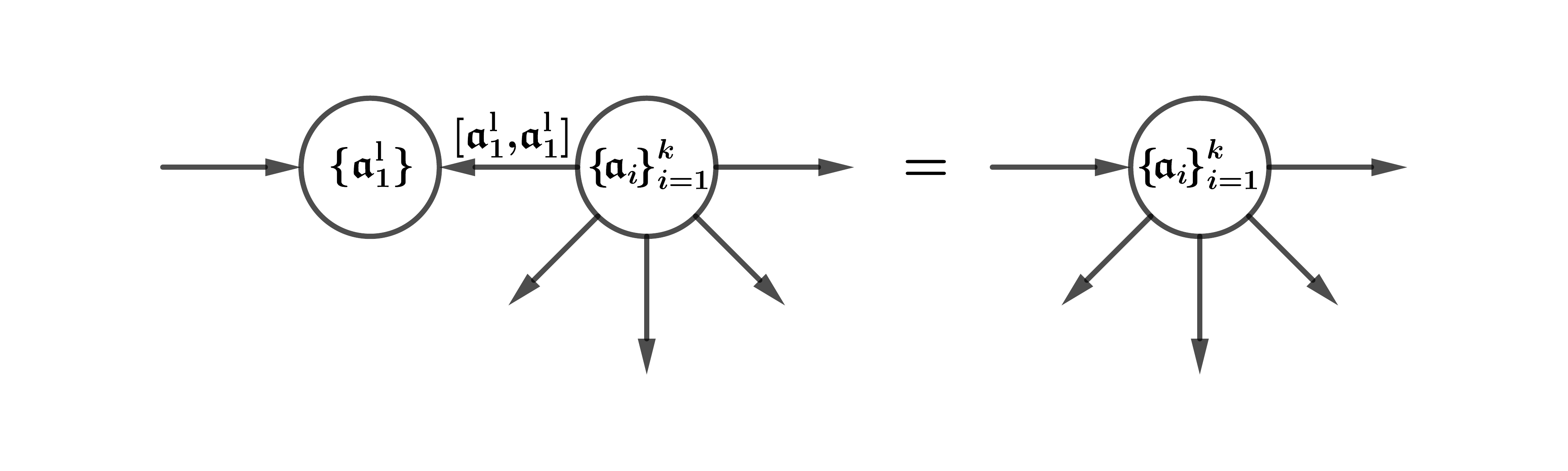}
\includegraphics[width=0.49\linewidth]{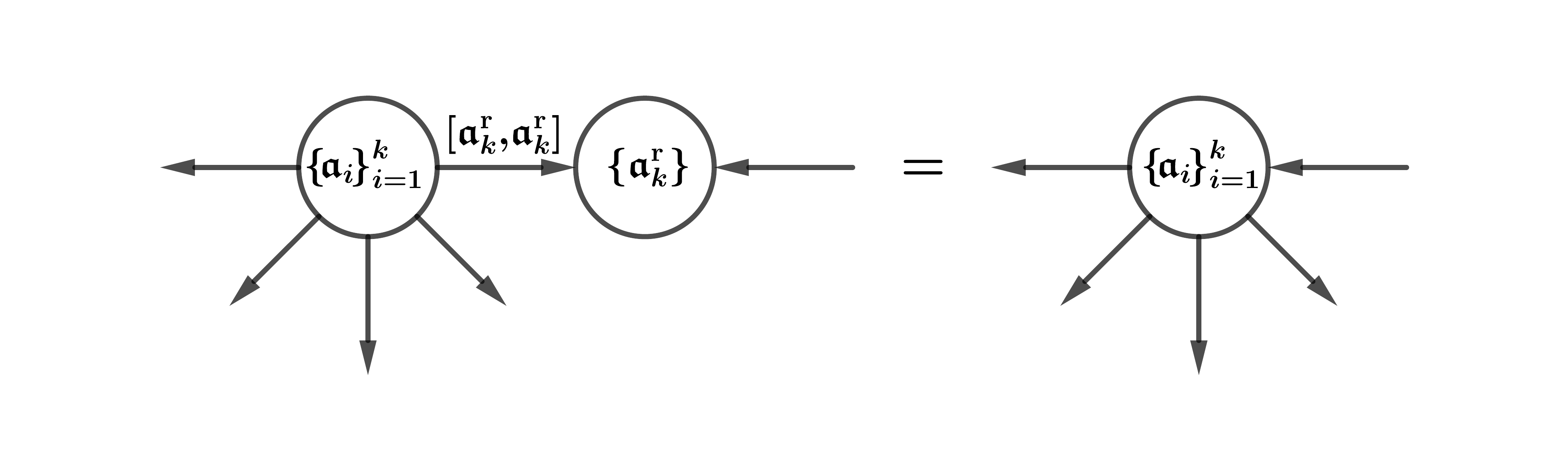}
\caption{The graphs for formulas~\eqref{AnglePlusMulticupLeft} and~\eqref{AnglePlusMulticupRight}.}
\label{fig_AnglePlusMcup}
\end{figure}

\eq{AnglePlusMulticupRight}{
\MTC(\{\mathfrak{a}_i\}_{i=1}^k) \biguplus  \Rt(\mathfrak{a}_k^{\mathrm r},\mathfrak{a}_k^{\mathrm r}) 
\biguplus \Ang(\mathfrak{a}_k^{\mathrm r})  = \MTTL(\{\mathfrak{a}_i\}_{i=1}^k).
}

\paragraph{{\bf Long chord + multibirdie = multitrolleybus}}
\eq{ChordalDomainPlusMultibirdieL}{
[g(a_0), g(\mathfrak{a}_1^{\mathrm l})] \biguplus \Rt(\mathfrak{a}_1^{\mathrm l},\mathfrak{a}_1^{\mathrm l})
\biguplus \MTB(\{\mathfrak{a}_i\}_{i=1}^k) = \MTTL(\{a_0\} \cup\{\mathfrak{a}_i\}_{i=1}^k),
}
where $[g(a_0), g(\mathfrak{a}_1^{\mathrm l})]$ is a long chord;

\eq{ChordalDomainPlusMultibirdieR}{
\MTB(\{\mathfrak{a}_i\}_{i=1}^k) \biguplus \Lt(\mathfrak{a}_k^{\mathrm r},\mathfrak{a}_k^{\mathrm r}) 
\biguplus \ [g(\mathfrak{a}_k^{\mathrm r}), g(a_{k+1})] =\MTTR(\{\mathfrak{a}_i\}_{i=1}^k \cup \{a_{k+1}\}),
}
where $[g(\mathfrak{a}_k^{\mathrm r}), g(a_{k+1})]$ is a long chord.

\begin{figure}[h]
\includegraphics[width = 0.49\linewidth]{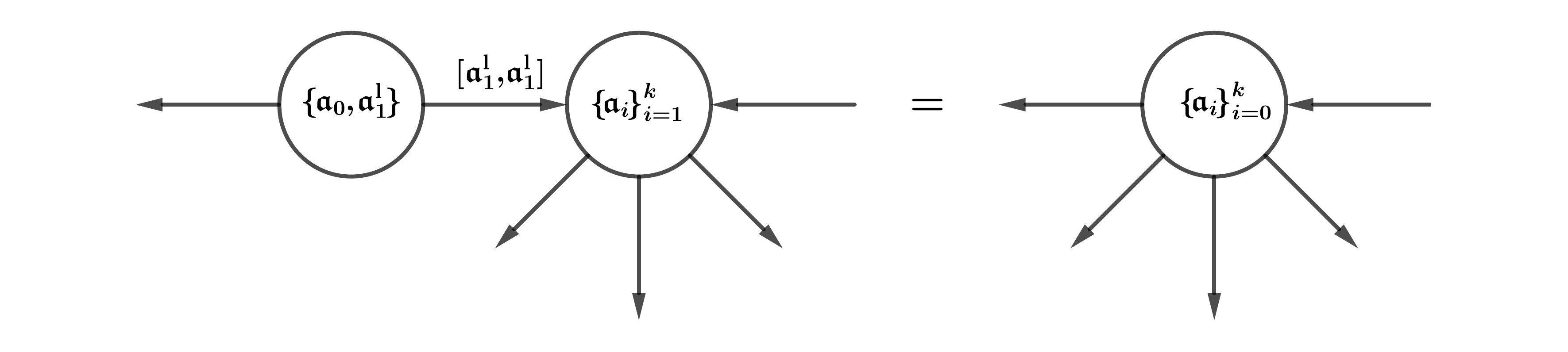}
\includegraphics[width = 0.49\linewidth]{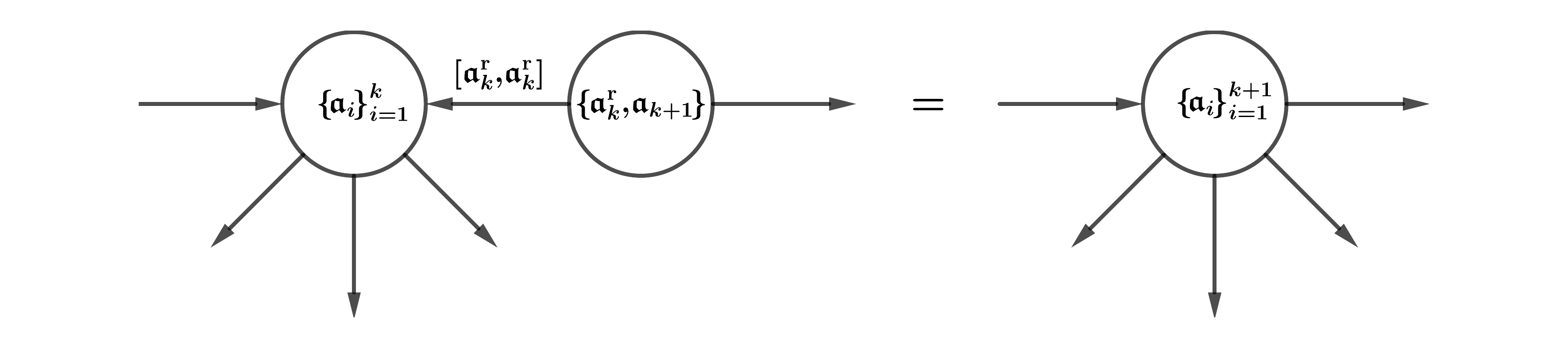}
\caption{The graphs for formulas~\eqref{ChordalDomainPlusMultibirdieL} and~\eqref{ChordalDomainPlusMultibirdieR}.}
\label{ChordalDomainPlusMultibirdieRGraph}
\end{figure}

\paragraph{{\bf Angle + multitrolleybus = multibirdie}}
\eq{AnglePlusMultiTrollebusR}{
\MTTR(\{\mathfrak{a}_i\}_{i=1}^k) \biguplus \Rt(\mathfrak{a}_k^{\mathrm r},\mathfrak{a}_k^{\mathrm r}) 
\biguplus \Ang(\mathfrak{a}_k^{\mathrm r}) = \MTB(\{\mathfrak{a}_i\}_{i=1}^k);
}
\eq{AnglePlusMultiTrollebusL}{
\Ang(\mathfrak{a}_1^{\mathrm l}) \biguplus \Lt(\mathfrak{a}_1^{\mathrm l},\mathfrak{a}_1^{\mathrm l}) 
\biguplus \MTTL(\{\mathfrak{a}_i\}_{i=1}^k)= \MTB(\{\mathfrak{a}_i\}_{i=1}^k).
}

\begin{figure}[h]
\includegraphics[width=0.49\linewidth]{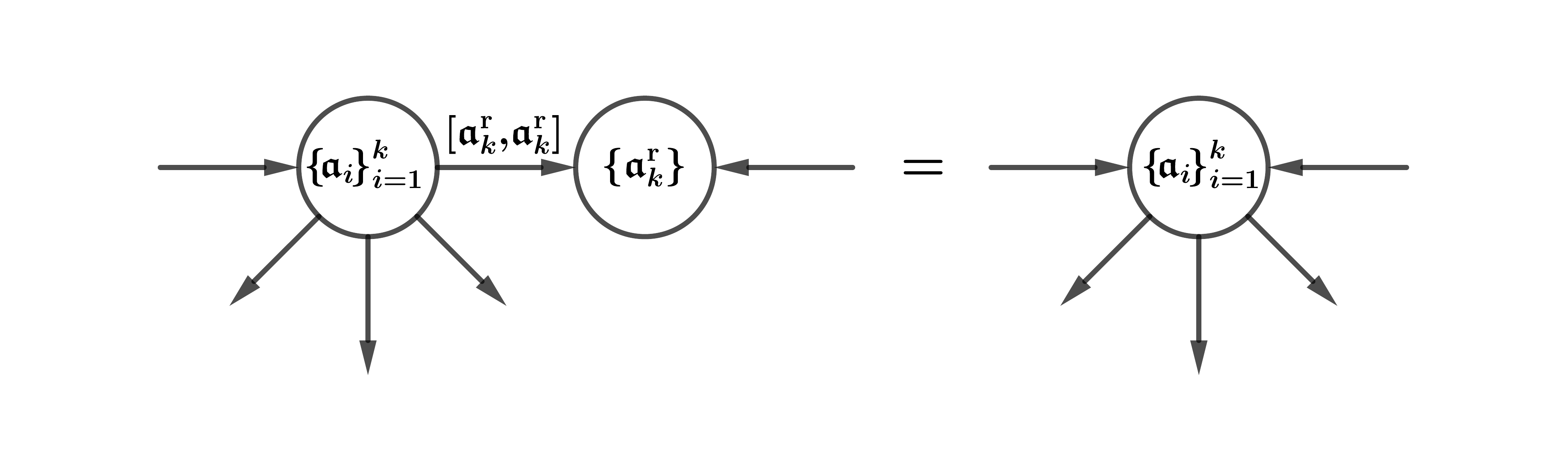}
\includegraphics[width=0.49\linewidth]{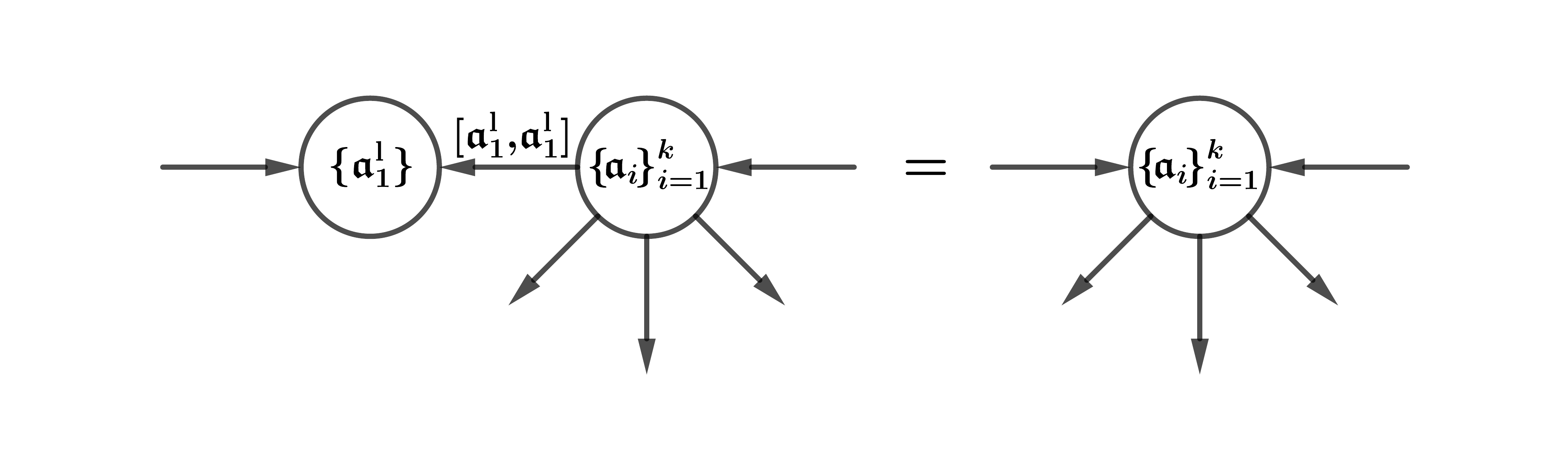}
\caption{The graphs for formulas~\eqref{AnglePlusMultiTrollebusR} and~\eqref{AnglePlusMultiTrollebusL}.}
\end{figure}

\paragraph{{\bf Long chord + multitrolleybus = multicup}}
\begin{equation}\label{ChordalDomainPlusMultitrolleybusR}
[g(a_0), g(\mathfrak{a}_1^{\mathrm l})]  \biguplus \Rt(\mathfrak{a}_1^{\mathrm l},\mathfrak{a}_1^{\mathrm l})
\biguplus \MTTR(\{\mathfrak{a}_i\}_{i=1}^k) = \MTC(\{a_0\} \cup\{\mathfrak{a}_i\}_{i=1}^k),
\end{equation}
where $[g(a_0), g(\mathfrak{a}_1^{\mathrm l})]$ is a long chord;

\begin{equation}\label{ChordalDomainPlusMultitrolleybusL}
\MTTL(\{\mathfrak{a}_i\}_{i=1}^k) \biguplus \Lt(\mathfrak{a}_k^{\mathrm r},\mathfrak{a}_k^{\mathrm r}) 
\biguplus\  [g(\mathfrak{a}_k^{\mathrm r}), g(a_{k+1})] = \MTC(\{\mathfrak{a}_i\}_{i=1}^k \cup \{a_{k+1}\}),
\end{equation}
where $[g(\mathfrak{a}_k^{\mathrm r}), g(a_{k+1})]$ is a long chord.
\begin{figure}[h]
\includegraphics[width=0.49\linewidth]{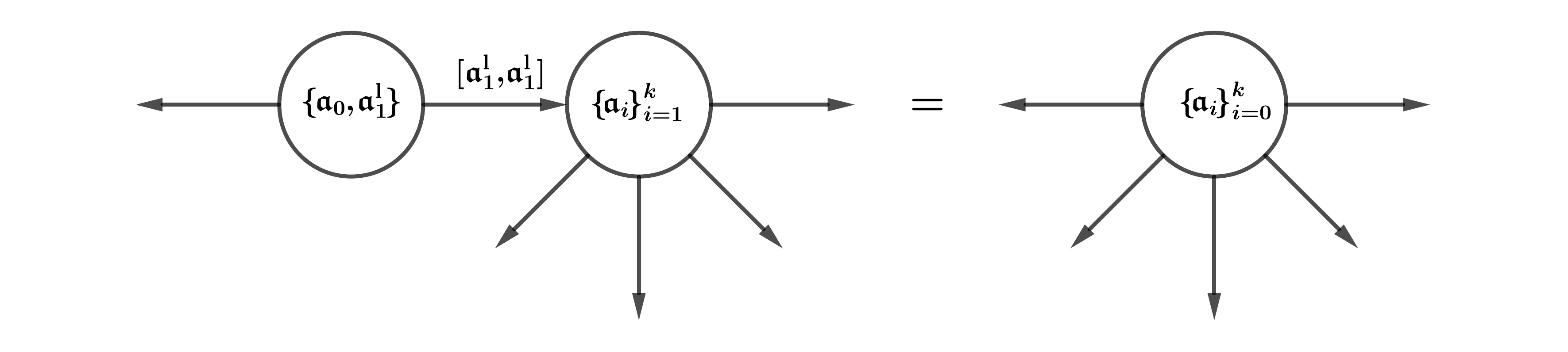}
\includegraphics[width=0.49\linewidth]{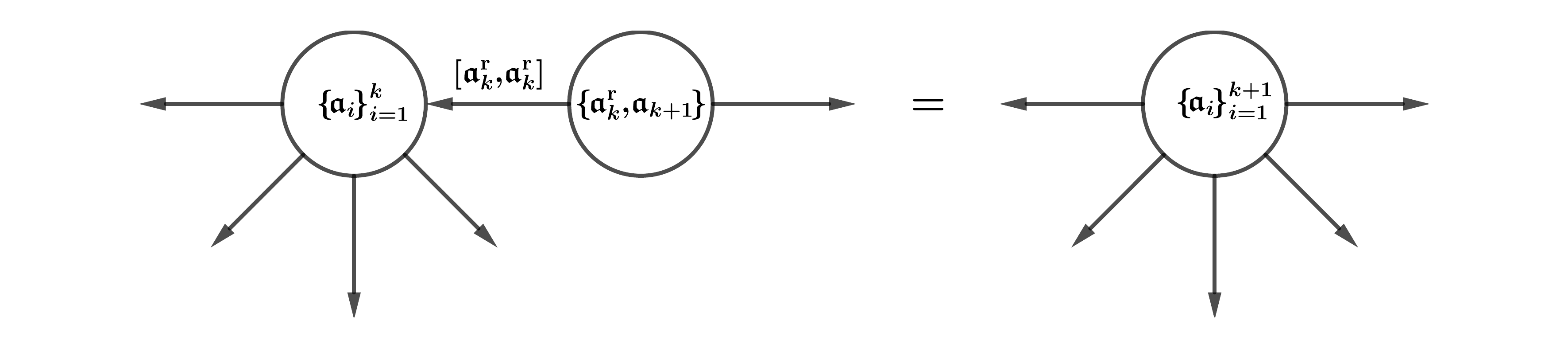}
\caption{The graphs for formulas~\eqref{ChordalDomainPlusMultitrolleybusR} and~\eqref{ChordalDomainPlusMultitrolleybusL}.}
\end{figure}

%\medskip
%\paragraph{{\bf Multicup + multitrolleybus = multicup}}
%\begin{equation}\label{MTCplusMTTR}
%\MTC(\{\mathfrak{a}_i\}_{i=1}^k) \biguplus \Rt(\mathfrak{a}_k^{\mathrm r},\mathfrak{a}_k^{\mathrm r}) 
%\biguplus \MTTR(\{\mathfrak{a}_i\}_{i=k+1}^{m}) = \MTC(\{\mathfrak{a}_i\}_{i=1}^m),\quad \mathfrak{a}_k^{\mathrm r} 
%= \mathfrak{a}_{k+1}^{\mathrm l}, m \geq k+1;
%\end{equation}
%\begin{equation}\label{MTCplusMTTL}
%\MTTL(\{\mathfrak{a}_i\}_{i=1}^k) \biguplus \Lt(\mathfrak{a}_k^{\mathrm r},\mathfrak{a}_k^{\mathrm r}) 
%\biguplus \MTC(\{\mathfrak{a}_i\}_{i=k+1}^{m}) = \MTC(\{\mathfrak{a}_i\}_{i=1}^m),\quad \mathfrak{a}_k^{\mathrm r} 
%= \mathfrak{a}_{k+1}^{\mathrm l}, m \geq k+1.
%\end{equation}

%\paragraph{{\bf Multicup + birdie = multitrolleybus}}
%\begin{equation}\label{MTCplusBirdieR}
%\MTC(\{\mathfrak{a}_i\}_{i=1}^k) \biguplus \Rt(\mathfrak{a}_k^{\mathrm r},\mathfrak{a}_k^{\mathrm r}) 
%\biguplus \MTB(\{\mathfrak{a}_i\}_{i=k+1}^{m}) = \MTTL(\{\mathfrak{a}_i\}_{i=1}^m),\quad \mathfrak{a}_k^{\mathrm r} 
%= \mathfrak{a}_{k+1}^{\mathrm l}, m \geq k+1;
%\end{equation}
%\begin{equation}\label{MTCplusBirdieL}
%\MTB(\{\mathfrak{a}_i\}_{i=1}^k) \biguplus \Lt(\mathfrak{a}_k^{\mathrm r},\mathfrak{a}_k^{\mathrm r}) 
%\biguplus \MTC(\{\mathfrak{a}_i\}_{i=k+1}^{m}) = \MTTR(\{\mathfrak{a}_i\}_{i=1}^m),\quad \mathfrak{a}_k^{\mathrm r} 
%= \mathfrak{a}_{k+1}^{\mathrm l}, m \geq k+1.
%\end{equation}

\paragraph{{\bf Multitrolleybus = trolleybus parade}}
\begin{equation}\label{RMultitrolleybusDesintegration}
\MTTR(\{\mathfrak{a}_i\}_{i=1}^k) \!=\! \Big(\biguplus_{i=1}^{k}\MTTR(\{\mathfrak{a}_i\})\Big) 
\biguplus \Big(\biguplus_{i=1}^{k-1}\!\big( \Rt(\mathfrak{a}_{i}^{\mathrm r},\mathfrak{a}_{i}^{\mathrm r}) 
\biguplus \RTroll(\mathfrak{a}_i^{\mathrm r},\mathfrak{a}_{i+1}^{\mathrm l}) 
\biguplus \Rt(\mathfrak{a}_{i+1}^{\mathrm l},\mathfrak{a}_{i+1}^{\mathrm l})\big)\Big);
\end{equation}
\begin{equation}\label{LMultitrolleybusDesintegration}
\MTTL(\{\mathfrak{a}_i\}_{i=1}^k) \!=\! \Big(\biguplus_{i=1}^{k}\MTTL(\{\mathfrak{a}_i\})\Big) 
\biguplus \Big(\biguplus_{i=1}^{k-1}\!\big( \Lt(\mathfrak{a}_{i}^{\mathrm r},\mathfrak{a}_{i}^{\mathrm r}) 
\biguplus \LTroll(\mathfrak{a}_i^{\mathrm r},\mathfrak{a}_{i+1}^{\mathrm l}) 
\biguplus \Lt(\mathfrak{a}_{i+1}^{\mathrm l},\mathfrak{a}_{i+1}^{\mathrm l})\big)\Big).
\end{equation}
In~\eqref{RMultitrolleybusDesintegration}, if some~$\mathfrak{a}_i$ is a single point, 
then we omit degenerate tangent domains~$\Rt(\mathfrak{a}_i^{\mathrm l},\mathfrak{a}_i^{\mathrm l})$ 
and $\Rt(\mathfrak{a}_i^{\mathrm r},\mathfrak{a}_i^{\mathrm r})$ and replace the multitrolleybus 
$\MTTR(\{\mathfrak{a}_i\})$ by the degenerate tangent domain~$\Rt(\mathfrak{a}_i,\mathfrak{a}_i)$. 
We perform similar replacements of the left tangent domains and multitrolleybuses 
in~\eqref{LMultitrolleybusDesintegration}.

\medskip
\paragraph{{\bf Multibirdie = right multitrolleybus + angle + left multitrolleybus}}
\begin{multline}\label{MultibirdieDesintegration}
\MTB(\{\mathfrak{a}_i\}_{i=1}^k) = 
\\
=\MTTR(\{\mathfrak{a}_i\}_{i=1}^{j-1}\cup \{\mathfrak{a}_j^{\mathrm{l}}\}) 
\biguplus\Rt(\mathfrak{a}_j^{\mathrm{l}},\mathfrak{a}_j^{\mathrm{l}})\biguplus
\MTB(\{\mathfrak{a}_j\}) \biguplus
\Lt(\mathfrak{a}_j^{\mathrm{r}},\mathfrak{a}_j^{\mathrm{r}})\biguplus
\MTTL(\{\mathfrak{a}_j^{\mathrm{r}}\}\cup\{\mathfrak{a}_i\}_{i=j+1}^{k}).
\end{multline}
Here~$j$ is an arbitrary number,~$j = 1,2, \ldots,k$. Both multitrolleybuses can be disintegrated according 
to~\eqref{RMultitrolleybusDesintegration} and~\eqref{LMultitrolleybusDesintegration}. 
If~$\mathfrak{a}_j$ is a single point, then one should change~$\MTB(\{\mathfrak{a}_j\})$ for~$\Ang(\mathfrak{a}_j)$. 
If~$\mathfrak{a}_j$ is a solid root, then $\MTB(\{\mathfrak{a}_j\})$ can be further disintegrated in two ways 
according to~\eqref{AnglePlusMultiTrollebusL} and~\eqref{AnglePlusMultiTrollebusR}. 
\begin{figure}[h!]
\includegraphics[width = 0.5\linewidth]{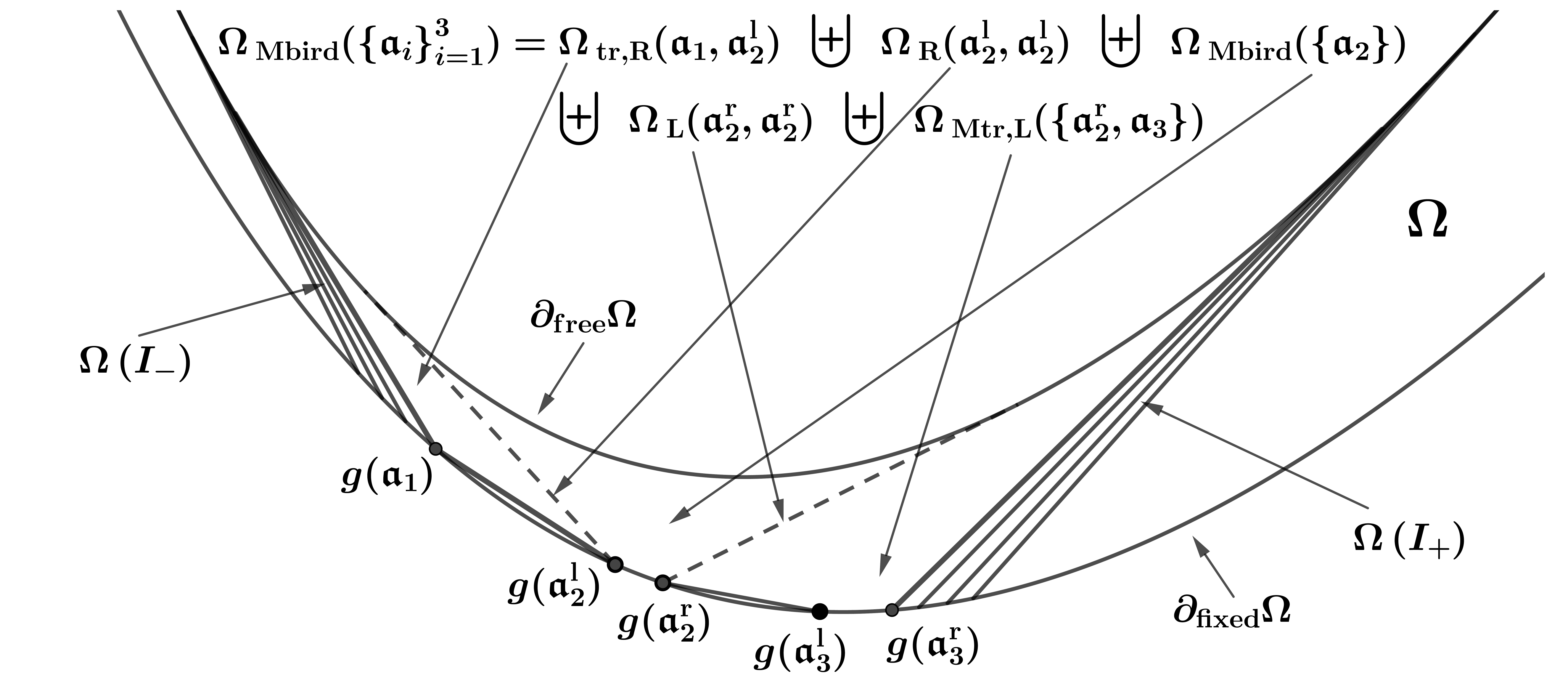}
\includegraphics[width = 0.5\linewidth]{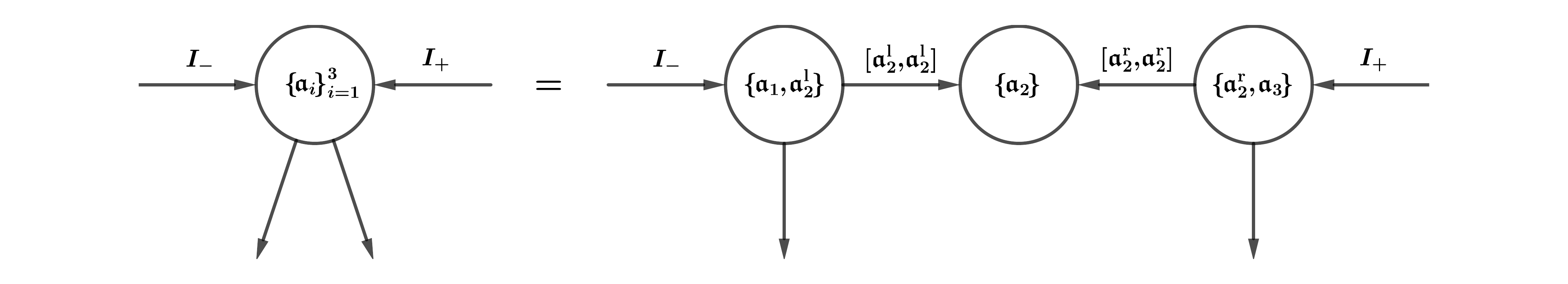}
\caption{An example for formula~\eqref{MultibirdieDesintegration} and its graphical representation.}
\label{MultibirdieDesintegrationGraph}
\end{figure}

\medskip
\paragraph{{\bf Closed multicup + trolleybus = multitrolleybus}}
\begin{equation}\label{ClosedMulticupTrolleybusR}
\ClMTC(\{\mathfrak{a_i}\}_{i=1}^k)\biguplus \Ch([\mathfrak{a}_1^{\mathrm l},\mathfrak{a}_k^{\mathrm r}],
[\mathfrak{a}_1^{\mathrm l},\mathfrak{a}_k^{\mathrm r}])  
\biguplus\RTroll(\mathfrak{a}_1^{\mathrm l},\mathfrak{a}_k^{\mathrm r}) = \MTTR(\{\mathfrak{a_i}\}_{i=1}^k);
\end{equation}
\begin{equation}\label{ClosedMulticupTrolleybusL}
\ClMTC(\{\mathfrak{a_i}\}_{i=1}^k)\biguplus \Ch([\mathfrak{a}_1^{\mathrm l},\mathfrak{a}_k^{\mathrm r}],
[\mathfrak{a}_1^{\mathrm l},\mathfrak{a}_k^{\mathrm r}])  
\biguplus\LTroll(\mathfrak{a}_1^{\mathrm l},\mathfrak{a}_k^{\mathrm r}) = \MTTL(\{\mathfrak{a_i}\}_{i=1}^k).
\end{equation}
\paragraph{{\bf Closed multicup + birdie = multibirdie}}
\begin{equation}\label{ClosedMulticupBirdie}
\ClMTC(\{\mathfrak{a_i}\}_{i=1}^k)\biguplus \Ch([\mathfrak{a}_1^{\mathrm l},\mathfrak{a}_k^{\mathrm r}],
[\mathfrak{a}_1^{\mathrm l},\mathfrak{a}_k^{\mathrm r}])  
\biguplus\Bird(\mathfrak{a}_1^{\mathrm l},\mathfrak{a}_k^{\mathrm r}) = \MTB(\{\mathfrak{a_i}\}_{i=1}^k).
\end{equation}

\subsection{General foliations}
\label{s345}

It is natural to draw a special graph~$\Gamma$\index{graph!}\index[symbol]{$\Gamma$} corresponding to a foliation to describe 
its combinatorial properties. The vertices correspond to the linearity domains. Two vertices are joined 
with an edge if there is a fence that is their common neighbor. Such a graph is drawn in the plane by 
the mapping~$\nabla B\colon \Omega_{\eps} \to \mathbb{R}^2$. However, we need to clarify some details.

We will use a small amount of graph  theory terminology. Since we study very special graphs, the use of the terminology 
will also be special. Our graphs are oriented trees (i.\,e.,  trees whose edges possess orientation). 
We call a vertex that does not have incoming edges a root, a vertex that does not have outgoing edges a leaf 
(a leaf may have several incoming edges). By a path we call an oriented path, i.\,e., we move from 
the beginning of the edge to its end while exploring the path. Other terminology is clear.

The vertices of the graph will be denoted by~$\{\mathfrak{L}_i\}_i$, the edges will be denoted 
by~$\{\mathfrak{E}_i\}_i$. Edges and vertices are of different types, moreover, they are also equipped 
with numerical parameters to be specified later. We begin the description with the edges.

Each edge~$\mathfrak{E}$ represents either a chordal domain~$\Ch([\ato,\bto],[\abot,\bbot])$ or one of 
tangent domains: either~$\Rt(\uul,\uur)$ or~$\Lt(\uul,\uur)$. The edge representing a chordal 
domain~$\Ch([\ato,\bto],[\abot,\bbot])$ is oriented from its upper neighbor to its lower neighbor. 
We consider the functions~$a$ and~$b$ associated with a chordal domain as its numeric parameters.
The edge of~$\Rt(\uul,\uur)$ is oriented from the vertex of its left neighbor to the vertex of its right neighbor. 
The edge representing~$\Lt(\uul,\uur)$ is oriented symmetrically. The closed interval $[\uul,\uur]$ is 
the numerical parameter of such an edge.

The vertices correspond to the linearity domains. For angles, trolleybuses, birdies, and multifigures 
the graphical representation was given in the subsections where they were introduced (\ref{ElementaryGraphs} 
and~\ref{s343}). These vertices are of their individual types (i.\,e., there are several vertices of the 
type ``angle'' in the graph, several vertices of the type ``birdie'', etc.). Each such vertex is equipped 
with its numerical characteristics that are the values of the parameters corresponding to its points on the 
fixed boundary $\dfi\Omega$. For example, a vertex of the type ``angle'' has one numerical parameter~$u$ 
of the point~$g(u)$ the angle is sitting on, whereas the collection of the intervals~$\{\mathfrak{a}_i\}_{i=1}^k$ 
plays the role of the numerical parameter for a vertex that has the type ``multicup'', or ``multitrolleybus'', 
or ``multibirdie''.

However, we also need some fictitious vertices, which do not correspond to any linearity domain of non-zero area. 
For example, on Figure~\ref{fig:cup_graph} the vertex representing the long chord is fictitious. There will be 
five types of such vertices.

First, there will be some~$\mathfrak{L}_i$ that correspond to long chords (the chords that are tangent to  the 
free boundary of~$\Omega_\eps$). Namely, suppose that we have a full chordal domain~$\Ch([\ato,\bto],*)$ 
such that~$\DR(\ato,\bto) \ne 0$ and~$\DL(\ato,\bto) \ne 0$, and two tangent domains,~$\Rt(\bto,u_2)$ 
and~$\Lt(u_1,\ato)$.  
Then, the vertex~$\mathfrak{L}$ corresponding to the chord~$[g(\ato),g(\bto)]$ has three outgoing edges 
representing~$\Ch([\ato,\bto],*)$, $\Rt(\bto,u_2)$, and~$\Lt(u_1,\ato)$. The set~$\{\ato,\bto\}$ is 
the numerical parameter for~$\mathfrak{L}$. The example is given on Figure~\ref{fig:cup_graph}. 

Second, there will be some vertices~$\mathfrak{L}_i$ that correspond to points of the fixed boundary. 
Suppose we have a chordal domain~$\Ch(*,[\abot,\bbot])$ with~$\abot = \bbot$ (we recall that such chordal 
domains are called cups). In our foliations, all such points will coincide with some~$c_j$ from 
Definition~\ref{roots}, see explanation around~\eqref{eqtor}. Then, the vertex~$\mathfrak{L}$ corresponding to~$\abot = \bbot=c_j$ has one 
incoming edge matching~$\Ch(*,[\abot,\bbot])$ and one numerical parameter that equals~$c_j$.

Third, sometimes we will need to paste a chord between two chordal domains (this will be done when one of 
the differentials vanishes, see Definition~\ref{differentials}). Suppose we have two chordal 
domains,~$\Ch([a_1,b_1],[a_2,b_2])$ and~$\Ch([a_2,b_2],[a_3,b_3])$. In such a case, we paste a vertex~$\mathfrak{L}$ that corresponds to the chord~$[g(a_2),g(b_2)]$, see Figure~\ref{fig:twochdomgraph}. 
It has one incoming edge and one outgoing edge and the numerical parameter~$\{a_2,b_2\}$. Long chords, one 
or both differentials of which vanish, are also considered as fictitious vertices of the third type.

Fourth, there might be one or two vertices at infinity. If we have a tangent domain~$\Omega(\uul,\uur)$ 
with $\uul=-\infty$, then there is a vertex~$\mathfrak{L}$ that corresponds to~$-\infty$. It has the numerical 
parameter~$-\infty$ and one edge representing~$\Omega(-\infty,\uur)$, that is outgoing for the case of $\Rt$ 
and incoming for the case of $\Lt$. Similarly, if we have $\uur=+\infty$, then we have a vertex~$\mathfrak{L}$ 
that corresponds to~$+\infty$ with the numerical parameter~$+\infty$ and one edge representing~$\Omega(\uul,+\infty)$, 
that is incoming for the case of $\Rt$ and outgoing for the case of $\Lt$. If $\uul=-\infty$ and $\uur=+\infty$ 
simultaneously, then we have both such vertices and one edge between them. 

Fifth, there might be a vertex corresponding to a single tangent. Suppose we have a tangent 
domain~$\Omega(u_1,u_2)$ such that~$\FF <0$ on~$[u_1,u_2]$ except for some point~$u$, where~$\FF$ 
equals zero\footnote{In such a situation,~$u = c_j$ for some root~$c_j$ (see Definition~\ref{roots}).}. 
For the case of right tangent domain, it is useful to decompose~$\Rt(u_1,u_2)$ as
\begin{equation*}
\Rt(u_1,u) \biguplus \Rt(u,u) \biguplus \Rt(u,u_2)
\end{equation*}
and paste a vertex representing~$\Rt(u,u)$ (alternatively, one can consider it as a multitrolleybus on 
a single point). It has one outgoing edge~$\Rt(u,u_2)$ and one incoming edge~$\Rt(u_1,u)$, see 
Figure~\ref{fig:twotandomgraph}. Its numerical parameter is~$u$. The same things can be done for the case 
of left tangents. We note that the fictitious vertices of the fifth type may be right and left (the same as 
the trolleybuses).

The rules listed above define the graph of the foliation. However, we provide further description to make 
its structure more transparent. It is useful to introduce a partial ordering on the set of linearity domains. 

\begin{Def}
\label{Ordering}
Let~$\mathfrak{L}_1$ and~$\mathfrak{L}_2$ be two linearity domains. We say that~$\mathfrak{L}_2$  is 
subordinate to~$\mathfrak{L}_1$  and write~$\mathfrak{L}_2 \prec \mathfrak{L}_1$  if~$\mathfrak{L}_1$  
separates~$\mathfrak{L}_2$  from the free boundary. 
\end{Def} 

Note that if $\mathfrak{L}_2 \prec \mathfrak{L}_1$, then $\mathfrak{L}_2$ is a closed multicup.  We can also let~$\mathfrak{L}_1$  
and~$\mathfrak{L}_2$  to be chords, and let~$\mathfrak{L}_2$  to be a point on the fixed boundary. 
%Therefore, the partial ordering introduced can be extended naturally to the set~$\mathfrak{L}$ of all 
%vertices of~$\Gamma$. 
One more thing to notice is that the numerical parameters of the vertices~$\mathfrak{L}_1$ 
and~$\mathfrak{L}_2$ are sufficient to define whether the statement~$\mathfrak{L}_2 \prec \mathfrak{L}_1$ is true.  

We explain how to construct the graph from a foliation. Our graph is a tree if we disregard the orientation. 
First, we describe its subgraph~$\GammaFree$ \label{gammafreedef}
spanned by the edges representing tangent domains. This subgraph 
describes the trace of the foliation on the free boundary. Formally, we can define~$\GammaFree$ to be the set 
of vertices that are not subordinated by any other vertex, and the edges between them. If we forget the 
orientation of edges,~$\GammaFree$ is a path, i.\,e., a tree whose vertices have degree two, except, possibly 
for two leaves at infinity. The leaves are usually fictitious vertices of the fourth type, however, if there 
is a multicup or a multitrolleybus that lasts to infinity, its vertex is a leaf (in such a case there is 
no  fictitious vertex representing the corresponding infinity). The orientation of edges has already been described. 
We only say that the roots of~$\GammaFree$ are the fictitious vertices of the first and third type (the latter, 
of course, should belong to~$\GammaFree$, i.\,e., represent a long chord), the vertices that correspond to 
multicups, and possibly, the vertices at infinity. The leaves  in~$\GammaFree$ correspond to angles, 
birdies, multibirdies, and possibly, vertices at infinities. The necessary and sufficient condition 
for~$\GammaFree$ to be a subgraph spanned by the edges corresponding to tangent domains of some foliation is 
that the foliation reconstructed from it covers the free boundary without intersections.

Second, we describe the graph~$\GammaFixed$ spanned by the edges corresponding to chordal domains. 
The graph~$\GammaFixed$ is a forest (i.\,e., a finite collection of trees). Each tree of the forest is 
oriented from its root, being any vertex of~$\GammaFree$ except for  fictitious vertices of the fourth or 
fifth types, and multifigures sitting on single arcs, to its leaves. The leaves of~$\GammaFixed$ are 
the  fictitious vertices of the second type (corresponding to the origins of cups) and closed multicups 
sitting on single arcs. All other vertices are closed multicups and  fictitious vertices of the third type. 
We note that this graph is generated by the ordering introduced in Definition~\ref{Ordering}: each 
edge~$\mathfrak{E}$ goes from~$\mathfrak{L}_1$ to $\mathfrak{L}_2$ if and only 
if~$\mathfrak{L}_2 \prec \mathfrak{L}_1$ and there are no vertices~$\mathfrak{L}_3$ such 
that~$\mathfrak{L}_2 \prec \mathfrak{L}_3 \prec \mathfrak{L}_1$. The necessary and sufficient condition 
for~$\GammaFixed$ to be a subgraph spanned by the edges corresponding to chordal domains of some foliation 
is that the linearity domains built from its vertices do not intersect, the edges are generated by 
the ordering from Definition~\ref{Ordering}. 

So, the graph of the foliation is a finite oriented tree whose vertices and ges have type (they correspond 
either to some figures or to fictitious constructions described above) and several numerical characteristics 
regarding their type. We warn the reader that we do not write down all the numerical parameters when we 
draw graphs, this makes our illustrations more clear. We would like to underline that the foliation could be restored from the graph and the numerical parameters determined by this foliation.

%% file: 4Chapter_text.tex
\chapter{Evolution of Bellman candidates}
\label{C4}

{In this chapter we provide an algorithm for building a special Bellman candidate on~$\Omega_\eps$ 
for each~$\eps$,~$0 < \eps < \eps_{\infty}$. In Chapter~\ref{C5} we will prove that this candidate coincides 
with the Bellman function~$\Bell(\,\cdot\,;\,f)$ using optimizers. 

The algorithm starts with sufficiently small~$\eps$. In such a case, the foliation for the Bellman candidate 
can be composed of cups (multicups), angles, and tangent domains. Then we increase~$\eps$, constructing the Bellman 
candidates for larger~$\eps$. Formally, there will be statements of two kinds (they can be called ``induction 
steps of the first and second kinds'')\index{induction steps}. The first ones state that the set of~$\eps$ 
for which there is a Bellman candidate of a given structure, is open. They are of the form: ``if for some~$\eta$ 
there is a Bellman candidate with the graph~$\Gamma$, then there is some positive~$\delta$ such that for all~$\eps$ in 
$[\eta, \eta + \delta]$ the foliation with the graph~$\Gamma$ and perturbed numerical parameters provides 
a Bellman candidate for~$f$ in~$\Omega_{\eps}$''. The second ones state that the set of those~$\eps$, for which 
there is a graph~$\Gamma$ and a collection of numerical parameters that provide a Bellman candidate for~$f$ 
and~$\eps$, is closed. They are of the form: ``if for each~$\eps_n$ there is a Bellman candidate with the 
graph~$\Gamma$,~$\eps_n \nearrow \eps$, and the numerical parameters converge to some limits as~$\eps_n\to\eps$, 
then~$\Gamma$ with the limiting parameters provide a foliation for~$f$ in~$\Omega_{\eps}$''. We note that the 
limits of numerical parameters may be degenerate in a sense (for example, a trolleybus may become a fictitious 
vertex of the fifth type), so~$\Gamma$ changes after passing to the limit. Each such induction step, in its turn, 
can be reduced to similar local statements, i.\,e., statements about the evolutional behavior of lonely figures, 
e.\,g. cups, angles, etc.

The main law that rules the evolution of the foliation is ``the forces decrease (grow in absolute value) 
as~$\eps$ grows''. As a result, long chords and multicups grow (Propositions~\ref{InductionStepForChordalDomain} 
and~\ref{InductionStepForMulticup}), trolleybuses decrease (Propositions~\ref{InductionStepForRightTrolleybus} 
and~\ref{InductionStepForLeftTrolleybus}), multitrolleybuses, birdies, and multibirdies disintegrate 
(Propositions~\ref{InductionStepForMultitrolleybus},~\ref{InductionStepForLeftMultitrolleybus}, 
and~\ref{MultibirdieDesintegrationSt}). What is more, single figures can crash, formally this happens 
in the induction steps of the second kind when one of the edges has ``zero length'' at the limit. 
In the case of a crash, we use formulas from Subsection~\ref{s344} to continue the evolution.

\section{Simple picture}
\label{s41}

\begin{Def}
\label{Simple picture}
Let~$\Gamma$ be a foliation graph. We call it \textup(and the foliation itself\textup) \emph{simple} if 
it has no oriented paths longer than one and no closed multicups. 
\end{Def}\index{graph! simple graph}

Simple foliations consist of alternating cups (or multicups on single arcs; by a multicup on a single arc 
we mean~$\MTC(\{\mathfrak{a}\})$, where~$\mathfrak{a}$ is an interval) and angles connected by tangent domains. 
If~$\Gamma$ is a graph of a simple foliation consisting of~$N$ edges, then there are either~$\frac{N}{2}$, 
or~$\frac{N-1}{2}$, or~$\frac{N-2}{2}$ angles in the foliation. In~$\GammaFree$, the vertices corresponding 
to angles alternate the vertices representing multicups and long chords. Each multicup is sitting on an arc 
whose convex hull is not contained in $\Omega_{\eps}$. Each long chord has a cup below it. 
See Figure~\ref{fig:exsimpict} for the visualization. 

\begin{wrapfigure}{r}{0.45\linewidth}
\begin{center}
\includegraphics[width = 0.85\linewidth]{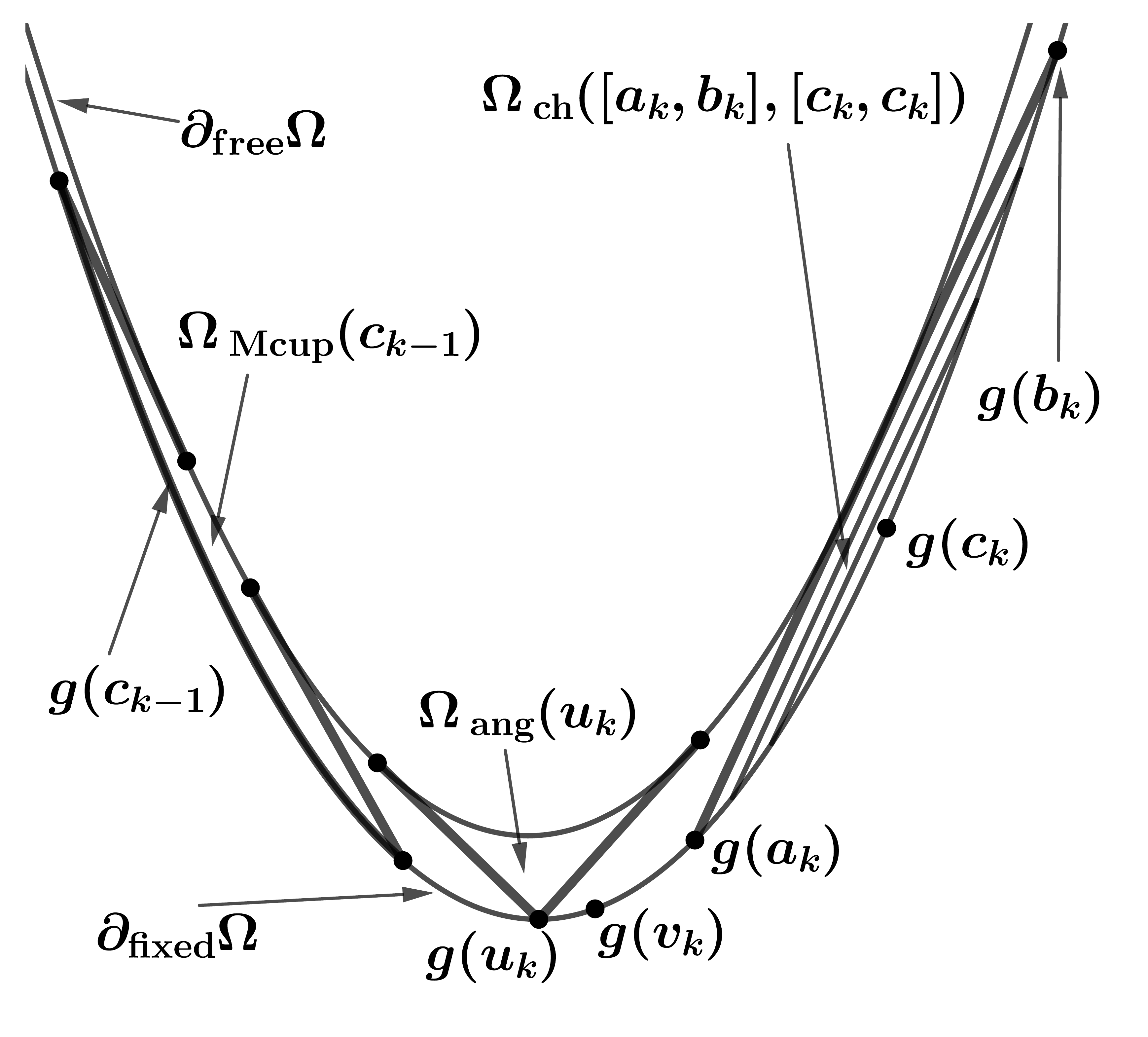}
\caption{An example of simple picture.}
\label{fig:exsimpict}
\end{center}
\end{wrapfigure}

For didactic reasons, we explain how does a simple graph generate a Bellman candidate~$B$ (similar essence 
for general graphs will be explained in Section~\ref{s44}). Suppose~$\Gamma$ is a simple graph. First, 
we consider its roots, which are long chords, multicups on solid roots, and vertices at infinities. For long 
chords and multicups, we build the standard candidate by linearity (see~\eqref{eq250901}). Second, 
consider the edges of~$\Gamma$. For the edges corresponding to chordal domains, we construct standard candidates 
again by linearity on chords  (see~\eqref{eq250901}). For each edge corresponding to a tangent domain, 
we continuously glue a standard candidate in this domain to the already built standard candidate corresponding to the 
source of the edge. This is done by choosing an appropriate~$\beta_2(t_0)$ in~\eqref{beta2new}. For tangent domains whose source is infinity, we do not have to glue anything, we simply consider 
the standard candidates on them, such standard candidates are uniquely defined, see Definition~\ref{211101}. 
In the angles, we choose the standard candidates by Proposition~\ref{St260901}. The constructed function $B$ 
is~$C^1$-smooth, and thus, by Proposition~\ref{ConcatenationOfConcaveFunctions}, it is locally concave.

In the theorem below, we use the notation for the essential roots, see Definition~\ref{roots}.

\begin{Th}
\label{SimplePicture}
For any function~$f$ satisfying Conditions~\textup{\ref{reg}} and~\textup{\ref{sum}} there exists~$\eps_1>0$ 
such that for any~$\eps < \eps_1$ there exists a simple graph and a collection of numerical parameters such 
that the function~$B$ constructed from this graph\textup,~$f$\textup, and~$\eps$ as described above is 
a~$C^1$-smooth Bellman candidate. Moreover\textup, its foliation satisfies the following properties\textup: 
the origins of the cups coincide with those~$c_i$ that are single points\textup; the multicups are sitting 
on those~$c_i$ that are intervals\textup; for any~$k = 1,2,\ldots,n,$ the parameter of the vertex~$u_k$ of 
the~$k$-th angle in~$\GammaFree$ tends to~$v_k$ as~$\eps \to 0$. 
\end{Th}
The proof of this theorem will be presented at the end of this section, because it needs some preparation.

\medskip

Consider two neighbor points~$c_k$ and~$v_{k+1}$, where the torsion $\torsion$ of the boundary curve~$\gamma$ changes its sign. If $c_k$ is a solid root, then for sufficiently small $\eps$  we can build a multicup on $c_k$ and 
define the standard candidate there. If $c_k$ is a single point, for small $\eps$ we use~Proposition~\ref{St3108} 
to build a full cup on it. Let its upper chord be~$[g(a_k),g(b_k)] = [g(a_k(\eps)),g(b_k(\eps))]$ (for a cup, 
we take its upper chord, whereas for a multicup, we consider the chord connecting its endpoints). Then, 
by~\eqref{eq101101}, we have~$\Fr(u;a_k,b_k;\eps)<0$ when~$u\in(b_k,v_{k+1}]$, because $\sign(\torsion)=\sign(\tors')$ 
(see~\eqref{eq211003}). Thus, the right tail of the cup or multicup built on~$c_k$ always contains~$v_{k+1}$. 
Similarly, the left tail of the cup or multicup built over~$c_{k+1}$ contains~$v_{k+1}$. The following lemma 
says that the ends of these tails tend to $v_{k+1}$ as~$\eps \to 0$.

\begin{Le}
\label{TailForSmallEps}
Let~$(g(a_k(\eps)),g(b_k(\eps)))$ be the upper chord of a cup or of a multicup built over~$c_k$\textup, 
let~$\tr_k = \tr_k(\eps)$ be the endpoint of its right tail. Then~$\tr_k \to v_{k+1}$ as~$\eps \to 0$. 
Similarly\textup, the endpoint~$\tl_k(\eps)$ of the left tail tends to~$v_k$. A similar convergence statement 
holds for the forces coming from the infinities.
\end{Le} 

\begin{proof}
We will deal with the case of the right tail. We will write $a$ for $a_k$ and $b$ for $b_k$. It suffices to prove 
that for each point~$w_{+}$ such that~$v_{k+1} < w_{+}$ (we also assume that~$w_+$ is not far from~$v_{k+1}$, 
we want~$\tors$ to increase on~$(v_{k+1},w_+)$), the inequality~$\tr_k < w_{+}$ holds eventually as~$\eps \to 0$. 
We use~\ref{eq101101} for the force $\Fr(t)= \Fr(t;a(\eps),b(\eps);\eps)$ outside the chordal domain 
to obtain
\eq{forcie}{
\Fr(w_+) = 
\int_{b}^{w_+} \exp\Big(\!-\!\!\int_{\tau}^{w_+} \frac{\kappa_2'}{\kappa_2-\kappa}\Big)\tors'(\tau)\,d\tau +
\exp\Big(\!-\!\!\int_{b}^{w_+}\!\frac{\kappa_2'}{\kappa_2-\kappa}\Big)\Fr(b).
}
It suffices to prove~$\Fr(w_{+}) > 0$ for $\eps$ sufficiently small. We first deal with the first summand, which will be split into two 
integrals $\int_{b}^{w_+} = \int_{b}^{v^{\mathrm{r}}_{k+1}}+\int_{v^{\mathrm{r}}_{k+1}}^{w_+}$ (as usual, 
by~$v^{\mathrm{r}}_{k+1}$ we denote the right endpoint of~$v_{k+1}$). First,
\eq{eq101103}{
\left|\int_b^{v^{\mathrm{r}}_{k+1}}\!\exp\Big(\!-\!\!\int_{\tau}^{w_+}\!\frac{\kappa_2'}{\kappa_2-\kappa}\Big)
\tors'(\tau)\,d\tau\right| \leq 
\exp\Big(\!-\!\!\!\int_{v^{\mathrm{r}}_{k+1}}^{w_+} \!\frac{\kappa_2'}{\kappa_2-\kappa}\Big)\cdot
\int_{b}^{v^{\mathrm{r}}_{k+1}}\!\left|\tors'(\tau)\right|\,d\tau.}

Take $w \in (v_{k+1}^{\mathrm{r}},w_+)$. Then,
\begin{multline}
\label{eq101104}
\int_{v^{\mathrm{r}}_{k+1}}^{w_+}\!\exp\Big(\!-\!\!\int_{\tau}^{w_+}\!\frac{\kappa_2'}{\kappa_2-\kappa}\Big)
\tors'(\tau)\,d\tau \geq
\int_{w}^{w_+} \!\exp\Big(\!-\!\!\int_{\tau}^{w_+} \!\frac{\kappa_2'}{\kappa_2-\kappa}\Big)\tors'(\tau)\,d\tau \geq
\\ 
\geq \exp\Big(\!-\!\!\int_{w}^{w_+} \!\frac{\kappa_2'}{\kappa_2-\kappa}\Big) \cdot \int_{w}^{w_+}\tors'(\tau)\,d\tau 
= \exp\Big(\!-\!\!\int_{w}^{w_+} \!\frac{\kappa_2'}{\kappa_2-\kappa}\Big) \cdot \big(\tors(w_+) - \tors(w)\big). 
\end{multline}
We multiply both sides of~\eqref{forcie} by $\exp\Big(\int_{w}^{w_+} \!\frac{\kappa_2'}{\kappa_2-\kappa}\Big)$,
use~\eqref{eq101103} and~\eqref{eq101104}, and obtain
\begin{multline}
\label{eq141102}
\exp\Big(\int_{w}^{w_+} \!\frac{\kappa_2'}{\kappa_2-\kappa}\Big)\Fr(w_+)\geq
\\
\tors(w_+) - \tors(w) - 
\exp\Big(\!-\!\!\!\int_{v^{\mathrm{r}}_{k+1}}^{w} \!\frac{\kappa_2'}{\kappa_2-\kappa}\Big) \cdot 
\int_{b}^{v^{\mathrm{r}}_{k+1}}\left|\tors'(\tau)\right|\,d\tau 
+ \exp\Big(\!-\!\!\int_{b}^{w}\!\frac{\kappa_2'}{\kappa_2-\kappa}\Big)\Fr(b).
\end{multline}
When $\eps \to 0$, $\kappa$ tends to $\kappa_2$ from below pointwise, therefore both exponents on the right 
hand side of~\eqref{eq141102} tend to zero. These exponents are multiplied by bounded factors. Indeed, 
$$
\int_{b}^{v^{\mathrm{r}}_{k+1}}\left|\tors'(\tau)\right|\,d\tau \leq \int_{c_k}^{v^{\mathrm{r}}_{k+1}}\left|
\tors'(\tau)\right|\,d\tau.
$$
Boundedness of the second factor is more cumbersome. If $(g(a),g(b))$ is 
the upper chord of a multicup, then $\Fr(b) = 0$ and there is nothing to do. Consider the case of a cup. 
In this case we use~\eqref{eq101101}: 
\eq{eq141101}{
\Fr(b) = \frac{\DR(a,b)}{g_1'(b)\kappa_2'(b)}.
}
Let us rewrite formula~\eqref{eq150802}.
After applying~\eqref{lincoef} and~\eqref{e334} its left hand side takes the following form
$$
\det
\begin{pmatrix}
\gamma''(b)\\
\gamma'(a)\\
\gamma(b)-\gamma(a)
\end{pmatrix} =
\det
\begin{pmatrix}
\gamma''(b)\\
\gamma'(a)\\
\CR\gamma'(b)-\CL\gamma'(a)
\end{pmatrix}=
\CR\det
\begin{pmatrix}
\gamma''(b)\\
\gamma'(a)\\
\gamma'(b)
\end{pmatrix}=
-\CR\DR \dett{g'(b)}{g'(a)}.
$$
Due to the same formula~\eqref{lincoef} the nonintegral term in the right hand side of~\eqref{eq150802} vanishes. 
Thus, we obtain 
$$
\CR\DR \dett{g'(b)}{g'(a)}=g_1'(b)g_1'(a)\kappa_2'(b)
\int_a^b\Big(\tors(u)-\tors(b)\Big)\kappa_2'(u)\big(g_1(b)-g_1(u)\big)\,du.
$$
The function $W(u)=\kappa_2'(u)\big(g_1(b)-g_1(u)\big)$ is nonnegative and its integral over $[a,b]$ is equal to
\begin{multline}
\int_a^b W(u)\,du = \kappa_2(u)\big(g_1(b)-g_1(u)\big)\Big|_a^b + \int_a^b \kappa_2(u)g_1'(u)\,du = 
\\
=-\kappa_2(a)\big(g_1(b)-g_1(a)\big)+g_2(b)-g_2(a)
=-\frac1{g_1'(a)}\dett{g(b)-g(a)}{g'(a)}=-\frac1{g_1'(a)}\dett{g'(b)}{g'(a)}\CR.
\end{multline}
Therefore,~\eqref{eq141101} takes the form
$$
\Fr(b) = \frac{\int_a^b \big(\tors(b)-\tors(u)\big)W(u)\,du}{\int_a^b W(u)\,du}. 
$$
By the mean value theorem, there exists some $u \in [a,b]$, such that $\Fr(b) =\tors(b)-\tors(u)$, 
which is bounded and moreover tends to zero when $\eps \to 0$. 
Therefore, the right hand side of~\eqref{eq141102} tends to $\tors(w_+)-\tors(w)$ which is positive. 
Thus, the left hand side of~\eqref{eq141102} is positive for sufficiently small $\eps$. 
\end{proof}

The notation in the following lemma is the same as in the previous one.

\begin{Le}
\label{moninttails}
The difference of forces\textup,~$\Fr(u;a_k,b_k)-\Fl(u;a_{k+1},b_{k+1})$\textup, is strictly increasing 
\textup(as a function of~$u$\textup) on the interval~$(\tl_{k+1},\tr_k) \cap (b_k,a_{k+1})$. 
\end{Le}

\begin{proof}
We differentiate the function in question with respect to~$u$, use~\eqref{leftder}, and obtain
\eq{difpowsum}{
\big(\Fr - \Fl\big)' =  -\beta_{2,\scriptscriptstyle\mathrm{R}}'+
\beta_{2,\scriptscriptstyle\mathrm{L}}'>0,
}
because~$\beta_{2,\scriptscriptstyle\mathrm{L}}' > 0> \beta_{2,\scriptscriptstyle\mathrm{R}}'$ 
on~$(\tl_{k+1},\tr_k)\cap (b_k,a_{k+1})$ by Definition~\ref{211101}. The proposition is proved.
\end{proof}

\begin{Cor}\index{equation! balance equation}
\label{BalanceEquationForSmallEps}
The balance equation 
\eq{eq250903}{
\Fr(u;a_k,b_k;\eps) = \Fl(u;a_{k+1},b_{k+1};\eps)
}
has a unique root~$u = u_{k+1}$ in~$(\tl_{k+1},\tr_k)$ for sufficiently small~$\eps$. 
\end{Cor}
\begin{proof}
First, by Lemma~\ref{TailForSmallEps},  we have~$[\tl_{k+1},\tr_k]\subset(b_k,a_{k+1})$ for sufficiently 
small~$\eps$. By definition of tails and continuity of forces, the force function is zero at the endpoint 
of its tail: $\Fr(\tr_k;a_k,b_k;\eps) = 0$ and~$\Fl(\tl_{k+1};a_{k+1},b_{k+1};\eps) = 0$. Therefore,
\begin{equation*}
\Fr(\tl_{k+1};a_k,b_k;\eps) - \Fl(\tl_{k+1};a_{k+1},b_{k+1};\eps) = \Fr(\tl_{k+1};a_k,b_k;\eps) < 0,
\end{equation*}
because~$\tl_{k+1} \in (b_k,\tr_k)$. Similarly,
\begin{equation*}
\Fr(\tr_k;a_k,b_k;\eps) - \Fl(\tr_k;a_{k+1},b_{k+1};\eps) = -\Fl(\tr_k;a_{k+1},b_{k+1};\eps) > 0.
\end{equation*}
By the Bolzano--Weierstrass principle, the balance equation has a root~$u_{k+1}$ on~$(\tl_{k+1},\tr_k)$. 
By Lemma~\ref{moninttails}, this root is unique.
\end{proof}

\begin{Rem}
\label{InfiniteVersion}
The results of the preceding lemma and the corollary hold true if one of the cups \textup(or both\textup) sit 
at infinity\textup, i.\,e.\textup, the corresponding~$c_k$ is infinite.
\end{Rem}

Now we have all the ingredients to prove Theorem~\ref{SimplePicture}.

\begin{proof}[Proof of Theorem~\ref{SimplePicture}.]
 First, we take~$\eps$ to be a small number such that we can build a full cup around each $c_k$ (or a multicup 
if $c_k$ is a solid root) with the help of Proposition~\ref{St3108}. Moreover, we take $\eps$ to be so small 
that all these figures have no intersections. This is possible because we have only finite number of roots by 
our assumptions. Then, if~$\eps$ satisfies the assumptions of Corollary~\ref{BalanceEquationForSmallEps} 
(together with Remark~\ref{InfiniteVersion}) for each~$k$, one can paste an angle with the vertex at the 
unique root of the balance equation~\eqref{eq250903} between each pair of consecutive cups or multicups, 
see~Corollary~\ref{cor250901}. The relation~$u_{k+1}\to v_{k+1}$ is an immediate consequence of 
Lemma~\ref{TailForSmallEps} and the inclusion~$u_{k+1} \in (\tl_{k+1},\tr_k)$.
\end{proof}

\section{Preparation to evolution}
\label{s42}

In this section, we collect technical lemmas that are useful for the evolution. There will be three groups of 
lemmas. The first group consists of lemmas that describe the places where the fictitious vertices of 
the third type may occur, the second is about tails and forces, and the third one works with the balance equation.

\subsection{Structural lemmas for chords}
\label{s421}\index{domain! chordal domain}

We make a convention on chordal domains: the inequalities~$\DL < 0$ and~$\DR<0$ hold true inside the chordal 
domain. Note that the same inequalities are required to build the standard candidate in a chordal domain.

\begin{Le}
\label{zerodifrootsInside}
Let~$\Ch([a,b],*)$ be a chordal domain. If~$\DL(a,b) = 0,$ then~$\tors$ decreases on the right of~$a;$ 
if~$\DR(a,b) = 0,$ then~$\tors$ increases on the left of~$b$. 
\end{Le}

\begin{proof}
We treat the case of the right differential only. The remaining case is symmetric. We will use 
formula~\eqref{eq121201}. Let 
$$
U(t) = \frac{\dett{g'(t)}{g'''(t)}}{\dett{g'(t)}{g''(t)}}-\frac{\dett{g''(t)}{g(t)-g(s)}}{\dett{g'(t)}{g(t)-g(s)}},
$$
where $s=s(t)$, $s<t$, is the corresponding to the chordal domain function defined in Section~\ref{s33}. The function 
$U$ is in fact a measure of bounded variation in a left neighborhood of $b$, therefore we can find an increasing 
function $A$, such that $A'+U$ is a positive measure in a left neighborhood of $b$.
Consider the function $V(t) = e^{A(t)}\DR(s,t)$ in the left neighborhood of $b$. We use equations~\eqref{eq121201} 
and~\eqref{eq28042001}:
$$
e^{-A(t)}V'(t) = \frac{\torsion(t)}{\dett{g'(t)}{g''(t)}}+ (A'(t)+U(t))\DR(s,t).
$$ 
If~$\tors$ does not increase to the left of~$b$, then $\torsion(t)<0$ in a small left neighborhood of $b$
according to~\eqref{eq211003} and Condition~\ref{reg}. Since $A'(t)+U(t)>0$ and $\DR(s,t)<0$, we obtain $V'(t)<0$. Therefore, $V$ decreases, which contradicts to the fact $V(t)<0$ for $t<b$ and $V(b)=0$. 
\end{proof}

\begin{Le}
\label{zerodifrootsOutside}
Suppose that a pair~$(a,b)$ satisfies the cup equation and the chord~$[g(a),g(b)]$ has nonzero tails 
\textup(i.\,e.\textup, $\tl<a$ and $b<\tr$\textup). If~$\DL(a,b) = 0$\textup, then~$\tors$ increases 
on the left of~$a$\textup; if~$\DR(a,b) = 0$\textup, then~$\tors$ decreases on the right of~$b$. 
\end{Le}

\begin{proof}
We treat the case of the right differential only. The remaining case is symmetric. We will use 
formula~\eqref{eq101101}. By Condition~\ref{reg} for the function~$f$, the function~$\tors$ either 
increases or decreases in a right neighborhood of~$b$. If it increases, then the force $\Fr$ is 
non-negative, which contradicts the assumption~$\tr > b$. Therefore,~$\tors$ decreases on the right of~$b$.
\end{proof}

Combining these two lemmas we obtain the following corollary: during the evolution, 
the differentials can vanish only in some very special situations (we use notation from Definition~\ref{roots} in the corollary below). 

\begin{Cor}
\label{RootCorollary}
Suppose that the chordal domain~$\Ch([a,b],*)$ has nonzero tails. If~$\DR(a,b)=0,$ then~$b = c_i$ 
for some~$i;$ if~$\DL(a,b)=0,$ then~$a = c_i$ for some~$i$.
\end{Cor}

\subsection{Tails growth lemmas}
\label{s422}\index{force! function}

\begin{Le}\label{biggercupsmallerforce}
Suppose~$\Ch([a_1,b_1],[a_0,b_0])$ is embedded into~$\Ch([a_2,b_2],[a_0,b_0]),$ i.\,e.\textup, the foliation 
of the former chordal domain coincides with some part of the foliation of the latter. Then the forces 
of~$\Ch([a_1,b_1],*)$ do not exceed the corresponding forces of~$\Ch([a_2,b_2],*)$. More precisely\textup, 
they are equal on~$[a_1,a_0]\cup[b_0,b_1],$ while outside~$[a_1,b_1]$ the inequalities are strict.
\end{Le}

\begin{proof}
For the right forces the statement follows from~\eqref{eq170101} and the fact that 
$\kappar(b)<\kappachord(b)<\kappa_2(b)$. Note that $\DR(a,b)<0$ therefore, $\Fr$ increases as 
the chordal domain enlarges, i.\,e., $b$ increases. Similar arguments work for the case of the left forces: 
we use~\eqref{eq170102} and the opposite relation between the coefficients: 
$\kappal(a)>\kappachord(a)>\kappa_2(a)$. Thus~$\Fl$ increases as the chordal domain enlarges, 
i.\,e., $a$ decreases. 
\end{proof}

\begin{Cor}
\label{TailsGrowthNonEv}
Suppose~$\Ch([a_1,b_1],*)$ is embedded into~$\Ch([a_2,b_2],*)$ in the sense of 
Lemma\textup{~\ref{biggercupsmallerforce}}. Then the tails of the former chordal domain strictly 
contain the tails of the latter. 
\end{Cor}

\begin{proof}
This follows from Lemma~\ref{biggercupsmallerforce} and the definition of tails. 
\end{proof}

The previous statements can be reformulated informally: the less the chordal domain is, the larger the tails 
are and the less (the larger in absolute value) the forces are inside the tails. The following lemma describes 
monotonicity of forces of a fixed chordal domain for a family of enlarging domains $\Omega_\eps$.

\begin{Le}\label
{EvolutionalTailGrowthLemma}
Let~$\Ch([a,b],*)$ be a chordal domain\textup, let $F$ be its left or right force. 
If~$t$ belongs to the closure of the corresponding tail of~$\Ch([a,b],*),$ then 
\begin{equation*}
\frac{\partial F(t;\eps)}{\partial \eps} \leq 0,
\end{equation*}
and the inequality is strict outside~$[a,b]$.
\end{Le}

\begin{proof}
We treat the case of the right force only, the other one is symmetric. First, we note that inside~$[a,b]$ 
the force does not depend on~$\eps$, therefore,~$\frac{\partial \Fr(t;\eps)}{\partial \eps} = 0$. 
Second, we use formula~\eqref{leftder} and see that  
\begin{equation*}
\Fr'(t;\eps) + \frac{\kappa_2'(t)}{\kappa_2(t)-\kappa(t;\eps)}\Fr(t;\eps) = \tors'(t).
\end{equation*}
outside this interval. We cautiously differentiate this equation with respect to~$\eps$ and see that 
\begin{equation*}
\frac{\partial }{\partial \eps}\Fr'(t;\eps) + 
\frac{\kappa_2'(t)}{\kappa_2(t)-\kappa(t;\eps)}\frac{\partial }{\partial \eps}\Fr(t;\eps)=
-\frac{\kappa_2'(t)  \Fr(t;\eps) }{(\kappa_2(t)-\kappa(t;\eps))^2}\frac{\partial }{\partial \eps}\kappa(t;\eps).
\end{equation*}
After interchanging the differentiations with respect to~$t$ and~$\eps$ in the first summand we see 
that~$y(t)=\frac{\partial }{\partial \eps}\Fr(t;\eps)$ is a solution of the first-order differential 
equation with respect to~$t$: 
\begin{equation*}
y'(t) + \frac{\kappa_2'(t)}{\kappa_2(t)-\kappa(t;\eps)} y(t) = 
-\frac{\kappa_2'(t)  \Fr(t;\eps) }{(\kappa_2(t)-\kappa(t;\eps))^2}\frac{\partial }{\partial \eps}\kappa(t;\eps).
\end{equation*}
We use~\eqref{firstordersolution} to express the solution: 
\begin{equation*}
\frac{\partial }{\partial \eps}\Fr(t;\eps) = \exp\Big(\!-\!\int_{b}^t\frac{\kappa_2'}{\kappa_2-\kappa}\Big)
\left(\int_{b}^t \exp\Big(\int_{b}^\tau\frac{\kappa_2'}{\kappa_2-\kappa}\Big)
\Big[\!-\!\frac{\kappa_2'  \Fr }{(\kappa_2-\kappa)^2}\frac{\partial \kappa}{\partial \eps}\,\Big]
(\tau)\,d\tau+\const(\eps) \right),
\end{equation*}
were the constant is zero, because~$\lim_{t \to b+}\frac{\partial }{\partial \eps}\Fr(t;\eps) = 0$ 
(this limit relation can be verified by a straightforward calculation using formula~\eqref{eq101101}).
Now, the result follows immediately, because~$\Fr < 0$ inside the tail, $\kappa_2'>0$ and 
$\frac{\partial \kappa(t;\eps)}{\partial \eps}<0$ by geometrical reasons.
\end{proof}

\begin{Cor}
\label{TailsGrowthCorollary}
Let~$\Ch([a,b],*)$ be a chordal domain with nonzero tails. If we increase~$\eps$ a little\textup, 
the tails of the chordal domain strictly enlarge. 
\end{Cor}

\begin{proof}
By Lemma~\ref{EvolutionalTailGrowthLemma}, the forces decrease (increase in absolute value) on the corresponding 
tails, therefore, the tails cannot decrease. What is more, if we increase~$\eps$ a little, the force at the 
end of the tail becomes negative, and the tail enlarges.
\end{proof}

\begin{Rem}\label{TailGrowthInfinity}
The results of Lemma~\ref{EvolutionalTailGrowthLemma} and Corollary~\ref{TailsGrowthCorollary} 
hold for the forces coming from the infinities and multicups as well.
\end{Rem}

For a moment, we let chordal domains fall out of $\Omega_\eps$, i.\,e., their chords may intersect the free 
boundary. Surely, when we were working with chordal domains, we did not need the upper boundary, therefore, 
such an assumption does not break all the results concerning chordal domains. 

\begin{Le}
\label{difequal}
Let $\Ch([a_2,b_2],[a_1,b_1])$ be a chordal domain with the standard Bellman candidate there. 
Suppose that $\eps_1<\eps_2$ and there are differentiable functions $a\colon[\eps_1,\eps_2]\to [a_2,a_1]$ 
and $b\colon[\eps_1,\eps_2]\to [b_1,b_2],$ such that the chord $[g(a(\eps)),g(b(\eps))]$ is tangent to 
the free boundary of $\Omega_\eps$ for any $\eps \in [\eps_1,\eps_2]$. 
Then, for any $\tilde\eps\in [\eps_1,\eps_2]$ we have 
\begin{equation*}
\frac{\partial}{\partial \eps}\bigg[F\big(t; a(\tilde\eps), b(\tilde\eps); \eps\big)\bigg]\Bigg|_{\eps = \tilde\eps}
= \quad\frac{\partial}{\partial \eps}\bigg[F\big(t; a(\eps), b(\eps); \eps\big)\bigg]\Bigg|_{\eps = \tilde\eps},
\end{equation*}
where $F$ stands for the right \textup(see~\eqref{eq101101}\textup) or the left 
\textup(see~\eqref{eq101102}\textup) force function. 
\end{Le}

\begin{proof}
This identity follows from Lemma~\ref{DifForcePoL} because $\kappar=\kappal=\kappachord$ when the chord is 
tangent to the free boundary.
\end{proof}

\begin{Le}
\label{FullCupEvolutionTails}
Under the hypothesis of Lemma~\textup{\ref{difequal}} the tails of~$\Ch([a(\eps),b(\eps)],*)$ 
strictly enlarge in~$\eps$. 
\end{Le}

\begin{proof}
Indeed, consider some point~$t$ that belongs to one of the tails of~$\Ch([a(\eps),b(\eps)],*)$ for some~$\eps$. 
First, we need to prove that~$F(t; a(\eps), b(\eps);\eps)$ decreases in~$\eps$. The derivative of this function 
with respect to~$\eps$ is non-positive, because, by Lemma~\ref{difequal}, for each~$\eps$ it equals the 
corresponding derivative taken as if the chordal domain had fixed upper chord, which is nonpositive by 
Lemma~\ref{EvolutionalTailGrowthLemma}. Thus, the tails do not decrease. Moreover, the derivative of 
the corresponding force is nonzero at the end of each tail, again by Lemma~\ref{EvolutionalTailGrowthLemma}, 
therefore, the tail grows.
\end{proof}

\begin{Def}\label{FlowOfChordalDomains}\index{domain! chordal domain! flow of chordal domains}
Let~$0 \leq \eps_1 < \eps_2 < \epsmax$. Let~$\Ch([a_1,b_1],[a_0,b_0])$ be a chordal domain. Assume that 
there are two continuous functions $a\colon[\eps_1,\eps_2] \to [a_1,a_0]$, $b\colon[\eps_1,\eps_2]\to[b_0,b_1]$ 
such that $[g(a(\eps),g(b(\eps))]$ is a chord of this chordal domain for any $\eps \in [\eps_1,\eps_2]$. 
We call the family 
$$
\{\Ch([a(\eps),b(\eps)],[a_0,b_0])\}_\eps
$$
a \emph{flow of chordal domains}. 
This flow generates the corresponding forces:
\begin{equation}\label{eq171101}
\begin{aligned}
\Fr\Big(u;\Ch\big(\big[a(\eps),b(\eps)\big],[a_0,b_0]\big);\eps\Big), \quad u \in [b_0,\tr(\eps)),\\
\Fl\Big(u;\Ch\big(\big[a(\eps),b(\eps)\big],[a_0,b_0]\big);\eps\Big), \quad u \in (\tl(\eps), a_0],
\end{aligned} 
\end{equation}
where~$\tr(\eps)$ and~$\tl(\eps)$ are the right and the left endpoints of the tails 
of~$\Ch\big(\big[a(\eps),b(\eps)\big],[a_0,b_0]\big)$. We say that a flow is \emph{decreasing} 
if~$b$ is decreasing ($a$ is increasing). We say that a flow is \emph{full} if the 
chord~$[g(a(\eps),g(b(\eps))]$ is tangent to the free boundary of 
$\Omega_\eps$ for all~$\eps \in [\eps_1,\eps_2]$.
\end{Def}

We gather the statements of the previous lemmas in the following corollary. 
\begin{Cor}\label{FlowMonotonicityCorollary}
Consider a decreasing or a full flow of chordal domains. The corresponding forces~\eqref{eq171101} are strictly decreasing functions of~$\eps$ outside the chordal domain and are constant inside the chordal domain. As a consequence\textup, 
the tails enlarge.
\end{Cor}

\begin{proof}
For the case of a decreasing flow, we use Lemmas~\ref{EvolutionalTailGrowthLemma} and~\ref{biggercupsmallerforce}. 
For the case of a full flow, we use Lemma~\ref{FullCupEvolutionTails}.
\end{proof}

\begin{Def}\label{MFF}\index{force! monotone force flow}
Let~$0 \leq \eps_1 < \eps_2 <\epsmax$\textup, let~$u_-\colon [\eps_1,\eps_2] \to  \mathbb{R}\cup\{-\infty\}$ 
be a non-increasing function and $u_+\colon [\eps_1,\eps_2]\to\mathbb{R}\cup\{+\infty\}$ a non-decreasing 
function such that $u_-\leq u_+$. Suppose that for any $\eps\in [\eps_1,\eps_2]$ there is a fence $\Omega(I_\eps)$, 
such that $(u_-(\eps),u_+(\eps))\subset I_\eps$.

We say that a continuous function~$\FFr\colon\{(u,\eps)\colon u\in[u_-(\eps),u_+(\eps)]\cap\mathbb{R},\ 
\eps\in[\eps_1,\eps_2]\}\to\mathbb{R}$ is a \emph{right monotone force flow\emph} if:
\begin{enumerate}
\item for any~$\eps \in [\eps_1,\eps_2]$ the fence $\Omega(I_\eps)$ is a right tangent domain, 
the function~$\FFr(\cdot\,;\eps)$ is the force of a standard candidate on $\Rt(I_\eps)$ introduced in Definition~\ref{211101};
\item $\FFr(u_+(\eps);\eps)=0$ provided that $u_+(\eps)<+\infty$;
\item for any~$\eta_1$ and~$\eta_2$ such that~$\eps_1<\eta_1<\eta_2<\eps_2$\textup, 
we have~$\FFr(v;\eta_1) > \FFr(v;\eta_2)$ whenever~$v \in (u_-(\eta_1),u_+(\eta_1)] \cap \mathbb{R}$.
\end{enumerate}

We say that a continuous function~$\FFl\colon \{(u,\eps)\colon u \in [u_-(\eps),u_+(\eps)]\cap \mathbb{R},\  
\eps \in [\eps_1,\eps_2]\}\to \mathbb{R}$ is a \emph{left monotone force flow\emph} if:
\begin{enumerate}
\item for any~$\eps \in [\eps_1,\eps_2]$ the fence $\Omega(I_\eps)$ is a left tangent domain, the 
function~$\FFl(\cdot\,;\eps)$ is a force of a standard candidate on $\Lt(I_\eps)$;
\item $\FFl(u_-(\eps);\eps)=0$ provided that $u_-(\eps)>-\infty$;
\item for any~$\eta_1$ and~$\eta_2$ such that~$\eps_1 < \eta_1 < \eta_2 < \eps_2$\textup, 
we have~$\FFl(v;\eta_1) > \FFl(v;\eta_2)$ whenever~$v \in [u_-(\eta_1),u_+(\eta_1))\cap \mathbb{R}$.
\end{enumerate}
\end{Def}

Monotone force flows may be generated by flows of chordal domains as in~\eqref{eq171101}. % (maybe fictive for the case of multifigures) 
%or infinities. 
Then, for the right case we have $u_+=\tr$, and $u_-=\tl$ for the left case, see Subsection~\ref{Subsec362} for the definition of $\tr$ and $\tl$.

\begin{Rem}\label{rem241101}
We underline that for any fixed $\eps \in [\eps_1,\eps_2]$ if $u_-(\eps)<u_+(\eps)$, then the monotone force flows $\FFr(\,\cdot\,;\eps)$ and $\FFl(\,\cdot\,;\eps)$ satisfy the differential equations
\begin{equation}\label{eq241101}
\FFr' = \tors'- \frac{\kappa_2'}{\kappa_2-\kappar}\FFr, \qquad
\FFl' = \tors'- \frac{\kappa_2'}{\kappa_2-\kappal}\FFl.
\end{equation}
Therefore, statements of Remarks~\ref{rem220203} and~\ref{Rem200401} hold true for the forces:
\begin{itemize}
\item if $u_-(\eps)<u_+(\eps)<\infty$, then $\tors'>0$ in a left neighborhood of $u_+(\eps)$;
\item if $u_+(\eps)>u_-(\eps)>-\infty$, then $\tors'<0$ in a right neighborhood of $u_-(\eps)$.
\end{itemize} 
\end{Rem}

\begin{Rem}\label{rem220201} Note that the function $u_+$ in the definition of the right monotone force 
flow is strictly increasing while it is finite. Similarly, the function $u_-$ in the definition of the left 
monotone force flow is strictly decreasing while it is finite.
\end{Rem}
%\begin{proof}
%Consider the case of the right flow. If the function $u_+$ is finite and constant on some interval $(\eta_-,\eta_+)$, then $\FFr(u_+(\eta_1);\eta_1)=0 = \FFr(u_+(\eta_2);\eta_2)$ for $\eta_1,\eta_2\in(\eta_-,\eta_+)$, which contradicts to the strict inequality in condition~3.
%\end{proof}

A simple example of a monotone force flow is given by a force of a chord that does not depend on~$\eps$ 
(the monotonicity follows from Lemma~\ref{EvolutionalTailGrowthLemma}). 
%Another trivial remark is that one 
%can take a restriction of a monotone force flow to a smaller domain to obtain another monotone force flow 
%(this will be implicitly used several times in what follows). 
\begin{Rem}\label{FCDMFF}
Let~$\{\Ch([a(\eps),b(\eps)],*)\},$ $\eps\in[\eta_1,\eta_2],$ be either decreasing or full flow of chordal domains. 
Then
\begin{equation*}
\FFr(u\,;\eps) = \Fr(u\,;\Ch([a(\eps),b(\eps)],*);\eps),\quad u \in [b(\eps),\tr(\eps)]\cap \mathbb{R}, 
\quad \eps \in [\eta_1,\eta_2],
\end{equation*}
is a right monotone force flow, and
\begin{equation*}
\FFl(u\,;\eps) = \Fl(u\,;\Ch([a(\eps),b(\eps)],*);\eps),\quad u \in [\tl(\eps),a(\eps)]\cap \mathbb{R}, 
\quad \eps \in [\eta_1,\eta_2],
\end{equation*}
is a left monotone force flow.
\end{Rem}

\subsection{Balance equation lemma}\label{s423}
\begin{Def}\label{baleq}\index{equation! balance equation}
Let~$\Fr$ and~$\Fl$ be forces of chordal domains, infinities, or multicups \textup(for the definitions 
see Subsection~\ref{Subsec362} and Remark~\ref{Rem211101}\textup) such that their domains intersect. 
Then the \emph{balance equation} is 
\begin{equation}\label{baleqformula}
\Fr(u) = \Fl(u),
\end{equation}
where~$u$ belongs to the intersection of the domains of the forces.
\end{Def} 

We are looking for solutions of balance equations. Lemma~\ref{moninttails} helped us to establish 
the existence of the solution in Corollary~\ref{BalanceEquationForSmallEps}. 
\begin{Le}\label{monbaleq}
Let~$\Fr$ and~$\Fl$ be two forces of chordal domains\textup, infinities\textup, multicups\textup, or simply chords 
such that their tails intersect. Then the function~$\Fr(u) - \Fl(u)$ strictly increases on this intersection.
\end{Le}
\begin{proof}
The proof is identical to that of Lemma~\ref{moninttails}.
\end{proof}

\section{Local evolutional theorems}\label{s43}
The form of all theorems in this section is: if for some~$\eps$ we can build a Bellman candidate on 
a specific domain using specific formulas, then, for a slightly larger~$\eps$, we can also build a Bellman 
candidate on a perturbed domain using similar formulas with perturbed parameters. The statements are rather 
formal and somewhat bulky, so, before each statement we give a short heuristic explanation. We also recall 
our convention that the forces are strictly negative inside chordal domains and any tale.

The proposition below says that any full chordal domain (with nonzero differentials of the upper chord) 
surrounded by tangent domains enlarges as~$\eps$ increases (in other words, in view of Definition~\ref{FlowOfChordalDomains}, 
a full chordal domain generates a full flow of chordal domains starting from it). 
 
\begin{St}[{\bf Induction step for a chordal domain}]\label{InductionStepForChordalDomain}\index{domain! chordal domain}
Let $\Ch([a_1,b_1],[a_0,b_0])$ be a full chordal domain\textup, 
$u_1 \!<\! a_1\!\leq\! a_0\!\leq\! b_0\!\leq\! b_1 \!<\! u_2,$ and $0\!<\!\eta_1\! <\!\epsmax$. 
Let a continuous function~$B$ coincide with the standard 
candidates on~$\Ch([a_1,b_1],[a_0,b_0])$\textup, $\Lt(u_1,a_1;\eta_1)$\textup, and~$\Rt(b_1,u_2;\eta_1)$. 
If~$\DL(a_1,b_1) < 0$ and~$\DR(a_1,b_1) < 0$\textup, then there exists~$\eta_2,$ $\eta_2 > \eta_1$, 
and a full flow $\{\Ch([a(\eps),b(\eps)],[a_0,b_0])\}_{\eps \in [\eta_1,\eta_2]}$ of chordal domains 
such that $a(\eta_1)=a_1,$ $b(\eta_1)=b_1$ and for each~$\eps \in [\eta_1,\eta_2]$ there exists a continuous 
function~$B_{\eps}$ that coincides with the standard candidates on 
$\Lt(u_1,a(\eps);\eps)$\textup,~$\Ch([a(\eps),b(\eps)],[a_0,b_0])$\textup, and~$\Rt(b(\eps),u_2;\eps)$. 
\end{St}

\begin{proof}
We use Remark~\ref{Rem230301} and Proposition~\ref{St010901} for the pair $(a_1,b_1)$ to find~$\delta>0$ 
and functions~$\tilde a$ and~$\tilde b$ acting from~$[0,\delta]$ to~$\mathbb{R}$ such that $\tilde a$ 
is decreasing,~$\tilde a(0) = a_1$, $\tilde b$ is increasing,
$\tilde b(0)=b_1$, $\tilde b(\tau)-\tilde a(\tau)=b_1-a_1+\tau$, and the pair~$(\tilde a(\tau),\tilde b(\tau))$ 
satisfies~\eqref{urlun1} as well as the inequalities $\DL(\tilde a(\tau),\tilde b(\tau)) < 0$ 
and~$\DR(\tilde a(\tau),\tilde b(\tau)) < 0$ for $\tau\in[0,\delta]$. We take $\eta_2>\eta_1$ in such a way that 
the chord $\big[g(\tilde a(\delta),g(\tilde b(\delta))\big]$ intersects the free boundary of~$\Omega_{\eta_2}$. 
Then, for any $\eps \in [\eta_1,\eta_2]$ there exists unique $\tau=\tau(\eps) \in [0,\delta]$ such that the 
chord $\big[g(\tilde a(\tau),g(\tilde b(\tau))\big]$ is tangent to the  free boundary of~$\Omega_{\eps}$.
We put $a(\eps)=\tilde a(\tau(\eps))$, $b(\eps)=\tilde b(\tau(\eps))$.

By our assumptions,~$u_2$ belongs to the right tail of~$\Ch([a_1,b_1],[a_0,b_0])$ and~$u_1$ belongs to 
its left tail. Due to Lemma~\ref{FullCupEvolutionTails},~$u_1$ and~$u_2$ belong to the left and the right 
tail of~$\Ch([a(\eps), b(\eps)],[a_0,b_0])$,~$\eps \in [\eta_1,\eta_2]$, correspondingly. This allows us to define the required~$B_\eps$ 
on~$\Ch([a(\eps),b(\eps)],[a_0,b_0])\cup \Lt(u_1,a(\eps);\eps) \cup \Rt(b(\eps),u_2;\eps)$ 
for~$\eps \in [\eta_1,\eta_2]$.  
\end{proof}

By Remark~\ref{FCDMFF}, the functions
\begin{equation*}
\FFr(u;\eps) = \Fr(u;[a(\eps),b(\eps)];\eps), \qquad \FFl(u;\eps) = \Fl(u;[a(\eps),b(\eps)];\eps)
\end{equation*}
are the right and left monotone force flows on the corresponding domains.

The following proposition describes how a non-full multicup (i.\,e., a multicup $\MTC(\{\mathfrak{a}_i\}_{i=1}^k;\eps)$ 
such that the chord~$\big[g(\mathfrak{a}_1^{\mathrm{l}}), g(\mathfrak{a}_k^{\mathrm{r}})\big]$ does not lie 
in $\Omega_\eps$) evolves in~$\eps$.

\begin{St}[{\bf Induction step for a multicup}]\label{InductionStepForMulticup}\index{multicup}
Let $\MTC(\{\mathfrak{a}_i\}_{i=1}^k;\eta_1),$ $\eta_1\in(0, \epsmax)$\textup, be a multicup such that
the chord~$\big[g(\mathfrak{a}_1^{\mathrm{l}}), g(\mathfrak{a}_k^{\mathrm{r}})\big]$ does not lie in 
$\Omega_{\eta_1}$\textup, $u_1 < \mathfrak{a}_1^{\mathrm{l}}$ and~$\mathfrak{a}_k^{\mathrm{r}} < u_2$. 
Let a continuous function~$B$ coincide with the standard candidates on~$\MTC(\{\mathfrak{a}_i\}_{i=1}^k;\eta_1),$
$\Lt(u_1,\mathfrak{a}_1^{\mathrm{l}};\eta_1),$ and~$\Rt({a}_k^{\mathrm{r}},u_2;\eta_1)$. Then there 
exists~$\eta_2$\textup, $\eta_1 < \eta_2,$ such that for each~$\eps \in [\eta_1,\eta_2]$ there exists 
a continuous function~$B_{\eps}$ that coincides with the standard candidates on 
$\Lt(u_1,\mathfrak{a}_1^{\mathrm{l}};\eps),$ $\MTC(\{\mathfrak{a}_i\}_{i=1}^k;\eps),$ 
and~$\Rt(\mathfrak{a}_k^{\mathrm{r}},u_2;\eps)$.
\end{St}

\begin{proof}
We may take any~$\eta_2>\eta_1$ such that the chord 
$\big[g(\mathfrak{a}_1^{\mathrm{l}}), g(\mathfrak{a}_k^{\mathrm{r}})\big]$ intersects the free boundary 
of $\Omega_{\eta_2}$. Take any~$\eps$ from the interval prescribed. By Remark~\ref{TailGrowthInfinity} 
the tails of the multicup enlarge with~$\eps$, therefore the points $u_1$ and $u_2$ belong to them. 
Thus one can build the required function~$B_{\eps}$.
\end{proof}

Clearly, the functions
\begin{equation*}
\Fr(u;[\mathfrak{a}_1^{\mathrm l},\mathfrak{a}_k^{\mathrm r}];\eps)\quad \text{and}  \quad
\Fl(u;[\mathfrak{a}_1^{\mathrm l},\mathfrak{a}_k^{\mathrm r}];\eps)
\end{equation*}
are the right and left monotone force flows  on the corresponding domains.

The following proposition says that a long chord with nonzero tails gives rise to a chordal domain. 
We note that this generalizes Proposition~\ref{InductionStepForChordalDomain} (in the latter case the 
differentials are nonzero, and thus the tails of the upper chord are nonzero). However, for didactic reasons, 
we prefer to separate these two propositions. In a sense, the cases where one or both differentials are 
zero differ from what is described in Proposition~\ref{InductionStepForChordalDomain}. Indeed, suppose 
that one of the differentials is zero and we have a chordal domain below a chord~$[g(a_0),g(b_0)]$ with 
the standard candidate on it. By Proposition~\ref{InductionStepForLongChords}, after we increase~$\eps$, 
we can build a chordal domain above~$[g(a_0),g(b_0)]$ and the standard candidate there. However, we cannot 
glue these standard candidates into a single one, because the differentials should be strictly negative 
for a standard candidate. We also note that the proposition below may be applied to the upper chord of 
a full multicup (i.\,e., a multicup such that the chord 
$\big[g(\mathfrak{a}_1^{\mathrm{l}}), g(\mathfrak{a}_k^{\mathrm{r}})\big]$ is tangent to the free boundary of $\Omega_\eps$). 

\begin{St}[{\bf Induction step for a long chord}]\label{InductionStepForLongChords}
Suppose that $a_0,b_0,u_1,u_2\!\in\!\mathbb{R},$ $a_0\!<\!b_0,$ the pair $(a_0,b_0)$ 
satisfies the cup equation~\eqref{urlun1}\textup, and the chord~$[g(a_0),g(b_0)]$ is tangent to the free boundary 
of $\Omega_{\eta_1},$ $\eta_1 \!\in\! (0, \epsmax)$. Let also~$u_1 \!<\! a_0$ and~$u_2 \!>\! b_0,$ 
and let~$u_1$ belong to the left tail of the 
chord $[g(a_0),g(b_0)],$ and~$u_2$ to its right tail. Then there exists~$\eta_2,$ $\eta_2>\eta_1,$ 
and a full flow $\{\Ch([a(\eps),b(\eps)],[a_0,b_0])\}_{\eps \in [\eta_1,\eta_2]}$ of chordal domains such that 
for each~$\eps \in [\eta_1,\eta_2]$ there exists a continuous function~$B_{\eps}$ that coincides with the 
standard candidates on $\Lt(u_1,a(\eps);\eps),$ $\Ch([a(\eps),b(\eps)],[a_0,b_0]),$ and~$\Rt(b(\eps),u_2;\eps)$. 
\end{St}

\begin{proof}
The proof is a repetition of the proof of Proposition~\ref{InductionStepForChordalDomain}, the only difference 
is that here we need to verify Condition~\ref{StartChord} directly in order to use Proposition~\ref{St010901}. 

We need to check that the quantity 
$$
L(t)=\dettt{\gamma'(t)}{\gamma'(a_0)}{\gamma'(b_0)} = g_1'(t)\det\begin{pmatrix}
1& \kappa_2(t)& \kappa_3(t)\\
g_1'(a_0)& g_2'(a_0)& f'(a_0)\\
g_1'(b_0)& g_2'(b_0)& f'(b_0)
\end{pmatrix}
$$
is positive for $t$ on the left of $a_0$ and is negative on the right of $b_0$. 
Let us consider the right case, the left one is analogous.  
Since the right tail of the chord~$[g(a_0),g(b_0)]$ is nonempty,~$\DR(a_0,b_0)\leq 0$. 
If $\DR(a_0,b_0)<0$, then the claim is trivial (see Remark~\ref{Rem230301}). 
If $\DR(a_0,b_0)=0$, then we use Lemma~\ref{zerodifrootsOutside} and conclude that $\tors$ decreases 
on $(b_0,b_0+\delta)$ for some $\delta>0$. Note that $L(b_0)=L'(b_0)=0$ and for $t\in (b_0,b_0+\delta)$ we have
$$
L_1(t) \df \frac{\Big(\frac{L(t)}{g_1'(t)}\Big)'}{\kappa_2'(t)} = \tors(t) \dett{g'(a_0)}{g'(b_0)}+\const < 
\tors(b_0)\dett{g'(a_0)}{g'(b_0)}+\const =L_1(b_0)=0.
$$ 
Since $\kappa_2'>0$, the function $\frac{L}{g_1'}$ strictly decreases on $(b_0,b_0+\delta)$, therefore 
$\frac{L}{g_1'}(t)<\frac{L}{g_1'}(b_0)=0$ and $L(t)<0$ for $t \in (b_0,b_0+\delta)$ because $g_1'>0$.
\end{proof}

Now we turn to trolleybuses. The next two propositions claim that the base of a trolleybus shrinks 
when~$\eps$ increases. On a more formal way, there exists a decreasing flow of chordal domains such that 
for each $\eps$ we can build a trolleybus on the corresponding chordal domain. In what follows, we will use the notation~$\bar{\mathbb{R}}$ for the extended real line~$\R\cup \{\pm\infty\}$.

\begin{St}[{\bf Induction step for a right trolleybus}]\label{InductionStepForRightTrolleybus}\index{trolleybus}
Let~$\eta_1,\eta_3 \!\in\!\mathbb{R},$ $0\!<\!\eta_1\!<\!\eta_3\!<\!\epsmax$. 
Suppose that~$a_1,b_1,a_2,b_2,u \!\in\! \mathbb{R},$ $a_2\!<\!a_1\!<\!b_1\!<\!b_2\!\leq\! u$. 
Let $\Omega([a_2,a_1])$ be the fence with the corresponding function 
$s\colon [a_2,a_1] \to [b_1,b_2],$ $s(a_i)=b_i,$ $i=1,2$. This fence coincides with~$\Ch([a_2,b_2],[a_1,b_1])$. Suppose that~$\FFr$ is a right monotone force flow 
with the corresponding functions $u_\pm\colon[\eta_1,\eta_3]\to\bar{\mathbb{R}}$ such 
that~$a_2\in[u_-(\eta_1),u_+(\eta_1)]$. Suppose that there exists a continuous function~$B$ that coincides 
with the standard candidates on~$\Rt(u_-(\eta_1),a_2;\eta_1),$ $\RTroll(a_2,b_2;\eta_1),$ $\Rt(b_2,u;\eta_1),$ 
and $\Ch([a_2,a_1])$. Moreover\textup, the force of $B$ on $\Rt(u_-(\eta_1),a_2;\eta_1)$ coincides 
with~$\FFr(\,\cdot\;;\eta_1)$. Then there exists~$\eta_2$\textup, $\eta_1<\eta_2<\eta_3$\textup, and 
a strictly increasing function~$a\colon [\eta_1,\eta_2]\to[a_2,a_1]$ such that~$a(\eta_1) = a_2$\textup, 
and for any $\eps \in [\eta_1,\eta_2]$ there exists a continuous function $B_\eps$ that coincides 
with the standard candidates on~$\Rt(u_-(\eps),a(\eps);\eps)$\textup, $\RTroll(a(\eps),b(\eps);\eps)$\textup, 
$\Ch([a(\eps),a_1])$\textup, and $\Rt(b(\eps),u;\eps)$\textup, where $b(\eps)=s(a(\eps))$. Moreover\textup, 
the force of $B_\eps$ on $\Rt(u_-(\eps),a(\eps);\eps)$ is~$\FFr(\,\cdot\;;\eps)$. The functions
\begin{equation}\label{eq060401} 
\Fr\Big(t;a(\eps),b(\eps);\eps\Big)\quad \text{and}\quad  
\Fl\Big(t;a(\eps),b(\eps);\eps\Big)
\end{equation}
are the right and the left monotone force flows on the corresponding domains.
\end{St}

We note that~$\{\Ch([a(\eps),a_1])\}, \eps \in [\eta_1,\eta_2]$, is a decreasing flow of chordal domains.
\begin{proof}
Our first step is to find $\eta_2\in(\eta_1,\eta_3)$ and a point $a^-$ in a right neighborhood of $a_2$ 
such that $a^-<u_+(\eps)$ and 
\begin{equation}\label{eq220201}
\FFr(a^-;\eps) > \Fl(a^-;\Ch([a_2,a_1])), \qquad \eps\in(\eta_1,\eta_2).
\end{equation}
We consider two cases: $u_+(\eta_1)>a_2$ and $u_+(\eta_1)=a_2$. 

First, if $u_+(\eta_1)>a_2$, then by Lemma~\ref{monbaleq} the function 
$\FFr(\,\cdot\;;\eta_1)-\Fl(\,\cdot\;;\Ch([a_2,a_1]))$ strictly increases in some right neighborhood of $a_2$. 
Since the point $a_2$ is the root of the balance equation for these forces, i.\,e., 
$\FFr(a_2;\eta_1) = \Fl(a_2;\Ch([a_2,a_1]))$, we may simply take a point $a^- \in (a_2,u_+(\eta_1))\cap (a_2,a_1)$ 
and obtain $\FFr(a^-;\eta_1) > \Fl(a^-;\Ch([a_2,a_1]))$. It follows from continuity of forces that 
for some $\eta_2 \in (\eta_1,\eta_3)$ inequality~\eqref{eq220201} holds.

Now we assume that $u_+(\eta_1)=a_2$. Then $\FFr(a_2;\eta_1) = \Fl(a_2;\Ch([a_2,a_1]))=0$, and according 
to Lemma~\ref{zerodifrootsInside} the function~$\tors$  decreases on an interval~$(a_2,a^-)$, $a^-<a_1$.  It follows from Remark~\ref{rem220201} that $u_+(\eps)>a_2$ for any 
$\eps\in(\eta_1,\eta_3]$, therefore  $u_+(\eps)>a^-$ by Remark~\ref{rem241101}. Lemma~\ref{monbaleq} implies 
that the function $\FFr(\,\cdot\;;\eps) - \Fl(\,\cdot\;;\Ch([a_2,a_1]))$ strictly increases on $(a_2,a^-)$. 
If there were no required $\eta_2$, then $\FFr(a^-;\eps)-\Fl(a^-;\Ch([a_2,a_1]))\leq 0$ for all $\eps$, 
$\eps\in(\eta_1,\eta_3)$. However,
$$
\FFr(a_2;\eps) \quad\longrightarrow \quad \FFr(a_2;\eta_1)=\Fl(a_2;\Ch([a_2,a_1])), \qquad \eps\to \eta_1+,
$$
therefore $\FFr(\,\cdot\;;\eps) - \Fl(\,\cdot\;;\Ch([a_2,a_1]))$ converges uniformly to zero on $[a_2,a^-]$ 
when $\eps \to \eta_1+$, or, equivalentely,
\eq{eq220202}{
\FFr(\,\cdot\;;\eps)-\tors \quad \longrightarrow \quad \Fl(\,\cdot\;;\Ch([a_2,a_1]))-\tors,\qquad \eps\to\eta_1+,
} 
uniformly on $[a_2,a^-]$. Due to Remark~\ref{Rem150401} the function on the right hand side of~\eqref{eq220202} 
is strictly decreasing while the function on the left hand side is strictly increasing. This contradicts the uniform convergence~\eqref{eq220202}. Thus, there exists $\eta_2 \in (\eta_1,\eta_3)$ such 
that~\eqref{eq220201} holds.

By Definition~\ref{MFF} of the monotone force flow,
\begin{equation*}
\FFr(a_2;\eps) - \Fl(a_2;\Ch([a_2,a_1])) < 0, \quad \eps \in (\eta_1,\eta_2),
\end{equation*}
since~$\Fl(a_2;\Ch([a_2,a_1]))$ does not depend on~$\eps$.
Therefore, there exists a point~$a = a(\eps) \in [a_2,a^-]$ that solves the balance equation
$\FFr(a;\eps)=\Fl(a;\Ch([a_2,a_1]))$. We note that the function~$a$ is increasing. For the existence 
of the desired function~$B_\eps$ we only need to verify that 
\begin{equation*}
\Fr(t;a(\eps),b(\eps);\eps) < 0, \quad \eps \in [\eta_1,\eta_2],\  t \in [b(\eps); u],
\end{equation*}
which follows from Corollary~\ref{FlowMonotonicityCorollary}, because~$\big\{\Ch([a(\eps),a_1])\big\}$ 
is a decreasing flow of chordal domains.

By Remark~\ref{FCDMFF}, the force functions \eqref{eq060401} form monotone force flows.
\end{proof}

\begin{St}[{\bf Induction step for a left trolleybus}]\label{InductionStepForLeftTrolleybus}
Let~$\eta_1,\eta_3 \in \mathbb{R},$ $0\!<\!\eta_1\!<\!\eta_3\!<\!\epsmax$. 
Suppose that $a_1,b_1,a_2,b_2,u \!\in\! \mathbb{R},$ $u \!\leq a_2\!<\!a_1\!<\!b_1\!<\!b_2$. 
Let $\Ch([b_1,b_2])$ be the fence with the corresponding function 
$s\colon[b_1,b_2]\to[a_2,a_1],$ $s(b_i)=a_i,$ $i=1,2$. This fence coincides with~$\Ch([a_2,b_2],[a_1,b_1])$. Suppose that~$\FFl$ is a left monotone force flow with 
the corresponding functions $u_\pm \colon [\eta_1,\eta_3] \to \bar{\mathbb{R}}$ such that
$b_2 \in [u_-(\eta_1),u_+(\eta_1)]$. Suppose that there exists a continuous function~$B$ that coincides 
with the standard candidates on~$\Lt(b_2,u_+(\eta_1);\eta_1),$ $\LTroll(a_2,b_2;\eta_1),$ 
$\Lt(u,a_2;\eta_1),$ and $\Ch([b_1,b_2])$. Moreover\textup, the force of $B$ on $\Lt(b_2,u_+(\eta_1);\eta_1)$ 
is~$\FFl(\,\cdot\;;\eta_1)$. Then there exists~$\eta_2$\textup, $\eta_1\!<\!\eta_2\!<\!\eta_3$\textup, 
and a strictly decreasing function~$b\colon [\eta_1,\eta_2] \to [b_1,b_2]$ such that $b(\eta_1) = b_2$\textup, 
and for any $\eps \in [\eta_1,\eta_2]$ there exists a continuous function $B_\eps$ that coincides with 
the standard candidates on~$\Lt(b(\eps),u_+(\eps);\eps)$\textup, $\LTroll(a(\eps),b(\eps);\eps)$\textup, 
$\Lt(u,a(\eps);\eps),$ and $\Ch([b_1,b(\eps)]),$ where $a(\eps)=s(b(\eps))$. Moreover\textup, the force 
of $B_\eps$ on $\Lt(b(\eps),u_+(\eps);\eps)$ is~$\FFl(\,\cdot\;;\eps)$. The forces~\eqref{eq060401} 
form monotone force flows.
\end{St}

\begin{Rem}\label{UniqueTrolley}
It follows from Lemma~\textup{\ref{monbaleq}} that the functions~$a$ and $b$ constructed in 
Propositions~\textup{\ref{InductionStepForRightTrolleybus}} and~\textup{\ref{InductionStepForLeftTrolleybus}} 
are unique \textup(at least when~$\eta_2 - \eta_1$ is sufficiently small\textup).
\end{Rem}

The following four propositions describe the evolutional behavior of multitrolleybuses. It appears that each 
multitrolleybus immediately splits into a trolleybus parade (by formulas~\eqref{RMultitrolleybusDesintegration} 
and~\eqref{LMultitrolleybusDesintegration}), and each of the trolleybuses decreases. We consider two simpler cases separately to make the presentation smoother.

\begin{St}[{\bf Induction step for a right multitrolleybus on a  solid root}]\label{Pr300301} 
Let $\eta_1,\eta_3\!\in\!\mathbb{R},$ $0\!<\!\eta_1\!<\!\eta_3\!<\!\epsmax$. Consider a right 
multitrolleybus~$\MTTR(\{\mathfrak{a}\};\eta_1)$ on a solid root 
$\mathfrak{a}=[\mathfrak{a}^{\mathrm{l}},\mathfrak{a}^{\mathrm{r}}]$ 
\textup(the case~$\mathfrak{a}^{\mathrm{l}}\!=\!\mathfrak{a}^{\mathrm{r}}$ is not excluded\textup). 
Let $u\!\in\!\mathbb{R}$, $\mathfrak{a}^{\mathrm{r}}\!\leq\! u$. Suppose that~$\FFr$ is a right monotone force 
flow with the corresponding functions $u_\pm \colon [\eta_1,\eta_3] \to \bar{\mathbb{R}}$
such that~$\mathfrak{a}^{\mathrm{l}} \in [u_-(\eta_1),u_+(\eta_1)]$. Suppose that there exists a continuous 
function~$B$ that coincides with the standard candidates on~$\Rt(u_-(\eta_1),\mathfrak{a}^{\mathrm{l}};\eta_1),$ 
$\MTTR(\{\mathfrak{a}\};\eta_1),$ and $\Rt(\mathfrak{a}^{\mathrm{r}},u;\eta_1)$. 
Moreover\textup, the force of $B$ on $\Rt(u_-(\eta_1),\mathfrak{a}^{\mathrm{l}};\eta_1)$ 
is~$\FFr(\,\cdot\;;\eta_1)$. Then for any $\eps \in (\eta_1,\eta_3]$ there exists a continuous 
function $B_\eps$ that coincides with the standard candidates on~$\Rt(u_-(\eps),u;\eps),$ 
and the force of $B_\eps$ on $\Rt(u_-(\eps), u;\eps)$ is~$\FFr(\,\cdot\;;\eps)$.
\end{St}

\begin{proof}
First we note that $\FFr(\mathfrak{a}^{\mathrm{l}}; \eta_1)=0$, therefore $\mathfrak{a}^{\mathrm{l}}=u_+(\eta_1)$. 
The only thing we need to prove is that $u_+(\eps)\geq u$ for any $\eps \in (\eta_1,\eta_3]$. 
By Remark~\ref{rem220201} we know that $u_+(\eps)>u_+(\eta_1)=\mathfrak{a}^{\mathrm{l}}$. 
By Remark~\ref{rem241101} $u_+(\eps) \notin \mathfrak{a}$, therefore $u_+(\eps)> \mathfrak{a}^{\mathrm{r}}$. 
From~\eqref{eq241101}, for any $t \in (\mathfrak{a}^{\mathrm{r}},\min(u_+(\eps), u)]$ we have 
\begin{equation}\label{eq241102}
\FFr(t;\eps) = 
\int_{\mathfrak{a}^{\mathrm{r}}}^t\exp\Big(-\int_\tau^t\frac{\kappa_2'}{\kappa_2-\kappa}\Big)\tors'(\tau)\,d\tau+
\exp\Big(-\int_{\mathfrak{a}^{\mathrm{r}}}^t\frac{\kappa_2'}{\kappa_2-\kappa}\Big)\FFr(\mathfrak{a}^{\mathrm{r}};\eps).
\end{equation}

The second summand in~\eqref{eq241102} is negative because $\mathfrak{a}^{\mathrm{r}} \in [u_-(\eps),u_+(\eps))$. The first summand in~\eqref{eq241102} is equal to $\Fr(t;\mathfrak{a}^{\mathrm{l}},\mathfrak{a}^{\mathrm{r}};\eps)$ (see~\eqref{eq101101}, $\DR(\mathfrak{a}^{\mathrm{l}},\mathfrak{a}^{\mathrm{r}})=0$). Furthermore, the inequality 
\eq{eq251101}{
\Fr(t;\mathfrak{a}^{\mathrm{l}},\mathfrak{a}^{\mathrm{r}};\eps)< 
\Fr(t;\mathfrak{a}^{\mathrm{l}},\mathfrak{a}^{\mathrm{r}};\eta_1)
}
follows from monotonicity of forces with respect to $\eps$. The right hand side of~\eqref{eq251101} is non-positive because $t \in (\mathfrak{a}^{\mathrm{r}},u]$ and 
$\Fr(\,\cdot\;;\mathfrak{a}^{\mathrm{l}},\mathfrak{a}^{\mathrm{r}};\eta_1)$ is the force of the standard 
candidate on $\Rt(\mathfrak{a}^{\mathrm{r}},u;\eta_1)$. Thus, in particular, $\FFr( \min(u_+(\eps), u);\eps)< 0$. 
Since $\FFr(u_+(\eps);\eps)=0$, we have $\min(u_+(\eps), u) \ne u_+(\eps)$, therefore $u_+(\eps)>u$.
\end{proof}

\begin{Rem}
Note that the case $\mathfrak{a}^{\mathrm{l}} =\mathfrak{a}^{\mathrm{r}}$ in Proposition~\ref{Pr300301} means 
that the multitrolleybus $\MTTR(\{\mathfrak{a}\};\eta_1)$ is a fictitious vertex~$\Rt(\mathfrak{a},\mathfrak{a})$ 
of the fifth type. Similar for the left case.
\end{Rem}

\begin{St}\label{Pr060401} {\bf (Induction step for a right multitrolleybus with one underlying chord\-al domain).} 
Let~$\eta_1,\eta_3\!\in\!\mathbb{R},$ $0\!<\!\eta_1\!<\!\eta_3\!<\!\epsmax$. Suppose that 
$a_1,b_1,a_2,b_2,u\!\in\!\mathbb{R},$ $a_2\!<\!a_1\!<\!b_1\!<\!b_2\!\leq\! u$. 
Let $\mathfrak{a}=[\mathfrak{a}^{\mathrm{l}},\mathfrak{a}^{\mathrm{r}}]$ be a solid root\textup, 
$\mathfrak{a}^{\mathrm{l}}\!<\!\mathfrak{a}^{\mathrm{r}}\!=\!a_2$. Consider a right 
multitrolleybus~$\MTTR(\{\mathfrak{a},b_2\};\eta_1)$. Let $\Ch([a_2,a_1])$ be the chordal domain with 
the corresponding function $s\colon[a_2,a_1] \to [b_1,b_2],$ $s(a_i)=b_i,$ $i=1,2$. Suppose 
that~$\FFr$ is a right monotone force flow with the corresponding functions 
$u_\pm \colon [\eta_1,\eta_3] \to \bar{\mathbb{R}}$ such that
$\mathfrak{a}^{\mathrm{l}} \in [u_-(\eta_1),u_+(\eta_1)]$. Suppose that there exists a continuous 
function~$B$ that coincides with the standard candidates on~$\Rt(u_-(\eta_1),\mathfrak{a}^{\mathrm{l}};\eta_1),$ 
$\MTTR(\{\mathfrak{a},b_2\};\eta_1),$ $\Rt(b_2,u;\eta_1),$ and $\Ch([a_2,a_1])$. Moreover\textup, the force 
of $B$ on $\Rt(u_-(\eta_1),\mathfrak{a}^{\mathrm{l}};\eta_1)$ is~$\FFr(\,\cdot\;;\eta_1)$. Then there exists
$\eta_2$\textup, $\eta_1\!<\!\eta_2\!<\!\eta_3$\textup, and a strictly increasing function
$a\colon [\eta_1,\eta_2] \to [a_2,a_1]$ such that~$a(\eta_1) = a_2$\textup, and for any $\eps \in (\eta_1,\eta_2]$ 
there exists a continuous function $B_\eps$ that coincides with the standard candidates 
on~$\Rt(u_-(\eps),a(\eps);\eps)$\textup, $\RTroll(a(\eps),b(\eps);\eps),$ $\Ch([a(\eps),a_1])$\textup, 
and $\Rt(b(\eps),u;\eps),$ where $b(\eps)=s(a(\eps))$. Moreover\textup, the force of $B_\eps$ on 
$\Rt(u_-(\eps),a(\eps);\eps)$ is~$\FFr(\,\cdot\;;\eps)$. The forces~\eqref{eq060401} form monotone force flows.
\end{St}

\begin{proof}
The force of $B$ on $\mathfrak{a}$ is equal zero, in particular $\FFr(\mathfrak{a}^{\mathrm{l}};\eta_1)=0$, 
therefore $u_+(\eta_1)=\mathfrak{a}^{\mathrm{l}}$. We cannot directly apply 
Proposition~\ref{InductionStepForRightTrolleybus} in this case, but its proof works, because for 
any $\eps \in(\eta_1,\eta_3)$ by Proposition~\ref{Pr300301} we have $u_+(\eps)>\mathfrak{a}^{\mathrm{r}}$, 
and the arguments of the second case in the proof of Proposition~\ref{InductionStepForRightTrolleybus} 
(where $u_+(\eta_1)=a_2$) can be applied literally. 
\end{proof}

\begin{St}[{\bf Induction step for a general right multitrolleybus}]\label{InductionStepForMultitrolleybus}\index{multitrolleybus}
Let~$\eta_1,\eta_3 \in \mathbb{R}$ and $\eta_1\!\!<\!\!\eta_3\!\!<\!\!\epsmax$. Consider a right multitrolleybus
$\MTTR(\{\mathfrak{a}_i\}_{i=1}^k;\eta_1)$. Let~$u\!\in\!\mathbb{R},$ $\mathfrak{a}_{k}^{\mathrm{r}}\!\leq\! u$. 
Suppose that~$\FFr$ is a right monotone force flow with the corresponding functions 
$u_\pm \colon [\eta_1,\eta_3] \to \bar{\mathbb{R}}$
such that $\mathfrak{a}_1^{\mathrm{l}} \in [u_-(\eta_1),u_+(\eta_1)]$. We also suppose that 
for each~$i = 1,2,\ldots,k-1$ there are chordal domains
$\Ch([\mathfrak{a}_i^{\mathrm{r}},\mathfrak{a}_{i+1}^{\mathrm{l}}],*)$ with the corresponding 
functions~$s_i,$ $s_i(\mathfrak{a}_i^{\mathrm{r}})=\mathfrak{a}_{i+1}^{\mathrm{l}}$. Suppose that 
there exists a continuous function~$B$ that coincides with the standard candidates 
on~$\Rt(u_-(\eta_1),\mathfrak{a}_1^{\mathrm{l}};\eta_1)$\textup, $\MTTR(\{\mathfrak{a}_i\}_{i=1}^k;\eta_1),$ 
$\Rt(\mathfrak{a}_k^{\mathrm{r}},u;\eta_1)$\textup, and 
every~$\Ch([\mathfrak{a}_i^{\mathrm{r}},\mathfrak{a}_{i+1}^{\mathrm{l}}],*)$. Moreover\textup, the force 
of $B$ on $\Rt(u_-(\eta_1),\mathfrak{a}_1^{\mathrm{l}};\eta_1)$ is~$\FFr(\,\cdot\;;\eta_1)$. 
Then\textup, there exists a number~$\eta_2,$ $\eta_2\!>\!\eta_1,$ and a collection of strictly increasing 
functions~$a_i\colon[\eta_1,\eta_2] \to \mathbb{R}$\textup, $a_i(\eta_1)=\mathfrak{a}_i^{\mathrm{r}}$\textup,
$i = 1,2,\ldots,k-1$\textup, such that for every~$\eps \in (\eta_1,\eta_2]$ there exists a continuous 
function~$B_{\eps}$ defined on the domain
\begin{equation*}
\begin{aligned}
\Rt\big(u_-(\eps),a_1(\eps);\eps\big) \cup \Big(\cup_{i=1}^{k-2} \Rt\big(b_i(\eps),a_{i+1}(\eps);\eps\big)\Big) 
\cup \Big(\cup_{i=1}^{k-1}\RTroll(a_i(\eps),b_{i}(\eps);\eps)\Big) \cup 
\\ 
\Big(\cup_{i=1}^{k-1}\Ch([a_i(\eps),b_{i}(\eps)],*)\Big)\cup \Rt\big(b_{k-1}(\eps),u;\eps\big); \qquad 
b_i(\eps)=s_i(a_i(\eps)),
\end{aligned}
\end{equation*}
that coincides with the standard candidate inside each subdomain of the partition. Moreover\textup, the force 
of $B_\eps$ on $\Rt(u_-(\eps),a_1(\eps);\eps)$ is~$\FFr(\,\cdot\;;\eps)$. The functions
\begin{equation*}
\Fr\Big(\,\cdot\,;a_i(\eps),b_i(\eps);\eps\Big) \qquad \text{and} \qquad 
\Fl\Big(\,\cdot\,;a_i(\eps),b_i(\eps);\eps\Big), \quad 1\leq i \leq k-1,
\end{equation*}
are the right and the left monotone force flows on the corresponding domains.
\end{St}

\begin{proof}
We split~$\MTTR(\{\mathfrak{a}_i\}_{i=1}^k;\eta_1)$ into the union of 
$\MTTR(\{\mathfrak{a}_i, \mathfrak{a}_{i+1}^{\mathrm{l}}\};\eta_1)$, $i=1,\dots, k-1$, and possibly 
$\MTTR(\{\mathfrak{a}_k\};\eta_1)$ if $\mathfrak{a}_{k}^{\mathrm{l}}<\mathfrak{a}_{k}^{\mathrm{r}}$. 
We apply Proposition~\ref{Pr060401} if $\mathfrak{a}_{i}^{\mathrm{l}}<\mathfrak{a}_{i}^{\mathrm{r}}$, 
and Proposition~\ref{InductionStepForRightTrolleybus} if 
$\mathfrak{a}_{i}^{\mathrm{l}}=\mathfrak{a}_{i}^{\mathrm{r}}$ to the multitrolleybuses 
$\MTTR(\{\mathfrak{a}_i, \mathfrak{a}_{i+1}^{\mathrm{l}}\};\eta_1)$ successively. Finally, we conclude 
our induction applying Proposition~\ref{Pr300301} to $\MTTR(\{\mathfrak{a}_k\};\eta_1)$ if 
$\mathfrak{a}_{k}^{\mathrm{l}}<\mathfrak{a}_{k}^{\mathrm{r}}$.
\end{proof}

\begin{St}[{\bf Induction step for a general left multitrolleybus}]\label{InductionStepForLeftMultitrolleybus}\index{multitrolleybus}
Let~$\eta_1,\eta_3 \in \mathbb{R}$\textup, $\eta_1\!<\!\eta_3\!<\!\epsmax$. Consider a left 
multitrolleybus~$\MTTL(\{\mathfrak{a}_i\}_{i=1}^k;\eta_1)$. Let~$u \in \mathbb{R},$ 
$u \leq \mathfrak{a}_{1}^{\mathrm{l}}$. Suppose that~$\FFl$ is a left monotone force flow with 
the corresponding functions $u_\pm \colon [\eta_1,\eta_3] \to \bar{\mathbb{R}}$ such 
that~$\mathfrak{a}_k^{\mathrm{r}} \in [u_-(\eta_1),u_+(\eta_1)]$. We also suppose that for each 
$i = 1,2,\ldots,k-1$ there are chordal domains~$\Ch([\mathfrak{a}_i^{\mathrm{r}},\mathfrak{a}_{i+1}^{\mathrm{l}}],*)$ 
with the corresponding functions~$s_i,$ $s_i(\mathfrak{a}_{i+1}^{\mathrm{l}})=\mathfrak{a}_i^{\mathrm{r}}$. 
Suppose that there exists a continuous function~$B$ that coincides with the standard candidates 
on~$\Lt(u,\mathfrak{a}_1^{\mathrm{l}};\eta_1)$\textup, $\MTTL(\{\mathfrak{a}_i\}_{i=1}^k;\eta_1)$\textup, 
$\Lt(\mathfrak{a}_k^{\mathrm{r}},u_+(\eta_1);\eta_1)$\textup, and every
$\Ch([\mathfrak{a}_i^{\mathrm{r}},\mathfrak{a}_{i+1}^{\mathrm{l}}],*)$. Moreover\textup, the force of $B$ on 
$\Lt(\mathfrak{a}_k^{\mathrm{r}},u_+(\eta_1);\eta_1)$ is~$\FFl(\,\cdot\;;\eta_1)$. Then\textup, there exists 
a number~$\eta_2$\textup, $\eta_2\!>\!\eta_1$\textup, and a collection of strictly decreasing functions 
$b_i\colon[\eta_1,\eta_2] \to \mathbb{R}$\textup, $b_i(\eta_1)=\mathfrak{a}_{i+1}^{\mathrm{l}}$\textup,
$i = 1,2,\ldots,k-1$\textup, such that for every~$\eps \in (\eta_1,\eta_2]$ there exists a continuous 
function~$B_{\eps}$ defined on the domain
\begin{equation*}
\begin{aligned}	
\Lt\big(u,a_1(\eps);\eps\big) \cup \Big(\cup_{i=1}^{k-2} \Lt\big(b_i(\eps),a_{i+1}(\eps);\eps\big)\Big) 
\cup \Big(\cup_{i=1}^{k-1}\LTroll(a_i(\eps),b_{i}(\eps);\eps)\Big) \cup 
\\ 
\Big(\cup_{i=1}^{k-1}\Ch([a_i(\eps),b_{i}(\eps)],*)\Big)\cup \Lt\big(b_{k-1}(\eps),u_+(\eps);\eps\big); 
\qquad a_i(\eps)=s_i(b_i(\eps)),
\end{aligned}
\end{equation*}
that coincides with the standard candidate inside each subdomain of the partition. Moreover\textup, 
the force of $B_\eps$ on $\Lt(b_{k-1}(\eps),u_+(\eps);\eps)$ is~$\FFl(\,\cdot\;;\eps)$. The functions
\begin{equation*} 
\Fr\Big(t;a_i(\eps),b_i(\eps);\eps\Big) \qquad \text{and} \qquad 
\Fl\Big(t;a_i(\eps),b_i(\eps);\eps\Big), \quad 1\leq i \leq k-1,
\end{equation*}
are the right and the left monotone force flows on the corresponding domains.
\end{St}

In the following proposition, we show that angles move continuously.

\begin{St}[{\bf Induction step for an angle}]\label{InductionStepAngle}\index{angle}
Let~$\eta_1,\eta_3 \in \mathbb{R},$ $0<\eta_1<\eta_3<\epsmax,$ let $u_0 \in \mathbb{R}$. Suppose that~$\FFl$ is a 
left monotone force flow with the corresponding functions $u_\pm^l\colon[\eta_1,\eta_3]\to\bar{\mathbb{R}},$ and
$\FFr$ is a right monotone force flow with the corresponding functions $u_\pm^r\colon[\eta_1,\eta_3]\to\bar{\mathbb{R}}$ 
such that $u_0 \in (u_-^r(\eta_1), u_+^r(\eta_1)] \cap [u_-^l(\eta_1), u_+^l(\eta_1))$. Suppose that there exists 
a continuous function~$B$ that coincides with the standard candidates on~$\Rt(u_-^r(\eta_1),u_0;\eta_1),$ 
$\Ang(u_0;\eta_1),$ and~$\Lt(u_0,u_+^l(\eta_1);\eta_1)$. Moreover\textup, the forces of $B$  
on~$\Rt(u_-^r(\eta_1),u_0;\eta_1)$ and~$\Lt(u_0,u_+^l(\eta_1);\eta_1)$ are~$\FFr(\,\cdot\;;\eta_1)$ and 
$\FFl(\,\cdot\;;\eta_1)$ correspondingly. Then there exists $\eta_2,$ $\eta_1<\eta_2<\eta_3,$ and a continuous 
function $u\colon[\eta_1,\eta_2]\to\mathbb{R}$ such that $u(\eta_1)=u_0$ and for any $\eps \in [\eta_1,\eta_2]$ 
we have $u(\eps) \in [u_-^l(\eps), u_+^l(\eps))\cap (u_-^r(\eps), u_+^r(\eps)]$ and there exists a continuous 
function $B_\eps$ that coincides with the standard candidates on~$\Rt(u_-^r(\eps),u(\eps);\eps),$ $\Ang(u(\eps);\eps),$ 
and~$\Lt(u(\eps),u_+^l(\eps);\eps)$. Moreover\textup, the forces of $B_\eps$ on~$\Rt(u_-^r(\eps),u(\eps);\eps)$ 
and~$\Lt(u(\eps),u_+^l(\eps);\eps)$ are~$\FFr(\,\cdot\;;\eps)$ and $\FFl(\,\cdot\;;\eps)$ correspondingly.
\end{St}

\begin{proof}
We first note that in order to construct a desired function $B_\eps$  it suffices to find a root $u(\eps)$ 
of the balance equation~\eqref{baleqformula} for $\FFr(\,\cdot\;;\eps)$ and $\FFl(\,\cdot\;;\eps)$ (see Corollary~\ref{cor250901}).
 
If $u_+^r(\eta_1)>u_0>u_-^l(\eta_1)$, then the proof is simple. The function 
$\FFr(\,\cdot\;;\eta_1)-\FFl(\,\cdot\;;\eta_1)$ is strictly increasing on 
$(u_-^l(\eta_1), u_+^l(\eta_1))\cap (u_-^r(\eta_1), u_+^r(\eta_1))$ by Lemma~\ref{monbaleq}, therefore 
it is positive on some right neihborhood of $u_0$ and is negative on a left one. By continuity with 
respect to $\eps$ the function $\FFr(\,\cdot\;;\eps)-\FFl(\,\cdot\;;\eps)$ has a root $u(\eps)$ in a fixed 
neighborhood of $u_0$ for $\eps$ sufficiently close to $\eta_1$, $\eps> \eta_1$. Again, by Lemma~\ref{monbaleq} 
this root is unique in the intersection of the tails.

If $u_+^r(\eta_1)=u_0$ or $u_0=u_-^l(\eta_1)$, then $\FFr(u_0;\eta_1)=\FFl(u_0;\eta_1)=0$, because $u_0$ is 
the root of the balance equation~\eqref{baleqformula} for $\FFr(\,\cdot\,;\eta_1)$ and $\FFl(\,\cdot\,;\eta_1)$. 
Therefore $u_+^r(\eta_1)=u_-^l(\eta_1)=u_0$. It follows from Remark~\ref{rem241101} 
that there exist $u_-,u_+ \in \mathbb{R}$ such that $u_-<u_0<u_+$, $\tors'>0$ on $(u_-,u_0)$, and $\tors'<0$ 
on $(u_0,u_+)$. Thus, $u_0=c_i$ for some $i$ (see Definition~\ref{roots}). For any $\eps \in (\eta_1,\eta_3)$ 
we have $u_-^l(\eps)<u_-$, and $u_+<u_+^r(\eps)$. The only thing we need to check is that for any $\delta>0$ 
the function $\FFr(\,\cdot\;;\eps)-\FFl(\,\cdot\;;\eps)$ has a root $u(\eps) \in (u_0-\delta,u_0+\delta)$ 
provided $\eps$ sufficiently close to $\eta_1$. 

If it is not the case, then there exists a positive $\delta$ and a sequence $\eps_n \to \eta_1$, $\eps_n > \eta_1$ 
such that the functions $\Phi_n(\,\cdot\,)=\FFr(\,\cdot\;;\eps_n)-\FFl(\,\cdot\;;\eps_n)$ have no roots on 
$(u_0-\delta,u_0+\delta)$. Without loss of generality, we may assume that $\Phi_n$ is negative on this intersection. 
The function $\Phi_n$ is strictly increasing on $[u_0,u_0+\delta]$, is negative there, and $\Phi_n(u_0)\to 0$, 
$n \to +\infty$. Therefore, $\Phi_n$ converges to zero uniformly on $[u_0,u_0+\delta]$. It follows that 
$$
\lim_{n \to \infty} \FFr(t;\eps_n)-\tors(t)= \lim_{n \to \infty} \FFl(t;\eps_n)-\tors(t) = 
\FFl(t;\eta_1)-\tors(t), \quad t \in [u_0,u_0+\delta],
$$
where the function on the right hand side is strictly decreasing on $(u_0,u_0+\delta)$ and the functions 
on the left hand side are strictly increasing on $(u_0,u_0+\delta)$ due to Remark~\ref{Rem150401}. 
This leads to the contradiction and proves the claim.
\end{proof}

\begin{St}[{\bf Induction step for a multibirdie}]\label{MultibirdieDesintegrationSt}\index{birdie}\index{multibirdie}
Let~$\eta_1,\eta_3 \in \mathbb{R},$ $\eta_1<\eta_3<\epsmax$. Consider a multibirdie
$\MTB(\{\mathfrak{a}_i\}_{i=1}^k;\eta_1)$. Suppose that~$\FFl$ and $\FFr$ are left and right monotone force 
flows with the corresponding functions $u_\pm^{\mathrm L}$ and $u_\pm^{\mathrm R}$ acting from $[\eta_1,\eta_3]$ 
to $\bar{\mathbb{R}}$ such that~$\mathfrak{a}_1^{\mathrm{l}} \in (u_-^{\mathrm R}(\eta_1),u_+^{\mathrm R}(\eta_1)]$ 
and~$\mathfrak{a}_k^{\mathrm{r}} \in [u_-^{\mathrm L}(\eta_1),u_+^{\mathrm L}(\eta_1))$. We also suppose 
that for each~$i,$ $i = 1,2,\ldots,k-1,$ there are chordal domains
$\Ch([\mathfrak{a}_i^{\mathrm{r}},\mathfrak{a}_{i+1}^{\mathrm{l}}],*)$. Suppose that there exists 
a continuous function~$B$ that coincides with the standard candidates 
on~$\Rt(u_-^{\mathrm R}(\eta_1), \mathfrak{a}_1^{\mathrm{l}};\eta_1),$ $\MTB(\{\mathfrak{a}_i\}_{i=1}^k;\eta_1),$ 
$\Lt(\mathfrak{a}_k^{\mathrm{r}},u_+^{\mathrm L}(\eta_1);\eta_1),$ and 
every~$\Ch([\mathfrak{a}_i^{\mathrm{r}},\mathfrak{a}_{i+1}^{\mathrm{l}}],*)$. Moreover\textup, the force of 
$B$ on~$\Rt(u_-^{\mathrm R}(\eta_1), \mathfrak{a}_1^{\mathrm{l}};\eta_1)$ is~$\FFr(\,\cdot\;;\eta_1),$ and 
the force of $B$ on~$\Lt(\mathfrak{a}_k^{\mathrm{r}},u_+^{\mathrm L}(\eta_1);\eta_1)$ is~$\FFl(\,\cdot\;;\eta_1)$. 
Then\textup, there exists a number~$\eta_2,$ $\eta_1\!<\!\eta_2\!<\!\eta_3,$ and a collection of strictly monotone 
functions~$a_i$ and $b_i$ acting from $[\eta_1,\eta_2]$ to $\mathbb{R}$ such that the $a_i$ are increasing 
and $a_i(\eta_1)=\mathfrak{a}_{i}^{\mathrm{r}},$ the $b_i$ are decreasing and 
$b_i(\eta_1)=\mathfrak{a}_{i+1}^{\mathrm{l}},$ and $[g(a_i(\eps)),g(b_i(\eps))]$ is a chord 
of~$\Ch([\mathfrak{a}_i^{\mathrm{r}},\mathfrak{a}_{i+1}^{\mathrm{l}}],*)$. Furthermore\textup, 
for every~$\eps \in (\eta_1,\eta_2]$ there exists an integer~$j = j(\eps)$\textup, $1 \leq j \leq k$\textup, 
and~$u(\eps) \in (b_{j-1}(\eps), a_{j}(\eps)]$ \textup(here we put~$b_0\df u_-^{\mathrm{R}}(\eps)$ and 
$a_k\df u_+^{\mathrm{L}}(\eps)$\textup) such that there exists a continuous function~$B_{\eps}$ on the domain
\begin{gather*}
\Big(\cup_{i=0}^{j-2}\Rt\big(b_i(\eps), a_{i+1}(\eps);\eps\big)\Big) \cup
\Big(\cup_{i=1}^{j-1}\RTroll\big(a_i(\eps),b_i(\eps);\eps\big)\Big) \cup
\\ 
\Rt\big(b_{j-1}(\eps),u(\eps);\eps\big) \cup
\Ang\big(u(\eps);\eps\big) \cup \Lt\big(u(\eps),a_{j}(\eps);\eps\big)\cup
\\
\Big(\cup_{i=j}^{k-1}\LTroll\big(a_i(\eps),b_i(\eps);\eps\big)\Big) \cup
\Big(\cup_{i=j}^{k-1}\Lt\big(b_i(\eps), a_{i+1}(\eps);\eps\big)\Big) \cup
\\
\Big(\cup_{i=1}^{k-1} \Ch\big(\big[a_i(\eps),b_i(\eps)\big],*\big)\Big)
\end{gather*}
that coincides with the standard candidate inside each subdomain of the partition. Moreover\textup, the force of 
$B_\eps$ coincides with~$\FFr(\,\cdot\;;\eps)$ in the right neighborhood of $u_-^{\mathrm{R}}(\eps)$ and 
with~$\FFl(\,\cdot\;;\eps)$ in the left neighborhood of $u_+^{\mathrm{L}}(\eps)$. The functions
\begin{equation}\label{eq280401}
{\FFr}_{,i}(\,\cdot\,;\eps) = \Fr\Big(\,\cdot\,;a_i(\eps),b_i(\eps);\eps\Big),\qquad 
{\FFl}_{,i}(\,\cdot\,;\eps) = \Fl\Big(\,\cdot\,;a_i(\eps),b_i(\eps);\eps\Big), \quad 1\leq i \leq k-1,
\end{equation}
are the right and the left monotone force flows on the corresponding domains.
\end{St}

\begin{proof}
We consider two multitrolleybuses~$\MTTR(\{\mathfrak{a}_i\}_{i=1}^k)$ and~$\MTTL(\{\mathfrak{a}_i\}_{i=1}^k)$. Application of Propositions~\ref{InductionStepForMultitrolleybus} and~\ref{InductionStepForLeftMultitrolleybus} 
to these foliations gives us numbers $\eta_2^{\mathrm{R}}, \eta_2^{\mathrm{L}}$ and collections of 
functions~$\{a_j,b_j\}_{j=1}^{k-1}$ (all of them are defined on an interval~$[\eta_1,\eta_2]$, where 
$\eta_2=\min(\eta_2^{\mathrm{R}},\eta_2^{\mathrm{L}})$), which we call~$\{a_{j}^{\mathrm{R}},b_{j}^{\mathrm{R}}\}$  
and~$\{a_{j}^{\mathrm{L}},b_{j}^{\mathrm{L}}\}$. Let~$\eps \in (\eta_1,\eta_2]$ be fixed.

We {\bf claim} that if~$b_{j}^{\mathrm{R}}\leq b_{j}^{\mathrm{L}}$ for some $j$, $1<j<k$, then~$b_{j-1}^{\mathrm{R}}<b_{j-1}^{\mathrm{L}}$. 
Indeed, we have the following chain of inequalities: 
\begin{equation}\label{eq280402}
\Fr\Big(b_{j-1}^{\mathrm{L}}; a_{j-1}^{\mathrm{L}},b_{j-1}^{\mathrm{L}};\eps\Big)=
\Fl\Big(b_{j-1}^{\mathrm{L}}; a_{j}^{\mathrm{L}},b_{j}^{\mathrm{L}};\eps\Big) \geq 
\Fl\Big(b_{j-1}^{\mathrm{L}}; a_{j}^{\mathrm{R}},b_{j}^{\mathrm{R}};\eps\Big) > 
\Fr\Big(b_{j-1}^{\mathrm{L}}; a_{j-1}^{\mathrm{R}},b_{j-1}^{\mathrm{R}};\eps\Big). 
\end{equation}
The first equality in~\eqref{eq280402} is simply the balance equation. The second inequality follows 
from Lemma~\ref{biggercupsmallerforce} and our assumption that $b_{j}^{\mathrm{R}} \leq b_{j}^{\mathrm{L}}$. 
We note that the point~$b_{j-1}^{\mathrm{L}}$ lies in the left tail of the chordal domain 
$\Ch([a_{j}^{\mathrm{R}},b_{j}^{\mathrm{R}}],*)$, because it lies in the left tail of the larger chordal 
domain~$\Ch([a_{j}^{\mathrm{L}},b_{j}^{\mathrm{L}}],*)$ (see Corollary~\ref{TailsGrowthNonEv}). 
Moreover, $b_{j-1}^{\mathrm{L}}$ lies on the left of the point $a_j^{\mathrm{R}}$, which is the root of the balance equation 
of the forces $\Fr\Big(\,\cdot\,; a_{j-1}^{\mathrm{R}},b_{j-1}^{\mathrm{R}};\eps\Big)$ and 
$\Fl\Big(\,\cdot\,; a_{j}^{\mathrm{R}},b_{j}^{\mathrm{R}};\eps\Big)$, therefore the last inequality 
in~\eqref{eq280402} follows from Lemma~\ref{monbaleq}. Inequality~\eqref{eq280402} and 
Lemma~\ref{biggercupsmallerforce} imply that~$b_{j-1}^{\mathrm{R}} < b_{j-1}^{\mathrm{L}}$. The {\bf claim} is proved.

It follows that there exists~$j = j(\eps) \in \{1,2,\ldots, k\}$ such that for any~$i$, $0<i < j$, 
the inequality~$b_i^{\mathrm{R}} < b_i^{\mathrm{L}}$ holds true and for any~$i$,~$j\leq i <k$, one 
has~$b_i^{\mathrm{R}} \geq b_i^{\mathrm{L}}$. We define~$b_i(\eps) = b_{i}^\mathrm{R}(\eps)$ for~$0<i < j$ 
and~$b_i(\eps) = b_{i}^\mathrm{L}(\eps)$ for~$j\leq i < k$. We also put $a_i(\eps)$, $1 \leq i \leq k-1,$ in 
such a way that $[g(a_i(\eps)),g(b_i(\eps))]$ is a chord 
of~$\Ch([\mathfrak{a}_i^{\mathrm{r}},\mathfrak{a}_{i+1}^{\mathrm{l}}],*)$. In what follows we use 
the notation~\eqref{eq280401} and also put ${\FFr}_{,0}=\FFr$ and ${\FFl}_{,k} = \FFl$. We only have to 
prove that for $\eps$ sufficiently close to $\eta_1$ there is~$u(\eps)\in[b_{j(\eps)-1}(\eps),a_{j(\eps)}(\eps)]$ 
solving the balance equation for~${\FFr}_{,j(\eps)-1}(\,\cdot\,;\eps)$ and ${\FFl}_{,j(\eps)}(\,\cdot\,;\eps)$.

First, consider the case $1<j(\eps)<k$. Let $u_- = \max(b_{j(\eps)-1}(\eps),t^{\mathrm{L}})$, where 
$t^{\mathrm{L}}$ is the left end of the tail of ${\FFl}_{,j(\eps)}(\,\cdot\,;\eps)$. We claim that 
${\FFr}_{,j(\eps)-1}(u_-;\eps)\leq {\FFl}_{,j(\eps)}(u_-;\eps)$. Indeed, if $u_-=t^{\mathrm{L}}$ then 
the claim is obvious:
$$
{\FFl}_{,j(\eps)}(t^{\mathrm{L}};\eps)=0\geq {\FFr}_{,j(\eps)-1}(t^{\mathrm{L}};\eps),
$$
because $t^{\mathrm{L}}$ lies in the tail of ${\FFr}_{,j(\eps)-1}(\,\cdot\,;\eps)$.

If $u_-=b_{j(\eps)-1}(\eps)$, then $u_-<b_{j(\eps)-1}^\mathrm{L}(\eps)$, therefore due to Lemma~\ref{monbaleq} we have
\eq{eq091001}{
\begin{aligned}
{\FFr}_{,j(\eps)-1}(u_-;\eps)&-{\FFl}_{,j(\eps)}(u_-;\eps)\leq 
{\FFr}_{,j(\eps)-1}(b_{j(\eps)-1}^\mathrm{L}(\eps);\eps)-{\FFl}_{,j(\eps)}(b_{j(\eps)-1}^\mathrm{L}(\eps);\eps)
\\
&<\Fr\Big(b_{j(\eps)-1}^{\mathrm{L}}; a_{j(\eps)-1}^{\mathrm{L}},b_{j(\eps)-1}^{\mathrm{L}};\eps\Big) - 
{\FFl}_{,j(\eps)}(b_{j(\eps)-1}^\mathrm{L}(\eps);\eps)=0,
\end{aligned}
}
where the second inequality follows from Lemma~\ref{biggercupsmallerforce} and inequality 
$b_{j(\eps)-1}(\eps)<b_{j(\eps)-1}^\mathrm{L}(\eps)$. Similarly, for $u_+=\min(a_{j(\eps)}(\eps),t^{\mathrm{R}})$, 
where $t^{\mathrm{R}}$ is the right end of the tail of ${\FFr}_{,j(\eps)-1}(\,\cdot\,;\eps)$, we have 
${\FFr}_{,j(\eps)-1}(u_+;\eps)\geq {\FFl}_{,j(\eps)}(u_+;\eps).$ Therefore, there exists a root $u(\eps) \in [u_-,u_+]$ of the balance equation for the forces ${\FFr}_{,j(\eps)-1}(\,\cdot\,;\eps)$ and ${\FFl}_{,j(\eps)}(\,\cdot\,;\eps)$, which is unique due to Lemma~\ref{monbaleq}.

Inequality~\eqref{eq091001} implies that $u(\eps)>b_{j(\eps)-1}(\eps)$, therefore 
$u(\eps) \in (b_{j(\eps)-1}(\eps),a_{j(\eps)}(\eps)]$.

Now we consider the case $j(\eps)=1$ (the case $j(\eps)=k$ is symmetric). 
Let $u_+=\min(a_1(\eps),t^{\mathrm{R}}(\eps))$, where $t^{\mathrm{R}}$ is the right end of the tail 
of ${\FFr}_{,0}(\,\cdot\,;\eps)$. We {\bf claim} that ${\FFr}_{,0}(u_+;\eps)\geq {\FFl}_{,1}(u_+;\eps)$. Indeed, 
if $u_+=t^{\mathrm{R}}$, then the argument is the same as before: 
${\FFr}_{,0}(t^{\mathrm{R}};\eps)=0\geq {\FFl}_{,1}(t^{\mathrm{R}};\eps)$. If $u_+=a_1(\eps)$, then 
$u_+\geq a_1^{\mathrm{R}}(\eps)$ and, applying Lemma~\ref{monbaleq} and Lemma~\ref{biggercupsmallerforce}, we have
\begin{multline*}
{\FFr}_{,0}(u_+;\eps)-{\FFl}_{,1}(u_+;\eps)\geq{\FFr}_{,0}(a_1^{\mathrm{R}};\eps)-{\FFl}_{,1}(a_1^{\mathrm{R}};\eps)
\\
= {\FFr}_{,0}(a_1^{\mathrm{R}};\eps)-\Fl\Big(a_1^{\mathrm{R}}; a_{1}^{\mathrm{L}},b_{1}^{\mathrm{L}};\eps\Big) \geq
{\FFr}_{,0}(a_1^{\mathrm{R}};\eps)-\Fl\Big(a_1^{\mathrm{R}}; a_{1}^{\mathrm{R}},b_{1}^{\mathrm{R}};\eps\Big)=0.
\end{multline*}
The {\bf claim} is proved.

Recall that we want to prove: for $\eps$ sufficiently close to $\eta_1$ there is~$u(\eps)\in[b_{0}(\eps),a_{1}(\eps)]$ 
solving the balance equation for~${\FFr}_{,0}(\,\cdot\,;\eps)$ and ${\FFl}_{,1}(\,\cdot\,;\eps)$. Assume the contrary: for some sequence $\eps_n \to \eta_1+$ the function 
$\Phi_n(\,\cdot\,) \!=\! {\FFr}_{,0}(\,\cdot\,;\eps_n)\!-\!{\FFl}_{,1}(\,\cdot\,;\eps_n)$ has no balance points 
on $(b_0(\eps_n),a_1(\eps_n)]$. Then $\Phi_n(u_+(\eps_n))\!>\!0$ and therefore $\Phi_n$ is positive on the 
intersection of the tails, therefore $t^{\mathrm{L}}$, the left end of the tail of ${\FFl}_{,1}(\,\cdot\,;\eps_n)$, 
is not greater than $b_0(\eps_n)$. $\Phi_n$ is a strictly increasing function by Lemma~\ref{monbaleq}, and 
$\Phi_n(a_1(\eta_1))$ tends to $0$ when $n \to +\infty$. Therefore, $\Phi_n$ tends to zero uniformly on 
$(b_0(\eta_1)),a_1(\eta_1))$ (we recall that $b_0(\eta_1)\geq b_0(\eps)$ for any $\eps>\eta_1$ by definition 
of a monotone force flow). Thus, we have
$$
\lim_{n \to +\infty} {\FFl}_{,1}(t;\eps_n)-\tors(t)  = \lim_{n \to +\infty} {\FFr}_{,0}(t;\eps_n)-\tors(t)= 
\FFr(t;\eta_1)-\tors(t), \quad t \in (b_0(\eta_1),a_1(\eta_1)),
$$
where the function on the right hand side is strictly decreasing on $(b_0(\eta_1),a_1(\eta_1))$ and 
the functions on the left hand side are strictly increasing on $(b_0(\eta_1),a_1(\eta_1))$ due to 
Remark~\ref{Rem150401}. This leads to the contradiction and proves the statement.
\end{proof}

\begin{Rem}
We note that $j(\eps)$ in Proposition~\ref{MultibirdieDesintegrationSt} indeed could depend on $\eps$\textup, 
i.\,e.\textup, during the evolution the angle could change its place between the trolleybuses. Moreover\textup, 
the function $j(\cdot)$ could have an infinite number of jumps even on a bounded interval\textup, 
see the example ``Oscillating birdie'' on page~123 of~\cite{ISVZ2018}.
\end{Rem}

\begin{Rem}\label{rem231001}
We have seen in the proof that the root $u(\eps)$ of the balance equation is in the semiclosed interval 
$(b_{j(\eps)-1}(\eps),a_{j(\eps)}(\eps)]$. It may occur that $u(\eps)=a_{j(\eps)}(\eps)$\textup, and in 
this case the angle $\Ang\big(u(\eps);\eps\big)$\textup, the tangent domain $\Lt\big(u(\eps),u(\eps);\eps\big)$, 
and the trolleybus $\LTroll\big(u(\eps),b_{j(\eps)}(\eps);\eps\big)$ glue together forming a birdie\textup, 
see formula~\eqref{LTrolleybusPlusAngle}. Moreover\textup, this equation could be valid for $\eps$ in some 
interval\textup; the birdie can shrink without disintegrating.
\end{Rem}

\section{Global evolution}\label{s44}
Before passing to formal statements, we describe the rules of the evolution. 

Recall that in Subsection~\ref{s345} we constructed the graph $\Gamma$ corresponding to the foliation of a Bellman candidate, and its subgraph $\GammaFree$ corresponding to subdomains of the foliation that are not separated from the free boundary~$\dfree\Omega$.
The vertices of the graph $\Gamma$ are of two types: the vertices corresponding to linearity domains and fictitious vertices. The edges of $\Gamma$ always correspond to fences: either chordal or tangent domains.  The vertices of~$\GammaFree$ correspond to linearity domains: multicups, 
angles, trolleybuses, multitrolleybuses, birdies, multibirdies, fictitious vertices of the first, third 
(corresponding to long chords), fourth, and fifth type. 
%There are 
%chordal domains attached to some of them from below. 
%{\vv Na kartinkakh nado by na vsekh ryobrakh
%nadpisat' intervaly, a etot passazh pro sily ya by prosto ubral. Kstati j silakh. Pochemu u trolleybusov oni
%tol'ko s odnoi stotony?} {\udalit To each edge~$\mathfrak{E}$ of~$\GammaFree$, we assign 
%a force in the domain it represents. For example, it is natural to equip the edge~$\Lt(u_1,a_0)$ on 
%Figure~\ref{fig:cup_graph} with~$\Fl(\cdot\,;a_0,b_0;\eps)$; and it is natural to assign the 
%force~$\Fr(\cdot\,;a_0,b_0;\eps)$ to~$\Rt(b_0,u_2)$ on the left graph of Figure~\ref{fig:gtr}; 
%the edges~$\Lt(u_1,\mathfrak{a}_1)$ and~$\Rt(\mathfrak{a}_4,u_2)$ of the graph on Figure~\ref{fig:multicup4} 
%match~$\Fl(\cdot\,;\mathfrak{a}_1,\mathfrak{a}_4;\eps)$ and~$\Fr(\cdot\,;\mathfrak{a}_1,\mathfrak{a}_4;\eps)$ 
%correspondingly. }

To each edge~$\mathfrak{E}$ of~$\GammaFree$, we assign 
a force by the formal rule described in the table below. In the 
first column there is the type of the vertex from where the edge starts. The numerical parameters of 
this vertex are placed in the second column. The force that is assigned to the tangent domain if it 
lies on the left of the figure is in the third column, and the force that is assigned to the tangent 
domain lying on the right of the figure is in the last one. 

\medskip
\label{pg231001}
\centerline{
\begin{tabular}{|l|c|c|c|}
\hline 
{\bf Vertex type}&{\bf Parameters}&{\bf Left Force}&{\bf Right Force}
\\
\hline
Right trolleybus&$\{a,b\}$&&$\Fr(\cdot\,;a,b;\eps)$
\\
\hline
Left trolleybus&$\{a,b\}$&$\Fl(\cdot\,;a,b;\eps)$&
\\
\hline
Multicup&$\{\mathfrak{a}_i\}_{i=1}^k$&$\Fl(\cdot\,;\mathfrak{a}_1^{\mathrm l},\mathfrak{a}_k^{\mathrm r};\eps)$
&$\Fr(\cdot\,;\mathfrak{a}_1^{\mathrm l},\mathfrak{a}_k^{\mathrm r};\eps)$
\\
\hline
Right multitrolleybus&$\{\mathfrak{a}_i\}_{i=1}^k$&
&$\Fr(\cdot\,;\mathfrak{a}_1^{\mathrm l},\mathfrak{a}_k^{\mathrm r};\eps)$
\\
\hline
Left multitrolleybus&$\{\mathfrak{a}_i\}_{i=1}^k$
&$\Fl(\cdot\,;\mathfrak{a}_1^{\mathrm l},\mathfrak{a}_k^{\mathrm r};\eps)$&
\\
\hline
Fictitious vertex of the first type&$\{a,b\}$&$\Fl(\cdot\,;a,b;\eps)$&$\Fr(\cdot\,;a,b;\eps)$
\\
\hline
Fictitious vertex of the third type&$\{a,b\}$&$\Fl(\cdot\,;a,b;\eps)$&$\Fr(\cdot\,;a,b;\eps)$
\\
\hline
Fictitious vertex of the fourth type&$-\infty$&&$\Fr(\cdot\,;-\infty;\eps)$
\\
\hline
Fictitious vertex of the fourth type&$+\infty$&$\Fl(\cdot\,;+\infty;\eps)$&
\\
\hline
Right fictitious vertex of the fifth type&$c_i$&&$\Fr(\cdot\,;c_i,c_i;\eps)$
\\
\hline
Left fictitious vertex of the fifth type&$c_i$&$\Fl(\cdot\,;c_i,c_i;\eps)$&\\
\hline
\end{tabular}
$\quad (\star)$
}

\medskip

All the foliations generated during the evolution satisfy the following rule: if~$\Rt(u_1,u_2)$ or~$\Lt(u_1,u_2)$ 
is represented by the edge~$\mathfrak{E}$ in~$\GammaFree$, then~$(u_1,u_2)$ belongs to the tail of the force 
corresponding to~$\mathfrak{E}$. This requirement for the foliation will be called the \emph{non-degeneracy 
force condition}\index{condition! non-degeneracy force condition}.

\begin{Cond}\label{NonDegeneracyForceCondition}
For any edge~$\mathfrak{E}$ in~$\GammaFree$ corresponding to a tangent domain~$\Rt(u_1,u_2)$ 
or~$\Lt(u_1,u_2)$\textup, the interval~$(u_1,u_2)$ belongs to the tail of the force assigned to~$\mathfrak{E}$.
\end{Cond}

A short inspection of definitions shows that Condition~\ref{NonDegeneracyForceCondition} holds true for 
all the graphs corresponding to the standard candidates constructed. In other words, all the forces in 
tangent domains are strictly negative. In particular, the following remark is important.
 
\begin{Rem}\label{nondegforcecondSimplePictureRem}
The Bellman candidate constructed for a simple picture in Section~\textup{\ref{s41}} fulfills the 
non-degeneracy force Condition~\textup{\ref{NonDegeneracyForceCondition}}.
\end{Rem}

As has already been said, the main rule of the evolution is that the forces decrease (grow in absolute value), 
see Subsection~\ref{s422}. As a consequence, the tails strictly grow (by this we mean that the~$\tr$ increase 
and the~$\tl$ decrease). Thus, full chordal domains grow (Proposition~\ref{InductionStepForChordalDomain}), 
the multicups are stable\footnote{In a sense, they also grow: the border tangents rise; however, the numerical 
parameters do not change.} (Proposition~\ref{InductionStepForMulticup}), the trolleybuses shrink 
(Propositions~\ref{InductionStepForRightTrolleybus} and~\ref{InductionStepForLeftTrolleybus}), the angles 
continuously wander from side to side (Proposition~\ref{InductionStepAngle}). These figures can be described 
as stable. If there are multitrolleybuses or multibirdies in the foliation for  a fixed~$\eps$, they immediately 
disintegrate (Propositions~\ref{InductionStepForMultitrolleybus},~\ref{InductionStepForLeftMultitrolleybus}, 
and~\ref{MultibirdieDesintegrationSt}). These figures are unstable. As for the birdie, it can shrink 
(see Remark~\ref{rem231001}), but in general it disintegrates. Thus, it is half-stable. 

There is also one useful condition all our graphs will satisfy. It is of structural character (and thus 
relies on Definition~\ref{roots}) and concerns mostly fictitious vertices. It is called the \emph{leaf-root 
condition}\index{condition! leaf-root condition}.

\begin{Cond}\label{Leaf-root condition}
Any arc of any multifigure that is not a single point coincides with one of the solid roots~$c_i$.
Numeric parameters of the fictitious vertices of the second type are some roots~$c_i$ that are single points. 
Each fictitious vertex of the third type corresponds to a chord~$[g(a_0),g(b_0)]$ with at least one vanishing differential. If $\DR(a_0,b_0)=0,$ then~$b_0 = c_i$ for some single point root $c_i$\textup; if~$\DL(a_0,b_0) = 0,$ then~$a_0 = c_j$\textup, where $c_j$ must be a single point root as well.
The numeric parameter of each vertex of the fifth type is also a single point root~$c_i$.
\end{Cond}

\begin{Rem}
All simple graphs constructed in Section~\textup{\ref{s41}} fulfill the leaf-root 
Condition~\textup{\ref{Leaf-root condition}}.
\end{Rem}

\begin{Def}\label{Admissible graph}\index{graph! admissible graph}
Let~$\eps < \epsmax$. We say that a graph~$\Gamma$ is \emph{admissible} for~$f$ and~$\eps$ if all figures 
corresponding to the vertices and edges of~$\Gamma$ satisfy their local propositions.  
\end{Def}
By ``all figures corresponding to the vertices and edges of~$\Gamma$ satisfy their local propositions''  
we mean the following: for each vertex or edge in~$\Gamma$ the parameters satisfy the assumptions of the
proposition indicated for this vertex or edge in the table below (in the third column).

\medskip
\centerline{
\label{tab201101}
\begin{tabular}{|l|c|c|c|}
\hline 
{\bf Vertex or edge type}&{\bf Formulas}&{\bf Verification}&{\bf Evolutional rule}
\\
\hline
Right tangent domain&\eqref{beta2},~\eqref{fenceB2}&\ref{NewRightTangentsCandidate}, 
\ref{NewRightTangentsCandidateInfty}&
\\
\hline
Left tangent domain&\eqref{beta2},~\eqref{fenceB2}&\ref{NewLeftTangentsCandidate}, \ref{NewLeftTangentsCandidateInfty}&
\\
\hline
Chordal domain&\eqref{eq250901}&\ref{NewLightChordalDomainCandidate}&\ref{InductionStepForChordalDomain}
\\
\hline
Angle&
\eqref{eq250902}&\ref{St260901}&\ref{InductionStepAngle}
\\
\hline
Right trolleybus&\eqref{eq201001}&\ref{St131001}&\ref{InductionStepForRightTrolleybus}
\\
\hline
Left trolleybus&\eqref{eq201001}&\ref{St131001}&\ref{InductionStepForLeftTrolleybus}
\\
\hline
Birdie&\eqref{eq201001}&\ref{St131001}&\ref{MultibirdieDesintegrationSt}
\\
\hline
Multicup&\eqref{CandidateInMultifigure},~\eqref{eq201001}&\ref{St211101}&\ref{InductionStepForMulticup}
\\
\hline
Full multicup&\eqref{CandidateInMultifigure},~\eqref{eq201001}&\ref{St211101}&\ref{InductionStepForLongChords}
\\
\hline
Right multitrolleybus&\eqref{CandidateInMultifigure},~\eqref{eq201001}&\ref{St211101}
&\ref{InductionStepForMultitrolleybus}
\\
\hline
Left multitrolleybus&\eqref{CandidateInMultifigure},~\eqref{eq201001}&\ref{St211101}
&\ref{InductionStepForLeftMultitrolleybus}
\\
\hline
Multibirdie&\eqref{CandidateInMultifigure},~\eqref{eq201001}&\ref{St211101}&\ref{MultibirdieDesintegrationSt}
\\
\hline
Closed multicup&\eqref{CandidateInMultifigure},~\eqref{eq201001}&\ref{St211101}& Stable
\\
\hline
Fictitious vertex of the first type&\eqref{eq250901}&\ref{St260901}&\ref{InductionStepForChordalDomain}
\\
\hline
Fictitious vertex of the third type, long chord&\eqref{eq250901}&\ref{St260901}&\ref{InductionStepForLongChords}
\\
\hline
Fictitious vertex of the fifth type& \eqref{Beta2InTheFifthType}& \ref{St260901}&\ref{Pr300301}
\\
\hline
\end{tabular}
}
\medskip

In the first column, there is the type of the vertex or edge, in the second there is a reference to formulas
that are used to construct the canonical function~$B$ in the corresponding figure, in the third column the number of the 
proposition that guarantees that this~$B$ is a Bellman candidate are stored. Finally, the last column contains 
the number of the proposition that describes the local evolution of the parameters for the figure. We have 
omitted fictitious vertices of the second and fourth types (as well as the vertices of the third type that 
correspond to short chords), because the value of the function~$B$ in the domains corresponding to them is 
defined trivially, and these figures are stable and have no evolutional scenarios.

Now we describe how to construct the function~$B$ from a graph. First, one constructs this function to be the 
standard candidate on all the domains corresponding to vertices and edges that participate 
in~$\Gamma\setminus\GammaFree$, because for their figures there is no additional information needed to 
construct~$B$ (no information from other figures). Second, we construct the function~$B$ to be a standard 
candidate on all the domains corresponding  to vertices of $\GammaFree$ not being leaves (i.\,e., except angles). 
Third, we construct the standard candidates for the edges of $\GammaFree$. For each such edge~$\mathfrak{E}$, 
the values of~$B$ in the figure corresponding to its beginning define the force function on the domain 
corresponding to the edge (see table~$(\star)$ on page~\pageref{pg231001}), thus one may construct~$B$ in 
the tangent domain corresponding to~$\mathfrak{E}$ if he knows the values of~$B$ on the domain of its 
source\footnote{There is one exception: for tangent domains coming from infinity, one does not need any 
boundary data.}. Finally, we construct $B$ on the domains corresponding to leaves of $\GammaFree$ (i.\,e., 
on angles), because we know the values of $B$ on the linear boundary of each such angle. Note that if~$\Gamma$ 
fulfills Condition~\ref{NonDegeneracyForceCondition}, then the restriction of~$B$ to each figure is a standard 
Bellman candidate there. Admissibility of the graph guarantees that the force function (defined locally on each element of the foliation by the rules from the table~$(\star)$ on page~\pageref{pg231001}) is a non-negative continuous function on~$\mathbb{R}$. 

\begin{Rem}\label{AdmissibleCondidate}
The function~$B$ constructed from an admissible graph is a Bellman candidate. 
\end{Rem} 

\begin{proof}
We need to verify conditions of Definition~\ref{candidate} for the function $B$. Looking at the table above 
we use the corresponding verification proposition for each vertex or edge and see that the function $B$ 
possesses the foliation on the entire $\Omega_\eps$. We also note that the function~$B$ is locally concave 
and~$C^1$-smooth not only on subdomains, but globally.
\end{proof}

Since during the evolution some figures grow and angles move, several figures might crash. For example, 
the vertex of an angle may coincide with the right endpoint of a long chord. In such a case, we look at 
formula~\eqref{CupPlusAngleR}, and see that now they form a trolleybus. Therefore, the graph of the foliation 
changes at this moment~$\eps$. We call such moments \emph{the critical points of the evolution}. The idea is 
that if a crash occurs, then the crashed figures compose a new one (with the help of formulas from 
Subsection~\ref{s344}), and we can proceed the evolution. Unfortunately, there might be infinitely many 
critical points (see the example on the page~123 of \cite{ISVZ2018}, where an angle flips the direction 
of a trolleybus infinitely many times). However, if one restricts his attention only to those critical points, 
at which the structure of the graph~$\GammaFixed$ essentially changes, he finds only a finite number of 
critical points. Such points are called \emph{essentially critical}\index{essentially critical points}. 
%{\vv Dima, ne nuzhno li zdes podrobnee poyasnit'? 
%Tut pochemu-to zskommentirovana sleduyushchaya fraza: (we will give the rigorous definition in the proof 
%of Theorem~\ref{BC} on page~\pageref{lb061101}). A bez opredeleniya chto takoe essentially critical points
%tut kak-to vsyo vyglyadit nelepo.
%}
The following definition is also useful.
\begin{Def}\label{Smoothgraph}\index{graph! smooth graph}
We say that a graph~$\Gamma$ is \emph{smooth} if there are no vertices representing full multicups\textup, 
multitrolleybuses\textup, multibirdies\textup, fictitious vertices of the third type that represent long 
chords\textup, and fictitious vertices of the fifth type in~$\Gamma$. 
\end{Def}

\begin{Th}\label{BC}
For any~$\eps < \epsmax$\textup, there exists a graph~$\Gamma(\eps)$ admissible for~$f$ and~$\eps$.
\end{Th}
We will not give a careful proof of the theorem because it repeats literally the proof of the same theorem 
for the BMO case (see Theorem~4.4.15 of~\cite{ISVZ2018}). Here we will describe the main steps of the proof.

First, we use Theorem~\ref{SimplePicture} to build a smooth admissible graph $\Gamma(\eps)$ for small~$\eps$. 
Then, we use local evolution theorems from Subsection~\ref{s43} collected in the table on page~\pageref{tab201101} 
to show that if there exists an admissible graph $\Gamma(\eta_1)$ for some $\eta_1$, then we can construct a 
smooth admissible graph $\Gamma(\eps)$ for $\eps$, $\eps>\eta_1$, sufficiently close to~$\eta_1$. If for 
some $\eta_2$ we have smooth graphs $\Gamma(\eps)$, $\eps \in (\eta_1,\eta_2)$, we can pass to the limit and 
construct a limit graph $\Gamma(\eta_2)$. It can happen that in the limit graph $\Gamma(\eta_2)$ some edges 
``have zero length''. In this case we modify the graph using formulas from Subsection~\ref{s344}. This modified 
graph can be non-smooth but it is admissible and we can continue evolution starting from it. It appears that 
under our assumptions there are only finite number of essentially critical points of the evolution, when the 
graph $\Gamma(\eps)$ is not smooth. In such a way we obtain the graph $\Gamma(\eps)$ for any $\eps<\epsmax$.}

%% file: 5Chapter_text.tex
\chapter{Optimizers}\label{C5}
{%\color{blue}
In the previous chapter, we constructed a Bellman candidate~$B$ of a special form (Theorem~\ref{BC} 
and Remark~\ref{AdmissibleCondidate}). We claim that it coincides with the Bellman function~$\Bell$. 
Subsection~\ref{s224} suggests a method to prove the claim.  We have to construct an 
optimizer\index{optimizer}~$\vf_x$ for each~$x \in \Omega_{\eps}$ (see Definition~\ref{Opt}). 
Here we will follow the same strategy as when we were constructing Bellman candidates: we will first 
study the local behavior of the optimizers (i.\,e., how do optimizers vary when~$x$ runs through one figure), 
this is done in Section~\ref{s52}, and then ``glue'' these local scenarios together in Section~\ref{s53}. 
The optimizers for the BMO case were built in~\cite{ISVZ2018}, and here we will follow a similar strategy for the general 
case. In Sections~\ref{s51} and~\ref{s52} we will not use evolution, so $\eps$ is fixed and we omit the subscript $\eps$ till Section~\ref{s53}. In particular, we will write $\Xi$ instead of $\Xi_\eps$ and $\Omega = \Xi_0\setminus \Xi$ instead of $\Omega_\eps$, see Section~\ref{s21} for the definitions of these objects. In Section~\ref{s54} we will consider the cases when Conditions~\eqref{eq111001} and~\eqref{eq111002} are violated.

\section{Abstract theory}\label{s51}
We begin with an abstract description of how do optimizers look like. First, as it was mentioned in 
Subsection~\ref{s224}, it is natural to construct monotone optimizers. It is not difficult to build a monotone 
function~$\vf_x$ such that~$\av{\vf_x}{} = x$ and~$B(x) = \av{\vf_x}{}$. The main difficulty is to verify 
that~$\vf_x \in \Class_\Omega$. It was noticed in~\cite{IOSVZ2015} that it is more natural to argue geometrically. 
The notion of a \emph{delivery curve}\index{delivery curve} is useful in this context.

\begin{Def}\label{DelCurve}
Let~$B$ be a Bellman candidate on the domain $\Omega$. Suppose~$\vf\colon [l,r] \to \dfi\Omega$ is 
an integrable function. The curve~$\gamma\colon (l,r] \to \Omega$ given by the formula
\begin{equation}\label{CurveGenerator}
\gamma(\tau) \df 
\av{\vf}{[l,\tau]},\quad \tau \in (l,r],
\end{equation}
is called a \emph{delivery curve} if~$B(\gamma(\tau)) = \av{\ff(\vf)}{[l,\tau]}$ for any~$\tau \in (l,r]$ 
\textup(in particular\textup,~$\gamma(\tau) \in \Omega$\textup). The function~$\vf$ is called the 
\emph{generating function} for~$\gamma$.
\end{Def} 
In other words,~$\gamma$ is a curve that ``delivers'' optimizers to the point. The word ``curve'' here 
means a parametrized curve, because the definition depends on the parametrization. The advantage of 
considering such a curve is that it allows to verify the condition that~$\vf$ is a test function 
(i.\,e.,~$\vf \in \Class_{\Omega}$). 

The main feature we will use is the formula
\begin{equation}\label{FirstDerivative}
\gamma(\tau) + (\tau-l)\gamma'(\tau) = \vf(\tau),
\end{equation}
which can be obtained by differentiation of~\eqref{CurveGenerator}. In particular, this formula shows that 
the tangent to~$\gamma$ at~$\tau$ points in the direction of~$\vf(\tau)$. Thus, one can reconstruct the values 
of~$\vf$ by looking at the points on the fixed boundary that ``are indicated'' by the tangents of the 
corresponding delivery curve. We will use this principle very often. Moreover, equation~\eqref{FirstDerivative} 
allows to reconstruct~$\vf$ from~$\gamma$. 

\begin{Le}\label{Monotone_Convex}
A curve given by formula~\textup{\eqref{CurveGenerator}} is convex if its generating function is monotone. 
%{\color{red} (Is the converse also true?)}
\end{Le}
\begin{proof}
We will give a proof for the case of an increasing generating function~$\vf$. The case of decreasing 
function is symmetric. 

Let us assume for a while that the function~$\vf$ is $C^2$-smooth. In such a case, we may differentiate~\eqref{FirstDerivative} and get
\begin{equation*}
(\tau-l)\gamma''(\tau) = -2\gamma'(\tau) + \vf'(\tau). 
\end{equation*}
Thus, the curvature of~$\gamma$, which is~$(\gamma'_1)^{-3} \dett{\gamma'}{\gamma''}$, has the same 
sign as~$\dett{\gamma'}{\vf'}$, because $\gamma_1'>0$. We use~\eqref{FirstDerivative} once again to express~$\gamma'$ 
and rewrite the determinant in the following form:
\begin{equation*}
\dett{\gamma'}{\vf'}=\frac{1}{\tau-l}\dett{\vf-\gamma}{\vf'}.
\end{equation*}
This expression is positive,~because~$\vf'$ is a tangent vector to~$\dfi\Omega$ at~$\vf(\tau)$ with $\vf'_1>0$,
and~$\gamma(\tau)$ belongs to~$\Omega$, see Figure~\ref{fig:DCLemma_monotone}.

\begin{figure}[h!]
\begin{center}
\includegraphics[width = 0.8 \linewidth]{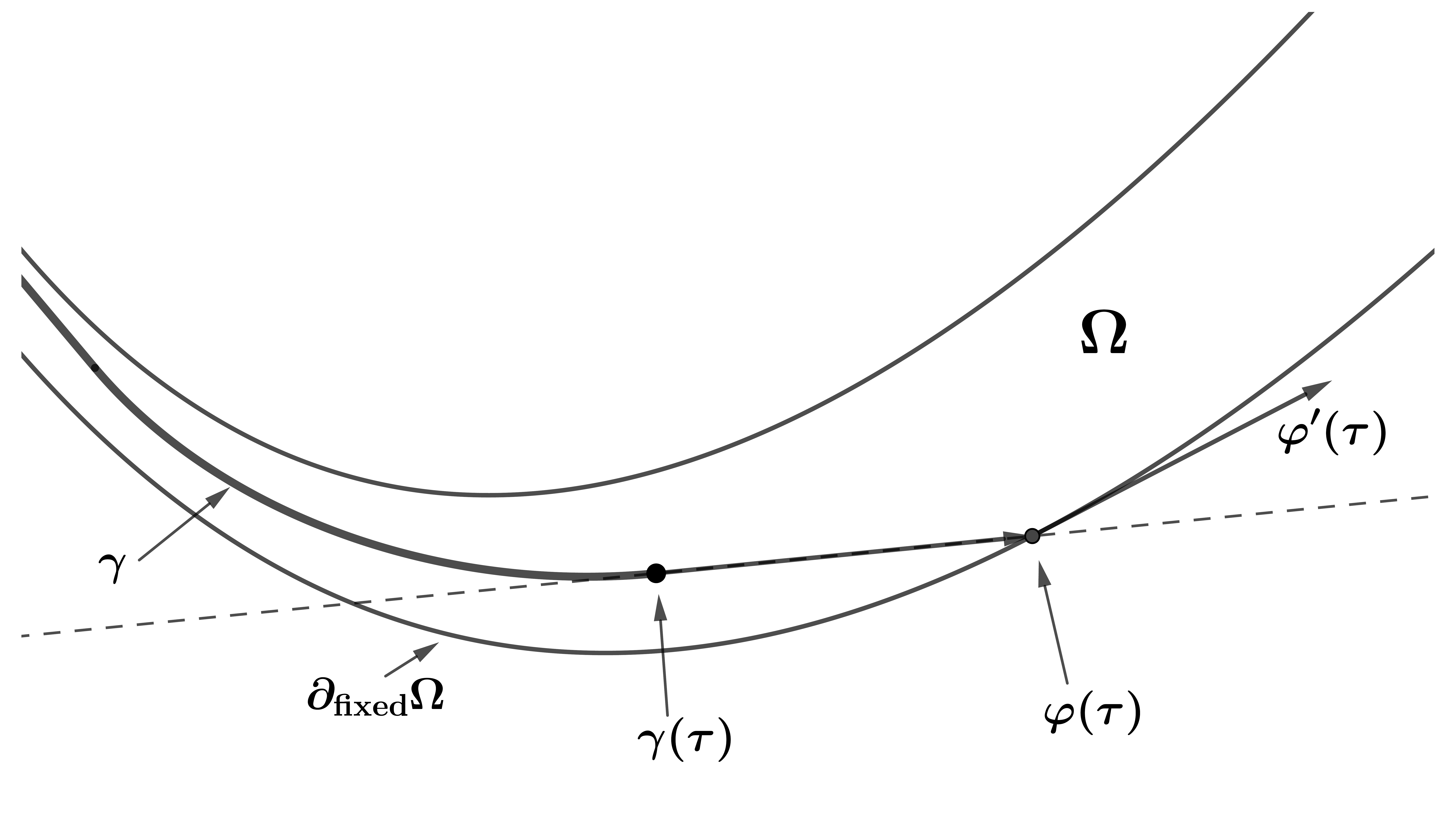}
\caption{Illustration to the proof of Lemma~\ref{Monotone_Convex}.}
\label{fig:DCLemma_monotone}
\end{center}
\end{figure}

We only have to get rid of the smoothness assumption. One can approximate~$\vf$ by smooth increasing functions
$\vf_n\colon (l,r] \to \dfi\Omega$ in such a way that the curves~$\gamma_n$ generated by~$\vf_n$ converge 
to~$\gamma$ pointwise. Each $\gamma_n$ is a convex curve (in the sense that these curves are the graphs of 
convex functions in the standard coordinates) and~$\gamma_n \to \gamma$ pointwise. Therefore~$\gamma$ is 
a convex curve itself.
\end{proof}

\begin{figure}[h!]
\begin{center}
\includegraphics[width = 0.8 \linewidth]{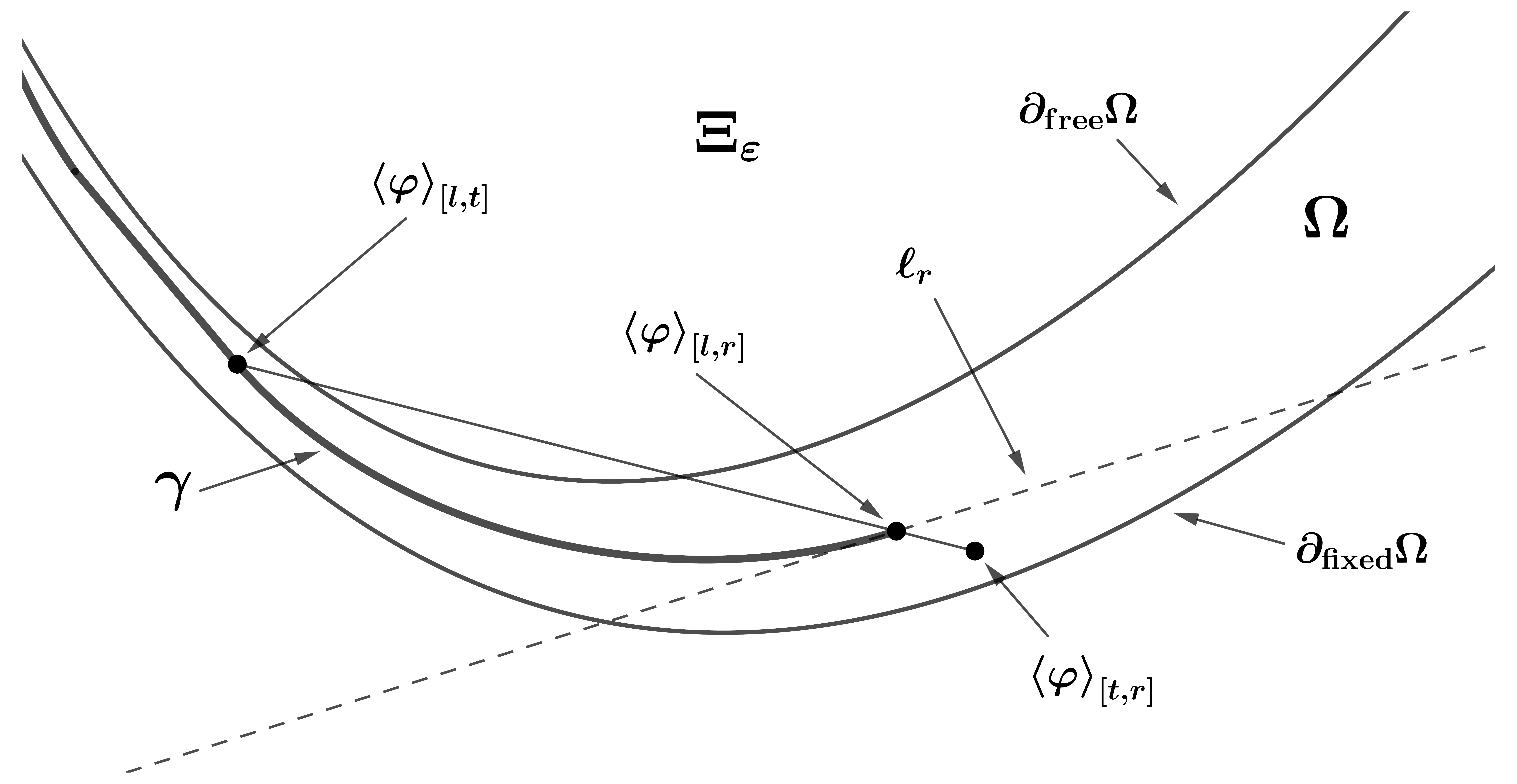}
\caption{Illustration to the proof of Lemma~\ref{DeliveryCurveLemma}.}
\label{fig:DCLemma}
\end{center}
\end{figure}

The following lemma links the condition that~$\gamma$ is convex with the condition~$\vf \in \Class_\Omega$ 
(see Figure~\ref{fig:DCLemma} for visualization of the proof). The symbol $\gamma'(r)$ in the lemma below means the left derivative, which always exists 
due to convexity. 

\begin{Le}\label{DeliveryCurveLemma}
Suppose~$\gamma$ to be a curve parametrized by the interval~$(l,r]$ and given by~\textup{\eqref{CurveGenerator}}. Let it be convex 
in the sense that it is the graph of a convex function in the standard coordinates. Suppose also that the tangent 
line 
\begin{equation*}
\ell_r = \{\gamma(r) + s\gamma'(r)\colon s\in\mathbb{R}\}
\end{equation*}
does not cross the domain $\Xi$ \textup(see Definition~\textup{\ref{cond}}\textup). Then\textup, 
for any~$t \in [l,r)$\textup, we have~$\av{\vf}{[t,r]}\in\Omega$.
\end{Le}

\begin{proof}
Since the curve $\gamma$ is convex, it lies above the line $\ell_r$. The domain~$\Xi$ also lies above the line $\ell_r$. We note that~$\av{\vf}{[l,r]}$ is 
a convex combination of~$\av{\vf}{[l,t]}$ and~$\av{\vf}{[t,r]}$. Thus,~$\av{\vf}{[t,r]}$ is separated 
from~$\Xi$ by~$\ell_r$. On the other hand,~$\av{\vf}{[t,r]}$ surely belongs 
to~$\Xi_0$, so it lies inside~$\Omega$. 
\end{proof}

In the following corollary we write $\gamma'(\tau)$, $\tau \in (l,r]$, meaning any of 
the one-sided derivatives.  

\begin{Cor}\label{DeliveryCurveCorollary}
Suppose~$\gamma$ to be a delivery curve on~$(l,r]$. Let it be convex in the sense that it is the graph of 
a convex function in the standard coordinates. Suppose also that the tangent line 
\begin{equation*}
\ell_\tau = \{\gamma(\tau) + s\gamma'(\tau)\colon s\in\mathbb{R}\}
\end{equation*}
does not cross the domain $\Xi$ for any~$\tau \in (l,r]$. Then\textup, the function~$\vf$ that 
generates~$\gamma$ belongs to~$\Class_\Omega$.
\end{Cor}

Before we pass to constructing specific delivery curves, we should postulate a heuristic principle that 
will help us to guess them. Since a delivery curve ``consists of optimizers'', it has to avoid the directions, 
in which the Bellman candidate is non-linear. Thus, we guess that delivery curves should go either along the 
extremal segments or along the free boundary. 

In Subsection \ref{s52} we will construct delivery curves for each elementary figure of the foliation. 
In the general case a figure will have special points on the free boundary, we call them \emph{incoming} 
and \emph{outgoing nodes}. The idea is as follows: if we have special delivery curves for the incoming nodes, 
then we can construct a delivery curve for any point in the domain. Moreover, if the domain has outgoing nodes, 
then we construct special delivery curves for them. Every outgoing node for some domain is at the same time 
an incoming node for its neighbor domain. Continuing delivery curves along these special nodes allows us 
to construct optimizers for all the points in $\Omega$.

\section{Local behavior of optimizers}\label{s52}

\subsection{Optimizers for tangent domains}\label{s521}\index{domain! tangent domain}
Consider a tangent domain~$\Rt(u_1,u_2)$ foliated by the segments $S(u)=[g(u),w(u)]$, $u \in [u_1,u_2]$, 
tangent to the free boundary of $\Omega$. Let $B$ be a standard candidate on~$\Rt(u_1,u_2)$. This domain 
has two linear parts of the boundary, namely $S(u_1)$ and $S(u_2)$. The tangency points $w(u_1)$ and $w(u_2)$ 
on the free boundary are the \emph{incoming} and \emph{outgoing nodes} of $\Rt(u_1,u_2)$ correspondingly. 
Recall the positive valued function~$\lambda$\index[symbol]{$\lambda$},
\eq{eq011801}{
\lambda(u)\big(g(u)-w(u)\big) = w'(u), \quad u\in [u_1,u_2],
}
introduced in~\eqref{eq150701}.

Suppose~$\psi$ is an optimizer for the incoming node $w(u_1)$ (see Figure~\ref{fig:RtOpt} below) defined 
on the interval~$(l,l_1]$. Our aim is to build the optimizers for all the points inside~$\Rt(u_1,u_2)$. 
We start with the points $w(u),$ $u \in [u_1,u_2]$, lying on the free boundary. We look 
for a function~$\vf$ on~$(l,r]$ for some $r>l_1$ such that~$\vf = \psi$ on~$(l,l_1]$ and its delivery 
curve~$\gamma$ goes along the free boundary from~$w(u_1)$ to~$w(u_2)$ on~$[l_1,r]$. 
We will find a monotone function $u\colon [l_1,r] \to [u_1,u_2]$ such that the function 
\eq{eq031801}{
\vf(t) = g(u(t))}
generates the required delivery curve $\gamma(t) = \av{\vf}{[l,t]}$ that will coincide with $w(u(t)),$ 
$t \in [l_1,r]$. This is equivalent to the equation
$$
(t-l) w(u(t))=\int_l^t \vf(\tau)\, d \tau.
$$
We differentiate this identity with respect to $t$ and obtain
$$
w(u(t))+(t-l)w'(u(t))u'(t) = \vf(t) = g(u(t)).
$$
Using~\eqref{eq011801} we obtain
$
(t-l)\lambda(u(t))u'(t)=1.
$
We solve this differential equation with the boundary condition~$u(l_1)=u_1$ and get
\eq{eq021801}{
\log \frac{t-l}{l_1-l} = \int_{u_1}^{u(t)}\lambda(v)\,dv.
}
Since the function $\lambda$ is positive and $C^1$-smooth (see~Subsection~\ref{Sec091201}), equality~\eqref{eq021801} defines the required 
$C^2$-smooth increasing function $u$ on the interval $[l_1,r]$ (and the function $\vf$ defined 
by~\eqref{eq031801}), where 
$$
r = l+ (l_1-l) \exp\Big(\int_{u_1}^{u_2}\lambda(v) \,dv\Big).
$$
We want to use Lemma~\ref{DeliveryCurveLemma} to verify that the function $\vf$ belongs to $\Class_\Omega$. 
By construction, the curve $\gamma(t)$, $t\in (l_1,r]$, coincides with the part of the free boundary $\dfree\Omega$.
Therefore, this part of the curve is convex and its tangents do not cross the domain~$\Xi$. In order 
to use Lemma~\ref{DeliveryCurveLemma}, we need the convexity of the curve $\gamma$ on the whole 
interval $(l,r).$ This consideration leads to the proposition below.

\begin{St}\label{OptimizersRightTangentDomain}
Let~$B$ be a candidate on~$\Rt(u_1,u_2)$. Suppose that there exists a non-decreasing optimizer~$\psi$ for~$B$ 
at the point~$w(u_1)$. Let also~$\psi \preceq g(u_1)$. Then\textup, there 
exists a non-decreasing optimizer~$\vf_x$ for~$B$ at every point~$x \in \Rt(u_1,u_2)$\textup, 
moreover\textup,~$\vf_x \preceq g(u_2)$.
\end{St}
Recall that the ordering $\preceq$ on $\dfi\Omega$ was introduced on page~\pageref{ordering}.  

\begin{proof}
We have constructed the desired function $\vf_x$ for the points $x \in \Rt(u_1,u_2)\cap \FreeBoundary\Omega$. 
For any point $x \in \Rt(u_1,u_2)$ there exists unique $u \in [u_1,u_2]$, such that $x \in S(u)$. We already 
know the optimizers at the endpoints~$w(u)$ and~$g(u)$ of this extremal segment. Namely, 
they are $\vf_{w(u)}$ on an interval $(l,r]$ and the constant function $\vf_{g(u)} \equiv g(u)$. We will obtain 
the desired optimizer $\vf_{x}$ extending the function $\vf_{w(u)}$ by the constant $g(u)$ to some interval 
$(r,r_1]$ of appropriate length. 

To prove the proposition, we need to show two things. First, we need to verify that the function $\vf_x$ lies in 
$\Class_{\Omega}$. Second, we need to prove the equality $B(x) = \av{\ff(\vf_x)}{}$ for the candidate $B$. 

We start with the verification of convexity of the curve $\gamma=\gamma_x$ generated by $\vf=\vf_x$. By Lemma~\ref{Monotone_Convex}, 
the curve generated by $\psi$, i.\,e, $\gamma|_{(l,l_1]}$, is convex and its tangent line at the point $w(u_1)$ 
coincides with the tangent line to the free boundary (this is a consequence of the condition $\psi \preceq g(u_1)$ and~\eqref{FirstDerivative}). 
Then the curve $\gamma$ goes along the free boundary, which is convex, and then it goes along the tangent line $S(u)$. 
Therefore, this curve $\gamma$ is convex. Let $l<t<\tau\leq r_1$. If $\tau\leq l_1$, then the point 
$\av{\vf}{[t,\tau]}$ lies on the curve generated by $\psi$, therefore, lies in $\Omega$ by the assumption. 
If $\tau> l_1$ then the tangent line to $\gamma$ at $\gamma(\tau)$ is the tangent line to the free boundary by construction. Therefore, this tangent line do not cross the domain~$\Xi$, and we use Lemma~\ref{DeliveryCurveLemma} to conclude that $\av{\vf}{[t,\tau]}$ lies in $\Omega$. This proves that $\vf \in \Class_\Omega$.

It suffices to show that~$B(\gamma(s)) = \av{\ff(\vf)}{[l,s]}$ for~$s \in [l_1,r_1]$, because for $s \in (l,l_1]$ 
this follows from the fact that~$\psi$ is an optimizer. Thus, we need to check the following identity:
\begin{equation}\label{NewtonLeibniz}
(s-l)B(\gamma(s)) - (l_1-l)B(\gamma(l_1)) = \int\limits_{l_1}^s \ff(\vf(\tau))\,d\tau.
\end{equation}
After differentiating with respect to $s$ we obtain the equivalent identity (because \eqref{NewtonLeibniz} 
holds for $s=l_1$):
\eq{eq012801}{
B(\gamma(s))+ (s-l) \scalp{\nabla B(\gamma(s))}{\gamma'(s)} = \ff(\vf(s))=B(\vf(s)).
}
Since $B$ is a Bellman candidate (see Definition~\ref{candidate}), the gradient $\nabla B$ is constant 
on the extremal segment containing $\gamma(s)$ and $\vf(s)$, therefore,
\begin{equation*}
B(\vf(s))-B(\gamma(s)) = \scalp{\nabla B}{\big(\vf(s)-\gamma(s) \big)}
\stackrel{\eqref{FirstDerivative}}{=} (s-l) \scalp{\nabla B}{\gamma'(s)}.
\hfill\qedhere
\end{equation*}
\end{proof}

\begin{figure}[h!]
\begin{center}
\includegraphics[width = 0.8 \linewidth]{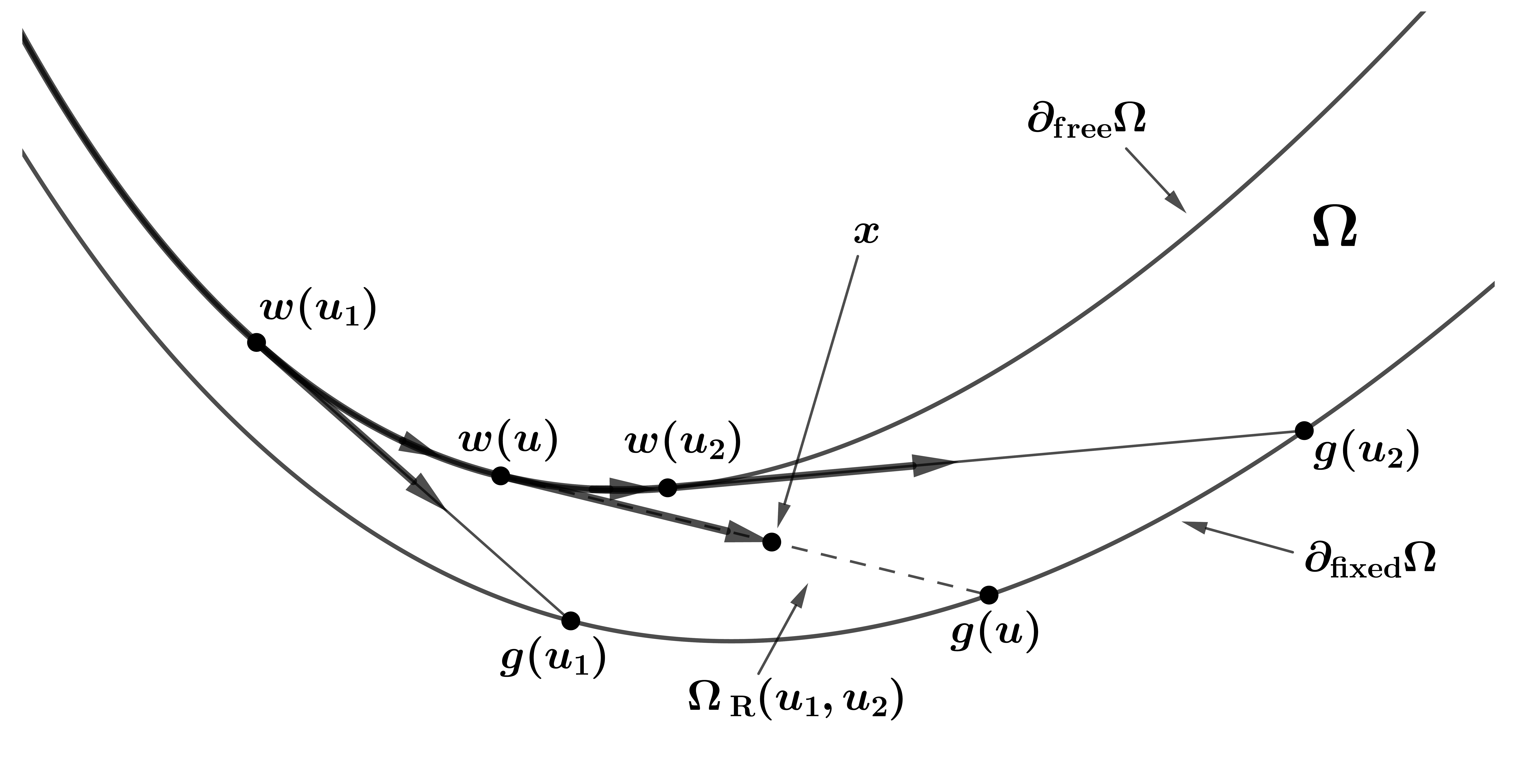}
\caption{Delivery curves inside a tangent domain.}
\label{fig:RtOpt}
\end{center}
\end{figure}

We briefly state a symmetric proposition.
\begin{St}\label{OptimizersLeftTangentDomain}
Let~$B$ be a candidate on~$\Lt(u_1,u_2)$. Suppose that there exists a non-increasing optimizer~$\psi$ for~$B$ 
at the point~$w(u_2)$. Let also~$g(u_2) \preceq \psi$. Then\textup, 
there exists a non-increasing optimizer~$\vf_x$ for~$B$ at every point~$x \in \Lt(u_1,u_2)$\textup, 
moreover\textup,~$g(u_1) \preceq  \vf_x$.
\end{St}

The proof of similar propositions for infinite domains is slightly more complicated and requires additional 
assumptions~\eqref{eq111001} and~\eqref{eq111002}. We will use these assumptions together with the 
following technical lemma.

\begin{Le}\label{Lem111001}
Let $\sigma \in \mathbb{R}$. Let $\vt$ and $Y$ be two piecewise monotone continuous functions on 
$[\sigma, +\infty)$ with a finite number of intervals of monotonicity. Suppose that $\lim_{+\infty} \vt =0$. 
Then\textup, the integration by parts formula is valid\textup:
\eq{eq070201}{
\int_\sigma^{+\infty}\vt dY = \lim_{\nu \to +\infty}\int_\sigma^{\nu}\vt dY = 
- Y(\sigma)\vt(\sigma) - \lim_{\nu \to +\infty}\int_\sigma^{\nu} Y d\vt = 
- Y(\sigma)\vt(\sigma) - \int_\sigma^{+\infty} Y d\vt,
}
where both limits exist \textup(finite or infinite\textup).
 
Symmetrically\textup, for piecewise monotone continuous functions $\vt$ and $Y$  on $(-\infty,\sigma]$ with a finite number of intervals of monotonicity\textup, if $\lim_{-\infty} \vt =0,$ then
\eq{eq170201}{
\int_{-\infty}^\sigma \vt dY = \lim_{\nu\to -\infty}\int_{\nu}^\sigma \vt dY = 
Y(\sigma)\vt(\sigma) - \lim_{\nu\to -\infty}\int_{\nu}^\sigma Y d\vt = 
Y(\sigma)\vt(\sigma) - \int_{-\infty}^\sigma Y d\vt,
}
where both limits exist \textup(finite or infinite\textup).
\end{Le}

\begin{proof}
We give the proof of~\eqref{eq070201}. The proof of~\eqref{eq170201} is similar. 

Since the integration by parts formula is valid on any finite interval, without loss of generality, we may 
assume that both $Y$ and $\vt$ are monotone and do not change the sign on $[\sigma,+\infty)$. Then, both 
integrals in~\eqref{eq070201} are monotone with respect to $\nu$, therefore both limits exist, finite 
or infinite.

First, assume that $\int_\sigma^{+\infty}\vt dY$ converges. 
Take any $\nu >\sigma$ and write the integration by parts formula on the interval $[\sigma,\nu]$:
\eq{eq300901}{
\int_{\sigma}^{\nu} \! \vt \ dY + Y(\sigma)\vt(\sigma) = Y(\nu)\vt(\nu) - \int_\sigma^{\nu} \! Y d\vt.
}
The left hand side of~\eqref{eq300901} has a finite limit when $\nu\to +\infty$. The first summand on 
the right hand side is of the same sign as the second one (because $\sign(\vt') = - \sign(\vt)$), therefore, 
both summands have finite limits. If the limit of the first summand were nonzero, the limit of the second 
one would be infinite. Indeed, in this case 
\eq{eq070202}{
|Y(\tau)|>\frac{\mathrm{const}}{|\vt(\tau)|}, \quad \tau \to +\infty,
}
therefore, 
\eq{eq070203}{
\Big|\int^{+\infty}Y(\tau)d\vt(\tau)\Big|>
\mathrm{const}\cdot\Big|\int^{+\infty}\frac{d\vt(\tau)}{\vt(\tau)}\Big|=+\infty,
}
because $\vt(+\infty)=0$. Thus, the first summand on the right hand side of~\eqref{eq300901} vanishes at $+\infty$. 
Formula~\eqref{eq070201} is proved.

Now, let us assume that $\int_\sigma^{+\infty} Y d\vt$ converges. The integral on the left hand side 
of~\eqref{eq300901} is monotone with respect to $\nu$, thus, it has a finite or infinite limit when 
$\nu$ tends to $+\infty$. Therefore, the first term on the right hand side of~\eqref{eq300901} also 
has a finite or infinite limit. If this limit is non-zero, then we have~\eqref{eq070202} and~\eqref{eq070203}, 
which contradicts our assumption. Thus, the first summand on the right hand side of~\eqref{eq300901} vanishes, 
the left hand side of~\eqref{eq300901} has a finite limit when $\nu$ tends to $+\infty$, and~\eqref{eq070201} 
is proved.

Finally, if both limits in~\eqref{eq070201} are infinite, we need to show that the infinities on both sides 
have the same sign. We may assume that $\vt$ is positive and decreasing. 
We may also assume that $Y$ is monotone and non-negative at infinity. If $Y$ were bounded at infinity, 
then the limit on the left hand side of~\eqref{eq070201} would be finite. Therefore, $Y$ increases at 
infinity, and both sides of~\eqref{eq070201} are equal $+\infty$.
\end{proof}

\begin{Le}\label{Lem190701}
Suppose that the condition~\eqref{eq111001} is fulfilled. Then\textup, the condition~\eqref{eq141001a} 
is equivalent to the fact that $\beta_2$ given by~\eqref{eq120701}\textup{,} i.\,e.\textup{,}
\eq{eq180202}{
\beta_2(v)=\tors(v)-\!\!\int_{-\infty}^v \exp\Big(\!-\!\int_\tau^v\frac{\kappa_2'}{\kappa_2-\kappa}\;\Big)
\;\tors'(\tau)\;d\tau,
}
satisfies the inequality $\beta_2(v)<+\infty$ for any $v \in \mathbb{R}$.  Moreover\textup, the following 
equality holds\textup:
\eq{eq140201}{
\int_{-\infty}^{v}\! f'(\tau)\; \exp\Big(\!-\!\int_{\tau}^{v}\lambda\;\Big)\, d\tau = 
-\big(w_1(v)-g_1(v)\big)\Big[\kappa_3(v)+\big(\kappa(v)-\kappa_2(v)\big)\beta_2(v)\Big].
}
Here $\lambda = \lamr,$ $\kappa = \kappar,$ and $w = \wr$.
\end{Le}

\begin{proof}
Let us modify the expression in the exponent on the right hand side of~\eqref{eq180202}. Using~\eqref{kappa'}, 
\eqref{eq150701}, and \eqref{eqKappas}, we obtain
\eq{eq111005}{
\int \frac{\kappa_2'}{\kappa_2-\kappa} = \log(\kappa_2-\kappa)+\log(g_1-w_1)+\int \lambda\;.
}
First, we {\bf claim} that 
\eq{eq111008}{
\int_{-\infty}^v \frac{\kappa_2'}{\kappa_2-\kappa} = +\infty,
} 
which is equivalent to
$$
\lim_{\tau\to-\infty}\big(\kappa_2(\tau)-\kappa(\tau)\big)\big(g_1(\tau)-w_1(\tau)\big)
\exp\Big(\!-\!\int_\tau^v\lambda\;\Big)=0
$$
by~\eqref{eq111005}.
Recall that $(g_1-w_1) \kappa = (g_2-w_2)$, and therefore
\eq{eq111007}{
\lim_{\tau \to -\infty} \kappa(\tau)\big(g_1(\tau)-w_1(\tau)\big)\exp\Big(\!-\!\int_\tau^v \lambda\;\Big)=
\lim_{\tau \to -\infty} \big(g_2(\tau)-w_2(\tau)\big)\exp\Big(\!-\!\int_\tau^v \lambda\;\Big)=0
}
due to condition~\eqref{eq111001}. The function $\kappa_2 = \frac{g_2'}{g_1'}$ is increasing by the convexity of the curve $g$. There are two cases. If $\kappa_2$ is bounded on $-\infty$, then 
\eq{eq111006}{
\lim_{\tau \to -\infty} \kappa_2(\tau)\big(g_1(\tau)-w_1(\tau)\big)\exp\Big(\!-\!\int_\tau^v \lambda\;\Big)=0
}
due to the same condition~\eqref{eq111001}. If $\kappa_2$ is not bounded on $-\infty$, then $\kappa_2(\tau) \to -\infty$ as $\tau \to -\infty$, 
and then we use the fact that $0>\kappa_2>\kappa$ provided~$\tau$ is sufficiently close to~$-\infty$. 
Therefore $|\kappa_2|<|\kappa|$ and~\eqref{eq111006} follows from~\eqref{eq111007}. Thus the {\bf claimed}  divergence 
in~\eqref{eq111008} is proved. 

We wish to use Lemma~\ref{Lem111001} with 
$$
\vt(\tau)=\exp\Big(\!-\!\int_\tau^v\frac{\kappa_2'}{\kappa_2-\kappa}\;\Big)\qquad\text{and}\qquad Y(\tau)=\tors(\tau)
$$
to integrate by parts and rewrite formula~\eqref{eq180202} for $\beta_2$. 
We verify the hypotheses of Lemma~\ref{Lem111001}. It is clear that $\kappa_2'>0$, $\kappa_2>\kappa$, 
therefore the function $\vt$ is monotone and tends to $0$ at $-\infty$ due to~\eqref{eq111008}. 
The function~$Y$ is piecewise monotone according to Condition~\ref{reg}.
Thus, by Lemma~\ref{Lem111001} we get: 
\eq{eq070603}{
\beta_2(v) = 
\!\int_{-\infty}^v\! \frac{\kappa_2'(\tau)}{\kappa_2(\tau)-\kappa(\tau)}\;
\exp\Big(\!-\!\int_{\tau}^v \frac{\kappa_2'}{\kappa_2-\kappa}\;\Big)\tors(\tau)\ d\tau=
\!\int_{-\infty}^v\! \frac{\kappa_3'(\tau)}{\kappa_2(\tau)-\kappa(\tau)}\;
\exp\Big(\!-\!\int_{\tau}^v \frac{\kappa_2'}{\kappa_2-\kappa}\;\Big)\ d\tau. 
}
Now we plug~\eqref{eq111005} into \eqref{eq070603} and obtain
\eq{eq091201}{
\beta_2(v) = \frac{1}{\kappa_2(v)-\kappa(v)}\cdot \frac{1}{g_1(v)-w_1(v)} \cdot 
\int_{\!-\!\infty}^v\!\kappa_3'(\tau)\big(g_1(\tau)-w_1(\tau)\big)\;\exp\Big(\!-\!\int_{\tau}^v \lambda\;\Big)\ d\tau. 
}
The next step is to integrate by parts again using the same Lemma~\ref{Lem111001} with 
$$
\vt(\tau) = \big(g_1(\tau) - w_1(\tau)\big) \exp\Big(\!-\!\int_{\tau}^v \lambda\;\Big)\qquad 
\text{and} \qquad Y(\tau) = \kappa_3(\tau).
$$  
We verify hypotheses of Lemma~\ref{Lem111001}. First,  $\vt \to 0$ when $\tau\to-\infty$ due to~\eqref{eq111001}. 
Second, we calculate $\vt'$ using the relation $\big(g_1 - w_1\big)\lambda  = w_1'$:
$$
\vt'(\tau) = \Big(\lambda(\tau) \big(g_1(\tau) - w_1(\tau)\big) +g_1'(\tau)-w_1'(\tau)\Big)\; 
\exp\Big(\!-\!\int_{\tau}^v \lambda\;\Big) \stackrel{\eqref{eq150701}}{=} 
g_1'(\tau) \exp\Big(\!-\int_{\tau}^v \lambda\Big) > 0.
$$
Thus $\vt$ is monotone. Finally, $Y$ is piecewise monotone by Remark~\ref{Rem100301}. Thus, we may apply 
Lemma~\ref{Lem111001} and integrate by parts in~\eqref{eq091201}:
$$
\beta_2(v) = \frac{\kappa_3(v)}{\kappa_2(v)-\kappa(v)} - \frac{1}{\kappa_2(v)-\kappa(v)}\cdot \frac{1}{g_1(v)-w_1(v)} \int_{-\infty}^v \kappa_3(\tau) g_1'(\tau) \exp\Big(-\int_{\tau}^v \lambda\Big)\ d\tau. 
$$
Since $\kappa_3 g_1' = f'$, we get:
\eq{}{
\beta_2(v) = \frac{\kappa_3(v)}{\kappa_2(v)-\kappa(v)} - \frac{1}{\kappa_2(v)-\kappa(v)}\cdot 
\frac{1}{g_1(v)-w_1(v)} \cdot \int_{\!-\!\infty}^v f'(\tau) \exp\Big(-\int_{\tau}^v \lambda\;\Big)\ d\tau. 
}
This completes the proof of~\eqref{eq140201}.

Condition~\eqref{eq141001a} states that the left hand side of~\eqref{eq140201} is greater than $-\infty$. 
According to~\eqref{eq140201} this is equivalent to $\beta_2<+\infty$.
\end{proof}

Let us turn back to consideration of optimizers.

\begin{St}\label{OptimizerRightTangentsInfty}
Let~$B$ be the standard candidate on~$\Rt(-\infty,u_2)$. There exists a non-decreasing optimizer~$\vf_x$ 
for~$B$ at every point~$x \in \Rt(-\infty,u_2)$\textup, moreover\textup,~$\vf_x \preceq g(u_2)$.
\end{St}

\begin{proof}
By Definition~\ref{211101} of a standard candidate in the domain~$\Rt(-\infty,u_2)$, the function $\beta_2$ 
given by~\eqref{eq120701} is finite. Thus, the integral in~\eqref{eq180202} converges.

We begin with the points on the free boundary. When this is done, we will automatically get the desired 
optimizer $\vf_x$ for any $x \in \Rt(-\infty,u_2)$. For this aim we pick any finite $u_1$ such 
that~$x \in \Rt(u_1,u_2)$ and use Proposition~\ref{OptimizersRightTangentDomain}. Moreover, it suffices 
to construct an optimizer for $x=w(u_2)$ only.

Similar to the case of a bounded domain $\Rt$ we would like to construct a function 
$\vf\colon (l,r] \to \FixedBoundary\Omega$ on some interval $(l,r]$ such that the curve 
$\gamma(t) = \av{\vf}{(l,t)}$ generated by $\vf$ goes along the free boundary from the infinity 
to $w(u_2)$. As before, we will find a function in the form $\vf(t) = g(u(t))$, where 
$u \colon (l,r] \to (-\infty, u_2]$ is a monotone function. The previous reasoning (see~\eqref{eq021801}) leads us 
to the following relation:
\eq{eq020402}{
\int_{u(t_1)}^{u(t_2)} \!\!\lambda(v) \, dv = \log \frac{t_2-l}{t_1-l}\,, \qquad l<t_1<t_2\leq r.
}
We want to have $u(r)=u_2$. We substitute $t_2=r$ into \eqref{eq020402} and obtain
\eq{eq030402}{
\int_{u(t_1)}^{u_2} \!\!\lambda(v) \, dv = \log \frac{r-l}{t_1-l}\,, \qquad l<t_1< r.
}
According to the condition in~\eqref{eq111001} we have 
\eq{eq111003}{
\int_{-\infty}^{u_2}\!\!\lambda(v) \, dv=+\infty\,,}
 therefore, \eqref{eq030402} defines the function $u$ on $(l,r)$.

We may choose the length of the segment $(l,r)$ to be equal $1$, namely, we may take $l=0$, $r=1$. 
We get the explicit formula for the inverse function $t = t(u)$: 
\eq{eq070602}{
t = \exp \Big(- \int_{u}^{u_2}\lambda(v) \, dv\Big).
}

Corollary~\ref{DeliveryCurveCorollary} guarantees $\vf \in \Class_\Omega$ provided we know that 
$$
\frac{1}{t}\int_0^t \vf  = w(u(t)), \qquad t \in (0,1].
$$

Therefore, for any $t \in (0,1]$ we need to check that 
\eq{eq110902}{
\int_0^t \vf  = t w(u(t)).
}  
Moreover, in order to prove that $\vf$ is an optimiser for $B$ we also need to verify that
\eq{eq070601}{
B(w(u(t))) = \av{\ff(\vf)}{[0,t]} = \frac{1}{t}\int_0^t f(u(s))\ ds.
} 

We start with~\eqref{eq110902}. 
First, we prove that the limit of the right hand side of~\eqref{eq110902} is zero as $t\to 0$. It is equivalent to 
\eq{eq111004}{
\lim_{u \to -\infty} w_i(u) \exp \Big(- \int_{u}^{u_2}\lambda(v) \, dv\Big)=0, \qquad i=1,2.
}
Fix $i=1,2$. If the function $w_i$ is bounded on $-\infty$, then~\eqref{eq111004} follows from~\eqref{eq111003}. 
If $w_i$ is not bounded on $-\infty$, then we apply L'H\^opital's rule:
$$
\lim_{u \to -\infty} \frac{w_i(u)}{\exp \big(\int_{u}^{u_2}\lambda \big) }  =
- \lim_{u \to -\infty} \frac{w_i'(u)}{\lambda(u)\exp \big( \int_{u}^{u_2}\lambda\big) } \stackrel{\eqref{eq011801}}{=}
\lim_{u \to -\infty} \big(w_i(u)-g_i(u)\big) \exp\Big(\!-\!\int_{u}^{u_2}\lambda\;\Big) = 0  
$$
due to condition~\eqref{eq111001}. 

We calculate the integral on the left hand side of~\eqref{eq110902}:
\eq{eq110903}{
\int_{0}^t \vf(\tau) d \tau = \int_0^t g(u(\tau)) d \tau
\stackrel{\eqref{eq070602}}{=}
\int_{-\infty}^{u(t)}\lambda(v)g(v) \exp\Big(\!-\!\int_{v}^{u_2}\lambda\;\Big)\;dv.
}
Note that 
$$
\int \lambda(v)g(v) \exp\Big(\!-\!\int_{v}^{u_2}\lambda\;\Big)\;dv \stackrel{\eqref{eq011801}}{=}
\int \big(w'(v)+\lambda(v)w(v)\big) \exp\Big(\!-\!\int_{v}^{u_2}\lambda\;\Big)\;dv = 
w(v)\exp \Big(\!-\!\int_{v}^{u_2}\lambda\;\Big),
$$
which converges to zero as $v \to -\infty$, see~\eqref{eq111004}. This proves that the integral 
in~\eqref{eq110903} converges and is equal to 
$$
w(u(t)) \exp \Big(\!-\!\int_{u(t)}^{u_2}\lambda\;\Big) = t w(u(t)).
$$
Relation~\eqref{eq110902} is now proved.

We turn to the proof of~\eqref{eq070601}. The right hand side of \eqref{eq070601} can be rewritten 
using~\eqref{eq070602} in the following way
\begin{align}\label{eq070604}
\frac{1}{t}\int_0^t f(u(s))\ ds &\stackrel{\phantom{\scriptstyle{\tau = u(s)}}}{=}  
\exp \Big(\int_{u(t)}^{u_2}\lambda\;\Big) \cdot \int_0^t f(u(s))\ 
d\left[\exp\Big(\!-\!\int_{u(s)}^{u_2}\lambda\;\Big)\right]=\notag
\\
&\stackrel{\scriptstyle{\tau = u(s)}}{=}\int_{-\infty}^{u(t)} f(\tau) \lambda(\tau) 
\exp\Big(\!-\!\int_\tau^{u(t)}\lambda\;\Big)  d\tau=
\\
&\stackrel{\phantom{\scriptstyle{\tau = u(s)}}}{=} 
f(u(t)) - \int_{-\infty}^{u(t)} f'(\tau) \exp\Big(\!-\!\int_\tau^{u(t)}\lambda\;\Big)\;d\tau,\notag
\end{align}
the integration by parts is guaranteed by~Lemma~\ref{Lem111001} with 
$\vt(\tau) = \exp\Big(-\int_{\tau}^{u(t)} \lambda\Big)$ and $Y = f$.
Using~\eqref{eq140201}, we obtain~\eqref{eq070601}, see representation~\eqref{fenceB2} for the relation 
between~$B$ and $\beta_2$. \qedhere
\end{proof}

\begin{St}\label{st140701}
If~\eqref{eq111001} holds but~\eqref{eq141001a} fails\textup, then $\boldsymbol{B}(x;\,f) = +\infty$ 
for any $x \in \Omega\setminus \dfi\Omega$. 
\end{St}
\begin{proof}
It suffices to prove $\boldsymbol{B}(w(v);\,f) = +\infty$ for any $v\in \mathbb{R}$. Condition~\eqref{eq111001} 
guaranties that the function $\vf$ constructed in the proof of Proposition~\ref{OptimizerRightTangentsInfty} is a test 
function for the point $w(v)$. The failure of~\eqref{eq141001a} means 
$$
\int_{-\infty}^v f'(\tau) \exp\Big(\!-\!\int_{\tau}^v \lamr\Big)\ d\tau = - \infty,
$$
whence by Lemma~\ref{Lem111001}
$$
\av{\ff(\vf)}{[0,1]} \buildrel{\eqref{eq070604}}\over=
\int_{-\infty}^v f(\tau) \lamr(\tau)\exp\Big(\!-\!\int_{\tau}^v \lamr\Big)\ d\tau = + \infty.
$$
\end{proof}

We state similar propositions for the case of the left tangents domain on $+\infty$.  
\begin{St}\label{OptimizerLeftTangentsInfty}
Let~$B$ be the standard candidate on~$\Lt(u_1,+\infty)$. There exists a non-increasing optimizer~$\vf_x$ 
for~$B$ at every point~$x \in \Lt(u_1,+\infty)$\textup, moreover\textup,~$g(u_1) \preceq \vf_x$.
\end{St}

\begin{St}\label{st140702}
If~\eqref{eq111002} holds but~\eqref{eq211001a} fails\textup, then $\boldsymbol{B}(x;\,f) = +\infty$ for 
any $x \in \Omega\setminus \dfi\Omega$. 
\end{St}

\subsection{Optimizers for all other figures}\label{s522}
It is very easy to construct the optimizers for chordal domains. Indeed, the proposition below is 
a straightforward consequence of formula~\eqref{eq250901}.

\begin{St}\label{OptimizersChordalDomain}\index{domain! chordal domain}
Let~$B$ be the standard candidate on a chordal domain~$\Ch([a_0,b_0],[a_1,b_1]).$ Let $g(a)$ and $g(b)$ be the endpoints of a chord from this chordal domain\textup, $a \in [a_0,a_1],$ $b \in [b_1,b_0]$. Let 
$$
x = \alpha g(a)+(1-\alpha)g(b),
$$ 
where $\alpha \in [0,1]$. Then\textup, the optimizer~$\vf_x\colon [0,1] \to \FixedBoundary \Omega$ is given by the formula
\begin{equation*}
\vf_x = g(a)\chi_{[0,\alpha]} + g(b)\chi_{(\alpha,1]}.
\end{equation*}
\end{St}

\begin{proof}
It is obvious that~$\av{\vf_x}{} = x$, and the linearity of $B$ on the chord 
implies~$B(x) = \av{\ff(\vf_x)}{[0,1]}$. It is clear that all the averages $\av{\vf_x}{J}$ 
lie on $[g(a),g(b)] \subset \Omega$  for $J\subset [0,1]$, therefore,~$\vf_x \in \Class_\Omega$.
\end{proof}

\begin{Rem}
The optimizer~$\vf_x$ we suggest for a chordal domain is non-decreasing. One can construct a non-increasing 
optimizer simply taking the function~$\vf_x = g(b)\chi_{[0,1-\alpha]} + g(a)\chi_{(1-\alpha,1]}$.
\end{Rem}

\begin{Def}
The construction above works also for the special case of a long chord. In this situation the tangency point 
with the free boundary is called the \emph{outgoing node} of this chord.
\end{Def}

Let us now pass to the case of multicups.
\begin{St}\label{OptimizersClosedMulticup}\index{multicup! closed multicup}
Let~$B$ be the standard candidate on~$\ClMTC(\{\mathfrak{a}_i\}_{i=1}^k)$. Then\textup, there exists a 
monotone  optimizer~$\vf_x$ for~$B$ at any point~$x \in \ClMTC(\{\mathfrak{a}_i\}_{i=1}^k)$\textup, 
moreover\textup,~$g(\mathfrak{a}_1^{\mathrm l}) \preceq \vf_x \preceq g(\mathfrak{a}_k^{\mathrm r})$ pointwise. 
\end{St}

\begin{proof}
Fix~$x\in \ClMTC(\{\mathfrak{a}_i\}_{i=1}^k)$ and represent it as a convex combination of three points~$g(a_1)$, $g(a_2)$, $g(a_3)$ for some 
$a_1,a_2,a_3 \in \cup_{i=1}^k\mathfrak{a}_i$:
\begin{equation*}
x = \alpha_1 g(a_1) + \alpha_2 g(a_2) + \alpha_3 g(a_3);\, \quad \alpha_1+\alpha_2+\alpha_3=1, \quad \alpha_j \geq 0.
\end{equation*}
Without loss of generality, we may assume that~$a_1 \leq a_2 \leq a_3$. We put
\begin{equation*}
\vf_x(\tau)= g(a_1)\chi_{[0,\alpha_1]} + g(a_2)\chi_{(\alpha_1,\alpha_1+\alpha_2]}+g(a_3)\chi_{(\alpha_1+\alpha_2,1]}.
\end{equation*}
The equality~$\av{\vf_x}{} = x$ is evident, and the relation~$B(x) = \av{\ff(\vf_x)}{[0,1]}$ follows from 
the linearity of $B$ on the closed multicup. The averages $\av{\vf_x}{J}$ lie in the multicup for all 
$J\subset [0,1]$, thus,~$\vf_x \in \Class_\Omega$.
\end{proof}

\begin{St}\label{OptimizersMulticup}\index{multicup}
Let~$B$ be the standard candidate on~$\MTC(\{\mathfrak{a}_i\}_{i=1}^k)$. Then\textup, there exists a monotone 
optimizer~$\vf_x$ for~$B$ at any point~$x \in \MTC(\{\mathfrak{a}_i\}_{i=1}^k)$\textup, moreover\textup,
$g(\mathfrak{a}_1^{\mathrm l}) \preceq \vf_x \preceq g(\mathfrak{a}_k^{\mathrm r})$ pointwise. 
\end{St}

\begin{proof}
If $x \in \FixedBoundary\Omega$, then $\vf_x=x\chi_{[0,1]}$ is the optimizer. So, in what follows we 
assume $x \notin\FixedBoundary\Omega$. Consider the open convex set~$\Omega'$ that is the interior of the convex 
hull of~$\Xi$ and the points~$g(\mathfrak{a}_1^{\mathrm{l}})$ and~$g(\mathfrak{a}_k^{\mathrm{r}})$. 
Since~$x \notin \Omega'$, by the separation theorem, there exists a line~$\varkappa = \varkappa(x)$ that passes 
through~$x$ and does not intersect~$\Omega'$. Let $y$ and $z$ be the leftmost and the rightmost points of the
intersection of the line $\varkappa(x)$ with the boundary of~$\MTC(\{\mathfrak{a}_i\}_{i=1}^k)$. There may 
occur one of the two variants for each crossing point ($y$ and $z$): either it lies on an 
arc~$g(\mathfrak{a}_i)$,~$i \in \{1,2,\ldots,k\}$, or on a 
segment~$[g(\mathfrak{a}_i^{\mathrm{r}}),g(\mathfrak{a}_{i+1}^{\mathrm{l}})]$,~$i \in \{1,2,\ldots,k-1\}$. 

\begin{figure}[h!]
\begin{center}
\includegraphics[width = 0.8 \linewidth]{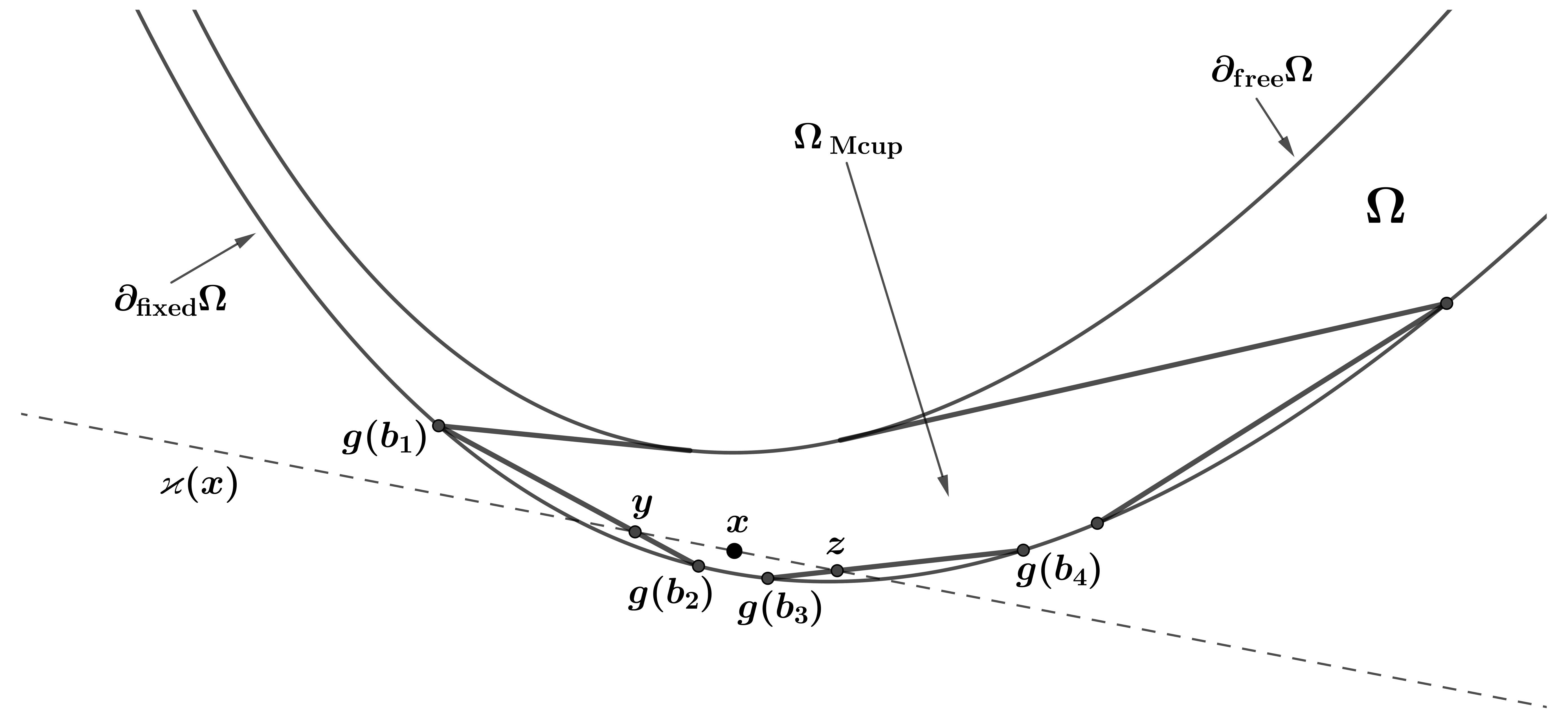}
\caption{Construction of an optimizer in a multicup.}
\label{fig:DCMC}
\end{center}
\end{figure}

If~$y$ lies on the chord~$[g(\mathfrak{a}_i^{\mathrm{r}}),g(\mathfrak{a}_{i+1}^{\mathrm{l}})]$, then we can write
\begin{equation*}
y= \alpha_y g(\mathfrak{a}_i^{\mathrm{r}}) + (1-\alpha_y) g(\mathfrak{a}_{i+1}^{\mathrm{l}});
\qquad \text{ for some } \quad \alpha_y \in [0,1].
\end{equation*}
Similarly, if~$z$ lies on the chord~$[g(\mathfrak{a}_j^{\mathrm{r}}),g(\mathfrak{a}_{j+1}^{\mathrm{l}})]$, then
\begin{equation*}
z= (1-\alpha_z) g(\mathfrak{a}_j^{\mathrm{r}}) + \alpha_z g(\mathfrak{a}_{j+1}^{\mathrm{l}});
\qquad \text{ for some } \quad \alpha_z \in [0,1].
\end{equation*}
So, in any case, $y= \alpha_y g(b_1) + (1-\alpha_y) g(b_2)$, $z = (1-\alpha_z) g(b_3) + \alpha_z g(b_4)$ 
(see Figure~\ref{fig:DCMC}), where~$b_j \in \cup_{i=1}^k\mathfrak{a}_i$ (if~$y$ is an 
intersection of~$\varkappa$ with an arc, then we may take~$b_1=b_2$; similarly with~$z$) such that
\begin{equation*}
g_1(b_1) \leq y_1 \leq g_1(b_2) \leq g_1(b_3) \leq z_1 \leq g_1(b_4).
\end{equation*} 
It is convenient to define the optimizer $\vf_x$ on the segment~$[y_1,z_1]$:
\begin{equation}\label{OptimizerFormulaMulticup}
\vf_x(t) = 
\begin{cases}
g(b_1),\quad &t \in J_1 = [y_1, y_1+ \alpha_y(z_1 - x_1)); \\ 
g(b_2),\quad &t \in  J_2 = [y_1+ \alpha_y(z_1 - x_1)), z_1-x_1+y_1); \\ 
g(b_3),\quad &t \in J_3 = [z_1-x_1+y_1, z_1-\alpha_z(x_1 - y_1));\\ 
g(b_4),\quad &t \in J_4 = [z_1-\alpha_z(x_1 - y_1), z_1]. 
\end{cases}
\end{equation}
As usual, the equalities~$\av{\vf_x}{} = x$ and~$\av{\ff(\vf_x)}{[y_1,z_1]} = B(x)$ are evident. However, 
if we draw the delivery curve for~$\vf_x$, we can see that in some cases it does not fall under the scope 
of Lemma~\ref{DeliveryCurveLemma} (the tangent may cross the free boundary), so, we prove directly that 
$\vf_x$ lies in $\Class_\Omega$.

We claim that a point~$\av{\vf_x}{J}$, where~$J \subset [y_1,z_1]$, either belongs to one of the 
segments~$[g(b_1),g(b_2)]$ and~$[g(b_3),g(b_4)]$, or is separated from the free boundary by~$\varkappa$. 
Once the claim is proved, we obtain~$\vf_x \in \Class_\Omega$. 

We will consider different cases of the position of~$J$ inside~$[y_1,z_1]$. If $J$ intersects not more than two 
of the intervals $J_1,J_2, J_3, J_4$ in~\eqref{OptimizerFormulaMulticup}, then the claim is obvious. If~$J$ 
intersects all four intervals, then we may represent~$x$ as a linear combination of~$\av{\vf_x}{J}$,~$g(b_1)$, 
and~$g(b_4)$. Since~$g(b_1)$ and~$g(b_4)$ lie above~$\varkappa$ (i.\,e., in the same half-plane with the 
free boundary), the point $\av{\vf_x}{J}$ lies below~$\varkappa$.

So, we may suppose that~$J$ intersects three intervals from~\eqref{OptimizerFormulaMulticup}. 
Without loss of generality, we may assume that~$J \cap J_4  = \varnothing$. 
Then,~$\av{\vf_x}{J}$ is a  convex combination of~$g(b_3)$ and a point from~$[y,g(b_2)]$ (since~$J_2\subset J$). 
Both these points lie below $\varkappa$. Therefore,~$\av{\vf_x}{J}$ is separated 
from the free boundary by~$\varkappa$.
\end{proof}

\begin{Rem}
The optimizer~$\vf_x$ we suggest for a multicup domain is non-decreasing. One can construct a non-increasing 
optimizer in a similar way. 
\end{Rem}

\begin{Def}
The endpoints of the arc of the free boundary that is the part of the boundary of a multicup are called 
the \emph{outgoing nodes} of the multicup. 
\end{Def}

\begin{St}\label{OptimizersAngle}\index{angle}
Let~$\Ang(u)$ be an angle with the vertex $g(u)$ and the boundary tangent lines $\Sl$ and $\Sr$ with the 
tangency points $\wl$ and $\wr$ correspondingly. Let $B$ be the standard candidate on~$\Ang(u)$. Let~$\psir$ 
be a non-decreasing optimizer for the point~$\wr$ and let~$\psil$ be a non-increasing optimizer for the 
point~$\wl$ such that~$\psir \preceq g(u)$ and~$g(u) \preceq \psil$ pointwise. Then\textup, there exists 
an optimizer for every point~$x \in \Ang(u)$.
\end{St}

\begin{proof}
The proof of this proposition is very similar to the proof of Proposition~\ref{OptimizersMulticup}. 
First, there exist numbers~$\alpha_1,\alpha_2,\alpha_3$ such that
\begin{equation*}
x = \alpha_1 \wr + \alpha_2 g(u) + \alpha_3\wl;\quad \alpha_1 + \alpha_2+\alpha_3 = 1,\quad \alpha_j \geq 0.
\end{equation*}
Second, by Remark~\ref{Rescaling}, we may model~$\psir$ and~$\psil$ on any interval. Suppose that~$\psir$ 
is defined on~$[0,\alpha_1]$, and~$\psil$ is defined on~$[0,\alpha_3]$. Define the optimizer~$\vf_x$ 
on~$[0,1]$ by the formula
\begin{equation*}
\vf_x(\tau) = 
\begin{cases}
\psir(\tau),\quad & \tau \in [0,\alpha_1);\\
g(u),\quad & \tau \in [\alpha_1,\alpha_1+\alpha_2);\\
\psil(1 - \tau),\quad & \tau \in [\alpha_1+\alpha_2,1].
\end{cases}
\end{equation*}
As usual, the equalities~$\av{\vf_x}{} = x$ and~$\av{\ff(\vf_x)}{[0,1]} = B(x)$ are evident. We have to verify 
that~$\vf_x \in \Class_\Omega$. Let~$\varkappa = \varkappa(x)$ be a line passing through~$x$ that does not 
intersect~$\Xi$. First, we prove that~$\gamma_{\psir}$ lies above~$\varkappa$ and~$\gamma_{\psil}$ lies 
above~$\varkappa$ as well. It suffices to prove the claim concerning~$\gamma_{\psir}$. Since~$\psir$ 
is non-decreasing, Lemma~\ref{Monotone_Convex} says that the curve~$\gamma_{\psir}$ is a graph of a convex 
function. Therefore it lies above its tangent line at the point $\wr$. Due to~\eqref{FirstDerivative} this 
tangent line contains the segment~$[\wr,g(u)]$. Therefore,~$\gamma_{\psir}$ lies above the tangent line to 
the free boundary at the point~$\wr$ and on the left of this point. Thus,~$\gamma_{\psir}$ lies above~$\varkappa$.

Let~$J \subset [0,1]$ be an interval, we have to prove that~$\av{\vf_x}{J} \in \Omega$. Consider several 
cases of position of~$J$ inside~$[0,1]$. 

Suppose that~$[\alpha_1,\alpha_1+\alpha_2] \subset J$. Then we claim that~$\av{\vf_x}{J}$ is separated from 
the free boundary by~$\varkappa$. Since~$x$ is a convex combination of~$\av{\vf}{[0,\alpha_1) \setminus J}$,
$\av{\vf}{[\alpha_1+\alpha_2,1] \setminus J}$, and~$\av{\vf}{J}$, and the first two points lie above~$\varkappa$, 
the point $\av{\vf}{J}$ lies below~$\varkappa$ indeed, see Figure~\ref{fig:Ang_opt}.

\begin{figure}[h!]
\begin{center}
\includegraphics[width = 0.8 \linewidth]{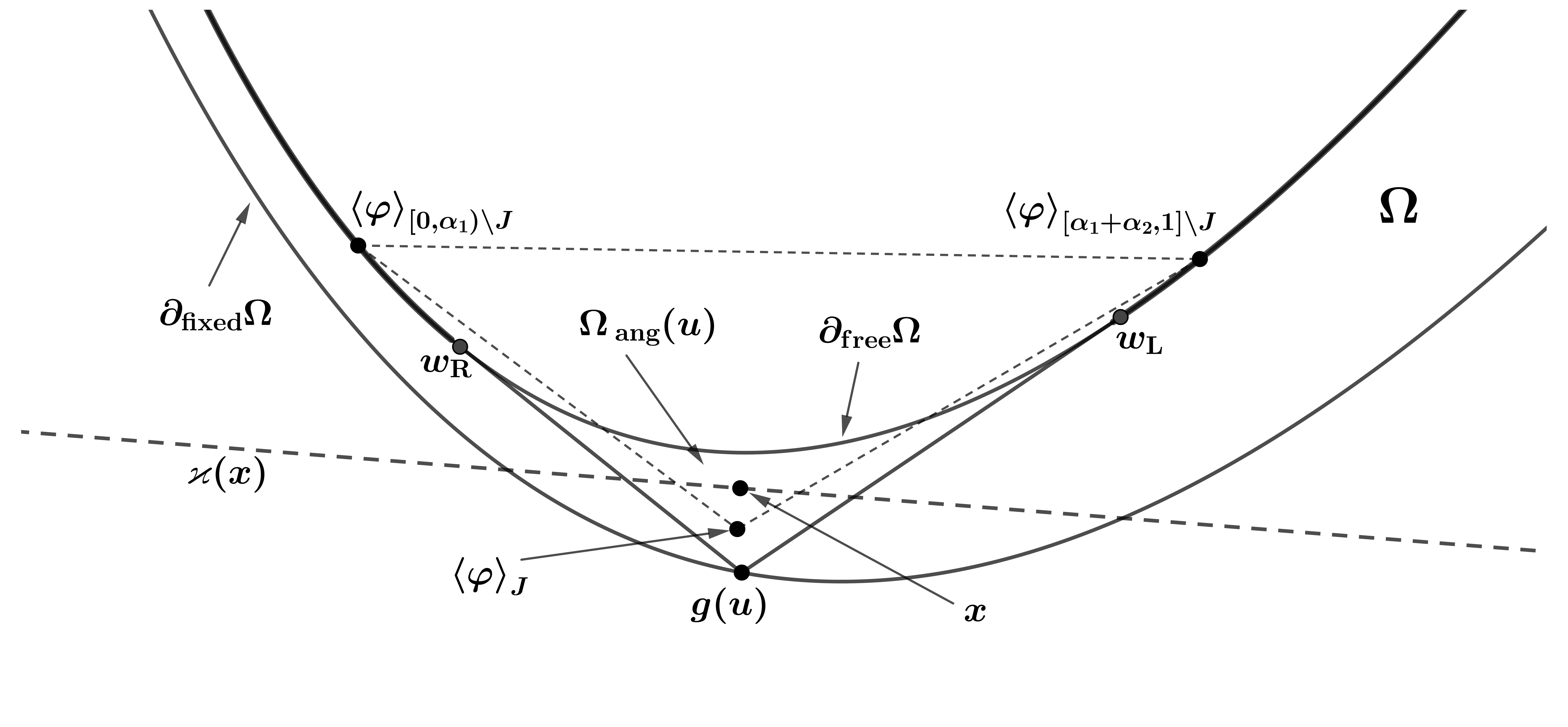}
\caption{Optimizer in $\Ang(u)$.}
\label{fig:Ang_opt}
\end{center}
\end{figure}

If $[\alpha_1,\alpha_1+\alpha_2]$ is not contained in~$J$, then 
we can apply Lemma~\ref{DeliveryCurveLemma}  to the function $\vf|_{J}$. Indeed, $\vf|_{J}$ is a monotone function, and the tangent to the corresponding curve 
at the endpoint does not intersect the free boundary.
\end{proof}

\begin{Def}
The points $\wr$ and $\wl$ introduced in Proposition~\ref{OptimizersAngle} are called the \emph{incoming nodes} of the angle.
\end{Def}

\begin{St}\label{OptimizersTrolleybusR}\index{trolleybus}
Let~$B$ be the standard candidate in the right trolleybus~$\RTroll(u_1,u_2)$. 
Suppose~$\psi$ to be a non-decreasing optimizer for the point~$\wr(u_1)$ such that~$\psi \preceq g(u_1)$ pointwise. 
Then\textup, for any~$x \in \RTroll(u_1,u_2)$ there exists a non-decreasing optimizer $\vf_x$ such that $\vf_x\preceq g(u_2)$.  
\end{St}

\begin{proof}
Choose any point~$x \in \RTroll(u_1,u_2)$. The trolleybus~$\RTroll(u_1,u_2)$ lies inside the triangle with 
the vertices~$g(u_1),g(u_2)$, and $\wr(u_1)$. Thus, there exist~$\alpha_0,\alpha_1,\alpha_2$ such that
\begin{equation*}
x = \alpha_0 \wr(u_1) + \alpha_1 g(u_1) + \alpha_2 g(u_2);\quad \alpha_0+\alpha_1+\alpha_2 = 1, \quad \alpha_j \geq 0.
\end{equation*}
By Remark~\ref{Rescaling}, we may assume that~$\psi$ is defined on~$[0,\alpha_0]$. Define the optimizer~$\vf_x$ 
on $[0,1]$ by the formula
\begin{equation*}
\vf_x(\tau) = 
\begin{cases}
\psi(\tau),\quad & \tau \in [0,\alpha_0);
\\
g(u_1),\quad & \tau \in [\alpha_0,\alpha_0+\alpha_1);
\\
g(u_2),\quad & \tau \in [\alpha_0+\alpha_1,1].
\end{cases}
\end{equation*}
It is clear that $\av{\vf_x}{[0,1]}=x$ and $\av{\ff(\vf_x)}{[0,1]}=B(x)$. We only have to verify~$\vf_x \in \Class_\Omega$, i.\,e., $\av{\vf_x}{J} \in \Omega$
for any interval~$J \subset[0,1]$. Take a line~$\varkappa = \varkappa(x)$ that passes 
through~$x$ and such that the point~$g(u_2)$ together with the domain $\Xi$ lie above this line, 
see Figure~\ref{fig:RTroll_DC}. As in the proof of Propositions~\ref{OptimizersMulticup} and~\ref{OptimizersAngle}, 
we will have to consider different cases of location of~$J$ inside~$[0,1]$.
\begin{figure}[h!]
\begin{center}
\includegraphics[width = 0.49 \linewidth]{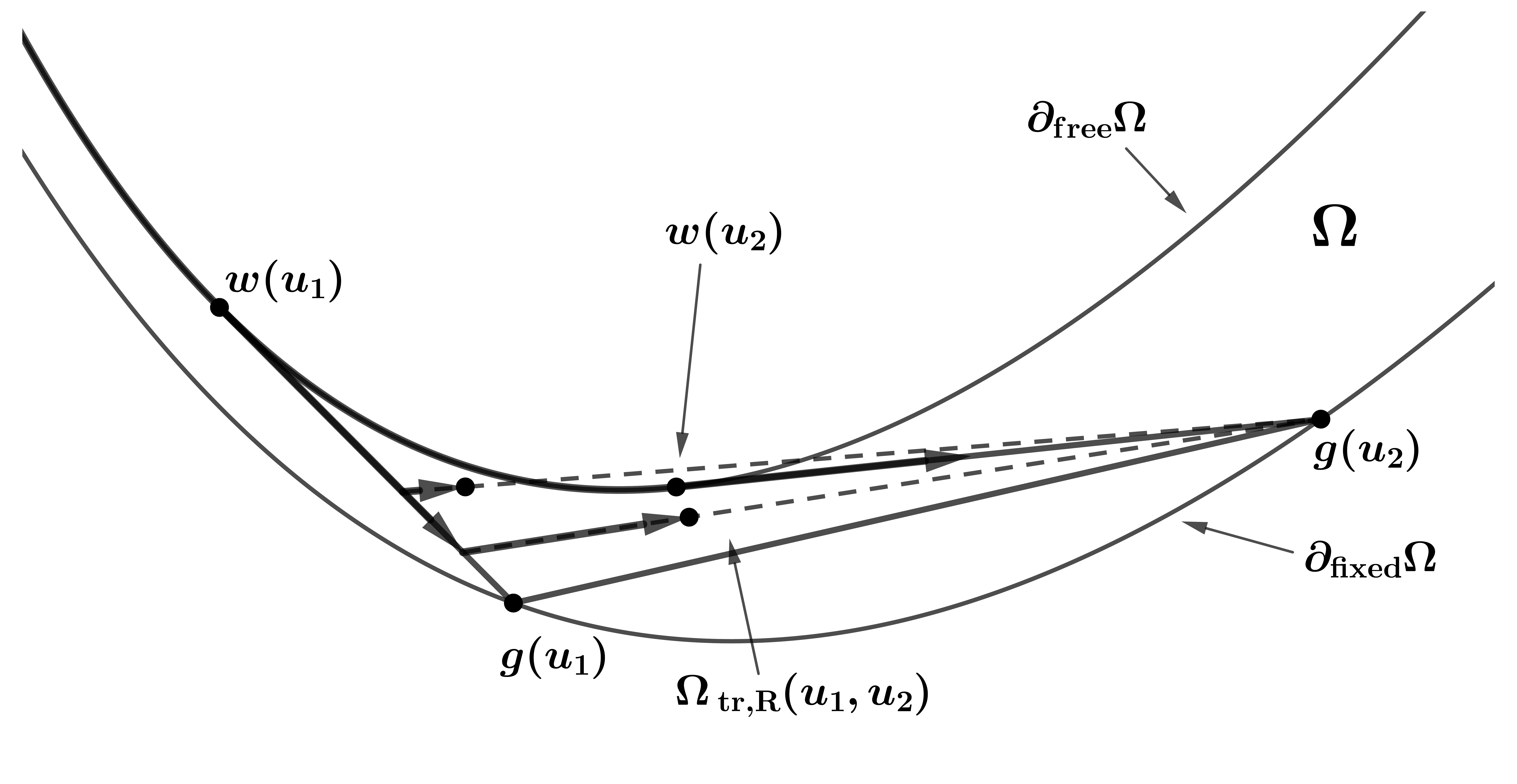}
\includegraphics[width = 0.49 \linewidth]{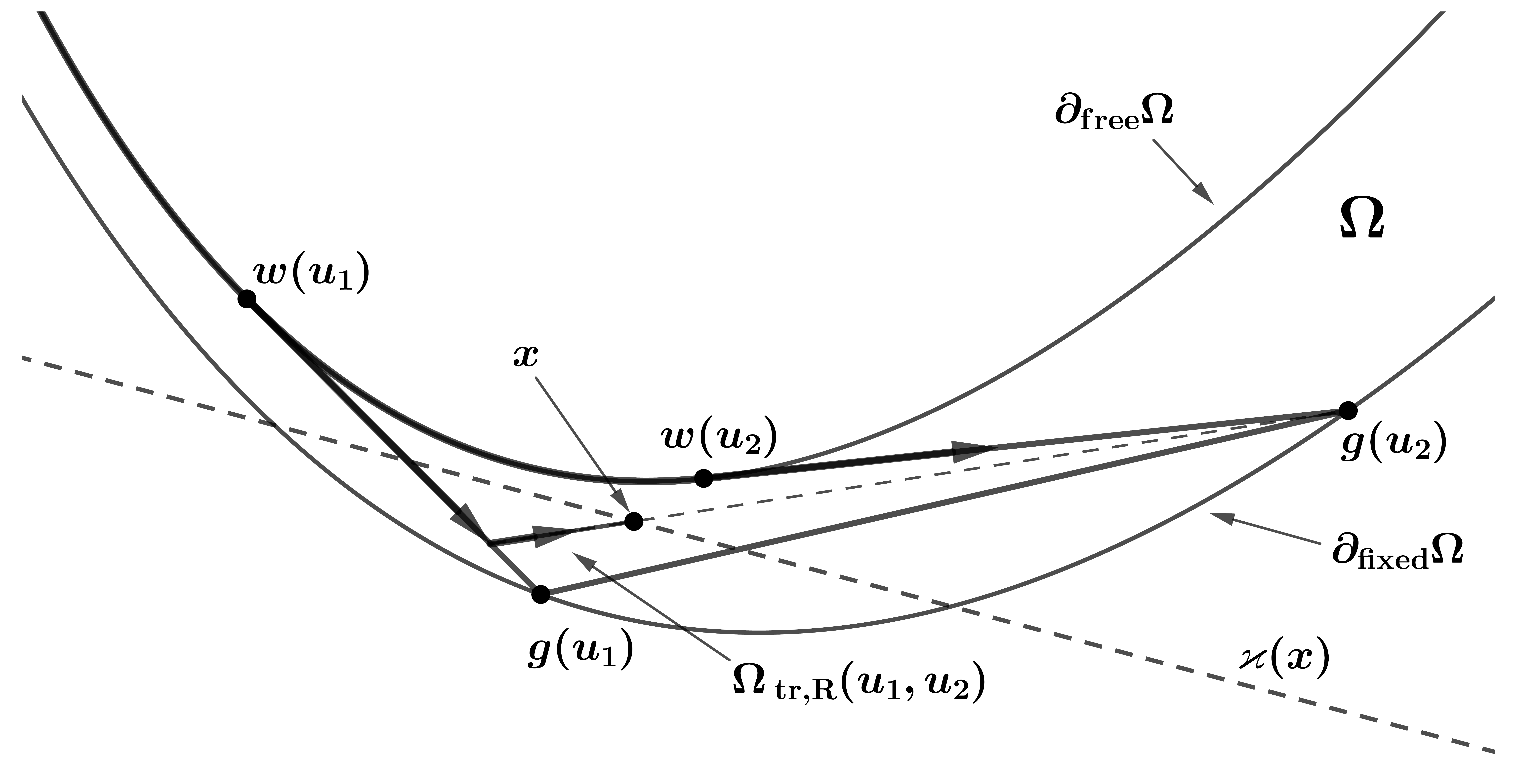}
\caption{Delivery curves inside a trolleybus.}
\label{fig:RTroll_DC}
\end{center}
\end{figure}

If~$[\alpha_0,\alpha_0+\alpha_1) \subset J$, then~$\av{\vf_x}{J}$ is separated from the free boundary 
by~$\varkappa_x$ (this reasoning is completely similar to the one we used in 
Propositions~\ref{OptimizersMulticup} and~\ref{OptimizersAngle}).

If~$J \cap [\alpha_0+\alpha_1,1] = \varnothing$, then the situation falls under the scope of 
Lemma~\ref{DeliveryCurveLemma}.

If~$J \cap [0,\alpha_0) = \varnothing$, then~$\av{\vf_x}{J} \subset [g(u_1),g(u_2)]$.

We have considered all the cases and verified that~$\vf_x \in \Class_\Omega$. Moreover, it follows 
from the construction that $\vf_x$ is non-decreasing and $\vf_x\preceq g(u_2)$.
\end{proof}

As usual, we have a symmetric proposition.

\begin{St}\label{OptimizersTrolleybusL}
Let~$B$ be the standard candidate in the left trolleybus~$\LTroll(u_1,u_2)$. 
Suppose~$\psi$ to be a non-increasing optimizer for the point~$\wl(u_2)$ such that~$\psi \succeq g(u_2)$. 
Then\textup, for any~$x \in \LTroll(u_1,u_2)$ there exists a non-increasing optimizer $\vf_x$ such that $\vf_x\succeq g(u_1)$.  
\end{St}

It remains to construct the optimizers for multitrolleybuses, birdies, and multibirdies. Formulas from 
Subsection~\ref{s344} will help us in this business.

\begin{St}\label{OptimizersMultitrolleybusR}\index{multitrolleybus}
Let~$B$ be the standard candidate in a right multitrolleybus~$\MTTR(\{\mathfrak{a}_i\}_{i=1}^k)$. 
Suppose~$\psi$ to be a non-decreasing optimizer for the point~$\wr(\mathfrak{a}_1^{\mathrm{l}})$ such 
that~$\psi \preceq g(\mathfrak{a}_1^{\mathrm{l}})$. Then\textup, for any~$x$ from 
$\MTTR(\{\mathfrak{a}_i\}_{i=1}^k)$ 
there exists a non-decreasing optimizer that does not exceed~$g(\mathfrak{a}_k^{\mathrm r})$.
\end{St}

\begin{proof}
We apply~\eqref{RMultitrolleybusDesintegration} and decompose a multitrolleybus in an alternating 
sequence of right trolleybuses and multitrolleybuses $\MTTR(\mathfrak{a}_i)$ on solid roots. We consider each such multitrolleybus as a tangent domain with the candidate $B$ on it (note that this is not a standard candidate for the tangent domain). Applying Proposition~\ref{OptimizersTrolleybusR} 
to the trolleybuses and Proposition~\ref{OptimizersRightTangentDomain} to the tangent domains consecutively, 
we build optimizers for all the points inside $\MTTR(\{\mathfrak{a}_i\}_{i=1}^k)$. 
\end{proof}

\begin{St}\label{OptimizersMultitrolleybusL}
Let~$B$ be the standard candidate in a left multitrolleybus~$\MTTL(\{\mathfrak{a}_i\}_{i=1}^k)$. 
Suppose~$\psi$ to be a non-increasing optimizer for the point~$\wl(\mathfrak{a}_k^{\mathrm{r}})$ such 
that~$\psi \succeq g(\mathfrak{a}_k^{\mathrm{r}})$. Then\textup, for any~$x$ in  
$\MTTL(\{\mathfrak{a}_i\}_{i=1}^k)$ 
there exists a non-increasing optimizer that is not less than~$g(\mathfrak{a}_1^{\mathrm l})$.
\end{St}

\begin{Def}\label{Rem020401}
Let $w_-$ and $w_+$ be the the left and the right endpoints of the arc of the free boundary that is 
the part of the boundary of a trolleybus \textup(or a multitrolleybus\textup). These points are called the 
\emph{incoming} and \emph{outgoing nodes} of the trolleybus \textup(or the multitrolleybus\textup)\textup: for the case 
of $\RTroll$ \textup(and $\MTTR$\textup) the incoming node is $w_-$ and outgoing is $w_+,$ and vice versa 
for the case of $\LTroll$ \textup(and $\MTTL$\textup).
\end{Def}

\begin{St}\label{OptimizersMultiBirdie}\index{multibirdie}
Let~$B$ be the standard candidate in a multibirdie~$\MTB(\{\mathfrak{a}_i\}_{i=1}^k)$. 
Let~$\psir$ be a non-decreasing optimizer for the point~$\wr(\mathfrak{a}_1^{\mathrm l})$ such 
that~$\psir\preceq g(\mathfrak{a}_1^{\mathrm l})$\textup, 
let~$\psil$ be a non-increasing optimizer for the point~$\wl(\mathfrak{a}_k^{\mathrm r})$ such 
that~$\psil\succeq g(\mathfrak{a}_k^{\mathrm r})$. Then\textup, there exists an optimizer for 
every point~$x$ in $\MTB(\{\mathfrak{a}_i\}_{i=1}^k)$.
\end{St}

\begin{proof}
We apply~\eqref{MultibirdieDesintegration} (with any choice of~$j$) and decompose the multibirdie into 
two multitrolleybuses and an angle. Applying corresponding previous propositions for each of the domains, 
we build optimizers for all the points inside~$\MTB(\{\mathfrak{a}_i\}_{i=1}^k)$.  
\end{proof}

\begin{Def}
The points $\wr(\mathfrak{a}_1^{\mathrm l})$ and $\wl(\mathfrak{a}_k^{\mathrm r})$ 
in Proposition~\ref{OptimizersMultiBirdie} are called the \emph{incoming nodes} of 
the multibirdie~$\MTB(\{\mathfrak{a}_i\}_{i=1}^k)$.
\end{Def}

\section{Global optimizers}\label{s53}
Before passing to global optimizers, we collect the information about \index{incoming node}incoming 
and \index{outgoing node}outgoing nodes for a vertex or an edge in the foliation graph. These nodes are 
points on the free boundary that are used to transfer the delivery curve from one domain to another. 
For some figures, the optimizers inside them are constructed by using the 
optimizer coming from the right or from the left (for example, to build optimizers in 
a trolleybus, 
we need an optimizer for a special point on the free boundary, see Propositions~\ref{OptimizersTrolleybusR} 
and~\ref{OptimizersTrolleybusL}). The incoming node is the point we start the delivery curve from, 
whereas the outgoing node is the point where it leaves the figure. In the table below, we give a precise 
description. We extend our notation from Definition~\ref{Rem020401} for other domains~$\mathfrak{L}$: let $w_-$ 
and $w_+$ be the the left and the right endpoints of the arc of the free boundary $\dfree\Omega$ which is 
the part of the boundary of~$\mathfrak{L}$. 

\medskip
\centerline{
\begin{tabular}{|p{1.6cm}|p{1.7cm}|p{2.0cm}|p{2.0cm}|p{2.0cm}|p{2.0cm}|}
\hline 
&{\small \rule{32pt}{0pt} Long chord}&
\rule{32pt}{0pt} $\scriptstyle\MTC(\{\mathfrak{a}_i\}_{i=1}^k)$&
$\scriptstyle\MTTR(\{\mathfrak{a}_i\}_{i=1}^k)$ $\scriptstyle\RTroll(u_1,u_2)$ $\scriptstyle\Rt(u_1,u_2)$&
$\scriptstyle\MTTL(\{\mathfrak{a}_i\}_{i=1}^k)$ $\scriptstyle\LTroll(u_1,u_2)$ $\scriptstyle\Lt(u_1,u_2)$&
$\scriptstyle\Ang(w)$ $\scriptstyle\MTB(\{\mathfrak{a}_i\}_{i=1}^k)$
\\
\hline
\hfil Outgoing \hfil& \hfil$w_-=w_+$ \hfil& \hfil$w_-, w_+$ \hfil& \hfil$w_+$ \hfil&\hfil $ w_-$ \hfil&
\\
\hline
\hfil Incoming \hfil& & &\hfil$w_-$\hfil& \hfil$ w_+$ \hfil& \hfil$w_-, w_+$\hfil
\\
\hline
\end{tabular}
}
\medskip

Other figures (separated from the upper boundary) have no incoming or outgoing nodes at all. Now we see 
that the propositions of Subsections~\ref{s521} and~\ref{s522} are of the following form: if there is 
a monotone optimizer for the incoming node satisfying a certain restriction (if there are incoming nodes 
for the figure in question), then we can build monotone optimizers satisfying a similar restriction for 
all the points in the domain, in particular for the outgoing nodes (if there are any). These propositions are well suited for induction. 
To state the general theorem about optimizers, we need Definition~\ref{Admissible graph} 
(see Remark~\ref{AdmissibleCondidate} and the paragraph before it, where it is explained that an 
admissible graph generates the Bellman candidate).

\begin{Th}\label{GlobalOptimizer}
Let~$\eps < \epsmax$\textup, let~$\Gamma(\eps)$ be an admissible graph for~$\eps$ and~$f$. 
The Bellman candidate~$B_\eps$ generated by~$\Gamma(\eps)$ admits an optimizer at every point of~$\Omega_{\eps}$.  
\end{Th}

\begin{proof}
The proof of the theorem consists in inductive application of the propositions from Subsection~\ref{s521} 
and~\ref{s522}. The details are the same as in the partial case of parabolic strip, see Theorem~5.3.1 in~\cite{ISVZ2018}.
\end{proof}

Theorem~\ref{GlobalOptimizer} we have just proved justifies the arguments given after Theorem~\ref{MT} 
on page~\pageref{MT}. Moreover, the candidate~$B_\eps$ constructed by the admissible graph~$\Gamma(\eps)$ 
coincides with the Bellman function $\Bell$ and with the minimal locally concave function $\BG_{\Omega_\eps,f}$.

\begin{Th}\label{Final}
For any~$\eps < \epsmax$ the constructed Bellman candidate~$B_{\eps}$ coincides with~$\Bell$ 
and~$\BG_{\Omega_\eps,f}$. 
\end{Th}

\section{Remarks concerning global conditions}\label{s54}

In this section we discuss conditions~\eqref{eq111001} and~\eqref{eq111002} introduced without any explanations. 

As the reader has seen, these technical conditions were used only during different steps of constructing 
optimizers. Here we discuss what happens when they fail. We note 
that construction of optimizers do not use any evolutional arguments and deals with the fixed $\eps$. 

\subsection{Behavior on $-\infty$}

Since we have the right-left symmetry, let us consider the behavior on $-\infty$, namely condition~\eqref{eq111001}. 
In what follows we omit the index $\mathrm{R}$,  i.\,e., we write $\lambda$, $w$, 
$\kappa$ instead of $\lamr$, 
$\wr$, $\kappar$. 

If~\eqref{eq111001} is not fulfilled, and the foliation for the function $B$ contains a domain~$\Rt(-\infty,v)$ 
of the right tangents coming from $-\infty$, the statement of Proposition~\ref{OptimizerRightTangentsInfty} is 
not valid in the following sense: there are no optimizers~$\vf_x$ for $x \in \Rt(-\infty,v) \setminus \dfi\Omega$. However, for any such $x$
there is a sequence of test functions $\vf_n \in \Class_{\Omega}$ such that 
$\av{\vf_n}{}=x$ and $B(x) = \lim_{n \to +\infty} \av{\ff(\vf_n)}{}$. We call such $\{\vf_n\}_{n}$ an \emph{optimizing sequence}\index{optimizer! optimizing sequence} for $x$. 
All the propositions from Subsection~\ref{s52} remain true if we replace the word optimizer by optimizing sequence. 

We give some details on how to construct these functions $\vf_n$ for $x = w(v)$. 
Consider sequences $v_n^-, v_n^+ \in \mathbb{R}$ such that $v_n^- < v_n^+ < v$, 
$v_n^+\to - \infty$, and the chord $[g(v_n^-),g(v_n^+)]$ is tangent to $\dfree\Omega$. The tangent point 
is $w(v_n^+)$. Consider the domain $\Rt(v_n^+,v)$ and let $\gamma_n$ be the delivery curve that starts at 
the point $g(v_n^-)$ and goes to $w(v_n^+)$ along the chord, and then proceeds along $\dfree\Omega$ till 
the point $w(v)$. Build a test function $\vf_n$ on $[0,1]$ that generates this curve $\gamma_n$:  
\eq{eq290602}{
\vf_n(t) =
\begin{cases}
g(v_n^-), & t \in [0, \alpha_1);
\\
g(v_n^+), & t \in [\alpha_1, \alpha_2);
\\
g(u(t)), & t \in [\alpha_2,1].
\end{cases}
}
The parameters $u=u(t)$ and $t=t(u)$ are related as in~\eqref{eq070602}:
\eq{eq280601}{
\log t = -\int_{u(t)}^{v}\lambda, \qquad t \in [\alpha_2,1],
}
and $u(\alpha_2) = v_n^+$, i.\,e.,
\eq{eq290603}{
\alpha_2 = \exp\Big(\!-\!\int_{v_n^+}^{v}\lambda\Big).
}
The value of $\alpha_1$ is determined by the relation that represents the point $w(v_n^+)$ as a convex 
combination of $g(v_n^-)$ and $g(v_n^+)$:
$$
\frac{\alpha_1}{\alpha_2}g(v_n^-) + \frac{\alpha_2-\alpha_1}{\alpha_2}g(v_n^+) = w(v_n^+),
$$
namely,
\eq{eq201201}{
\frac{\alpha_1}{\alpha_2} = \frac{g_1(v_n^+)-w_1(v_n^+)}{g_1(v_n^+)-g_1(v_n^-)}\,.
}

\begin{Rem}\label{rem010701}
Since $\vf_n$ defined by~\eqref{eq290602} generates the delivery curve $\gamma_n$, we have $
\av{g(\vf_n)}{[0,1]} = w(v).
$
Moreover, Corollary~\ref{DeliveryCurveCorollary} guaranties 
$\vf_n$ lies in $\Class_\Omega$, therefore it is a test function for $w(v)$.
\end{Rem}

The following theorem generalizes Proposition~\ref{OptimizerRightTangentsInfty} for the case 
when condition~\ref{eq111001} fails.

\begin{Th}\label{th120701}
Let~$B$ be the standard candidate on~$\Rt(-\infty, v)$. Then\textup, for every point~$x \in \Rt(-\infty,v)$ 
there exists a sequence of non-decreasing test functions~$\vf_{n,x}$ being an optimizing sequence for~$B$ at $x$\textup, 
moreover\textup,~$\vf_{n,x} \preceq g(v)$.
\end{Th}

The proof of Theorem~\ref{th120701} will be given for different cases of condition~\ref{eq111001} failure 
separately, namely, we consider three cases:
\begin{itemize}
\item[\textbf{(A)}] $\exp(-\int_{-\infty}^v \lambda) \ne 0$, i.\,e., the integral $\int_{-\infty}^v \lambda$ converges;
\item[\textbf{(B)}] $\int_{-\infty}^v \lambda = +\infty$, and the function \index[symbol]{$\vt_1$}
\eq{eq010701}{
\vt_1(\tau) \df \big(g_1(\tau) - w_1(\tau)\big)\exp\Big(\!-\!\int_\tau^v \lambda\;\Big)
}
does not tend to zero as $\tau \to -\infty$;\index[symbol]{$\vt_2$}
\item[\textbf{(C)}] $\int_{-\infty}^v \lambda = +\infty$, $\vt_1(-\infty) = 0$, and the function 
\eq{eq010703}{
\vt_2(\tau) \df \big(g_2(\tau) - w_2(\tau)\big)\exp\Big(\!-\!\int_\tau^v \lambda\;\Big)
}
does not tend to zero as $\tau \to -\infty$.
\end{itemize} 

\begin{Rem}
For any function $\psi \colon \mathbb{R} \to \mathbb{R}$ we use the notation $\psi(-\infty)$ either for 
finite or for infinite limit $\lim_{\tau \to -\infty} \psi(\tau)$.  We will often use this notation for the 
functions that are monotone in a neighborhood of $-\infty$, e.\,g., $f$, $\kappa_3$, $\tors$ that are 
monotone by Condition~\ref{reg} and Remark~\ref{Rem100301}, as well as for $\vt_1$ and $\vt_2$, whose 
monotonicity will be proved in Lemma~\ref{lem010701}.
\end{Rem}

\begin{Le}\label{lem010701}
The functions $\vt_1$ and $\vt_2$ defined by~\eqref{eq010701} and~\eqref{eq010703} are monotone in a neighborhood 
of $-\infty$ and have finite limits at $-\infty$. Moreover\textup, $\vt_1$ and $\vt_1'$ are positive\textup, and
\eq{eq050701}{
\sign\vt_2'=\sign\vt_2 = \sign g_2' = \sign \kappa_2 = \sign \kappa
}
in a neighborhood of $-\infty$.
\end{Le}

\begin{proof}
The functions $\vt_1$ is positive and increasing:
\eq{eq010702}{
\vt_1'(\tau) \stackrel{\eqref{eq011801}}{=} g_1'(\tau)\exp\Big(\!-\!\int_{\tau}^v \lambda\;\Big)>0.
}
Therefore, $\vt_1(-\infty)$ is finite.

The similar formula for the derivative of $\vt_2$,
\eq{eq070104}{
\vt_2'(\tau) = g_2'(\tau)\exp\Big(\!-\!\int_\tau^v \lambda\;\Big),
}
implies that $\sign \vt_2'(\tau) = \sign \kappa_2(\tau)$, whereas $\sign \vt_2(\tau) = \sign \kappa(\tau)$.  
From the geometric assumption on the domain~$\Omega$, the point $w(\tau)$ lies on a segment connecting 
$g(\tau)$ and some $g(\tau_-)$, $\tau_-<\tau$. It follows from convexity of $\dfi \Omega$ that 
\eq{eq010705}{
\kappa_2(\tau_-) < \kappa(\tau) < \kappa_2(\tau).
} Whence, $\sign\kappa_2=\sign\kappa$ in a neighborhood of~$-\infty$, and $\sign\vt_2'=\sign\vt_2$ as well. 
Therefore, the limit $\vt_2(-\infty)$ is finite.
\end{proof}

Preparing for the proof of Theorem~\ref{th120701}, we calculate $\av{\ff(\vf_n)}{[0,1]}$:
\begin{gather}
\label{eq290601}
\av{\ff(\vf_n)}{[0,1]} = \alpha_1 f(v_n^-) + (\alpha_2-\alpha_1)f(v_n^+) + \int_{\alpha_2}^1 f(u(t))\, dt=
\\
\notag 
\stackrel{\eqref{eq280601}}{=}\alpha_1 f(v_n^-)+(\alpha_2-\alpha_1)f(v_n^+)+\int_{v_n^+}^v f(u)\lambda(u)\exp\Big(\!-\!\int_u^v\lambda\;\Big)\,du=
\\
\notag 
=\alpha_1 f(v_n^-)+(\alpha_2-\alpha_1)f(v_n^+)+f(u)\exp\Big(\!-\!\int_u^v\lambda\;\Big)\Big|_{u=v_n^+}^{u=v}-
\int_{v_n^+}^v f'(u) \exp\Big(-\int_u^v \lambda\Big)\,du=
\\
\label{eq300601}
=\alpha_1 \big(f(v_n^-)-f(v_n^+)\big)+f(v)-\int_{v_n^+}^v f'(u)\exp\Big(\!-\!\int_u^v \lambda\;\Big)\,du=
\\
\label{eq280602}
=-\frac{f(v_n^+)-f(v_n^-)}{g_1(v_n^+)-g_1(v_n^-)}\vt_1(v_n^+)+f(v)-
\int_{v_n^+}^v f'(u)\exp\Big(\!-\!\int_u^v\lambda\;\Big)\,du\,;
\end{gather}
in the latter equality we use~\eqref{eq201201}, \eqref{eq290603}, and~\eqref{eq010701}.
We apply the Cauchy mean value theorem to find $v_n \in [v_n^-, v_n^+]$ such that 
$$
\frac{f(v_n^+)-f(v_n^-)}{g_1(v_n^+)-g_1(v_n^-)} = \frac{f'(v_n)}{g_1'(v_n)} = \kappa_3(v_n).
$$
We use this to rewrite~\eqref{eq280602}:
\eq{eq280603}{
\av{\ff(\vf_n)}{[0,1]} = -\kappa_3(v_n)\vt_1(v_n^+)+f(v)-\int_{v_n^+}^v f'(u)\exp\Big(\!-\!\int_u^v\lambda\;\Big)du.
} 

\subsection{The case \textbf{(A)}\,: $\int_{-\infty}^v \lambda$ converges}

\begin{Le}\label{lem290601}
If $\int_{-\infty}^v\lambda<+\infty$\textup, then the limits $g(-\infty)$ and 
$w(-\infty)$ are finite and coincide. Thus\textup, $\vt_1(-\infty)=0$ and $\vt_2(-\infty)=0$.
\end{Le}

\begin{proof}
We note that $\lambda>0$, therefore the hypothesis of the lemma is equivalent to convergence of the integral. 
From~\eqref{eq010702} we have
$$
\vt_1'(\tau) \leq g_1'(\tau) \leq  \vt_1'(\tau) \exp\Big(\!\int_{-\infty}^v \!\lambda\;\Big).
$$ 
By Lemma~\ref{lem010701} $\vt_1(-\infty)$ is finite, and therefore $\int_{-\infty}^v \vt_1' < +\infty$. 
Thus, $\int_{-\infty}^v  g_1' < +\infty$, what means $g_1(-\infty)$ is finite.

Since the limit $g_1(-\infty)$ is finite and the curve $(g_1,g_2)$ is convex, we see that the limit of $g_2$ 
at $-\infty$ is not $-\infty$. Therefore, $g_2(-\infty)$ is either finite or $+\infty$. If $g_2(-\infty)=+\infty$, 
then $g_2'<0$ in a neighborhood of $-\infty$. From~\eqref{eq070104} we have 
$$
\vt_2'(\tau) \geq g_2'(\tau) \geq  \vt_2'(\tau) \exp\Big(\!\int_{-\infty}^v \!\lambda\;\Big).
$$ 
By Lemma~\ref{lem010701} $\vt_2(-\infty)$ is finite, therefore $\int_{-\infty}^v \vt_2' > -\infty$. 
Thus, $\int_{-\infty}^v  g_2' > -\infty$, what means $g_2$ has a finite limit at $-\infty$. This is a contradiction.

We have proved that $g= (g_1,g_2)$ has a finite limit at $-\infty$. From the geometric assumption on the 
domain~$\Omega$, any point $w(\tau)$ lies on a segment connecting $g(\tau)$ and some $g(\tau_-)$, $\tau_-<\tau$. 
Since both points $g(\tau)$ and $g(\tau_-)$ tend to the same limit, the limit of $w(\tau)$ as $\tau\to-\infty$ 
exists and $w(-\infty)=g(-\infty)$.
\end{proof}

\begin{Cor}
\label{cor290601}
If $\int_{-\infty}^v \lambda<+\infty$\textup, then $(\kappa_2-\kappa)(g_1-w_1)$ tends to $0$ at $-\infty$.
\end{Cor}

\begin{proof}
From Lemma~\ref{lem290601} we know that $g_1-w_1$ tends to $0$ at $-\infty$.

If $\kappa_2>0$ in a neighborhood of $-\infty$, then $\kappa_2(-\infty)$ is finite because $\kappa_2'>0$. 
Since any point $w(\tau)$ lies on a segment connecting $g(\tau)$ and some $g(\tau_-)$, $\tau_-<\tau$, we have an 
inequality $\kappa_2(\tau_-)<\kappa(\tau) < \kappa_2(\tau)$. Therefore, $\kappa(-\infty) = \kappa_2(-\infty)$, and the claim is proved. 

If $\kappa_2<0$ in a neighborhood of $-\infty$, then $0<\kappa_2-\kappa<-\kappa$. Therefore by Lemma~\ref{lem290601} we have
$$
|(\kappa_2-\kappa)(g_1-w_1)|\leq |\kappa (g_1-w_1)| = |g_2-w_2| \to 0 
$$
as the argument tends to $-\infty$.
\end{proof}

\begin{Cor}
\label{cor290602}
If $\int_{-\infty}^v \lambda<+\infty$\textup, then $\int_{-\infty}^v \frac{\kappa_2'}{\kappa_2-\kappa} = +\infty$.
\end{Cor}

\begin{proof}
The statement follows from~\eqref{eq111005} and Corollary~\ref{cor290601}.  
\end{proof}

\begin{Le}
\label{lem290602}
If $\int_{-\infty}^v \lambda<+\infty$ and $f(-\infty) = -\infty$\textup, then $\tors(-\infty)= -\infty$.
\end{Le}

\begin{proof}
Assume the contrary: let $\tors > C_1$ in a neighborhood of $-\infty$ for some $C_1\in \mathbb{R}$. 
Then, $\kappa_3'> C_1 \kappa_2'$ because $\kappa_2'>0$. Integration of this inequality yields 
$\kappa_3 < C_1\kappa_2 + C_2$ for some $C_2\in\mathbb{R}$. Multiply this by positive $g_1'$ and obtain
$f' < C_1g_2'+C_2g_1'$. Integrating this inequality we get $f > C_1 g_2+C_2 g_1 +C_3$ for some 
$C_3 \in \mathbb{R}$. This inequality contradicts the hypothesis that $f(-\infty)=-\infty$, 
since the functions $g_1$ and $g_2$ have finite limits at $-\infty$ by Lemma~\ref{lem290601}.
\end{proof}

\begin{St}\label{st130701}
If $\int_{-\infty}^v \lambda<+\infty$ and there is a standard candidate $B$ on $\Rt(-\infty, v)$\textup, 
then $f(-\infty) > -\infty$. 
\end{St}

\begin{proof}
Assume the contrary: let $f(-\infty) = -\infty$. Then, by Lemma~\ref{lem290602} we have $\tors(-\infty)=-\infty$, 
therefore $\tors'>0$ in a neighborhood of $-\infty$. By~\eqref{eq120702}
$\beta_2'>0$ in the same neighborhood of $-\infty$. This contradicts to the existence of the standard candidate 
$B$ on $\Rt(-\infty, v)$, see Definition~\ref{211101}.
\end{proof}

\begin{St}\label{st130702}
If $\int_{-\infty}^v \lambda<+\infty$ and $f(-\infty)=+\infty$\textup, then $\boldsymbol{B}(x;\,f)=+\infty$ 
for any \hbox{$x \in \Omega\setminus \dfi\Omega$}.
\end{St}

\begin{proof}
It suffices to prove $\boldsymbol{B}(w(v);\,f) = +\infty$. We construct a sequence of test functions $\vf_n$ for 
the point~$w(v)$ by~\eqref{eq290602}. The condition~$\int_{-\infty}^v \lambda<+\infty$ together 
with~\eqref{eq290603} guarantees that $\alpha_2$ is separated from $0$. Therefore, from the condition 
$f(-\infty)=+\infty$ and~\eqref{eq290601}, 
we obtain that $\av{\ff(\vf_n)}{[0,1]}$ tends to $+\infty$, and thus $\boldsymbol{B}(w(v);\,f) = +\infty$. 
\end{proof}

\begin{St}\label{st130703}
Suppose that $\int_{-\infty}^v \lambda<+\infty$ and the limit $f(-\infty)$ is finite. Let~$B$ be the standard 
candidate on~$\Rt(-\infty, v)$. Then\textup, for every point~$x \in \Rt(-\infty,v)$ there exists a sequence 
of non-decreasing test functions~$\vf_{n,x}$ being an optimizing sequence for~$B$ at $x$\textup, 
moreover\textup,~$\vf_{n,x} \preceq g(v)$.
\end{St}

\begin{proof}
As before, it suffices to prove the claim for the case $x=w(v)=\wr(v)$. In what follows we omit the subscript $x$ in our notation. We 
use the sequence of functions $\vf_n$ defined by~\eqref{eq290602}. All these functions are non-decreasing and satisfy 
the inequality~$\vf_{n}\preceq g(v)$. Therefore, we only need to verify that $\av{\ff(\vf_n)}{[0,1]}\to B(w(v))$. 

Recall~\eqref{fenceB2}: 
\eq{eq010706}{
B(w(v)) = f(v) + \big(w_1(v)-g_1(v)\big)\Big[\kappa_3(v)+\big(\kappa(v)-\kappa_2(v)\big)\beta_2(v)\Big],
}
where $\beta_2$ is given by~\eqref{eq120701}:
\begin{align}
\notag
\beta_2(v) &=\tors(v)-\int_{-\infty}^v\exp\Big(\!-\!\int_{\tau}^v\frac{\kappa_2'}{\kappa_2-\kappa}\;\Big)\,
\tors'(\tau)\,d\tau=
\\
&=\int_{-\infty}^v \exp\Big(\!-\!\int_{\tau}^v\! \frac{\kappa_2'}{\kappa_2-\kappa}\;\Big) 
\frac{\kappa_3'(\tau)}{\kappa_2(\tau)-\kappa(\tau)} \, d\tau.
\label{eq010707}
\end{align}
Here we have used Lemma~\ref{Lem111001} with $\vt = \exp\big(-\int_{\tau}^v\frac{\kappa_2'}{\kappa_2-\kappa}\big)$ 
and $Y = \tors$, note that $\vt(-\infty)=0$ by Corollary~\ref{cor290602}. 

We use~\eqref{eq300601} to evaluate the limit of $\av{\ff(\vf_n)}{[0,1]}$:
\begin{align*}
\lim_{n\to \infty}\av{\ff(\vf_n)}{[0,1]}
&=f(v)-\int_{-\infty}^v f'(u) \exp\Big(\!-\!\int_{u}^v \lambda\;\Big)\,du =
\\
&=f(v) - \int_{-\infty}^v \kappa_3(u) g_1'(u)\exp\Big(\!-\!\int_u^v\lambda\;\Big)\,du\stackrel{\eqref{eq010702}}=
f(v) - \int_{-\infty}^v \kappa_3(u) \vt_1'(u)\, du.
\end{align*}
Integrate by parts using Lemma~\ref{Lem111001} with $Y = \kappa_3$ and $\vt=\vt_1$ 
(by Lemma~\ref{lem290601} we have $\vt_1(-\infty)=0$) and obtain
$$
\lim_{n\to\infty}\av{\ff(\vf_n)}{[0,1]}=f(v)-\kappa_3(v)\big(g_1(v)-w_1(v)\big)+\int_{-\infty}^v\kappa_3'(u)\vt_1(u)\,du.
$$

Therefore, to check that the limit coincides with~\eqref{eq010706}, it suffices to verify the identity:
\eq{eq010708}{
\int_{-\infty}^v \kappa_3'(u) \vt_1(u) \, du
= \big(\kappa(v)-\kappa_2(v)\big)\big(w_1(v)-g_1(v)\big)\int_{-\infty}^v \exp\Big(\!-\!\int_{u}^v\! 
\frac{\kappa_2'}{\kappa_2-\kappa} \Big) \frac{\kappa_3'(u)}{\kappa_2(u)-\kappa(u)} \, du,
}
which follows from~\eqref{eq111005}:
\eq{eq110702}{
\vt_1(u) = (g_1(u)-w_1(u)) \exp\Big(\!-\!\int_u^v \lambda\;\Big) = 
\frac{\big(\kappa_2(v)-\kappa(v)\big)\big(g_1(v)-w_1(v)\big)}{\kappa_2(u)-\kappa(u)}\exp\Big(\!-\!\int_u^v\! 
\frac{\kappa_2'}{\kappa_2-\kappa} \Big).
}
\end{proof}

Propositions~\ref{st130701},~\ref{st130702}, and~\ref{st130703} prove Theorem~\ref{th120701} in the case~\textbf{(A)}.

\subsection{The case~\textbf{(B)}\,: $\int_{-\infty}^v \lambda = +\infty,$ $\vt_1(-\infty)>0$}

\begin{St}\label{st130704}
If $\int_{-\infty}^v \lambda = +\infty,$ $\vt_1(-\infty)>0,$ and $\kappa_3(-\infty)=-\infty,$ then 
$\boldsymbol{B}(x;\,f) = +\infty$ for any $x \in \Omega\setminus \dfi\Omega$.  
\end{St}

\begin{proof}
It suffices to prove $\boldsymbol{B}(w(v))=+\infty$. We consider the sequence $\vf_n$ defined in~\eqref{eq290602}, they are test functions for~$w(v)$, see Remark~\ref{rem010701}. Since the 
sign of $\kappa_3$ coincides with the sign of $f'$, from~\eqref{eq280603} we see that 
$\av{\ff(\vf_n)}{[0,1]} \to +\infty$.
\end{proof}

\begin{Le}\label{lem010702}
If $\int_{-\infty}^v \lambda = +\infty$, $\vt_1(-\infty)>0$, then the limit $\kappa_2(-\infty)$ is finite 
and $\int_{-\infty}^v \frac{\kappa_2'}{\kappa_2-\kappa} = +\infty$.
\end{Le}

\begin{proof}
By Lemma~\ref{lem010701} both $\vt_1(-\infty)$ and $\vt_2(-\infty)$ are finite. Since $\vt_1(-\infty)>0$, 
$\kappa(-\infty) = \frac{\vt_2(-\infty)}{\vt_1(-\infty)}$ is finite as well. From~\eqref{eq010705} and monotonicity of $\kappa_2$ we 
conclude that $\kappa_2(-\infty) = \kappa(-\infty)$.

From~\eqref{eq111005} we have
$$
\int_{\tau}^v \frac{\kappa_2'}{\kappa_2-\kappa} = 
\log\frac{\kappa_2(v)-\kappa(v)}{\kappa_2(\tau)-\kappa(\tau)} + \log \frac{\vt_1(v)}{\vt_1(\tau)},
$$ 
and the right hand side tends to $+\infty$ as $\tau$ goes to $-\infty$.
\end{proof}

\begin{St}\label{st130705}
If $\int_{-\infty}^v \lambda = +\infty,$ $\vt_1(-\infty)>0,$ and there is a standard candidate $B$ 
on $\Rt(-\infty, v),$ then $\kappa_3(-\infty) < +\infty$. 
\end{St}

\begin{proof}
Assume the contrary, let $\kappa_3(-\infty) = +\infty$. First, we claim that $\tors(-\infty) = -\infty$. If not, 
then $\tors > C_1$ for some $C_1 \in \mathbb{R}$ in a neighborhood of $-\infty$, i.\,e., $\kappa_3'> C_1 \kappa_2'$. 
Integration of this inequality yields $\kappa_3 < C_1\kappa_2+ C_2$ for some $C_2 \in \mathbb{R}$. This contradicts 
to the assumption $\kappa_3(-\infty) = +\infty$ by Lemma~\ref{lem010702}. Thus, $\tors(-\infty) = -\infty$ and 
$\tors'>0$ in a neighborhood of $-\infty$. By~\eqref{eq120702} $\beta_2'>0$ in the same neighborhood of $-\infty$, 
this contradicts to the existence of the standard candidate $B$ on $\Rt(-\infty, v)$. We conclude that 
$\kappa_3(-\infty) < +\infty$.
\end{proof}

\begin{St}\label{st130706}
Suppose that $\int_{-\infty}^v \lambda = +\infty$, $\vt_1(-\infty)>0$, and $\kappa_3(-\infty)$ is finite. 
Let~$B$ be the standard candidate on~$\Rt(-\infty, v)$. Then for every point~$x \in \Rt(-\infty,v)$ there 
exists a sequence of non-decreasing test functions~$\vf_{n,x}$ being an optimizing sequence for~$B$ at $x$\textup, 
moreover\textup,~$\vf_{n,x} \preceq g(v)$.
\end{St} 

\begin{proof}
It suffices to prove the claim for the case $x = w(v) = \wr(v)$. In what follows we omit the 
subscript $x$ in our notation. We use the sequence of functions $\vf_n$ defined by~\eqref{eq290602}. All these functions 
are non-decreasing  and satisfy the inequality~$\vf_{n} \preceq g(v)$. Therefore, we only need to verify 
that $\av{\ff(\vf_n)}{[0,1]} \to B(w(v))$. 

We calculate the limit using~\eqref{eq280603}:
\begin{align*}
\lim_{n\to \infty}\av{\ff(\vf_n)}{[0,1]} =
& -\kappa_3(-\infty)\vt_1(-\infty)+f(v)-\int_{-\infty}^v f'(\tau)\exp\Big(\!-\!\int_\tau^v \lambda\;\Big)\,d\tau =
\\
=&-\kappa_3(-\infty)\vt_1(-\infty) + f(v) - \int_{-\infty}^v \kappa_3(\tau) \vt_1'(\tau) \, d\tau =
\\
=& f(v) - \kappa_3(v)(g_1(v)-w_1(v))+\int_{-\infty}^v \kappa_3'(\tau) \vt_1(\tau)\, d\tau. 
\end{align*}
Then, we use~\eqref{eq010706},~\eqref{eq010707}, and~\eqref{eq010708} to conclude that 
$\av{\ff(\vf_n)}{[0,1]} \to B(w(v))$.
\end{proof}

Propositions~\ref{st130704},~\ref{st130705}, and~\ref{st130706} prove Theorem~\ref{th120701} in the case~\textbf{(B)}.

\subsection{The case~\textbf{(C)}\,: $\int_{-\infty}^v \lambda = +\infty,$ $\vt_1(-\infty)= 0,$ $\vt_2(-\infty)\ne 0$}

\begin{Le}\label{lem150701}
If $\int_{-\infty}^v \lambda = +\infty,$  $\vt_1(-\infty) = 0,$ and $\vt_2(-\infty) \ne 0,$ then 
$\kappa_2(-\infty) = \kappa(-\infty) = -\infty$. In particular\textup, $g_2'<0$ in a neighborhood of $-\infty$ 
and $\vt_2(-\infty)<0$.
\end{Le}

\begin{proof} 
We easily obtain that $\kappa(-\infty) = \frac{\vt_2(-\infty)}{\vt_1(-\infty)}$ is infinite. 
From~\eqref{eq010705} and monotonicity of~$\kappa_2$ we conclude that $\kappa_2(-\infty) = \kappa(-\infty)$ is infinite as well. 
Since the function $\kappa_2$ is increasing, both these limits are $-\infty$. Then, $g_2' = g_1'\kappa_2 <0$ 
in a neighborhood of $-\infty$ and from~\eqref{eq050701} we obtain that $\vt_2(-\infty)<0$.
\end{proof}

\begin{St}\label{st130707}
If $\int_{-\infty}^v \lambda = +\infty,$ $\vt_1(-\infty) = 0,$ $\vt_2(-\infty)\ne 0,$ and 
$\tors(-\infty)=+\infty,$
then $\boldsymbol{B}(x;\,f) = +\infty$ for any $x \in \Omega\setminus \dfi\Omega$.  
\end{St}

\begin{proof}
It suffices to prove that $\boldsymbol{B}(w(v))=+\infty$. We consider the sequence $\vf_n$ defined in~\eqref{eq290602}, they are test 
functions for~$w(v)$, see Remark~\ref{rem010701}. 
Using~\eqref{eq010703} we rewrite~\eqref{eq300601} in the form
\begin{align}
\notag
\av{\ff(\vf_n)}{[0,1]}=& -\frac{f(v_n^+)-f(v_n^-)}{g_2(v_n^+)-g_2(v_n^-)}\vt_2(v_n^+)+f(v)-
\int_{v_n^+}^v f'(u)\exp\Big(\!-\!\int_u^v\lambda\;\Big)\,du\,=
\\
\label{eq110701}
=& -\frac{f'(v_n)}{g_2'(v_n)}\vt_2(v_n^+)+f(v)-
\int_{v_n^+}^v f'(u)\exp\Big(\!-\!\int_u^v\lambda\;\Big)\,du\,,
\end{align}
for some $v_n \in (v_n^-, v_n^+)$. Since $\vt_2(-\infty)<0$ and $g_2'<0$ in a neighborhood of $-\infty$ 
(see Lemma~\ref{lem150701}) it suffices to prove that $\frac{f'}{g_2'}(-\infty) = +\infty$. This follows 
from the identity $\frac{f'}{g_2'}  = \frac{\kappa_3}{\kappa_2}$ and the L'H\^opital rule: 
$\kappa_2(-\infty)=-\infty$, and $\frac{\kappa_3'}{\kappa_2'}(-\infty)= \tors(-\infty) = +\infty$. 
We see that $\av{\ff(\vf_n)}{[0,1]} \to +\infty$.
\end{proof}

\begin{St}\label{st130708}
If $\int_{-\infty}^v \lambda = +\infty,$ $\vt_1(-\infty)=0,$ $\vt_2(-\infty) \ne 0,$ and there is 
a standard candidate $B$ on $\Rt(-\infty, v),$ then $\tors(-\infty) > -\infty$. 
\end{St}

\begin{proof}
Assume the contrary, let $\tors(-\infty) = -\infty$. Then $\tors'>0$ in a neighborhood of $-\infty$. 
By~\eqref{eq120702} we see that $\beta_2'>0$, which is impossible for the standard candidate~$B$, see Definition~\ref{211101}.
\end{proof}

\begin{St}\label{st130709}
Let $\int_{-\infty}^v \lambda = +\infty,$ $\vt_1(-\infty)=0,$ $\vt_2(-\infty) \ne 0,$ and let $\tors(-\infty)$ 
be finite. If~$B$ is the standard candidate on~$\Rt(-\infty, v),$ then for every point~$x\in\Rt(-\infty,v)$ 
there exists a sequence of non-decreasing test functions~$\vf_{n,x}$ being an optimizing sequence for~$B$ at $x$\textup, 
moreover\textup,~$\vf_{n,x} \preceq g(v)$.
\end{St} 

\begin{proof}
It suffices to prove the claim for the case $x = w(v) = \wr(v)$. In what follows we omit the 
subscript $x$ in our notation. We use the sequence of functions $\vf_n$ defined by~\eqref{eq290602}. All these functions 
are non-decreasing  and satisfy the inequality~$\vf_{n} \preceq g(v)$. Therefore, it suffices to verify 
that $\av{\ff(\vf_n)}{[0,1]} \to B(w(v))$. 

First, we note that $\frac{f'}{g_2'} = \frac{\kappa_3}{\kappa_2}$ and $\kappa_2(-\infty)=-\infty$ by 
Lemma~\ref{lem150701}, therefore by the L'H\^opital rule, 
$$
\frac{f'}{g_2'} (-\infty) = \frac{\kappa_3'}{\kappa_2'}(-\infty) = \tors(-\infty).
$$

We calculate  the limit using~\eqref{eq110701}:
\begin{align*}
\lim_{n\to \infty}\av{\ff(\vf_n)}{[0,1]} 
=& -\tors(-\infty)\vt_2(-\infty)+f(v)-\int_{-\infty}^v f'(\tau) \exp\Big(\!-\!\int_\tau^v \lambda\;\Big)\,d\tau=
\\
=& -\tors(-\infty)\vt_2(-\infty)+f(v)-\int_{-\infty}^v \kappa_3(\tau) \vt_1'(\tau)\,d\tau =
\\
=& -\tors(-\infty)\vt_2(-\infty)+f(v)-\kappa_3(v)(g_1(v)-w_1(v))+\int_{-\infty}^v\kappa_3'(\tau)\vt_1(\tau)\,d\tau;
\end{align*}
in the latter equality we integrate by parts using Lemma~\ref{Lem111001} with $Y = \kappa_3$ and $\vt=\vt_1$.
Then, we use~\eqref{eq010708} to rewrite this formula in the form
\begin{align*}
\lim_{n\to \infty}\av{\ff(\vf_n)}{[0,1]}
=& f(v) + \big(w_1(v)-g_1(v)\big)\kappa_3(v)+ \big(w_1(v)-g_1(v)\big)\big(\kappa(v)-\kappa_2(v)\big) \times
\\
\times & \left(\frac{-\tors(-\infty)\vt_2(-\infty)}{\big(w_1(v)-g_1(v)\big)\big(\kappa(v)-\kappa_2(v)\big)}+ 
\int_{-\infty}^v \exp\Big(\!-\!\int_{\tau}^v\! \frac{\kappa_2'}{\kappa_2-\kappa} \Big) 
\frac{\kappa_3'(\tau)}{\kappa_2(\tau)-\kappa(\tau)}\,d\tau\right).
\end{align*}
We need to verify that the expression in the large parentheses coincides with $\beta_2$ given by~\eqref{eq120701}. We have the following formula for $\beta_2$:
$$
\tors(v)-\!\int_{-\infty}^v \!\exp\Big(\!-\!\int_{\tau}^v\!\frac{\kappa_2'}{\kappa_2-\kappa}\,\Big)\tors'(\tau)\,d\tau = 
\tors(-\infty)\exp\Big(\!-\!\int_{-\infty}^v\!\frac{\kappa_2'}{\kappa_2-\kappa}\,\Big) + 
\!\int_{-\infty}^v \!\exp\Big(\!-\!\int_{\tau}^v\!\frac{\kappa_2'}{\kappa_2-\kappa}\,\Big)
\frac{\kappa_3'(\tau)}{\kappa_2(\tau)-\kappa(\tau)}\,d\tau.
$$
Therefore it suffices to verify that
\begin{align}
\label{eq110704}
-\vt_2(-\infty) &\buildrel{?}\over= \big(g_1(v)-w_1(v)\big)(\kappa_2(v)-\kappa(v)\big) 
\exp\Big(\!-\!\!\int_{-\infty}^v\!\frac{\kappa_2'}{\kappa_2-\kappa}\,\Big)=
\\
\notag
&= \big(g_1(v)-w_1(v)\big)(\kappa_2(v)-\kappa(v)\big)  \lim_{u \to -\infty} 
\exp\Big(\!-\!\int_{u}^v\!\frac{\kappa_2'}{\kappa_2-\kappa}\,\Big)\buildrel{\eqref{eq110702}}\over=
\\
\notag&=\lim_{u\to -\infty} \big(\kappa_2(u)-\kappa(u)\big)\vt_1(u) 
= \lim_{u\to -\infty} \Big(\frac{\kappa_2(u)}{\kappa(u)}-1\Big)\vt_2(u),
\end{align}
where all the limits exist. Since the limit $\vt_2(-\infty)$ is finite and nonzero, it suffices to prove 
that $\frac{\kappa_2}{\kappa}(-\infty) =0$; note that this limit exists. 
We rewrite $\kappa_2$ and $\kappa$ in terms of $\vt_1$ and $\vt_2$: 
$$
\kappa = \frac{\vt_2}{\vt_1}, \qquad \kappa_2 = \frac{g_2'}{g_1'} = \frac{\vt_2'}{\vt_1'},
$$
therefore
\eq{eq110703}{
\frac{\kappa_2}{\kappa}(-\infty) = \frac{(\log|\vt_2|)'}{(\log\vt_1)'}(-\infty).
}
Since $\log \vt_1(-\infty) = -\infty$, and the limit~\eqref{eq110703} exists, we finish the proof 
by the L'H\^opital rule:
$$
0=\frac{\log|\vt_2|}{\log\vt_1}(-\infty)=\frac{(\log|\vt_2|)'}{(\log\vt_1)'}(-\infty)=\frac{\kappa_2}{\kappa}(-\infty).
$$
\end{proof}

\begin{Rem}
As opposed to the cases~\textbf{(A)} and~\textbf{(B)}, in the case \textbf{(C)} we have 
$\int_{-\infty}^v \frac{\kappa_2'}{\kappa_2-\kappa}<\infty$, this follows from~\eqref{eq110704}.
\end{Rem}

Propositions~\ref{st130707},~\ref{st130708}, and~\ref{st130709} prove Theorem~\ref{th120701} 
in the case~\textbf{(C)}. Thus, we complete the proof of the theorem.

\subsection{Behavior on $+\infty$}

We formulate the analog of Theorem~\ref{th120701} for the behavior on $+\infty$ that generalizes 
Proposition~\ref{OptimizerLeftTangentsInfty} for the case when Condition~\ref{eq111002} fails.

\begin{Th}\label{th140701}
Let~$B$ be the standard candidate on~$\Lt(v,+\infty)$. Then\textup, for every point~$x \in \Lt(v,+\infty)$ 
there exists a sequence of non-increasing test functions~$\vf_{n,x}$ being an optimizing sequence for~$B$ at $x$\textup, 
moreover\textup,~$g(v) \preceq \vf_{n,x}$.
\end{Th}

Since the situation is absolutely symmetric, we omit the proof of this theorem.

\subsection{When $\boldsymbol{B}=+\infty$?}\label{sec100101}

We collect the conditions on the behavior of the function $f$ at infinities that guarantee 
$\boldsymbol{B}(x;\,f) = +\infty$ for any $x \in \Omega\setminus \dfi\Omega$.

Propositions~\ref{st130702}, \ref{st130704}, \ref{st130707}, and~\ref{st140701} describe the conditions on $-\infty$: 

\bigskip
\centerline{
\begin{tabular}{|p{9cm}|p{6cm}|}
\hline
\hfil Condition on $\Omega\rule[-10pt]{0pt}{25pt}$ \hfil & \hfil Condition on $f$ \hfil 
\\
\hline
\hfil \textbf{(A)}$\colon \rule[-10pt]{0pt}{25pt} \exp\big(-\int_{-\infty}^v \lamr\big) \ne 0$ \hfil& 
\hfil $f(-\infty) = +\infty$ \hfil 
\\
\hline
\hfil \textbf{(B)}$\colon \rule[-10pt]{0pt}{25pt} \exp\big(-\int_{-\infty}^v \lamr\big) = 0$, $\vt_1(-\infty)\ne 0$ 
\hfil& \hfil $\kappa_3(-\infty)=-\infty$ \hfil
\\
\hline
\hfil \textbf{(C)}$\colon \rule[-10pt]{0pt}{25pt} \exp\big(-\int_{-\infty}^v \lamr\big) = 0$, 
$\vt_1(-\infty)= 0$, $\vt_2(-\infty)\ne 0$ \hfil& \hfil $\tors(-\infty)=+\infty$ \hfil
\\
\hline
\hfil \textbf{\eqref{eq111001}}$\colon \rule[-10pt]{0pt}{25pt} \exp\big(-\int_{-\infty}^v \lamr\big) = 0$, 
$\vt_1(-\infty)= 0$, $\vt_2(-\infty)= 0$ \hfil&
\hfil $\int_{-\infty}^v f'(\tau) \exp\Big(-\int_{\tau}^v \lamr\Big)\ d\tau = - \infty$ \hfil
\\
\hline
\end{tabular}
}
\bigskip

Here $v \in\mathbb{R}$ is arbitrary, the conditions do not depend on $v$.
We underline that Condition~\ref{reg} is of crucial importance for these results. If  Condition~\ref{reg}  is violated, the problem of characterizing the functions $f$ for which $\boldsymbol{B}$ is infinite, is more involved, see Lemma~6.1.4 in~\cite{ISVZ2018}. 

\begin{Rem}
In case~\textbf{(A)} the conditions $f(-\infty) = +\infty$ and 
$\int_{-\infty}^v f'(\tau) \exp\Big(-\int_{\tau}^v \lamr\Big)\ d\tau = - \infty$ are equivalent. 
In cases~\textbf{(B)} and~\textbf{(C)} the conditions on $f$ look differently, but in the case~\textbf{(C)} 
the condition $\tors(-\infty)=+\infty$ is equivalent to $\frac{f'}{g_2'}(-\infty) = +\infty$, which is symmetric 
to the condition $\frac{f'}{g_1'}(-\infty) = \kappa_3(-\infty)=-\infty$ that appears in case~\textbf{(B)}.
\end{Rem}

As usual we omit the symmetrical statements for the case of $+\infty$ but collect the corresponding results:

\bigskip

\centerline{
\begin{tabular}{|p{9cm}|p{6cm}|}
\hline
\hfil Condition on $\Omega\rule[-10pt]{0pt}{25pt}$ \hfil & \hfil Condition on $f$ \hfil 
\\
\hline
\hfil \textbf{(A)}$\colon \rule[-10pt]{0pt}{25pt} \exp\big(\int_v^{+\infty} \laml\big) \ne 0$ \hfil& 
\hfil $f(+\infty) = +\infty$ \hfil \\
\hline
\hfil \textbf{(B)}$\colon \rule[-10pt]{0pt}{25pt} \exp\big(\int_v^{+\infty} \laml\big) = 0$, 
${\vtl_1}(+\infty)\ne 0$ \hfil& \hfil $\kappa_3(+\infty)=+\infty$ \hfil
\\
\hline
\hfil \textbf{(C)}$\colon \rule[-10pt]{0pt}{25pt} \exp\big(\int_v^{+\infty} \laml\big) = 0$, 
$\vtl_1(+\infty)= 0$, $\vtl_2(+\infty)\ne 0$ \hfil& \hfil $\tors(+\infty)=+\infty$ \hfil
\\
\hline
\hfil \textbf{\eqref{eq111002}}$\colon \rule[-10pt]{0pt}{25pt} \exp\big(\int_v^{+\infty} \laml\big) = 0$, 
$\vtl_1(+\infty)= 0$, $\vtl_2(+\infty)= 0$ \hfil& \hfil 
$\int_v^{+\infty} f'(\tau) \exp\Big(\int_v^{\tau} \laml\Big)\ d\tau = + \infty$ \hfil
\\
\hline
\end{tabular}
}

\bigskip

Here 
$$
\vtl(\tau) = \big(g(\tau) - \wl(\tau)\big)  \exp\Big(\int_v^{\tau} \laml\Big), \qquad \tau \in \mathbb{R}, 
$$
where $v$ is a fixed parameter, $v \in \mathbb{R}$.

\begin{Rem}\label{Rem150701}
The function~$\boldsymbol{B}(\;\cdot\;;\,f)$ is infinite on $\Omega\setminus \dfi\Omega$ if the pair 
$(\Omega, f)$ satisfies one of the conditions listed in the tables presented in this subsection. 
If $(\Omega, f)$ does not satisfy all the conditions, then the function~$\boldsymbol{B}(\;\cdot\;;\,f)$ 
is finite  on $\Omega$ and can be constructed using the procedure described in this paper.
\end{Rem}
}

\subsection{Examples of domains}\label{sec280301}
In Section~\ref{cond} we have seen several examples of the domains that satisfy conditions~\eqref{eq111001} and~\eqref{eq111002}. Now we present the domains that satisfy conditions~\textbf{(A)}, \textbf{(B)}, or~\textbf{(C)}. We provide examples of the required behavior when $t$ tends to $-\infty$. The case of $+\infty$ is obviously symmetric.
\begin{itemize}
\item{\textbf{(A)}} 
$\displaystyle\qquad
g(t) = \big(e^t + e^{t/2}, e^{2t} + 2e^{3t/2}\big),\qquad  w(t) = \big(e^t,e^{2t}\big),\qquad \lambda(t) = e^{t/2}.
$
\item{\textbf{(B)}} 
$\displaystyle\qquad
g(t) = \big(t, (1+t+t^2)e^{-t^2}\big),\qquad  w(t) = \big(-t^2,e^{-t^2}\big),\qquad \lambda(t) = -\frac{2}{t+1}.
$
\item{\textbf{(C)}} 
$\displaystyle\qquad
g(t) = \Big(t+\frac{t^2}{1+t^2}, \frac{e^{-t}}{1+t^2}\Big),\qquad  w(t) = \big(t,e^{-t}\big),\qquad \lambda(t) = 1+\frac{1}{t^2}.
$
\end{itemize}